\documentclass{amsart}
\usepackage{amsmath,amssymb}
\usepackage{mathrsfs}
\newtheorem{df}{Definition}
\newtheorem{prp}{Proposition}
\newtheorem{thm}{Theorem}
\newtheorem{lem}{Lemma}
\newtheorem{rem}{Remark}
\newtheorem{ex}{Example}
\newcommand{\relmiddle}[1]{\mathrel{}\middle#1\mathrel{}}
\numberwithin{equation}{section}
\numberwithin{figure}{section}
\numberwithin{prp}{section}
\numberwithin{df}{section}
\numberwithin{thm}{section}
\numberwithin{lem}{section}
\numberwithin{rem}{section}
\numberwithin{ex}{section}

\begin{document}

\title[Crystal interpretation of a formula on the branching rule]
{Crystal interpretation of a formula on the branching rule of types $B_{n}$, $C_{n}$, and $D_{n}$}
\author{Toya Hiroshima}
\address{Department of Pure and Applied Mathematics,
Graduate School of Information Science and Technology,
Osaka University,
1-5 Yamadaoka, Suita, Osaka 565-0871, Japan}
\email{t-hiroshima@ist.osaka-u.ac.jp}
\date{}

\begin{abstract}
The branching coefficients of the tensor product of finite-dimensional irreducible $U_{q}(\mathfrak{g})$-modules, where $\mathfrak{g}$ is $\mathfrak{so}(2n+1,\mathbb{C})$ ($B_{n}$-type), $\mathfrak{sp}(2n,\mathbb{C})$ ($C_{n}$-type), and $\mathfrak{so}(2n,\mathbb{C})$ ($D_{n}$-type), are expressed in terms of Littlewood-Richardson (LR) coefficients in the stable region.
We give an interpretation of this relation by Kashiwara's crystal theory 
by providing an explicit surjection from the LR crystal of type $C_{n}$ to the disjoint union of Cartesian product of LR crystals of $A_{n-1}$-type and by proving that LR crystals of types $B_{n}$ and $D_{n}$ are identical to the corresponding LR crystal of type $C_{n}$ in the stable region.
\end{abstract}


\maketitle

\section{Introduction}

The generalized Littlewood-Richardson (LR) rule in Kashiwara's crystal theory~\cite{Kas1,Kas2} 
is one of the most remarkable applications of crystals to the representation theory of quantum groups.
Let $U_{q}(\mathfrak{g})$ be the quantum group of classical Lie algebra $\mathfrak{g}$ 
and let $V_{q}(\Tilde{\lambda})$ be the finite-dimensional irreducible $U_{q}(\mathfrak{g})$-module of a dominant integral weight $\Tilde{\lambda}$, where $\mathfrak{g}$ is $\mathfrak{so}(2n+1,\mathbb{C})$ ($B_{n}$-type), $\mathfrak{sp}(2n,\mathbb{C})$ ($C_{n}$-type), and $\mathfrak{so}(2n,\mathbb{C})$ ($D_{n}$-type).
Let $\lambda$ be the Young diagram (partition) corresponding to $\Tilde{\lambda}$.
The generalized LR rule asserts that the multiplicity of $V_{q}(\Tilde{\lambda})$ in the tensor product 
$V_{q}(\Tilde{\mu})\otimes V_{q}(\Tilde{\nu})$ is given by the cardinality of the LR crystal.
The multiplicity $d_{\mu\nu}^{\lambda}$ 
is expressed by the celebrated LR coefficients as~\cite{Kin,Koi}
\begin{equation} \label{eq:KK1}
d_{\mu\nu}^{\lambda}=
{\textstyle\sum\limits_{\xi,\zeta,\eta\in\mathcal{P}_{n}}}
c_{\xi\zeta}^{\lambda}c_{\zeta\eta}^{\mu}c_{\eta\xi}^{\nu}
\end{equation}
in the stable region, i.e., $l(\mu)+l(\nu)\leq n$, 
where $l(\lambda)$ denotes the length of $\lambda$ and $\mathcal{P}_{n}$ denotes the set of all Young diagrams with at most $n$ rows.
The LR coefficient itself is also given by the cardinality of the LR crystal of type $A$.

In this paper, we give an interpretation of Eq.~\eqref{eq:KK1} in terms of crystals. 
More precisely, we construct an explicit surjection from the LR crystal of $C_{n}$-type whose cardinality is the left-hand side of Eq.~\eqref{eq:KK1} to the disjoint union of the Cartesian product of LR crystals of $A_{n-1}$-type corresponding to 
$\sum_{\xi,\zeta,\eta\in\mathcal{P}_{n}} c_{\xi\zeta}^{\lambda}c_{\zeta\eta}^{\mu}$, 
where the cardinality of the kernel of the surjection gives the missing $c_{\eta\xi}^{\nu}$.
We also show that LR crystals of types $B_{n}$ and $D_{n}$ are identical to the corresponding LR crystal of type $C_{n}$ in the stable region, which provides the crystal interpretation of Eq.~\eqref{eq:KK1} in $B_{n}$ and $C_{n}$ cases.
In the crystal theory, the LR coefficient is interpreted as the cardinality of the LR crystal.
Thus, the formulas are not in the final form from our point of view and the formulas should be understood as a shadow of the underlying set-theoretical bijections defined for LR crystals.
In this spirit, Kwon~\cite{Kwo3} studied the branching rule of classical group by his spinor model~\cite{Kwo1,Kwo2} which is a combinatorial model of classical crystals. 
Our method is different and we have a surjective map from the LR crystal of type $B_{n}$, $C_{n}$, and $D_{n}$ to the disjoint union of the products of two LR crystals of types $A$ such that each fiber gives the third LR crystal of type $A$.

This paper is organized as follows.
Section~\ref{sec:SP} is devoted to the background on crystals that we need in the sequel, 
which includes the axiomatic definition of crystals, the construction of crystals of $C_{n}$-type, 
and LR crystals of type $C_{n}$.
In Section~\ref{sec:column}, we describe the properties of single-column tableaux of $C_{n}$-type 
($C_{n}$-columns), which includes the summary of known facts as well as newly obtained results.
Section~\ref{sec:main1} presents the main theorem on $C_{n}$ case (Theorem~\ref{thm:main1}), 
which involves the maps on tableaux of $C_{n}$-type constructed based on the operations on $C_{n}$-columns.
This result is divided into two propositions (Proposition~\ref{prp:main11} and Proposition~\ref{prp:main12}), 
which are proven in Section~\ref{sec:main11} and Section~\ref{sec:main12}.
In Section~\ref{sec:Phi} and Section~\ref{sec:Psi}, the properties of maps introduced in Section~\ref{sec:main1} are investigated.
In Section~\ref{sec:main2}, we describe LR crystals of types $B_{n}$ and $D_{n}$ and prove that they are identical to the corresponding LR crystal of type $C_{n}$ in the stable region (Theorem~\ref{thm:main_B} and Theorem~\ref{thm:main_D}).

\section{Crystals of $C_{n}$-type} \label{sec:SP}

\subsection{Axioms of crystals}

Let us recall the axiomatic definition of a crystal~\cite{HK}.
Let $\mathfrak{g}$ be a symmetrizable Kac-Moody algebra with 
$P$ the weight lattice, 
$I$ the index set for the vertices of the Dynkin diagram of $\mathfrak{g}$, 
$A=(a_{ij})_{i,j\in I}$ the Cartan matrix,  
$\left\{  \alpha_{i}\in P \relmiddle| i\in I\right\}  $ the set of simple roots,  
$\left\{  \alpha_{i}^{\vee}\in P^{\ast} \relmiddle| i\in I\right\}  $ the set of simple coroots, and
$\left\langle \alpha_{i}^{\vee},\alpha_{j}\right\rangle =a_{ij}\;(i,j\in I)$.
Let $U_{q}(\mathfrak{g})$ be the quantized universal enveloping algebra or quantum group of $\mathfrak{g}$.
A $U_{q}(\mathfrak{g})$-crystal is defined as follows.

\begin{df} \label{df:crystal1}
A set $\mathcal{B}$ together with the maps
$\mathrm{wt}:\mathcal{B}\rightarrow P$ and 
$\Tilde{e_{i}},\Tilde{f_{i}}:\mathcal{B}\rightarrow \mathcal{B}\sqcup\{0\}$
is called a (semiregular) $U_{q}(\mathfrak{g})$-crystal if the following properties are satisfied $(i\in I)$:
when we define 
\[
\varepsilon_{i}(b)=\max\left\{  k\geq0 \relmiddle| \Tilde{e_{i}}^{k}b\in \mathcal{B}\right\},
\]
and
\[
\varphi_{i}(b)=\max\left\{  k\geq0 \relmiddle| \Tilde{f_{i}}^{k}b\in \mathcal{B}\right\},
\]
for $b\in \mathcal{B}$, then
\begin{itemize}
\item[(1)]
$\varepsilon_{i},\varphi_{i}:\mathcal{B}\rightarrow\mathbb{Z}_{\geq0}$ 
and 
$\varphi_{i}(b)=\varepsilon_{i}(b)+
\left\langle \alpha_{i}^{\vee},\mathrm{wt}(b)\right\rangle $,
\item[(2)]
if $\Tilde{e_{i}}b\neq0$, then 
$\mathrm{wt}(\Tilde{e_{i}}b)=\mathrm{wt}(b)+\alpha_{i}$, 
$\varepsilon_{i}(\Tilde{e_{i}}b)=\varepsilon_{i}(b)-1$, and 
$\varphi_{i}(\Tilde{e_{i}}b)=\varphi_{i}(b)+1$,
\item[(3)]
if $\Tilde{f_{i}}b\neq0$, then 
$\mathrm{wt}(\Tilde{f_{i}}b)=\mathrm{wt}(b)-\alpha_{i}$, 
$\varepsilon_{i}(\Tilde{f_{i}}b)=\varepsilon_{i}(b)+1$, and 
$\varphi_{i}(\Tilde{f_{i}}b)=\varphi_{i}(b)-1$,
\item[(4)]
for $b,b^{\prime}\in \mathcal{B}$, $\Tilde{f_{i}}b=b^{\prime}\Longleftrightarrow\Tilde{e_{i}}b^{\prime}=b$.
\end{itemize}
\end{df}
The maps $\Tilde{e_{i}}$ and $\Tilde{f_{i}}$ are called Kashiwara operators ($i\in I$) 
and $\mathrm{wt}(b)$ is called the weight of $b$.
A crystal $\mathcal{B}$ can be viewed as an oriented colored graph with colors $i\in I$ 
when we define $b\stackrel{i}{\longrightarrow} b^{\prime}$ if 
$\Tilde{f_{i}}b=b^{\prime}\; (b,b^{\prime}\in \mathcal{B})$.
This graph is called a crystal graph.

\begin{df}[tensor product rule] \label{df:tensor}
Let $\mathcal{B}_{1}$ and $\mathcal{B}_{2}$ be crystals.
The tensor product $\mathcal{B}_{1}\otimes \mathcal{B}_{2}$ is defined to be 
the set 
$\mathcal{B}_{1}\times \mathcal{B}_{2} = 
\left\{  b_{1}\otimes b_{2}\right\vert \left.  b_{1}
\in \mathcal{B}_{1},b_{2}\in \mathcal{B}_{2}\right\}  $
whose crystal structure is defined by
\begin{itemize}
\item[(1)]
$\mathrm{wt}(b_{1}\otimes b_{2})=\mathrm{wt}(b_{1})+\mathrm{wt}(b_{2})$,
\item[(2)]
$\varepsilon_{i}(b_{1}\otimes b_{2})=\max\left\{  \varepsilon_{i}(b_{1}),\varepsilon
_{i}(b_{2})-\left\langle \alpha_{i}^{\vee},\mathrm{wt}(b_{1})\right\rangle
\right\}$,
\item[(3)]
$\varphi_{i}(b_{1}\otimes b_{2})=\max\left\{  \varphi_{i}(b_{1})+\left\langle
\alpha_{i}^{\vee},\mathrm{wt}(b_{2})\right\rangle ,\varphi_{i}(b_{2})\right\}$,
\item[(4)]
$\Tilde{e_{i}}(b_{1}\otimes b_{2})=
\begin{cases}
\Tilde{e_{i}}b_{1}\otimes b_{2} & (\varphi_{i}(b_{1})\geq\varepsilon_{i}(b_{2})),\\
b_{1}\otimes\Tilde{e_{i}}b_{2} & (\varphi_{i}(b_{1})<\varepsilon_{i}(b_{2})),
\end{cases}
$
\item[(5)]
$\Tilde{f_{i}}(b_{1}\otimes b_{2})=
\begin{cases}
\Tilde{f_{i}}b_{1}\otimes b_{2} & (\varphi_{i}(b_{1})>\varepsilon_{i}(b_{2})),\\
b_{1}\otimes\Tilde{f_{i}}b_{2} & (\varphi_{i}(b_{1})\leq\varepsilon_{i}(b_{2})).
\end{cases}
$
\end{itemize}
\end{df}

\begin{df}
Let $\mathcal{B}_{1}$ and $\mathcal{B}_{2}$ be crystals.
A \emph{crystal morphism} 
$\Psi:\mathcal{B}_{1}\rightarrow \mathcal{B}_{2}$
is a map 
$\Psi:\mathcal{B}_{1}\sqcup\{0\}\rightarrow \mathcal{B}_{2}\sqcup\{0\}$
such that
\begin{itemize}
\item[(1)]
$\Psi(0)=0$,

\item[(2)]
if $b\in \mathcal{B}_{1}$ and $\Psi(b)\in \mathcal{B}_{2}$, then 
$\mathrm{wt}\left(  \Psi(b)\right)  =\mathrm{wt}(b)$, 
$\varepsilon_{i}\left(  \Psi(b)\right)  =\varepsilon_{i}(b)$ and 
$\varphi_{i}\left(  \Psi(b)\right)  =\varphi_{i}(b)$ 
$(\forall i\in I)$.

\item[(3)]
if $b,b^{\prime}\in \mathcal{B}_{1}$, 
$\Psi(b),\Psi(b^{\prime})\in \mathcal{B}_{2}$, 
and $\Tilde{f_{i}}b=b^{\prime}$, then 
$\Tilde{f_{i}}\Psi(b)=\Psi(b^{\prime})$ and
$\Psi(b)=\Tilde{e_{i}}\Psi(b^{\prime})$ $(\forall i\in I)$.
\end{itemize}
\end{df}

\begin{df} \label{df:morphism2}
\begin{itemize}
\item[(1)]
A crystal morphism $\Psi:\mathcal{B}_{1}\longrightarrow \mathcal{B}_{2}$ is called 
an \emph{embedding} 
if $\Psi$ induces an injective map from 
$\mathcal{B}_{1}\sqcup\{0\}$ to $\mathcal{B}_{2}\sqcup\{0\}$.
\item[(2)]
A crystal morphism $\Psi:\mathcal{B}_{1}\longrightarrow \mathcal{B}_{2}$ is called 
an \emph{isomorphism} 
if $\Psi$ is a bijection from 
$\mathcal{B}_{1}\sqcup\{0\}$ to $\mathcal{B}_{2}\sqcup\{0\}$.
\end{itemize}
\end{df}

\subsection{Crystals associated with finite-dimensional irreducible $U_{q}(\mathfrak{sp}_{2n})$-modules} \label{sec:crystal_C}

Let us describe crystals associated with finite-dimensional irreducible $U_{q}(\mathfrak{sp}_{2n})$-modules.
The symplectic Lie algebra $\mathfrak{sp}(2n,\mathbb{C})=\mathfrak{sp}_{2n}$ is the classical Lie algebra of $C_{n}$-type, where 
the simple roots are expressed as
\begin{align*}
\alpha_{i}  & =\epsilon_{i}-\epsilon_{i+1}\quad(i=1,2,\ldots,n-1),\\
\alpha_{n}  & =2\epsilon_{n},
\end{align*}
and fundamental weights as
\[
\omega_{i}=\epsilon_{1}+\epsilon_{2}+\cdots+\epsilon_{i} \quad (i=1,2,\ldots,n)
\]
with $\epsilon_{i}\in\mathbb{Z}^{n}$ 
being the standard $i$-th unit vector.

Let $\Tilde{\lambda}=a_{1}\omega_{1}+\cdots+a_{n}\omega_{n}$ ($a_{i}\in\mathbb{Z}_{\geq0}$) be a dominant integral weight.
Then $\Tilde{\lambda}$ can be written as 
$\Tilde{\lambda}=\lambda_{1}\epsilon_{1}+\cdots+\lambda_{n}\epsilon_{n}$, where
\begin{align*}
\lambda_{1}  & =a_{1}+a_{2}+\cdots+a_{n},\\
\lambda_{2}  & =a_{2}+\cdots+a_{n},\\
& \vdots\\
\lambda_{n}  & =a_{n}.
\end{align*}
Hence we can associate a Young diagram $\lambda=(\lambda_{1},\ldots,\lambda_{n})$ to $\Tilde{\lambda}$.


\begin{df}[\cite{HK,N}]
Let $\lambda$ be a Young diagram with at most $n$ rows.
A $C_{n}$-semistandard tableau of shape $\lambda$ is 
the semistandard tableau of shape $\lambda$ with letters (entries) taken from the set
\[\mathscr{C}_{n}:=\{1,2,\ldots,n,\Bar{n},\ldots,\Bar{1}\}\]
equipped with the total order
\[
1\prec2\prec\cdots\prec n\prec \Bar{n}\prec\cdots\prec \Bar{2}\prec \Bar{1}.
\]
\end{df}
We define 
$\mathscr{C}_{n}^{(+)}:=\{1,2,\ldots,n\}$ and 
$\mathscr{C}_{n}^{(-)}:=\{\Bar{1},\Bar{2},\ldots,\Bar{n}\}$.
In the sequel, a letter in $\mathscr{C}_{n}^{(+)}$ (resp. $\mathscr{C}_{n}^{(-)}$) is called a 
$\mathscr{C}_{n}^{(+)}$ (resp. $\mathscr{C}_{n}^{(-)}$)-letter and 
the usual order $<$ will be used within $\mathscr{C}_{n}^{(+)}$-letters instead of $\prec$.
We denote by $C_{n}\text{-}\mathrm{SST}(\lambda)$ 
the set of all  $C_{n}$-semistandard tableaux of shape $\lambda$ and set 
$C_{n}\text{-}\mathrm{SST}:=
\cup_{\lambda\in \mathcal{P}_{n}}C_{n}\text{-}\mathrm{SST}(\lambda)$.
We use the convention $C_{n}\text{-}\mathrm{SST}(\emptyset)=\{\emptyset \}$, 
where $\emptyset$ in the left-hand side is referred to as the Young diagram without any boxes.
For a $T\in C_{n}\text{-}\mathrm{SST}(\lambda)$, 
we define its weight to be
\begin{equation*}
\mathrm{wt}(T)=
{\textstyle\sum\limits_{i=1}^{n}}
(k_{i}-\overline{k_{i}})\epsilon_{i},
\end{equation*}
where $k_{i}$ (resp. $\overline{k_{i}}$) is the number of $i's$ (resp. $\Bar{i}'s$) appearing in $T$.

\begin{df}[\cite{HK,N}] \label{df:KN_C}
$T\in C_{n}\text{-}\mathrm{SST}(\lambda)$
is said to be \emph{KN-admissible} when the following conditions (C1) and (C2) are satisfied.
\begin{itemize}
\item[(C1)]
If $T$ has a column of the form
\setlength{\unitlength}{12pt}
\begin{center}
\begin{picture}(3,6)
\put(2,0){\line(0,1){6}}
\put(3,0){\line(0,1){6}}
\put(2,0){\line(1,0){1}}
\put(2,1){\line(1,0){1}}
\put(2,2){\line(1,0){1}}
\put(2,4){\line(1,0){1}}
\put(2,5){\line(1,0){1}}
\put(2,6){\line(1,0){1}}
\put(2,1){\makebox(1,1){$\Bar{b}$}}
\put(2,4){\makebox(1,1){$a$}}
\put(0,1){\makebox(2,1){$q\rightarrow$}}
\put(0,4){\makebox(2,1){$p\rightarrow$}}
\end{picture},
\end{center}
then we have
\[
(q-p)+\max (a,b) > N,
\]
where $N$ is the length of the column and 
$a (\in \mathscr{C}_{n}^{(+)})$ is at the $p$-th box from the top and 
$\Bar{b} (\in \mathscr{C}_{n}^{(-)})$ is at the $q$-th box from the top.

\item[(C2)]
If $T$ has a pair of adjacent columns having one of the following configurations with 
$p\leq q<r\leq s$, $a_{1}\leq b_{1}$, and $a_{2}\leq b_{2}$ ($a_{1}, b_{1}\in \mathscr{C}_{n}^{(+)}$):
\setlength{\unitlength}{12pt}
\begin{center}
\begin{picture}(8,5)
\put(4,0){\line(0,1){5}}
\put(7,0){\line(0,1){5}}
\put(3,4){\makebox(1,1){$a_{1}$}}
\put(4,0){\makebox(1,1){$\overline{a_{2}}$}}
\put(4,1){\makebox(1,1){$\overline{b_{2}}$}}
\put(4,3){\makebox(1,1){$b_{1}$}}
\put(6,1){\makebox(1,1){$\overline{b_{2}}$}}
\put(6,3){\makebox(1,1){$b_{1}$}}
\put(6,4){\makebox(1,1){$a_{1}$}}
\put(7,0){\makebox(1,1){$\overline{a_{2}}$}}
\put(5,0){\makebox(1,1)[l]{$,$}}
\put(0,0){\makebox(2,1){$s\rightarrow$}}
\put(0,1){\makebox(2,1){$r\rightarrow$}}
\put(0,3){\makebox(2,1){$q\rightarrow$}}
\put(0,4){\makebox(2,1){$p\rightarrow$}}
\end{picture},
\end{center}
then we have
\[
(q-p)+(s-r)<\max(b_{1},b_{2})-\min(a_{1},a_{2}).
\]
\end{itemize}
\end{df}
We denote by $C_{n}\text{-}\mathrm{SST}_{\mathrm{KN}}(\lambda)$ 
the set of all KN-admissible $C_{n}$-semistandard tableaux of shape $\lambda$ and set 
$C_{n}\text{-}\mathrm{SST}_{\mathrm{KN}}:=
\bigcup_{\lambda \in \mathcal{P}_{n}}C_{n}\text{-}\mathrm{SST}_{\mathrm{KN}}(\lambda)$.

\begin{rem}
Conditions (C1) and (C2) in Definition~\ref{df:KN_C} are equivalent to the following conditions (C1') and (C2'), respectively.
\begin{itemize}
\item[(C1')]
If $T$ has a column of the form
\setlength{\unitlength}{12pt}
\begin{center}
\begin{picture}(3,6)
\put(2,0){\line(0,1){6}}
\put(3,0){\line(0,1){6}}
\put(2,0){\line(1,0){1}}
\put(2,1){\line(1,0){1}}
\put(2,2){\line(1,0){1}}
\put(2,4){\line(1,0){1}}
\put(2,5){\line(1,0){1}}
\put(2,6){\line(1,0){1}}
\put(2,1){\makebox(1,1){$\Bar{a}$}}
\put(2,4){\makebox(1,1){$a$}}
\put(0,1){\makebox(2,1){$q\rightarrow$}}
\put(0,4){\makebox(2,1){$p\rightarrow$}}
\end{picture},
\end{center}
then we have
\[
(q-p)+a > N,
\]
where $N$ is the length of the column and 
$a (\in \mathscr{C}_{n}^{(+)})$ is at the $p$-th box from the top and 
$\Bar{b} (\in \mathscr{C}_{n}^{(-)})$ is at the $q$-th box from the top.

\item[(C2')]
If $T$ has a pair of adjacent columns having one of the following configurations with 
$p\leq q<r\leq s$ and $a\leq b$:
\setlength{\unitlength}{12pt}
\begin{center}
\begin{picture}(8,5)
\put(4,0){\line(0,1){5}}
\put(7,0){\line(0,1){5}}
\put(3,4){\makebox(1,1){$a$}}
\put(4,0){\makebox(1,1){$\overline{a}$}}
\put(4,1){\makebox(1,1){$\overline{b}$}}
\put(4,3){\makebox(1,1){$b$}}
\put(6,1){\makebox(1,1){$\overline{b}$}}
\put(6,3){\makebox(1,1){$b$}}
\put(6,4){\makebox(1,1){$a$}}
\put(7,0){\makebox(1,1){$\overline{a}$}}
\put(5,0){\makebox(1,1)[l]{$,$}}
\put(0,0){\makebox(2,1){$s\rightarrow$}}
\put(0,1){\makebox(2,1){$r\rightarrow$}}
\put(0,3){\makebox(2,1){$q\rightarrow$}}
\put(0,4){\makebox(2,1){$p\rightarrow$}}
\end{picture},
\end{center}
then we have
\[
(q-p)+(s-r)<b-a.
\]
\end{itemize}
\end{rem}

Now we can give the definition of a crystal $\mathcal{B}^{\mathfrak{sp}_{2n}}(\lambda)$ 
associated with the finite-dimensional irreducible $U_{q}(\mathfrak{sp}_{2n})$-module $V_{q}^{\mathfrak{sp}_{2n}}(\Tilde{\lambda})$ 
associated with a dominant integral weight $\Tilde{\lambda}$. 
As a set, the crystal $\mathcal{B}^{\mathfrak{sp}_{2n}}(\lambda)$ is 
$C_{n}\text{-}\mathrm{SST}_{\mathrm{KN}}(\lambda)$.
Kashiwara operators are determined by the following crystal graph of the vector representation $\mathbf{B}:=\mathcal{B}^{\mathfrak{sp}_{2n}}(\square)$ 
of the quantum group  $U_{q}(\mathfrak{sp}_{2n})$.
\setlength{\unitlength}{12pt}
\begin{center}
\begin{picture}(23,1.5)
\put(0,0){\framebox(1,1){$1$}}
\put(3,0){\framebox(1,1){$2$}}
\put(9.5,0){\framebox(1,1){$n$}}
\put(12.5,0){\framebox(1,1){$\Bar{n}$}}
\put(19,0){\framebox(1,1){$\Bar{2}$}}
\put(22,0){\framebox(1,1){$\Bar{1}$}}
\put(1,0.5){\vector(1,0){2}}
\put(4,0.5){\vector(1,0){2}}
\put(7.5,0.5){\vector(1,0){2}}
\put(10.5,0.5){\vector(1,0){2}}
\put(13.5,0.5){\vector(1,0){2}}
\put(17,0.5){\vector(1,0){2}}
\put(20,0.5){\vector(1,0){2}}
\put(6.5,0){\makebox(1,1){$\cdots$}}
\put(16,0){\makebox(1,1){$\cdots$}}
\put(1,0.5){\makebox(2,1){$1$}}
\put(4,0.5){\makebox(2,1){$2$}}
\put(7.5,0.5){\makebox(2,1){$n-1$}}
\put(10.5,0.5){\makebox(2,1){$n$}}
\put(13.5,0.5){\makebox(2,1){$n-1$}}
\put(17,0.5){\makebox(2,1){$2$}}
\put(20,0.5){\makebox(2,1){$1$}}
\end{picture},
\end{center}
where
$\mathrm{wt}\left(\framebox[12pt]{$i$\rule{0pt}{8pt}}\right)=\epsilon_{i}$ and 
$\mathrm{wt}\left(\framebox[12pt]{$\Bar{i}$\rule{0pt}{8pt}}\right)=-\epsilon_{i}$
$(i=1,2,\ldots,n)$.
Explicitly, for $i = 1,2,\ldots,n-1$,
\[
\Tilde{f_{i}}\,\framebox[12pt]{$j$} = 
\begin{cases}
\framebox[24pt]{$i+1$\rule{0pt}{8pt}} & (j=i), \\
\rule{0pt}{15pt}\framebox[12pt]{$\Bar{i}$} & (j=\overline{i+1}), \\
0 & (\text{otherwise}),
\end{cases}
\]
and
\[
\Tilde{f_{n}}\framebox[12pt]{$n$\rule{0pt}{8pt}} =\framebox[12pt]{$\Bar{n}$\rule{0pt}{8pt}}
\]
($\Tilde{e_{i}}$ is determined by these and Definition~\ref{df:crystal1}).
The crystal structure of $\mathcal{B}^{\mathfrak{sp}_{2n}}(\lambda)$ is realized by the embedding
$\Psi:\mathcal{B}^{\mathfrak{sp}_{2n}}(\lambda) \hookrightarrow \mathbf{B}^{\otimes\left\vert \lambda\right\vert }$
equipped with the tensor product rule (Definition~\ref{df:tensor}).
This embedding or reading is defined as follows.

\begin{df} \label{df:FE}
Suppose $T \in C_{n}\text{-}\mathrm{SST}_{\mathrm{KN}}(\lambda)$.
We read the entries in $T$ each column from the top to the bottom 
and from the rightmost column to the leftmost column.
Let the resulting sequence of entries be $m_{1},m_{2},\ldots,m_{N}$.
Then we define the following embedding. 
\[
\Psi:\mathcal{B}^{\mathfrak{sp}_{2n}}(\lambda)\hookrightarrow \mathbf{B}^{\otimes N}\quad
\left(T\longmapsto \framebox[15pt]{$m_{1}$\rule{0pt}{8pt}}\otimes 
\cdots \otimes \framebox[15pt]{$m_{N}$\rule{0pt}{8pt}}\right).
\]
\end{df}
This reading of $T$ in Definition~\ref{df:FE} is called the \emph{far-eastern reading} and is denoted by
\[
\mathrm{FE}(T)=\framebox[15pt]{$m_{1}$\rule{0pt}{8pt}}\otimes \cdots 
\otimes \framebox[15pt]{$m_{N}$\rule{0pt}{8pt}}.
\] 
Thanks to the KN admissible conditions ((C1) and (C2) in Definition~\ref{df:KN_C}), this reading is shown to be the embedding in the sense of Definition~\ref{df:morphism2}~\cite{HK}.

One of the most remarkable applications of crystals is the generalized LR rule described below.
Let us give a definition.

\begin{df} \label{df:addition}
Let $\lambda=(\lambda_{1},\lambda_{2},\ldots,\lambda_{n})$ be a Young diagram.
For a letter $i\in \mathscr{C}_{n}^{(+)}$ and a letter $\Bar{i}\in \mathscr{C}_{n}^{(-)}$, 
we define
\begin{equation*}
\lambda [i]:=(\lambda_{1},\ldots,\lambda_{i}+1,\ldots,\lambda_{n}),
\end{equation*}
and
\begin{equation*}
\lambda [\Bar{i}]:=(\lambda_{1},\ldots,\lambda_{i}-1,\ldots,\lambda_{n}).
\end{equation*}
In general, 
for a letter $m_{k}\in \mathscr{C}_{n}$\; $(k=1,2,\ldots,N)$, 
we define
\begin{equation*}
\lambda\lbrack m_{1},\ldots,m_{k}]:=\lambda\lbrack m_{1},\ldots,m_{k-1}][m_{k}]
\end{equation*}
$(\lambda\lbrack m_{0}]=\lambda)$, which is not necessarily a Young diagram.
If $\lambda\lbrack m_{1},\ldots,m_{k}]$ is a Young diagram for all $k=1,\ldots,N$, 
we say the sequence of letters $m_{1},m_{2}.\ldots,m_{N}$ is \emph{smooth} on $\lambda$ or 
$M:=\{ m_{1},m_{2},\ldots,m_{N} \}$ is smooth on $\lambda$, 
where $M$ is considered as the sequence of letters $m_{1},m_{2}.\ldots,m_{N}$.
If the sequence of letters $m_{1},m_{2}.\ldots,m_{N}$ comes from the far-eastern reading of a tableau $T$, 
we write $\lambda\lbrack \mathrm{FE}(T)]:=\lambda\lbrack m_{1},\ldots,m_{N}]$ and 
if such a sequence is smooth on $\lambda$, we say $\mathrm{FE}(T)$ is smooth on $\lambda$.
\end{df}

\begin{thm}[\cite{HK,KN,N}]
Let 
$\Tilde{\mu}$ and $\Tilde{\nu}$ be dominant integral weights, and  
$\mu$ and $\nu$ be the corresponding Young diagrams, respectively.
Then we have the following isomorphism:
\begin{equation} \label{eq:gLR1}
\mathcal{B}^{\mathfrak{sp}_{2n}}(\mu)\otimes\mathcal{B}^{\mathfrak{sp}_{2n}}(\nu)\simeq
{\textstyle\bigoplus\limits_{\substack{T\in \mathcal{B}^{\mathfrak{sp}_{2n}}(\nu)\\
\mathrm{FE}(T)=
\framebox[16pt]{$m_{1}$}\otimes
\cdots\otimes 
\framebox[16pt]{$m_{N}$}
}}}
\mathcal{B}^{\mathfrak{sp}_{2n}}\left( \mu\lbrack m_{1},m_{2},\ldots,m_{N}]\right),
\end{equation}
where $N=\left\vert \nu\right\vert$.
In the right-hand side of Eq.~\eqref{eq:gLR1}, we set 
$\mathcal{B}^{\mathfrak{sp}_{2n}}\left(  \mu\lbrack m_{1},\ldots,m_{N}]\right)  =\emptyset$ 
if the sequence of letters $m_{1},\ldots,m_{N}$ is not smooth on $\mu$.
\end{thm}
Let us denote by $d_{\mu\nu}^{\lambda}$ the multiplicity of $\mathcal{B}^{\mathfrak{sp}_{2n}}(\lambda)$ 
in the right-hand side of Eq.~\eqref{eq:gLR1}.
Then Eq.~\eqref{eq:gLR1} takes the form
\begin{equation} \label{eq:gLR2}
\mathcal{B}^{\mathfrak{sp}_{2n}}(\mu)\otimes\mathcal{B}^{\mathfrak{sp}_{2n}}(\nu)\simeq
{\textstyle\bigoplus\limits_{\lambda\in \mathcal{P}_{n}}}
\mathcal{B}^{\mathfrak{sp}_{2n}}(\lambda)^{\oplus d_{\mu\nu}^{\lambda}}
\quad(\mu,\nu\in\mathcal{P}_{n}).
\end{equation}
This corresponds to the decomposition of the tensor product of finite-dimensional irreducible 
$U_{q}(\mathfrak{sp}_{2n})$-modules $V_{q}^{\mathfrak{sp}_{2n}}(\Tilde{\mu})$ and $V_{q}^{\mathfrak{sp}_{2n}}(\Tilde{\nu})$.
\begin{equation} \label{eq:LRV1}
V_{q}^{\mathfrak{sp}_{2n}}(\Tilde{\mu})\otimes V_{q}^{\mathfrak{sp}_{2n}}(\Tilde{\nu})\simeq
{\textstyle\bigoplus\limits_{\lambda\in\mathcal{P}_{n}}}
V_{q}^{\mathfrak{sp}_{2n}}(\Tilde{\lambda})^{\oplus d_{\mu\nu}^{\lambda}}
\quad (\mu,\nu\in\mathcal{P}_{n}).
\end{equation}
Equation~\eqref{eq:gLR1} or \eqref{eq:gLR2} is called the generalized LR rule~\cite{HK,KN,N}.
It follows from Eqs.~\eqref{eq:gLR1} and \eqref{eq:gLR2} that the multiplicity $d_{\mu\nu}^{\lambda}$ 
is given by the cardinality of the following set
\begin{equation} \label{eq:LRcrystal2}
\mathbf{B}_{n}^{\mathfrak{sp}_{2n}}(\nu)_{\mu}^{\lambda}:=\left\{
T\in \mathcal{B}^{\mathfrak{sp}_{2n}}(\nu)
\relmiddle|
\begin{array}
[c]{c}
\text{$\mathrm{FE}(T)=
\framebox[16pt]{$m_{1}$\rule{0pt}{8pt}}\otimes 
\framebox[16pt]{$m_{2}$\rule{0pt}{8pt}}\otimes\cdots\otimes 
\framebox[16pt]{$m_{N}$\rule{0pt}{8pt}}$
($N=\left\vert \nu\right\vert$)
}\\
\text{is smooth on $\mu$ and $\mu \lbrack m_{1},\ldots,m_{N}]=\lambda$}
\end{array}
\right\},
\end{equation}
which is called the LR crystal of $C_{n}$-type.

It is established that 
the multiplicity $d_{\mu\nu}^{\lambda}$ can be expressed in terms of LR coefficients.
More precisely, we have
\begin{equation} \label{eq:KK2}
d_{\mu\nu}^{\lambda}=
{\textstyle\sum\limits_{\xi,\zeta,\eta\in\mathcal{P}_{n}}}
c_{\xi\zeta}^{\lambda}c_{\zeta\eta}^{\mu}c_{\eta\xi}^{\nu}
\end{equation}
in the stable region, i.e., $l(\mu)+l(\nu)\leq n$~\cite{Kin,Koi}.
The LR coefficient $c_{\mu\nu}^{\lambda}$ is also given by the cardinality of the set (Eq.~\eqref{eq:LRcrystal2}) with 
$\mathcal{B}^{\mathfrak{sp}_{2n}}(\lambda)$ being replaced by $\mathcal{B}^{\mathfrak{sl}_{n}}(\lambda)$ 
the crystal associated with the finite-dimensional irreducible $U_{q}(\mathfrak{sl}_{n})$-module 
$V_{q}^{\mathfrak{sl}_{n}}(\lambda)$~\cite{HK}.
This set is called the LR crystal of $A_{n-1}$-type.
Formally a crystal $\mathcal{B}^{\mathfrak{sl}_{n}}(\lambda)$ is obtained by eliminating all tableaux 
containing $\mathscr{C}_{n}^{(-)}$-letters from $\mathcal{B}^{\mathfrak{sp}_{2n}}(\lambda)$.
In this paper, we provide the interpretation of Eq.~\eqref{eq:KK2} in terms of crystals.
For that purpose, we will need the following definitions.

\begin{df} \label{df:LR-crystal3}
For Young diagrams $\lambda$, $\mu$, and $\nu$, we define
\[
\mathbf{B}^{(+)}_{n}(\nu)_{\mu}^{\lambda}:=\left\{
T\in \mathcal{B}^{\mathfrak{sp}_{2n}}(\nu)
\relmiddle|
\begin{array}
[c]{c}
\text{All entries in $T$ are $\mathscr{C}_{n}^{(+)}$-letters.}\\
\text{
$\mathrm{FE}(T)=
\framebox[15pt]{$i_{1}$\rule{0pt}{8pt}}\otimes 
\framebox[15pt]{$i_{2}$\rule{0pt}{8pt}}\otimes\cdots\otimes 
\framebox[15pt]{$i_{N}$\rule{0pt}{8pt}}$ 
($N=\left\vert \nu\right\vert$)
}\\
\text{is smooth on $\mu$ and $\mu\lbrack i_{1},\ldots,i_{N}]=\lambda$}
\end{array}
\right\},
\]
and
\[
\mathbf{B}^{(-)}_{n}(\nu)_{\mu}^{\lambda}:=\left\{
T\in \mathcal{B}^{\mathfrak{sp}_{2n}}(\nu)
\relmiddle|
\begin{array}
[c]{c}
\text{All entries in $T$ are $\mathscr{C}_{n}^{(-)}$-letters.}\\
\text{
$\mathrm{FE}(T)=
\framebox[15pt]{$\overline{i_{1}}$}\otimes 
\framebox[15pt]{$\overline{i_{2}}$}\otimes\cdots\otimes 
\framebox[15pt]{$\overline{i_{N}}$}$ 
($N=\left\vert \nu\right\vert$)
}\\
\text{is smooth on $\lambda$ and $\lambda\lbrack \overline{i_{1}},\ldots,\overline{i_{N}}]=\mu$}
\end{array}
\right\}.
\]
\end{df}
Note that the set $\mathbf{B}^{(+)}_{n}(\nu)_{\mu}^{\lambda}$ is identical with the LR crystal of type 
$A_{n-1}$ whose cardinality is the LR coefficient $c_{\mu\nu}^{\lambda}$.

\section{$C_{n}$-columns} \label{sec:column}

Let us call a $C_{n}$-semistandard tableau with shape $(1^{N})$ a $C_{n}$-column 
of length $N$.
We denote by $C_{n}\text{-}\mathrm{Col}(N)$ (=$C_{n}\text{-}\mathrm{SST}((1^{N}))$) 
the set of all $C_{n}$-columns of length $N$ and set 
$C_{n}\text{-}\mathrm{Col}:=\bigcup_{N\in \mathbb{Z}_{>0}}C_{n}\text{-}\mathrm{Col}(N)$.
In this section, we describe the properties of $C_{n}$-columns.

For a $C_{n}$-column
\setlength{\unitlength}{15pt}
\begin{center}
\begin{picture}(3.2,4)
\put(2,0){\framebox(1.2,1){$m_{N}$}}
\put(2,1){\framebox(1.2,2){$\vdots$}}
\put(2,3){\framebox(1.2,1){$m_{1}$}}
\put(0,1.5){\makebox(2,1){$C=$}}
\end{picture},
\end{center}
let us write $w(C)=m_{1}m_{2}\cdots m_{N}$ ($m_{i} \in \mathscr{C}_{n}, i=1,2,\ldots,N$).
A part of $C$ that consists of consecutive boxes is called a block.
A block of $C$ that consists of boxes from the $p$-th position to the $q$-th position is denoted by 
$\Delta C[p,q]$ ($p \leq q$).

\setlength{\unitlength}{15pt}
\begin{center}
\begin{picture}(6.25,5)
\put(2,0){\line(0,1){5}}
\put(3,0){\line(0,1){5}}
\put(2,1){\line(1,0){1}}
\put(2,2){\line(1,0){1}}
\put(2,3.2){\line(1,0){1}}
\put(2,4.2){\line(1,0){1}}
\put(0,1){\makebox(2,1){$q \rightarrow$}}
\put(0,3){\makebox(2,1){$p \rightarrow$}}
\put(2,1){\makebox(1,1){$m_{q}$}}
\put(2,2){\makebox(1,1.2){$\vdots$}}
\put(2,3){\makebox(1,1){$m_{p}$}}
\put(2.25,1){\makebox(4,3.2){
$\left.
\begin{array}
[c]{ccc}%
\\
\\
\\
\\
\end{array}
\right\}  \Delta C[p,q]$}}
\end{picture}.
\end{center}
If the two-column tableau $C_{1}C_{2}$ is semistandard, then we write $C_{1} \preceq C_{2}$, 
where $C_{i}$ is the $i$-th column ($i=1,2$).
Let us denote by $C_{n}\text{-}\mathrm{Col}_{\mathrm{KN}}(N)$ 
the set of all $C_{n}$-columns ($\in C_{n}\text{-}\mathrm{Col}(N)$) that are KN-admissible 
and set $C_{n}\text{-}\mathrm{Col}_{\mathrm{KN}}:=
\cup_{N\in \mathbb{Z}_{>0}}C_{n}\text{-}\mathrm{Col}_{\mathrm{KN}}(N)$.
The necessary and sufficient condition that $C \in C_{n}\text{-}\mathrm{Col}(N)$ 
be KN-admissible has been given by the first condition (C1) in Definition~\ref{df:KN_C}.
Yet another but equivalent condition is given by the following.

\begin{df} \label{df:split}
Suppose that $C\in C_{n}\text{-}\mathrm{Col}(N)$ such that 
$w(C)=i_{1}\cdots i_{a}\overline{j_{b}}\cdots \overline{j_{1}}$
where $N=a+b$, $i_{k}\in \mathscr{C}_{n}^{(+)\;}(k=1,2,\ldots,a)$, and 
$\overline{j_{k}}\in \mathscr{C}_{n}^{(-)}\;(k=1,2,\ldots,b)$.
Set $\mathscr{I}:=\{i_{1},\ldots,i_{a}\}$ and 
$\mathscr{J}:=\{j_{1},\ldots,j_{b}\}$, 
and define $\mathscr{L}:=\mathscr{I}\cap \mathscr{J}=\{l_{1},\ldots,l_{c}\}$ with  
$l_{1}<l_{2}<\cdots<l_{c}$.
The letters in $\mathscr{I}$, $\mathscr{J}$, and $\mathscr{L}$ are called 
$\mathscr{I}$-letters, $\mathscr{J}$-letters, and $\mathscr{L}$-letters, respectively.
The column $C$ can be split~\cite{DeC} when there exist $\mathscr{C}_{n}^{(+)}$-letters
$l_{1}^{\ast},\ldots,l_{c}^{\ast}$, which are called $\mathscr{L}^{\ast}$-letters, 
determined by the following algorithm
(if $\mathscr{L}=\emptyset$, then $\{l_{1}^{\ast},\ldots,l_{c}^{\ast}\}=\emptyset$ and $C$ can be always split).
\begin{itemize}
\item[(i)]
$l_{c}^{\ast}$ is the largest $\mathscr{C}_{n}^{(+)}$-letter satisfying 
$l_{c}^{\ast}<l_{c}$ and $l_{c}^{\ast}\notin \mathscr{I}\cup \mathscr{J}$,
\item[(ii)]
for $k=c-1,\ldots,1$, $l_{k}^{\ast}$ is the largest $\mathscr{C}_{n}^{(+)}$-letter 
satisfying $l_{k}^{\ast}<l_{k}$, $l_{k}^{\ast}\notin \mathscr{I}\cup \mathscr{J}$, and 
$l_{k}^{\ast}\notin\{l_{k+1}^{\ast},\ldots,l_{c}^{\ast}\}$.
\end{itemize}
\end{df}
Throughout this paper, the sets of letters such as $\mathscr{I}$, $\mathscr{J}$, $\mathscr{L}$, 
and $\mathscr{L}^{\ast}=\{l_{1}^{\ast},\ldots,l_{c}^{\ast}\}$ are also considered as the ordered sequences of letters with respect to the order $<$.
Keeping the notation in Definition~\ref{df:split}, we define  
$\overline{\mathscr{I}}:=\{\overline{i_{a}},\ldots,\overline{i_{1}}\}$, 
$\overline{\mathscr{J}}:=\{\overline{j_{b}},\ldots,\overline{j_{1}}\}$, 
$\overline{\mathscr{L}}:=\{\overline{l_{c}},\ldots,\overline{l_{1}}\}$, and  
$\overline{\mathscr{L}^{\ast}}:=\{\overline{l_{c}^{\ast}},\ldots,\overline{l_{1}^{\ast}}\}$, 
which are also considered as the ordered sequence of letters with respect to the order $\prec$.
The letters in $\overline{\mathscr{I}}$, $\overline{\mathscr{J}}$, $\overline{\mathscr{L}}$, and $\overline{\mathscr{L}^{\ast}}$ are called  $\overline{\mathscr{I}}$-letters, $\overline{\mathscr{J}}$-letters, $\overline{\mathscr{L}}$-letters, and $\overline{\mathscr{L}^{\ast}}$-letters, respectively.

The equivalence between the condition (C1) in Definition~\ref{df:KN_C} 
and the condition in Definition~\ref{df:split} is proven in ~\cite{She}.

\begin{thm}[C. Lecouvey~\cite{Lec}]
A column $C\in C_{n}\text{-}\mathrm{Col}(N)$ is KN-admissible if and only if it can be split.
\end{thm}

\begin{rem} \label{rem:algorithm1}
According to the algorithm in Definition~\ref{df:split}, 
$\mathscr{L}^{\ast}$-letters $l_{1}^{\ast},\ldots,l_{c}^{\ast}$ can be written as follows.
\[
l_{c}^{\ast}=
\begin{cases}
i_{p}-1& (\exists i_{p}\in \mathscr{I}\backslash \mathscr{L})\\
& \text{or} \\
j_{q}-1& (\exists j_{q}\in \mathscr{J}).
\end{cases}  
\]
For $k=1,2,\ldots,c-1$,
\[
l_{k}^{\ast}=
\begin{cases}
i_{p}-1& (\exists i_{p}\in \mathscr{I}\backslash \mathscr{L})\\
& \text{or}\\
j_{q}-1& (\exists j_{q}\in \mathscr{J})\\
& \text{or}\\
l_{k+1}^{\ast}-1.
\end{cases} 
\]
\end{rem}

We also need the notion of a KN-coadmissible column~\cite{Lec,She}.

\begin{df} \label{df:coad}
Let $C\in C_{n}\text{-}\mathrm{Col}(N)$ be the $C_{n}$-column described in Definition~\ref{df:split}.
For each $l\in \mathscr{L}$, 
denote by $N^{\ast}(l)$ the number of letters in $C$ satisfying
$l\preceq x\preceq \Bar{l}$.
Then the column $C$ is said to be \emph{KN-coadmissible} if
$N^{\ast}(l)\leq n-l+1$ $(\forall l\in \mathscr{L})$.
\end{df}
If $\mathscr{L}=\emptyset$, then $C$ is always KN-coadmissible.
Let us denote by $C_{n}\text{-}\mathrm{Col}_{\overline{\mathrm{KN}}}(N)$ 
the set of all $C_{n}$-columns ($\in C_{n}\text{-}\mathrm{Col}(N)$) 
that are KN-coadmissible and set 
$C_{n}\text{-}\mathrm{Col}_{\overline{\mathrm{KN}}}:=
\cup_{N\in Z_{>0}}C_{n}\text{-}\mathrm{Col}_{\overline{\mathrm{KN}}}(N)$.
The following lemma characterizes the KN-coadmissible $C_{n}$-columns.
The proof is analogous to that of Lemma 8.3.4. in \cite{HK}.

\begin{lem} \label{lem:KN_coad}
Suppose that $C\in C_{n}\text{-}\mathrm{Col}_{\overline{\mathrm{KN}}}$ takes the form
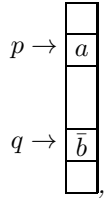
\setlength{\unitlength}{12pt}
\begin{center}
\begin{picture}(3,6)
\put(2,0){\line(0,1){6}}
\put(3,0){\line(0,1){6}}
\put(2,0){\line(1,0){1}}
\put(2,1){\line(1,0){1}}
\put(2,2){\line(1,0){1}}
\put(2,4){\line(1,0){1}}
\put(2,5){\line(1,0){1}}
\put(2,6){\line(1,0){1}}
\put(2,1){\makebox(1,1){$\Bar{b}$}}
\put(2,4){\makebox(1,1){$a$}}
\put(0,1){\makebox(2,1){$q\rightarrow$}}
\put(0,4){\makebox(2,1){$p\rightarrow$}}
\end{picture},
\end{center}
then we have 
$(q-p)+\min(a,b)\leq n$.
\end{lem}

\begin{proof}
If $a=b$, the claim is just Definition~\ref{df:coad}. 
Let us assume that $a<b$.
Let $j$ be the smallest entry such that $j>b$ and both $j$ and $\Bar{j}$ appear in $C$.
Assume that $j$ (resp. $\Bar{j}$) lies at the $k$-th (resp. $l$-th) position.
The column $C$ has the following configuration, 
where the left (resp. right) configuration is the $\mathscr{C}_{n}^{(-)}$ (resp. $\mathscr{C}_{n}^{(+)}$)-letters part 
($p< k <l <q$).

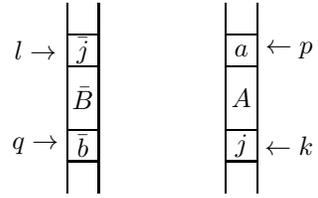
\setlength{\unitlength}{12pt}
\begin{center}
\begin{picture}(11,6)
\put(3,0){\line(0,1){6}}
\put(4,0){\line(0,1){6}}
\put(3,1){\line(1,0){1}}
\put(3,2){\line(1,0){1}}
\put(3,4){\line(1,0){1}}
\put(3,5){\line(1,0){1}}
\put(3,1){\makebox(1,1){$\Bar{b}$}}
\put(3,4){\makebox(1,1){$\Bar{j}$}}
\put(3,2){\makebox(1,2){$\Bar{B}$}}
\put(1,1){\makebox(2,1){$q \rightarrow$}}
\put(1,4){\makebox(2,1){$l \rightarrow$}}
\put(8,0){\line(0,1){6}}
\put(9,0){\line(0,1){6}}
\put(8,1){\line(1,0){1}}
\put(8,2){\line(1,0){1}}
\put(8,4){\line(1,0){1}}
\put(8,5){\line(1,0){1}}
\put(8,1){\makebox(1,1){$j$}}
\put(8,4){\makebox(1,1){$a$}}
\put(8,2){\makebox(1,2){$A$}}
\put(9,1){\makebox(2,1){$\leftarrow k$}}
\put(9,4){\makebox(2,1){$\leftarrow p$}}
\end{picture}.
\end{center}
Let us consider the following two cases separately:
\begin{description}
\item[(a)]
$b\in A$.
\item[(b)]
$b\notin A$.
\end{description}

\textbf{Case (a).}
Suppose that the entry $b$ lies at the $p^{\prime}$-th position.
The number of boxes between the box containing $a$ and that containing $b$ is $p^{\prime}-p-1$ 
and entries in these boxes are taken from the set $\{a+1,\ldots,b-1\}$($=\emptyset$ if $b=a+1$).
Since $\left\vert \{a+1,\ldots,b-1\}\right\vert =b-a-1$, we have 
$p^{\prime}-p-1\leq b-a-1$, while 
$q-p^{\prime}+b\leq n$
by the definition of KN-coadmissible columns.
Hence, we have 
$(q-p)+\min(a,b)\leq n$.

\textbf{Case (b).}
We divide this case further into the following two cases:
\begin{description}
\item[(b-1)]
$a<b-1$.
\item[(b-2)]
$a=b-1$.
\end{description}
In case \textbf{(b-1)},
$A\cap B=\emptyset$
and
$A\cup B\subseteq\{a+1,\ldots,b-1,b+1,\ldots,j-1\}$
so that 
$\left\vert A\right\vert +\left\vert \Bar{B}\right\vert =\left\vert A\cup
B\right\vert \leq j-a-2$.
In case \textbf{(b-2)},
$A\cap B=\emptyset$
and 
$A\cup B\subseteq\{a+2(=b+1),\ldots,j-1\}$
so that 
$\left\vert A\right\vert +\left\vert \Bar{B}\right\vert =\left\vert A\cup
B\right\vert \leq j-a-2$.
In both cases, we have
$(k-p-1)+(q-l-1)\leq j-a-2$, while 
$l-k+j\leq n$
by the definition of KN-coadmissible columns.
Hence, we have
$(q-p)+\min(a,b)\leq n$.

If the pair of entries $j$ and $\Bar{j}$
($j>b$) does not appear in $C$, then the column $C$ has the following configuration.

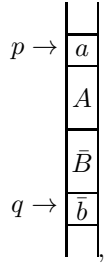
\setlength{\unitlength}{12pt}
\begin{center}
\begin{picture}(3,8)
\put(2,0){\line(0,1){8}}
\put(3,0){\line(0,1){8}}
\put(2,1){\line(1,0){1}}
\put(2,2){\line(1,0){1}}
\put(2,4){\line(1,0){1}}
\put(2,6){\line(1,0){1}}
\put(2,7){\line(1,0){1}}
\put(2,1){\makebox(1,1){$\Bar{b}$}}
\put(2,2){\makebox(1,2){$\Bar{B}$}}
\put(2,4){\makebox(1,2){$A$}}
\put(2,6){\makebox(1,1){$a$}}
\put(0,1){\makebox(2,1){$q\rightarrow$}}
\put(0,6){\makebox(2,1){$p\rightarrow$}}
\end{picture},
\end{center}
where $A$ (resp. $\Bar{B}$) is the block filled with $\mathscr{C}_{n}^{(+)}$ (resp. $\mathscr{C}_{n}^{(-)}$)-letters.
If $b \in A$, then we have $(q-p)+\min(a,b)\leq n$ by the previous argument.
If $b\notin A$, then $A\cap B=\emptyset$ and 
$A\cup B\subseteq\{a+1,\ldots,n\}\backslash\{b\}$
so that $\left\vert A\right\vert +\left\vert \Bar{B}\right\vert=\left\vert A\cup B\right\vert \leq n-a-1$, 
while $\left\vert A\right\vert +\left\vert \Bar{B}\right\vert=q-p-1$.
Hence, we have $(q-p)+\min(a,b)\leq n$.
The proof for the case $a>b$ is analogous.
\end{proof}

Let $C\in C_{n}\text{-}\mathrm{Col}(N)$ be the $C_{n}$-column 
described in Definition~\ref{df:split} and assume that it is KN-admissible.
Denote by $C^{\ast}$ the $C_{n}$-column obtained by filling the shape of $C$, i.e., $(1^{N})$ 
with letters taken from the set 
$(\mathscr{I}\backslash \mathscr{L})\sqcup
(\overline{\mathscr{J}}\backslash \overline{\mathscr{L}})\sqcup
\mathscr{L}^{\ast}\sqcup \overline{\mathscr{L}^{\ast}}$.
Then the map
\begin{equation} \label{eq:map1}
\phi:C\mapsto C^{\ast}
\end{equation}
is a bijection between $C_{n}\text{-}\mathrm{Col}_{\mathrm{KN}}(N)$ and 
$C_{n}\text{-}\mathrm{Col}_{\overline{\mathrm{KN}}}(N)$
\cite{Lec}.
The inverse map $\phi ^{-1}=:\psi$ is therefore given by the following algorithm.
Suppose $C\in C_{n}\text{-}\mathrm{Col}_{\overline{\mathrm{KN}}}(N)$ such that
$w(C)=i_{1}\cdots i_{a}\overline{j_{b}}\cdots \overline{j_{1}}$
where $N=a+b$, $i_{k}\in \mathscr{C}_{n}^{(+)\;}(k=1,2,\ldots,a)$, and 
$\overline{j_{k}}\in \mathscr{C}_{n}^{(-)}\;(k=1,2,\ldots,b)$.
Set $\mathscr{I}:=\{i_{1},\ldots,i_{a}\}$ and 
$\mathscr{J}:=\{j_{1},\ldots,j_{b}\}$, 
and define $\mathscr{L}:=\mathscr{I}\cap \mathscr{J}=\{l_{1},\ldots,l_{c}\}$ with  
$l_{1}<l_{2}<\cdots<l_{c}$.
As in Definition~\ref{df:split}, the letters in $\mathscr{I}$, $\mathscr{J}$, and $\mathscr{L}$ are called 
$\mathscr{I}$-letters, $\mathscr{J}$-letters, and $\mathscr{L}$-letters.
Find $\mathscr{C}_{n}^{(+)}$-letters 
$l_{1}^{\dag},\ldots,l_{c}^{\dag}$, which are called $\mathscr{L}^{\dag}$-letters,  
by the following procedure ($\mathscr{L}^{\dag}=\{l_{1}^{\dag},\ldots,l_{c}^{\dag}\}$).
\begin{itemize}
\item[(i)]
$l_{1}^{\dag}$ is the smallest $\mathscr{C}_{n}^{(+)}$-letter satisfying $l_{1}^{\dag}>l_{1}$ 
and $l_{1}^{\dag}\notin \mathscr{I}\cup \mathscr{J}$,
\item[(ii)]
for $k=2,\ldots,c$, $l_{k}^{\dag}$ is the smallest $\mathscr{C}_{n}^{(+)}$-letter satisfying 
$l_{k}^{\dag}>l_{k}$, $l_{k}^{\dag}\notin \mathscr{I}\cup \mathscr{J}$ and 
$l_{k}^{\dag}\notin\{l_{1}^{\dag},\ldots,l_{k-1}^{\dag}\}$.
\end{itemize}
Denote by $C^{\dag}$ the $C_{n}$-column obtained by filling the shape of $C$, i.e., $(1^{N})$ 
with letters taken from the set
$(\mathscr{I}\backslash \mathscr{L})\sqcup
(\overline{\mathscr{J}}\backslash \overline{\mathscr{L}})\sqcup
\mathscr{L}^{\dag}\sqcup \overline{\mathscr{L}^{\dag}}$.
Then
\begin{equation} \label{eq:map2}
\psi:C\mapsto C^{\dag}.
\end{equation}
By construction, both maps $\phi$ and $\psi$ are weight-preserving.

\begin{rem} \label{rem:algorithm2}
$\mathscr{L}^{\dag}$-letters $l_{1}^{\dag},\ldots,l_{c}^{\dag}$ can be written as follows.
\[
l_{1}^{\dag}=
\begin{cases}
i_{p}+1& (\exists i_{p}\in \mathscr{I}\backslash \mathscr{L})\\
& \text{or}\\
j_{q}+1& (\exists j_{q}\in \mathscr{J}).
\end{cases}
\]
For $k=2,\ldots,c$,
\[
l_{k}^{\dag}=
\begin{cases}
i_{p}+1& (\exists i_{p}\in \mathscr{I}\backslash \mathscr{L})\\
& \text{or}\\
j_{q}+1& (\exists j_{q}\in \mathscr{J})\\
& \text{or}\\
l_{k-1}^{\dag}+1.
\end{cases}
\]
\end{rem}

The actual implementation of the above algorithm to compute $\phi (C)$ for 
$C\in C_{n}\text{-}\mathrm{Col}$ is as follows.
For $k=c,c-1,\ldots,1$, we delete entries $l_{k}$ and $\overline{l_{k}}$ and 
relocate entries $l_{k}^{\ast}$ and $\overline{l_{k}^{\ast}}$ in the column to obtain the updated $C_{n}$-column.
This is called the operation for $l_{k} \rightarrow l_{k}^{\ast}$.
Note that the position of 
$l_{k}^{\ast}$($\overline{l_{k}^{\ast}}$) may be changed by subsequent operations for 
$l_{k-1}\rightarrow l_{k-1}^{\ast},\ldots,l_{1}\rightarrow l_{1}^{\ast}$.
We refer to this algorithm as the \emph{first kind} algorithm for $\phi$.
The first kind algorithm for $\psi$ is prescribed similarly.

\begin{ex} \label{ex:algorithm1}
For a $C_{n}$-column with entries $\{2,5,6,7,\Bar{7},\Bar{5},\Bar{4}\}$, 
$\mathscr{L}=\{5,7\}$ and $\mathscr{L}^{\ast}=\{1,3\}$.
The updating process of the column is shown in Fig.~\ref{fig:algorithm1}.
\end{ex} 

\begin{figure}
\setlength{\unitlength}{12pt}
\centering
\begin{picture}(7,7)
\put(0,0){\line(0,1){7}}
\put(1,0){\line(0,1){7}}
\put(3,0){\line(0,1){7}}
\put(4,0){\line(0,1){7}}
\put(6,0){\line(0,1){7}}
\put(7,0){\line(0,1){7}}

\put(0,0){\line(1,0){1}}
\put(0,1){\line(1,0){1}}
\put(0,2){\line(1,0){1}}
\put(0,3){\line(1,0){1}}
\put(0,4){\line(1,0){1}}
\put(0,5){\line(1,0){1}}
\put(0,6){\line(1,0){1}}
\put(0,7){\line(1,0){1}}

\put(0,0){\makebox(1,1){$\Bar{4}$}}
\put(0,1){\makebox(1,1){$\Bar{5}$}}
\put(0,2){\makebox(1,1){$\Bar{7}$}}
\put(0,3){\makebox(1,1){$7$}}
\put(0,4){\makebox(1,1){$6$}}
\put(0,5){\makebox(1,1){$5$}}
\put(0,6){\makebox(1,1){$2$}}

\put(1,3){\makebox(2,1){$\rightarrow$}}

\put(3,0){\line(1,0){1}}
\put(3,1){\line(1,0){1}}
\put(3,2){\line(1,0){1}}
\put(3,3){\line(1,0){1}}
\put(3,4){\line(1,0){1}}
\put(3,5){\line(1,0){1}}
\put(3,6){\line(1,0){1}}
\put(3,7){\line(1,0){1}}

\put(3,0){\makebox(1,1){$\Bar{3}$}}
\put(3,1){\makebox(1,1){$\Bar{4}$}}
\put(3,2){\makebox(1,1){$\Bar{5}$}}
\put(3,3){\makebox(1,1){$6$}}
\put(3,4){\makebox(1,1){$5$}}
\put(3,5){\makebox(1,1){$3$}}
\put(3,6){\makebox(1,1){$2$}}

\put(4,3){\makebox(2,1){$\rightarrow$}}

\put(6,0){\line(1,0){1}}
\put(6,1){\line(1,0){1}}
\put(6,2){\line(1,0){1}}
\put(6,3){\line(1,0){1}}
\put(6,4){\line(1,0){1}}
\put(6,5){\line(1,0){1}}
\put(6,6){\line(1,0){1}}
\put(6,7){\line(1,0){1}}

\put(6,0){\makebox(1,1){$\Bar{1}$}}
\put(6,1){\makebox(1,1){$\Bar{3}$}}
\put(6,2){\makebox(1,1){$\Bar{4}$}}
\put(6,3){\makebox(1,1){$6$}}
\put(6,4){\makebox(1,1){$3$}}
\put(6,5){\makebox(1,1){$2$}}
\put(6,6){\makebox(1,1){$1$}}
\end{picture}
\caption{Example of the first kind algorithm for $\phi$.}
\label{fig:algorithm1}
\end{figure}
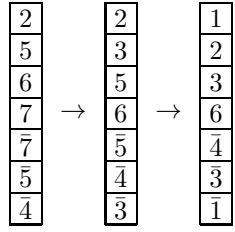

In order to view a $C_{n}$-column, we also use the \emph{filling diagram} explained below.
This is basically the circle diagram introduced by Sheats~\cite{She} 
and is useful to keep track of the change of entries when we update the column by the above algorithm.
It is constructed on $2\times n$ grid and the pair of the $k$-th squares from the left in the top and bottom rows is called the $k$-th slot. 
For example, the initial column in Fig.~\ref{fig:algorithm1}, i.e., 
the $C_{n}$-column with entries $\{2,5,6,7,\Bar{7},\Bar{5},\Bar{4}\}$, the filling diagram reads

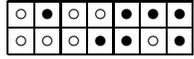
\setlength{\unitlength}{10pt}
\begin{center}
\begin{picture}(7,2)
\put(0,0){\line(1,0){7}}
\put(0,1){\line(1,0){7}}
\put(0,2){\line(1,0){7}}
\put(0,0){\line(0,1){2}}
\put(1,0){\line(0,1){2}}
\put(2,0){\line(0,1){2}}
\put(3,0){\line(0,1){2}}
\put(4,0){\line(0,1){2}}
\put(5,0){\line(0,1){2}}
\put(6,0){\line(0,1){2}}
\put(7,0){\line(0,1){2}}
\put(0,0){\makebox(1,1){$\circ$}}
\put(0,1){\makebox(1,1){$\circ$}}
\put(1,0){\makebox(1,1){$\circ$}}
\put(1,1){\makebox(1,1){$\bullet$}}
\put(2,0){\makebox(1,1){$\circ$}}
\put(2,1){\makebox(1,1){$\circ$}}
\put(3,0){\makebox(1,1){$\bullet$}}
\put(3,1){\makebox(1,1){$\circ$}}
\put(4,0){\makebox(1,1){$\bullet$}}
\put(4,1){\makebox(1,1){$\bullet$}}
\put(5,0){\makebox(1,1){$\circ$}}
\put(5,1){\makebox(1,1){$\bullet$}}
\put(6,0){\makebox(1,1){$\bullet$}}
\put(6,1){\makebox(1,1){$\bullet$}}
\end{picture}.
\end{center}
The slot 
{\setlength{\tabcolsep}{3pt}\footnotesize
\begin{tabular}{|c|} \hline
$\circ$ \\ \hline
$\circ$ \\ \hline
\end{tabular}}, 
{\setlength{\tabcolsep}{3pt}\footnotesize
\begin{tabular}{|c|} \hline
$\bullet$ \\ \hline
$\circ$ \\ \hline
\end{tabular}}, 
{\setlength{\tabcolsep}{3pt}\footnotesize
\begin{tabular}{|c|} \hline
$\circ$ \\ \hline
$\bullet$ \\ \hline
\end{tabular}}, and  
{\setlength{\tabcolsep}{3pt}\footnotesize
\begin{tabular}{|c|} \hline
$\bullet$ \\ \hline
$\bullet$ \\ \hline
\end{tabular}}     
are called $\emptyset$-slot, $(+)$-slot, $(-)$-slot, and $(\pm)$-slot, respectively.
If the $k$-th slot in the filling diagram for a $C_{n}$-column is $\emptyset$-slot, 
then both entries $k$ and $\Bar{k}$ do not appear in the column.
If the $k$-th slot is $(+)$-slot (resp. $(-)$-slot), then the entry $k$ (resp. $\Bar{k}$) appears in the column, 
while the entry $\Bar{k}$ (resp. $k$) does not appear.
If the $k$-th slot is $(\pm)$-slot, then both entries $k$ and $\Bar{k}$ appear in the column.
According to the algorithm for $\phi$, the filling diagram of the $C_{n}$-column 
in Example~\ref{ex:algorithm1} changes as follows.

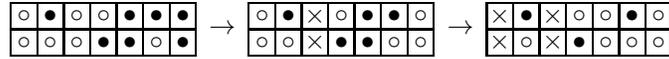
\setlength{\unitlength}{10pt}
\begin{center}
\begin{picture}(25,2)
\put(0,0){\line(1,0){7}}
\put(0,1){\line(1,0){7}}
\put(0,2){\line(1,0){7}}
\put(0,0){\line(0,1){2}}
\put(1,0){\line(0,1){2}}
\put(2,0){\line(0,1){2}}
\put(3,0){\line(0,1){2}}
\put(4,0){\line(0,1){2}}
\put(5,0){\line(0,1){2}}
\put(6,0){\line(0,1){2}}
\put(7,0){\line(0,1){2}}
\put(0,0){\makebox(1,1){$\circ$}}
\put(0,1){\makebox(1,1){$\circ$}}
\put(1,0){\makebox(1,1){$\circ$}}
\put(1,1){\makebox(1,1){$\bullet$}}
\put(2,0){\makebox(1,1){$\circ$}}
\put(2,1){\makebox(1,1){$\circ$}}
\put(3,0){\makebox(1,1){$\bullet$}}
\put(3,1){\makebox(1,1){$\circ$}}
\put(4,0){\makebox(1,1){$\bullet$}}
\put(4,1){\makebox(1,1){$\bullet$}}
\put(5,0){\makebox(1,1){$\circ$}}
\put(5,1){\makebox(1,1){$\bullet$}}
\put(6,0){\makebox(1,1){$\bullet$}}
\put(6,1){\makebox(1,1){$\bullet$}}
\put(7,0){\makebox(2,2){$\rightarrow$}}

\put(9,0){\line(1,0){7}}
\put(9,1){\line(1,0){7}}
\put(9,2){\line(1,0){7}}
\put(9,0){\line(0,1){2}}
\put(10,0){\line(0,1){2}}
\put(11,0){\line(0,1){2}}
\put(12,0){\line(0,1){2}}
\put(13,0){\line(0,1){2}}
\put(14,0){\line(0,1){2}}
\put(15,0){\line(0,1){2}}
\put(16,0){\line(0,1){2}}
\put(9,0){\makebox(1,1){$\circ$}}
\put(9,1){\makebox(1,1){$\circ$}}
\put(10,0){\makebox(1,1){$\circ$}}
\put(10,1){\makebox(1,1){$\bullet$}}
\put(11,0){\makebox(1,1){$\times$}}
\put(11,1){\makebox(1,1){$\times$}}
\put(12,0){\makebox(1,1){$\bullet$}}
\put(12,1){\makebox(1,1){$\circ$}}
\put(13,0){\makebox(1,1){$\bullet$}}
\put(13,1){\makebox(1,1){$\bullet$}}
\put(14,0){\makebox(1,1){$\circ$}}
\put(14,1){\makebox(1,1){$\bullet$}}
\put(15,0){\makebox(1,1){$\circ$}}
\put(15,1){\makebox(1,1){$\circ$}}
\put(16,0){\makebox(2,2){$\rightarrow$}}

\put(18,0){\line(1,0){7}}
\put(18,1){\line(1,0){7}}
\put(18,2){\line(1,0){7}}
\put(18,0){\line(0,1){2}}
\put(19,0){\line(0,1){2}}
\put(20,0){\line(0,1){2}}
\put(21,0){\line(0,1){2}}
\put(22,0){\line(0,1){2}}
\put(23,0){\line(0,1){2}}
\put(24,0){\line(0,1){2}}
\put(25,0){\line(0,1){2}}
\put(18,0){\makebox(1,1){$\times$}}
\put(18,1){\makebox(1,1){$\times$}}
\put(19,0){\makebox(1,1){$\circ$}}
\put(19,1){\makebox(1,1){$\bullet$}}
\put(20,0){\makebox(1,1){$\times$}}
\put(20,1){\makebox(1,1){$\times$}}
\put(21,0){\makebox(1,1){$\bullet$}}
\put(21,1){\makebox(1,1){$\circ$}}
\put(22,0){\makebox(1,1){$\circ$}}
\put(22,1){\makebox(1,1){$\circ$}}
\put(23,0){\makebox(1,1){$\circ$}}
\put(23,1){\makebox(1,1){$\bullet$}}
\put(24,0){\makebox(1,1){$\circ$}}
\put(24,1){\makebox(1,1){$\circ$}}
\end{picture},
\end{center} 
where the slot 
{\setlength{\tabcolsep}{3pt}\footnotesize
\begin{tabular}{|c|} \hline
$\times$ \\ \hline
$\times$ \\ \hline
\end{tabular}}     
is called $(\times)$-slot. 
If the $k$-th slot in the filling diagram for the updated column is $(\times)$-slot, 
then a pair of entries $l^{\ast}(=k)$ and $\overline{l^{\ast}}(=\Bar{k})$ newly appears and 
a pair of entries $l$ and $\Bar{l}$ disappears in the column, 
where $l \in \mathscr{L}$ with $\mathscr{L}$ being the set of $\mathscr{L}$-letters in the original column 
and $l^{\ast}\in \mathscr{L}^{\ast}$ with $\mathscr{L}^{\ast}$ being the set of $\mathscr{L}^{\ast}$-letters in the updated column.
We also use the filling diagram to view the updating process of a $C_{n}$-column by $\psi$.
In this case, the role of $\mathscr{L}^{\ast}$-letters is replaced by that of $\mathscr{L}^{\dag}$-letters.

\begin{lem} \label{lem:position1}
Suppose that $C\in C_{n}\text{-}\mathrm{Col}_{\mathrm{KN}}$ 
and let the set of $\mathscr{L}$-letters of $C$ be $\{l_{1},\ldots,l_{c}\}$.
Let $p_{k}$ (resp. $p_{k}^{\ast}$) be the position of $l_{k}$ (resp. $l_{k}^{\ast}$) in $C$ (resp. $\phi(C)$) and 
$q_{k}$ (resp. $q_{k}^{\ast}$) be the position of $\overline{l_{k}}$ (resp. $\overline{l_{k}^{\ast}}$) in $C$ (resp. $\phi(C)$).
Suppose that a series of operations for $l_{c}\rightarrow l_{c}^{\ast},\ldots,l_{k+1}\rightarrow l_{k+1}^{\ast}$ is finished.
The filling diagram of the updated column has the following configuration.

\setlength{\unitlength}{15pt}
\begin{center}
\begin{picture}(6,3)
\put(1,1){\line(0,1){2}}
\put(2,1){\line(0,1){2}}
\put(4,1){\line(0,1){2}}
\put(5,1){\line(0,1){2}}
\put(0,1){\line(1,0){6}}
\put(0,3){\line(1,0){6}}
\put(0,2){\line(1,0){2}}
\put(4,2){\line(1,0){2}}
\put(1,1){\makebox(1,1){$\circ$}}
\put(1,2){\makebox(1,1){$\circ$}}
\put(4,1){\makebox(1,1){$\bullet$}}
\put(4,2){\makebox(1,1){$\bullet$}}
\put(2,1){\makebox(2,2){$(0)$}}
\put(1,0){\makebox(1,1){$l_{k}^{\ast}$}}
\put(4,0){\makebox(1,1){$l_{k}$}}
\end{picture}.
\end{center}
Then we have $p_{k}-p_{k}^{\ast}=\alpha$ and  $q_{k}^{\ast}-q_{k}=\beta$, 
where $\alpha$ and $\beta$ are the number of $(+)$-slots and that of $(-)$-slots in region $(0)$, respectively.
\end{lem}

\begin{proof}
Between the $l_{k}^{\ast}$-th slot and the $l_{k}$-th slot (region $(0)$), there are no $\emptyset$-slots 
by the choice of $l_{k}^{\ast}$.
Let us assume that the number of $(\pm)$-slots and that of $(\times)$-slots are 
$\gamma$ and $\delta$ in the region $(\ast)$, respectively.
When the relocation of $\mathscr{L}^{\ast}$-letters down to $l_{k+1}^{\ast}$ is finished, 
the position of the box containing $l_{k}$ is changed from $p_{k}$ to $p_{k}+\delta$ 
because $\delta$ $\mathscr{L}^{\ast}$-letters appears 
above this box.
When the relocation of $l_{k}^{\ast}$ is finished, the position of box containing $l_{k}^{\ast}$ is changed 
from $p_{k}+\delta$ to $p_{k}+\delta-(\alpha+\gamma+\delta)=p_{k}-\alpha-\gamma$.
However, $\gamma$ $\mathscr{L}$-letters below the box containing $l_{k}^{\ast}$ are 
transformed to the corresponding $\mathscr{L}^{\ast}$-letters and are 
relocated above the box containing $l_{k}^{\ast}$ in $\phi(C)$ 
so that the position of $l_{k}^{\ast}$ in $\phi(C)$ is $p_{k}^{\ast}=p_{k}-\alpha$.
Similarly, we have $q_{k}^{\ast}=q_{k}+\beta$.
\end{proof}

The following result may be proven in much the same way as in Lemma~\ref{lem:position1}

\begin{lem} \label{lem:position2}
Suppose that $C\in C_{n}\text{-}\mathrm{Col}_{\overline{\mathrm{KN}}}$ 
and let the set of $\mathscr{L}$-letters of $C$ be $\{l_{1},\ldots,l_{c}\}$.
Let $p_{k}$ (resp. $p_{k}^{\dag}$) be the position of $l_{k}$ (resp. $l_{k}^{\dag}$) in $C$ (resp. $\psi(C)$) and 
$q_{k}$ (resp. $q_{k}^{\dag}$) be the position of $\overline{l_{k}}$ (resp. $\overline{l_{k}^{\dag}}$) in $C$ (resp. $\psi(C)$).
Suppose that a series of operations for $l_{1}\rightarrow l_{1}^{\dag},\ldots,l_{k-1}\rightarrow l_{k-1}^{\dag}$ is finished.
The filling diagram of the updated column has the following configuration.

\setlength{\unitlength}{15pt}
\begin{center}
\begin{picture}(6,3)
\put(1,1){\line(0,1){2}}
\put(2,1){\line(0,1){2}}
\put(4,1){\line(0,1){2}}
\put(5,1){\line(0,1){2}}
\put(0,1){\line(1,0){6}}
\put(0,3){\line(1,0){6}}
\put(0,2){\line(1,0){2}}
\put(4,2){\line(1,0){2}}
\put(1,1){\makebox(1,1){$\bullet$}}
\put(1,2){\makebox(1,1){$\bullet$}}
\put(4,1){\makebox(1,1){$\circ$}}
\put(4,2){\makebox(1,1){$\circ$}}
\put(2,1){\makebox(2,2){$(0)$}}
\put(1,0){\makebox(1,1){$l_{k}$}}
\put(4,0){\makebox(1,1){$l_{k}^{\dag}$}}
\end{picture}.
\end{center}
Then we have $p_{k}^{\dag}-p_{k}=\alpha$ and  $q_{k}-q_{k}^{\dag}=\beta$, 
where $\alpha$ and $\beta$ are the number of $(+)$-slots and that of $(-)$-slots in region $(0)$, respectively.
\end{lem}

Given $C\in C_{n}\text{-}\mathrm{Col}_{\mathrm{KN}}$, the computation of $\phi(C)$ can also be achieved 
by the following algorithm, which we refer to as the algorithm of the \emph{second kind} for $\phi$.
Suppose that $C\in C_{n}\text{-}\mathrm{Col}_{\mathrm{KN}}$
and let the set of $\mathscr{L}$-letters of $C$ be
$\{l_{1},\ldots,l_{c}\}$.
For $k=c,c-1,\ldots,1$, the following procedure is applied.
Firstly, we compute $l_{k}^{\ast}$ for $l_{k}$.
Secondly, we apply the operation (A) followed by the operation (B) described below.
A pair of operations (A) and (B) is called the operation for $l_{k}\rightarrow l_{k}^{\ast}$ as in the first kind algorithm.

\textbf{Operation (A).}

Set
\[
\{i_{p+1},\ldots,i_{p+r}\}:=\left\{i \relmiddle| l_{k}^{\ast}<i<l_{k},i\in C\right\}
\]
and 
\[
\{j_{q+1},\ldots,j_{q+s}\}:=
\left\{j \relmiddle| \overline{l_{k}}\prec \Bar{j} \prec\overline{l_{k}^{\ast}},\Bar{j}\in C\right\}.
\]
The block filled with $i_{p+1},\ldots,i_{p+r}$ and $l_{k}$ is replaced by 
the block filled with $l_{k}^{\ast}$ and $i_{p+1},\ldots,i_{p+r}$.
Similarly, the block filled with $\overline{l_{k}}$ and $\overline{j_{q+s}},\ldots,\overline{j_{q+1}}$ is replaced by 
the block filled with $\overline{j_{q+s}},\ldots,\overline{j_{q+1}}$ and $\overline{l_{k}^{\ast}}$.

\setlength{\unitlength}{15pt}
\begin{center}
\begin{picture}(17,4)
\put(2,0){\line(0,1){4}}
\put(3.5,0){\line(0,1){4}}
\put(2,0){\line(1,0){1.5}}
\put(2,1){\line(1,0){1.5}}
\put(2,2){\line(1,0){1.5}}
\put(2,3){\line(1,0){1.5}}
\put(2,4){\line(1,0){1.5}}
\put(0,0){\makebox(2,1){$p_{k}\rightarrow$}}
\put(2,0){\makebox(1.5,1){$l_{k}$}}
\put(2,1){\makebox(1.5,1){$i_{p+r}$}}
\put(2,2){\makebox(1.5,1){$\vdots$}}
\put(2,3){\makebox(1.5,1){$i_{p+1}$}}
\put(4,1){\makebox(1,2){$\longrightarrow$}}
\put(5.5,0){\line(0,1){4}}
\put(7,0){\line(0,1){4}}
\put(5.5,0){\line(1,0){1.5}}
\put(5.5,1){\line(1,0){1.5}}
\put(5.5,2){\line(1,0){1.5}}
\put(5.5,3){\line(1,0){1.5}}
\put(5.5,4){\line(1,0){1.5}}
\put(5.5,0){\makebox(1.5,1){$i_{p+r}$}}
\put(5.5,1){\makebox(1.5,1){$\vdots$}}
\put(5.5,2){\makebox(1.5,1){$i_{p+1}$}}
\put(5.5,3){\makebox(1.5,1){$l_{k}^{\ast}$}}
\put(8,1){\makebox(2,2){and}}
\put(12,0){\line(0,1){4}}
\put(13.5,0){\line(0,1){4}}
\put(12,0){\line(1,0){1.5}}
\put(12,1){\line(1,0){1.5}}
\put(12,2){\line(1,0){1.5}}
\put(12,3){\line(1,0){1.5}}
\put(12,4){\line(1,0){1.5}}
\put(12,0){\makebox(1.5,1){$\overline{j_{q+1}}$}}
\put(12,1){\makebox(1.5,1){$\vdots$}}
\put(12,2){\makebox(1.5,1){$\overline{j_{q+s}}$}}
\put(12,3){\makebox(1.5,1){$\overline{l_{k}}$}}
\put(10,3){\makebox(2,1){$q_{k}\rightarrow$}}
\put(14,1){\makebox(1,2){$\longrightarrow$}}
\put(15.5,0){\line(0,1){4}}
\put(17,0){\line(0,1){4}}
\put(15.5,0){\line(1,0){1.5}}
\put(15.5,1){\line(1,0){1.5}}
\put(15.5,2){\line(1,0){1.5}}
\put(15.5,3){\line(1,0){1.5}}
\put(15.5,4){\line(1,0){1.5}}
\put(15.5,0){\makebox(1.5,1){$\overline{l_{k}^{\ast}}$}}
\put(15.5,1){\makebox(1.5,1){$\overline{j_{q+1}}$}}
\put(15.5,2){\makebox(1.5,1){$\vdots$}}
\put(15.5,3){\makebox(1.5,1){$\overline{j_{q+s}}$}}
\end{picture}.
\end{center}

\textbf{Operation (B).}

Set 
\[
\{l_{t+1},\ldots,l_{t+\gamma}=l_{k-1}\}:=\{i_{p+1},\ldots,i_{p+r}\}\cap\{j_{q+1},\ldots,j_{q+s}\},
\]
assuming $\gamma \geq 1$ (if $\gamma =0$, then this operation is not necessary).
We extract non $\mathscr{L}$-letters from $\{i_{p+1},\ldots,i_{p+r}\}$ and $\{j_{q+1},\ldots,j_{q+s}\}$;
\[
\{i_{p_{1}},i_{p_{2}},\ldots,i_{p_{\alpha}}\}:=\{i_{p+1},\ldots,i_{p+r}\}\backslash\{l_{t+1},\ldots,l_{k-1}\},
\]
and
\[
\{j_{q_{1}},j_{q_{2}},\ldots,j_{q_{\beta}}\}:=\{j_{q+1},\ldots,j_{q+s}\}\backslash\{l_{t+1},\ldots,l_{k-1}\},
\]
where $r=\alpha+\gamma$ and $s=\beta+\gamma$.
The replaced blocks in the operation (A) are further replaced by the following blocks.

\setlength{\unitlength}{15pt}
\begin{center}
\begin{picture}(7.5,7)
\put(0,0){\line(0,1){7}}
\put(1.5,0){\line(0,1){7}}
\put(0,0){\line(1,0){1.5}}
\put(0,1){\line(1,0){1.5}}
\put(0,2){\line(1,0){1.5}}
\put(0,3){\line(1,0){1.5}}
\put(0,4){\line(1,0){1.5}}
\put(0,5){\line(1,0){1.5}}
\put(0,6){\line(1,0){1.5}}
\put(0,7){\line(1,0){1.5}}
\put(0,0){\makebox(1.5,1){$i_{p_{\alpha}}$}}
\put(0,1){\makebox(1.5,1){$\vdots$}}
\put(0,2){\makebox(1.5,1){$i_{p_{1}}$}}
\put(0,3){\makebox(1.5,1){$l_{k}^{\ast}$}}
\put(0,4){\makebox(1.5,1){$l_{k-1}$}}
\put(0,5){\makebox(1.5,1){$\vdots$}}
\put(0,6){\makebox(1.5,1){$l_{t+1}$}}
\put(2.5,3){\makebox(2.5,1){and}}
\put(6,0){\line(0,1){7}}
\put(7.5,0){\line(0,1){7}}
\put(6,0){\line(1,0){1.5}}
\put(6,1){\line(1,0){1.5}}
\put(6,2){\line(1,0){1.5}}
\put(6,3){\line(1,0){1.5}}
\put(6,4){\line(1,0){1.5}}
\put(6,5){\line(1,0){1.5}}
\put(6,6){\line(1,0){1.5}}
\put(6,7){\line(1,0){1.5}}
\put(6,0){\makebox(1.5,1){$\overline{l_{t+1}}$}}
\put(6,1){\makebox(1.5,1){$\vdots$}}
\put(6,2){\makebox(1.5,1){$\overline{l_{k-1}}$}}
\put(6,3){\makebox(1.5,1){$\overline{l_{k}^{\ast}}$}}
\put(6,4){\makebox(1.5,1){$\overline{j_{q_{1}}}$}}
\put(6,5){\makebox(1.5,1){$\vdots$}}
\put(6,6){\makebox(1.5,1){$\overline{j_{q_{\beta}}}$}}
\end{picture}.
\end{center}
That is, $\mathscr{L}$ (resp. $\Bar{\mathscr{L}}$)-letters in the obtained blocks in the operation (A) 
are expelled and relocated just above (resp. below) the box containing $l_{k}^{\ast}$ (resp. $\overline{l_{k}^{\ast}}$).
Note that these blocks are not semistandard 
because $l_{k-1}>l_{k}^{\ast}$ and $\overline{l_{k-1}}\prec \overline{l_{k}^{\ast}}$ 
and that $l_{k}$ (resp. $\overline{l_{k}}$) in the operation (A) for $l_{k}\rightarrow l_{k}^{\ast}$ is always lies 
at the upper (resp. lower) position of $l_{k+1}^{\ast}$ (resp. $\overline{l_{k+1}^{\ast}}$) 
because even when $l_{k+1}^{\ast}<l_{k}$ (resp. $\overline{l_{k}}\prec\overline{l_{k+1}^{\ast}}$), 
$l_{k}$ (resp. $\overline{l_{k}}$) is relocated just above $l_{k+1}^{\ast}$ (resp. below $\overline{l_{k+1}^{\ast}}$)
by the operation (B) for $l_{k+1}\rightarrow l_{k+1}^{\ast}$.
In particular, $p_{k}$ (resp. $q_{k}$) in the operation (A) is not necessarily the original position of $l_{k}$ (resp. $\overline{l_{k}}$) in $C$.
After the operation (B) for $l_{k}\rightarrow l_{k}^{\ast}$ is finished, the subsequent operations for 
$l_{k-1}\rightarrow l_{k-1}^{\ast}$
do not affect the positions of 
$i_{p_{1}},\ldots,i_{p_{\alpha}}$ ($\overline{j_{q_{\beta}}},\ldots,\overline{j_{q_{1}}}$ and $\overline{l_{k}^{\ast}}$) in the updated column.
We define $\Delta_{k}(C)$ and $\overline{\Delta_{k}}(C)$ as

\setlength{\unitlength}{15pt}
\begin{center}
\begin{picture}(12.5,4)
\put(0,1.5){\makebox(3,1){$\Delta_{k}(C):=$}}
\put(3,0){\line(0,1){4}}
\put(4.5,0){\line(0,1){4}}
\put(3,0){\line(1,0){1.5}}
\put(3,1){\line(1,0){1.5}}
\put(3,2){\line(1,0){1.5}}
\put(3,3){\line(1,0){1.5}}
\put(3,4){\line(1,0){1.5}}
\put(3,0){\makebox(1.5,1){$i_{p_{\alpha}}$}}
\put(3,1){\makebox(1.5,1){$\vdots$}}
\put(3,2){\makebox(1.5,1){$i_{p_{1}}$}}
\put(3,3){\makebox(1.5,1){$l_{k}^{\ast}$}}
\put(5,1.5){\makebox(2.5,1){and}}
\put(8,1.5){\makebox(3,1){$\overline{\Delta_{k}}(C):=$}}
\put(11,0){\line(0,1){4}}
\put(12.5,0){\line(0,1){4}}
\put(11,0){\line(1,0){1.5}}
\put(11,1){\line(1,0){1.5}}
\put(11,2){\line(1,0){1.5}}
\put(11,3){\line(1,0){1.5}}
\put(11,4){\line(1,0){1.5}}
\put(11,0){\makebox(1.5,1){$\overline{l_{k}^{\ast}}$}}
\put(11,1){\makebox(1.5,1){$\overline{j_{q_{1}}}$}}
\put(11,2){\makebox(1.5,1){$\vdots$}}
\put(11,3){\makebox(1.5,1){$\overline{j_{q_{\beta}}}$}}
\end{picture}.
\end{center}
When the operation (A) for $l_{1}\rightarrow l_{1}^{\ast}$ is completed 
(the operation (B) is not necessary for $l_{1}\rightarrow l_{1}^{\ast}$), the column turns out to be 
$\phi(C)$ (semistandard).
The second kind algorithm for $\psi$ is prescribed similarly.

\begin{ex} \label{ex:2nd}
Let $C$ be the KN-admissible $C_{n}$-column filled with entries $2$,$7$,$8$,$9$,$\Bar{9}$,$\Bar{8}$,$\Bar{7}$,$\Bar{5}$.
Then $\mathscr{L}=\{7,9\}$ and $\mathscr{L}^{\ast}=\{1,3\}$.
The updating process for $9\rightarrow 9^{\ast}=3$ is depicted in Fig.~\ref{fig:algorithm2}.
\end{ex}

\begin{figure}
\setlength{\unitlength}{12pt}
\centering
\begin{picture}(10,7)

\put(0,0){\line(0,1){7}}
\put(1,0){\line(0,1){7}}
\put(0,0){\line(1,0){1}}
\put(0,1){\line(1,0){1}}
\put(0,2){\line(1,0){1}}
\put(0,3){\line(1,0){1}}
\put(0,4){\line(1,0){1}}
\put(0,5){\line(1,0){1}}
\put(0,6){\line(1,0){1}}
\put(0,7){\line(1,0){1}}

\put(0,0){\makebox(1,1){$\Bar{5}$}}
\put(0,1){\makebox(1,1){$\Bar{7}$}}
\put(0,2){\makebox(1,1){$\Bar{9}$}}
\put(0,3){\makebox(1,1){$9$}}
\put(0,4){\makebox(1,1){$8$}}
\put(0,5){\makebox(1,1){$7$}}
\put(0,6){\makebox(1,1){$2$}}

\put(2,3){\makebox(2,1){$\longrightarrow$}}
\put(2,2){\makebox(2,1){$(A)$}}

\put(5,0){\line(0,1){7}}
\put(6,0){\line(0,1){7}}
\put(5,0){\line(1,0){1}}
\put(5,1){\line(1,0){1}}
\put(5,2){\line(1,0){1}}
\put(5,3){\line(1,0){1}}
\put(5,4){\line(1,0){1}}
\put(5,5){\line(1,0){1}}
\put(5,6){\line(1,0){1}}
\put(5,7){\line(1,0){1}}

\put(5,0){\makebox(1,1){$\Bar{3}$}}
\put(5,1){\makebox(1,1){$\Bar{5}$}}
\put(5,2){\makebox(1,1){$\Bar{7}$}}
\put(5,3){\makebox(1,1){$8$}}
\put(5,4){\makebox(1,1){$7$}}
\put(5,5){\makebox(1,1){$3$}}
\put(5,6){\makebox(1,1){$2$}}

\put(7,3){\makebox(2,1){$\longrightarrow$}}
\put(7,2){\makebox(2,1){$(B)$}}

\put(10,0){\line(0,1){7}}
\put(11,0){\line(0,1){7}}
\put(10,0){\line(1,0){1}}
\put(10,1){\line(1,0){1}}
\put(10,2){\line(1,0){1}}
\put(10,3){\line(1,0){1}}
\put(10,4){\line(1,0){1}}
\put(10,5){\line(1,0){1}}
\put(10,6){\line(1,0){1}}
\put(10,7){\line(1,0){1}}

\put(10,0){\makebox(1,1){$\Bar{7}$}}
\put(10,1){\makebox(1,1){$\Bar{3}$}}
\put(10,2){\makebox(1,1){$\Bar{5}$}}
\put(10,3){\makebox(1,1){$8$}}
\put(10,4){\makebox(1,1){$3$}}
\put(10,5){\makebox(1,1){$7$}}
\put(10,6){\makebox(1,1){$2$}}

\end{picture}
\caption{Example of the second kind algorithm for $\phi$.}
\label{fig:algorithm2}
\end{figure}
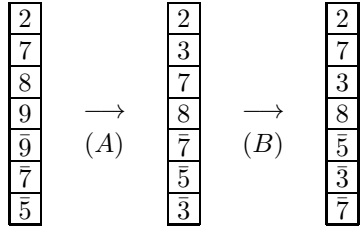

From the above procedure, the following result is obvious.

\begin{lem} \label{lem:col_sst1}
Suppose that $C\in C_{n}\text{-}\mathrm{Col}_{\mathrm{KN}}$.
If $l(\in \mathscr{L})$ lies at the $p$-th position in $C$,  
then the entry in the $p$-th position in $\phi(C)$ is strictly smaller than $l$.
Likewise, if $\overline{l}(\in\Bar{\mathscr{L}})$ lies at the $q$-th position in $C$, 
then the entry at the $q$-th position in $\phi(C)$ is strictly larger than $\Bar{l}$.
Furthermore, let $C_{+}$ (resp. $C_{-}$) be the $\mathscr{C}_{n}^{(+)}$ (resp. $\mathscr{C}_{n}^{(-)}$)-letters part of $C$ 
and $C_{+}^{\ast}$ (resp. $C_{-}^{\ast}$) be the $\mathscr{C}_{n}^{(+)}$ (resp. $\mathscr{C}_{n}^{(-)}$) part of $\phi(C)$.
Then we have $C_{+}^{\ast}\preceq C_{+}$ and $C_{-}\preceq C_{-}^{\ast}$.
\end{lem}

Similarly, we have the following.

\begin{lem} \label{lem:col_sst2}
Suppose that $C \in C_{n}\text{-}\mathrm{Col}_{\overline{\mathrm{KN}}}$.
If $l(\in \mathscr{L})$ lies at the $p$-th position in $C$,  
then the entry in the $p$-th position in $\psi(C)$ is strictly larger than $l$.
Likewise, if $\overline{l}(\in\Bar{\mathscr{L}})$ lies at the $q$-th position in $C$, 
then the entry at the $q$-th position in $\psi(C)$ is strictly smaller than $\Bar{l}$.
Furthermore, let $C_{+}$ (resp. $C_{-}$) be the $\mathscr{C}_{n}^{(+)}$ (resp. $\mathscr{C}_{n}^{(-)}$)-letters part of $C$ 
and $C_{+}^{\dag}$ (resp. $C_{-}^{\dag}$) be the $\mathscr{C}_{n}^{(+)}$ (resp. $\mathscr{C}_{n}^{(-)}$) part of $\psi(C)$.
Then we have $C_{+}\preceq C_{+}^{\dag}$ and $C_{-}^{\dag}\preceq C_{-}$.
\end{lem}

\begin{lem} \label{lem:coKN_lower}
Suppose that $C\in C_{n}\text{-}\mathrm{Col}_{\mathrm{KN}}(N)$.
Let $\{l_{1},\ldots,l_{c}\}$ be the set of $\mathscr{L}$-letters of $C$ and 
$\{l_{1}^{\ast},\ldots,l_{c}^{\ast}\}$ be the set of the corresponding $\mathscr{L}^{\ast}$-letters .
Let $p_{k}^{\ast}$ (resp. $q_{k}^{\ast}$) be the position of $l_{k}^{\ast}$ (resp. $\overline{l_{k}^{\ast}}$) in $\phi(C)$ .
Then we have
\[
q_{k}^{\ast}-p_{k}^{\ast}+l_{k}^{\ast}\geq N-\gamma_{k}^{\ast},
\]
where
$\gamma_{k}^{\ast}:=\sharp\left\{l\in \mathscr{L} \relmiddle| l_{k}^{\ast}<l<l_{k}\right\}$
$(k=c,c-1,\ldots,1)$.
\end{lem}

\begin{proof}
We proceed by induction on $k=c,c-1,\ldots,1$.
We follow the algorithm of the first kind for $\phi$ here.
Let $p_{i}$ (resp. $q_{i}$) be the position of $l_{i}$ (resp. $\overline{l_{i}}$) in $C$ and 
$p_{i}^{\ast}$ (resp. $q_{i}^{\ast}$) be the position of $l_{i}^{\ast}$ (resp. $\overline{l_{i}^{\ast}}$) in 
$\phi (C)$ ($i=1,2,\ldots,c$).

\textbf{(I).}
For $k=c$, the filling diagram of the initial column $C$ has the following configuration.

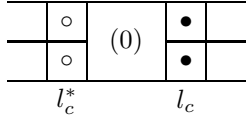
\setlength{\unitlength}{15pt}
\begin{center}
\begin{picture}(6,3)
\put(1,1){\line(0,1){2}}
\put(2,1){\line(0,1){2}}
\put(4,1){\line(0,1){2}}
\put(5,1){\line(0,1){2}}
\put(0,1){\line(1,0){6}}
\put(0,2){\line(1,0){2}}
\put(4,2){\line(1,0){2}}
\put(0,3){\line(1,0){6}}
\put(1,1){\makebox(1,1){$\circ$}}
\put(1,2){\makebox(1,1){$\circ$}}
\put(4,1){\makebox(1,1){$\bullet$}}
\put(4,2){\makebox(1,1){$\bullet$}}
\put(1,0){\makebox(1,1){$l_{c}^{\ast}$}}
\put(4,0){\makebox(1,1){$l_{c}$}}
\put(2,1){\makebox(2,2){$(0)$}}
\end{picture}.
\end{center}
Region $(0)$ consists of $(+)$-slots, $(-)$-slots, and $(\pm)$-slots.
The $(\times)$-slots and $\emptyset$-slots do not exist in this region.
Let us assume that the numbers of $(+)$-slots and  $(-)$-slots are
$\alpha$ and $\beta$, respectively.
The number of $(\pm)$-slots in this region is $\gamma_{c}^{\ast}$.
Then we have 
$p_{c}^{\ast}=p_{c}-\alpha$, $q_{c}^{\ast}=q_{c}+\beta$ by Lemma~\ref{lem:position1}, and 
$l_{c}^{\ast}=l_{c}-(\alpha+\beta+\gamma_{c}^{\ast})-1$ so that 
$q_{c}^{\ast}-p_{c}^{\ast}+l_{c}^{\ast}=q_{c}-p_{c}+l_{c}-\gamma_{c}^{\ast}-1\geq N-\gamma_{c}^{\ast}$, 
where the last inequality is due to the KN-admissibility, $q_{c}-p_{c}+l_{c}\geq N+1$.

\textbf{(II).}
Suppose that $\mathscr{L}$-letters, $l_{c},\ldots,l_{k+1}$ are transformed to the corresponding 
$\mathscr{L}^{\ast}$-letters, $l_{c}^{\ast},\ldots,l_{k+1}^{\ast}$ 
and relocated in the column ($k=c-1,\ldots,1$).
If $l_{k+1}^{\ast}>l_{k}$, then the situation is the same as in \textbf{(I)} so that we have 
$q_{k}^{\ast}-p_{k}^{\ast}+l_{k}^{\ast}\geq N-\gamma_{k}^{\ast}$.
If $l_{k+1}^{\ast}<l_{k}$, then the filling diagram of the updated column has the following configuration.

\setlength{\unitlength}{15pt}
\begin{center}
\begin{picture}(15,3)
\put(1,1){\line(0,1){2}}
\put(2,1){\line(0,1){2}}
\put(4,1){\line(0,1){2}}
\put(5,1){\line(0,1){2}}
\put(7,1){\line(0,1){2}}
\put(8,1){\line(0,1){2}}
\put(10,1){\line(0,1){2}}
\put(11,1){\line(0,1){2}}
\put(13,1){\line(0,1){2}}
\put(14,1){\line(0,1){2}}
\put(0,1){\line(1,0){15}}
\put(0,3){\line(1,0){15}}
\put(0,2){\line(1,0){2}}
\put(4,2){\line(1,0){1}}
\put(7,2){\line(1,0){1}}
\put(10,2){\line(1,0){1}}
\put(13,2){\line(1,0){2}}
\put(2,1){\makebox(2,2){$(0)$}}
\put(1,1){\makebox(1,1){$\circ$}}
\put(1,2){\makebox(1,1){$\circ$}}
\put(4,1){\makebox(1,1){$\times$}}
\put(4,2){\makebox(1,1){$\times$}}
\put(7,1){\makebox(1,1){$\bullet$}}
\put(7,2){\makebox(1,1){$\bullet$}}
\put(10,1){\makebox(1,1){$\bullet$}}
\put(10,2){\makebox(1,1){$\bullet$}}
\put(13,1){\makebox(1,1){$\circ$}}
\put(13,2){\makebox(1,1){$\circ$}}
\put(1,0){\makebox(1,1){$l_{k}^{\ast}$}}
\put(4,0){\makebox(1,1){$l_{k+1}^{\ast}$}}
\put(6,0){\makebox(3,1){$l_{k+1-\gamma_{k+1}^{\ast}}$}}
\put(10,0){\makebox(1,1){$l_{k}$}}
\put(13,0){\makebox(1,1){$l_{k+1}$}}
\put(8,1){\makebox(2,2){$\cdots$}}
\end{picture}.
\end{center}
There are no $\emptyset$-slots between the $l_{k}^{\ast}$-th slot and the $l_{k+1}$-th slot but 
are $\gamma _{k+1}^{\ast}$ $(\pm)$-slots between the $l_{k+1}^{\ast}$-th slot and the $l_{k+1}$-th slot.
Let us assume that region $(0)$ contains $\gamma_{0}$ $(\pm)$-slots and that the total number of $(+)$ and that of $(-)$ 
between the $l_{k}^{\ast}$-slot and the $l_{k}$-th slot are $\alpha$ and $\beta$, respectively. 
Then we have
$p_{k}^{\ast}=p_{k}-\alpha$, $q_{k}^{\ast}=q_{k}+\beta$ by Lemma~\ref{lem:position1}, and 
$l_{k}^{\ast}=l_{k}-(\alpha+\beta+\gamma_{0}+\gamma_{k+1}^{\ast}-1)-1$ so that
$q_{k}^{\ast}-p_{k}^{\ast}+l_{k}^{\ast}=q_{k}-p_{k}+l_{k}-(\gamma_{0}+\gamma_{k+1}^{\ast})$.
Since $\gamma_{k}^{\ast}=\gamma_{k+1}^{\ast}-1+\gamma_{0}$, we have 
$q_{k}^{\ast}-p_{k}^{\ast}+l_{k}^{\ast}\geq N-\gamma_{k}^{\ast}$.
From \textbf{(I)} and \textbf{(II)}, the claim follows. 
\end{proof}

The following result may be proven in much the same way as in Lemma~\ref{lem:coKN_lower}.

\begin{lem} \label{lem:KN_upper}
Suppose that $C\in C_{n}\text{-}\mathrm{Col}_{\overline{\mathrm{KN}}}$.
Let $\{l_{1},\ldots,l_{c}\}$ be the set of $\mathscr{L}$-letters in $C$ and 
$\{l_{1}^{\dag},\ldots,l_{c}^{\dag}\}$ be the set of corresponding $\mathscr{L}^{\dag}$-letters.
Let $p_{k}^{\dag}$ (resp. $q_{k}^{\dag}$) be the position of $l_{k}^{\dag}$ (resp. $\overline{l_{k}^{\dag}}$) in $\psi(C)$ .
Then we have
\[
q_{k}^{\dag}-p_{k}^{\dag}+l_{k}^{\dag}\leq n+\gamma_{k}^{\dag}+1,
\]
where
$\gamma_{k}^{\dag}:=\sharp\left\{l\in \mathscr{L} \relmiddle| l_{k}<l<l_{k}^{\dag}\right\}$
$(k=1,2,,\ldots,c)$.
\end{lem}

\section{Main Theorem I} \label{sec:main1}

Let us begin by giving some definitions.
For $T\in C_{n}\text{-}\mathrm{SST}$ ($T$ is not necessarily KN-admissible), 
we write $T=C_{1}C_{2}\ldots C_{n_{c}}$, where
$C_{x}$ ($x=1,2,\ldots,n_{c}$) is the $x$-th column (from the left) of $T$.

\begin{df} \label{df:Phi}
For $T=C_{1}C_{2}\cdots C_{n_{c}}\in C_{n}\text{-}\mathrm{SST}$, 
let $C_{-}^{(x)}$ (resp. $C_{+}^{(y)}$) be the $\mathscr{C}_{n}^{(-)}$ (resp. $\mathscr{C}_{n}^{(+)}$)-letters part of the $x$-th (resp. $y$-th) column of $T$ and let $C^{(x,y)}$ be the $C_{n}$-column 
whose $\mathscr{C}_{n}^{(-)}$ (resp. $\mathscr{C}_{n}^{(+)}$)-letters part is $C_{-}^{(x)}$ (resp. $C_{+}^{(y)}$).
Let $C_{-}^{(x)\ast}$ (resp. $C_{+}^{(y)\ast}$) be the $\mathscr{C}_{n}^{(-)}$ (resp. $\mathscr{C}_{n}^{(+)}$)-letters part of 
$\phi(C^{(x,y)})$ assuming that $C^{(x,y)}\in C_{n}\text{-}\mathrm{Col}_{\mathrm{KN}}$.
Replace $C_{-}^{(x)}$ (resp. $C_{+}^{(y)}$) in $T$ by $C_{-}^{(x)\ast}$ (resp. $C_{+}^{(y)\ast}$) and denote by $T^{\ast}$ the resulting tableau.
Then we define 
\[
\phi^{(x,y)}(T):=
\begin{cases}
T^{\ast} & (C^{(x,y)}\in C_{n}\text{-}\mathrm{Col}_{\mathrm{KN}}),\\
\emptyset & (otherwise),
\end{cases}
\]
and $\phi^{(x,y)}(\emptyset):=\emptyset$.
Using these maps, we define $\Phi^{(x)}:=\phi^{(x,n_{c})}\circ\cdots\circ\phi^{(x,x)}$, 
$\overline{\Phi^{(x)}}:=\Phi^{(x)}\circ\cdots\circ\Phi^{(n_{c})}$ ($1\leq x\leq n_{c}$), and
$\Phi :=\overline{\Phi^{(1)}}=\Phi^{(1)}\circ\cdots\circ\Phi^{(n_{c})}$.
\end{df}

Provided that $\Phi$ is well-defined on $T\in C_{n}\text{-}\mathrm{SST}_{\mathrm{KN}}$, i.e., 
$\Phi(T) \neq \emptyset$,  
$\Phi$ preserves the shape and weight of $T$ by construction.

\begin{df} \label{df:decomp}
Suppose that $T\in \mathbf{B}_{n}^{\mathfrak{sp}_{2n}}(\nu)_{\mu}^{\lambda}$.
Let $\Phi(T)^{(+)}$ be the part filled with 
$\mathscr{C}_{n}^{(+)}$-letters in $C_{n}$-semistandard tableau $\Phi(T)$, 
which is a semistandard tableau on some Young diagram.
On the other hand, let $\Phi(T)^{(-)}$ be the part filled with 
$\mathscr{C}_{n}^{(-)}$-letters in $C_{n}$-semistandard tableau $\Phi(T)$, 
which is a semistandard tableau on some skew Young diagram (a skew semistandard tableau).
For $T,T^{\prime}\in \mathbf{B}_{n}^{\mathfrak{sp}_{2n}}(\nu)_{\mu}^{\lambda}$ we write $T\sim T^{\prime}$, 
if $\Phi(T)^{(+)}=\Phi(T^{\prime})^{(+)}$ and 
$\mathrm{Rect}\left(  \Phi(T)^{(-)}\right)  =\mathrm{Rect}\left(  \Phi(T^{\prime})^{(-)}\right)$, 
where $\mathrm{Rect}(S)$ denotes the rectification of the skew semistandard tableau $S$~\cite{F} 
with the total order $\prec$.
\end{df}

\begin{thm} \label{thm:main1}
For all $T\in \mathbf{B}_{n}^{\mathfrak{sp}_{2n}}(\nu)_{\mu}^{\lambda}$, $\Phi$ is well-defined on $T$, i.e., 
$\Phi(T)\neq\emptyset$.
Furthermore, if $l(\mu)+l(\nu)\leq n$, we have the following surjection.
\begin{align} \label{eq:main1}
\mathbf{B}_{n}^{\mathfrak{sp}_{2n}}(\nu)_{\mu}^{\lambda} \twoheadrightarrow
{\textstyle\coprod\limits_{\xi,\zeta,\eta\in\mathcal{P}_{n}}}
\mathbf{B}^{(+)}_{n}(\xi)_{\zeta}^{\lambda}\times \mathbf{B}^{(-)}_{n}(\eta)_{\zeta}^{\mu} \\
\left( T  \longmapsto \left(  \Phi(T)^{(+)},\mathrm{Rect}(\Phi(T)^{(-)})\right)  \right). \notag
\end{align}
Hence, we have
\begin{equation} \label{eq:main2}
\mathbf{B}_{n}^{\mathfrak{sp}_{2n}}(\nu)_{\mu}^{\lambda}/\sim \simeq
{\textstyle\coprod\limits_{\xi,\zeta,\eta\in\mathcal{P}_{n}}}
\mathbf{B}^{(+)}_{n}(\xi)_{\zeta}^{\lambda}\times \mathbf{B}^{(-)}_{n}(\eta)_{\zeta}^{\mu}.
\end{equation}
\end{thm}

\begin{rem}
$\left\vert \mathbf{B}_{n}^{\mathfrak{sp}_{2n}}(\nu)_{\mu}^{\lambda}\right\vert =d_{\mu\nu}^{\lambda}$ and 
$\left\vert \mathbf{B}_{n}^{(+)}(\xi)_{\zeta}^{\lambda}\right\vert =c_{\xi\zeta}^{\lambda}$. 
In the stable region, i.e., $l(\mu)+l(\nu)\leq n$, 
$\left\vert \mathbf{B}_{n}^{(-)}(\eta)_{\zeta}^{\mu}\right\vert$ must be 
$c_{\zeta\eta}^{\mu}$.
This is explained as follows.
Let the shape of $\Phi(T)^{(-)}$ be $\nu /\xi$ and 
$\mathrm{Rect}(\Phi(T)^{(-)}) \in \mathbf{B}_{n}^{(-)}(\eta)_{\zeta}^{\mu}$.
The number of tableaux $T$ satisfying the condition of Definition~\ref{df:decomp} is given by the cardinality of the set
\[
\{\text{skew tableaux } S \text{ on } \nu /\xi \text{ such that } \mathrm{Rect}(S)=\mathrm{Rect}(\Phi(T)^{(-)}) \},
\]
which is the LR coefficient $c_{\eta\xi}^{\nu}$~\cite{F} so that $\left\vert \mathbf{B}_{n}^{(-)}(\eta)_{\zeta}^{\mu}\right\vert =c_{\zeta\eta}^{\mu}$ by the branching rule~\eqref{eq:KK1}.
\end{rem}

\begin{ex} \label{ex:main}
Let $\lambda=(3,3,1)$, $\mu=(3,3)$, and $\nu=(3,2,1,1)$, 
$\mathbf{B}_{n}^{\mathfrak{sp}_{2n}}(\nu)_{\mu}^{\lambda}$ consists of four elements shown below ($d_{\mu\nu}^{\lambda}=4$).

\setlength{\unitlength}{12pt}
\[
\begin{picture}(5,4)
\put(2,0){\line(0,1){4}}
\put(3,0){\line(0,1){4}}
\put(4,2){\line(0,1){2}}
\put(5,3){\line(0,1){1}}
\put(2,0){\line(1,0){1}}
\put(2,1){\line(1,0){1}}
\put(2,2){\line(1,0){2}}
\put(2,3){\line(1,0){3}}
\put(2,4){\line(1,0){3}}
\put(2,0){\makebox(1,1){$\Bar{3}$}}
\put(2,1){\makebox(1,1){$\Bar{4}$}}
\put(2,2){\makebox(1,1){$3$}}
\put(2,3){\makebox(1,1){$2$}}
\put(3,2){\makebox(1,1){$4$}}
\put(3,3){\makebox(1,1){$3$}}
\put(4,3){\makebox(1,1){$\Bar{2}$}}
\put(0,1.5){\makebox(2,1){$T_{1}=$}}
\end{picture},
\begin{picture}(5,4)
\put(2,0){\line(0,1){4}}
\put(3,0){\line(0,1){4}}
\put(4,2){\line(0,1){2}}
\put(5,3){\line(0,1){1}}
\put(2,0){\line(1,0){1}}
\put(2,1){\line(1,0){1}}
\put(2,2){\line(1,0){2}}
\put(2,3){\line(1,0){3}}
\put(2,4){\line(1,0){3}}
\put(2,0){\makebox(1,1){$\Bar{3}$}}
\put(2,1){\makebox(1,1){$\Bar{4}$}}
\put(2,2){\makebox(1,1){$4$}}
\put(2,3){\makebox(1,1){$2$}}
\put(3,2){\makebox(1,1){$\Bar{2}$}}
\put(3,3){\makebox(1,1){$3$}}
\put(4,3){\makebox(1,1){$3$}}
\put(0,1.5){\makebox(2,1){$T_{2}=$}}
\end{picture},
\begin{picture}(5,4)
\put(2,0){\line(0,1){4}}
\put(3,0){\line(0,1){4}}
\put(4,2){\line(0,1){2}}
\put(5,3){\line(0,1){1}}
\put(2,0){\line(1,0){1}}
\put(2,1){\line(1,0){1}}
\put(2,2){\line(1,0){2}}
\put(2,3){\line(1,0){3}}
\put(2,4){\line(1,0){3}}
\put(2,0){\makebox(1,1){$\Bar{4}$}}
\put(2,1){\makebox(1,1){$4$}}
\put(2,2){\makebox(1,1){$3$}}
\put(2,3){\makebox(1,1){$2$}}
\put(3,2){\makebox(1,1){$\Bar{3}$}}
\put(3,3){\makebox(1,1){$3$}}
\put(4,3){\makebox(1,1){$\Bar{2}$}}
\put(0,1.5){\makebox(2,1){$T_{3}=$}}
\end{picture},
\begin{picture}(6.5,4)
\put(3.5,0){\line(0,1){4}}
\put(4.5,0){\line(0,1){4}}
\put(5.5,2){\line(0,1){2}}
\put(6.5,3){\line(0,1){1}}
\put(3.5,0){\line(1,0){1}}
\put(3.5,1){\line(1,0){1}}
\put(3.5,2){\line(1,0){2}}
\put(3.5,3){\line(1,0){3}}
\put(3.5,4){\line(1,0){3}}
\put(3.5,0){\makebox(1,1){$\Bar{4}$}}
\put(3.5,1){\makebox(1,1){$4$}}
\put(3.5,2){\makebox(1,1){$2$}}
\put(3.5,3){\makebox(1,1){$1$}}
\put(4.5,2){\makebox(1,1){$\Bar{1}$}}
\put(4.5,3){\makebox(1,1){$3$}}
\put(5.5,3){\makebox(1,1){$\Bar{2}$}}
\put(0,1.5){\makebox(3,1){$\text{and } T_{4}=$}}
\end{picture}.
\]
By $\Phi$ these elements are mapped to
\setlength{\unitlength}{12pt}
\[
\begin{picture}(6.5,4)
\put(3.5,0){\line(0,1){4}}
\put(4.5,0){\line(0,1){4}}
\put(5.5,2){\line(0,1){2}}
\put(6.5,3){\line(0,1){1}}
\put(3.5,0){\line(1,0){1}}
\put(3.5,1){\line(1,0){1}}
\put(3.5,2){\line(1,0){2}}
\put(3.5,3){\line(1,0){3}}
\put(3.5,4){\line(1,0){3}}
\put(3.5,0){\makebox(1,1){$\Bar{1}$}}
\put(3.5,1){\makebox(1,1){$\Bar{2}$}}
\put(3.5,2){\makebox(1,1){$2$}}
\put(3.5,3){\makebox(1,1){$1$}}
\put(4.5,2){\makebox(1,1){$3$}}
\put(4.5,3){\makebox(1,1){$2$}}
\put(5.5,3){\makebox(1,1){$\Bar{2}$}}
\put(0,1.5){\makebox(2.5,1)[l]{$\Phi(T_{1})=$}}
\end{picture},
\begin{picture}(6.5,4)
\put(3.5,0){\line(0,1){4}}
\put(4.5,0){\line(0,1){4}}
\put(5.5,2){\line(0,1){2}}
\put(6.5,3){\line(0,1){1}}
\put(3.5,0){\line(1,0){1}}
\put(3.5,1){\line(1,0){1}}
\put(3.5,2){\line(1,0){2}}
\put(3.5,3){\line(1,0){3}}
\put(3.5,4){\line(1,0){3}}
\put(3.5,0){\makebox(1,1){$\Bar{1}$}}
\put(3.5,1){\makebox(1,1){$\Bar{2}$}}
\put(3.5,2){\makebox(1,1){$2$}}
\put(3.5,3){\makebox(1,1){$1$}}
\put(4.5,2){\makebox(1,1){$\Bar{2}$}}
\put(4.5,3){\makebox(1,1){$2$}}
\put(5.5,3){\makebox(1,1){$3$}}
\put(0,1.5){\makebox(2.5,1)[l]{$\Phi(T_{2})=$}}
\end{picture},
\begin{picture}(6.5,4)
\put(3.5,0){\line(0,1){4}}
\put(4.5,0){\line(0,1){4}}
\put(5.5,2){\line(0,1){2}}
\put(6.5,3){\line(0,1){1}}
\put(3.5,0){\line(1,0){1}}
\put(3.5,1){\line(1,0){1}}
\put(3.5,2){\line(1,0){2}}
\put(3.5,3){\line(1,0){3}}
\put(3.5,4){\line(1,0){3}}
\put(3.5,0){\makebox(1,1){$\Bar{1}$}}
\put(3.5,1){\makebox(1,1){$3$}}
\put(3.5,2){\makebox(1,1){$2$}}
\put(3.5,3){\makebox(1,1){$1$}}
\put(4.5,2){\makebox(1,1){$\Bar{2}$}}
\put(4.5,3){\makebox(1,1){$2$}}
\put(5.5,3){\makebox(1,1){$\Bar{2}$}}
\put(0,1.5){\makebox(2.5,1)[l]{$\Phi(T_{3})=$}}
\end{picture},
\]
and
\begin{center}
\begin{picture}(6.5,4)
\put(3.5,0){\line(0,1){4}}
\put(4.5,0){\line(0,1){4}}
\put(5.5,2){\line(0,1){2}}
\put(6.5,3){\line(0,1){1}}
\put(3.5,0){\line(1,0){1}}
\put(3.5,1){\line(1,0){1}}
\put(3.5,2){\line(1,0){2}}
\put(3.5,3){\line(1,0){3}}
\put(3.5,4){\line(1,0){3}}
\put(3.5,0){\makebox(1,1){$\Bar{2}$}}
\put(3.5,1){\makebox(1,1){$3$}}
\put(3.5,2){\makebox(1,1){$2$}}
\put(3.5,3){\makebox(1,1){$1$}}
\put(4.5,2){\makebox(1,1){$\Bar{1}$}}
\put(4.5,3){\makebox(1,1){$2$}}
\put(5.5,3){\makebox(1,1){$\Bar{2}$}}
\put(0,1.5){\makebox(2.5,1)[l]{$\Phi(T_{4})=$}}
\end{picture},
\end{center}
respectively.
In this example, $\mathrm{Rect}\left(  \Phi(T_{i})^{(-)}\right)  $ ($i=1,\ldots,4$) are the same and are given by
\setlength{\unitlength}{12pt}
\begin{center}
\begin{picture}(2,2)
\put(0,0){\line(0,1){2}}
\put(1,0){\line(0,1){2}}
\put(2,1){\line(0,1){1}}
\put(0,0){\line(1,0){1}}
\put(0,1){\line(1,0){2}}
\put(0,2){\line(1,0){2}}
\put(0,0){\makebox(1,1){$\Bar{1}$}}
\put(0,1){\makebox(1,1){$\Bar{2}$}}
\put(1,1){\makebox(1,1){$\Bar{2}$}}
\end{picture}
\end{center}
so that $\eta=(2,1)$ and $\zeta=\mu\lbrack \Bar{2},\Bar{2},\Bar{1}]=(2,1)$ for all $T_{i}$ ($i=1,\ldots,4$).
Since the process $\mu\rightarrow\mu\lbrack \Bar{2},\Bar{2},\Bar{1}]$ is 
\setlength{\unitlength}{10pt}
\begin{center}
\begin{picture}(17,2)
\put(0,0){\line(0,1){2}}
\put(1,0){\line(0,1){2}}
\put(2,0){\line(0,1){2}}
\put(3,0){\line(0,1){2}}
\put(0,0){\line(1,0){3}}
\put(0,1){\line(1,0){3}}
\put(0,2){\line(1,0){3}}
\put(3,0){\makebox(2,2){$\rightarrow$}}
\put(5,0){\line(0,1){2}}
\put(6,0){\line(0,1){2}}
\put(7,0){\line(0,1){2}}
\put(8,1){\line(0,1){1}}
\put(5,0){\line(1,0){2}}
\put(5,1){\line(1,0){3}}
\put(5,2){\line(1,0){3}}
\put(8,0){\makebox(2,2){$\rightarrow$}}
\put(10,0){\line(0,1){2}}
\put(11,0){\line(0,1){2}}
\put(12,1){\line(0,1){1}}
\put(13,1){\line(0,1){1}}
\put(10,0){\line(1,0){1}}
\put(10,1){\line(1,0){3}}
\put(10,2){\line(1,0){3}}
\put(13,0){\makebox(2,2){$\rightarrow$}}
\put(15,0){\line(0,1){2}}
\put(16,0){\line(0,1){2}}
\put(17,1){\line(0,1){1}}
\put(15,0){\line(1,0){1}}
\put(15,1){\line(1,0){2}}
\put(15,2){\line(1,0){2}}
\end{picture},
\end{center}
$\mathrm{FE}\left(  \mathrm{Rect}\left(  \Phi(T_{i})^{(-)}\right)  \right)  $ is smooth on $\mu$ 
($i=1,\ldots,4$).
We observe that $\mathrm{FE}\left(\Phi (T_{4})^{(-)} \right)$, which is not identical to 
$\mathrm{Rect}\left(\mathrm{FE}\left(\Phi (T_{4})^{(-)} \right)\right)$, 
is also smooth on $\mu$.
This is not a mere coincidence; it holds in general (Proposition~\ref{prp:skew}).
We can check that $\mathrm{FE}\left( \Phi(T_{i})^{(+)} \right)$ is smooth on $\zeta$ and 
$\zeta\left[  \mathrm{FE}\left(  \Phi(T_{i})^{(+)}\right)  \right]  =\lambda$ ($i=1,\ldots,4$).
Indeed, the process 
$\zeta\rightarrow\zeta\left[  \mathrm{FE}\left(  \Phi(T_{1})^{(+)}\right)  \right]=\zeta\lbrack2,3,1,2]$ is
\setlength{\unitlength}{10pt}
\begin{center}
\begin{picture}(20,3)
\put(0,1){\line(0,1){2}}
\put(1,1){\line(0,1){2}}
\put(2,2){\line(0,1){1}}
\put(0,1){\line(1,0){1}}
\put(0,2){\line(1,0){2}}
\put(0,3){\line(1,0){2}}
\put(2,1){\makebox(2,2){$\rightarrow$}}
\put(4,1){\line(0,1){2}}
\put(5,1){\line(0,1){2}}
\put(6,1){\line(0,1){2}}
\put(4,1){\line(1,0){2}}
\put(4,2){\line(1,0){2}}
\put(4,3){\line(1,0){2}}
\put(6,1){\makebox(2,2){$\rightarrow$}}
\put(8,0){\line(0,1){3}}
\put(9,0){\line(0,1){3}}
\put(10,1){\line(0,1){2}}
\put(8,0){\line(1,0){1}}
\put(8,1){\line(1,0){2}}
\put(8,2){\line(1,0){2}}
\put(8,3){\line(1,0){2}}
\put(10,1){\makebox(2,2){$\rightarrow$}}
\put(12,0){\line(0,1){3}}
\put(13,0){\line(0,1){3}}
\put(14,1){\line(0,1){2}}
\put(15,2){\line(0,1){1}}
\put(12,0){\line(1,0){1}}
\put(12,1){\line(1,0){2}}
\put(12,2){\line(1,0){3}}
\put(12,3){\line(1,0){3}}
\put(15,1){\makebox(2,2){$\rightarrow$}}
\put(17,0){\line(0,1){3}}
\put(18,0){\line(0,1){3}}
\put(19,1){\line(0,1){2}}
\put(20,1){\line(0,1){2}}
\put(17,0){\line(1,0){1}}
\put(17,1){\line(1,0){3}}
\put(17,2){\line(1,0){3}}
\put(17,3){\line(1,0){3}}
\end{picture},
\end{center}
that of 
$\zeta\rightarrow\zeta\left[  \mathrm{FE}\left(  \Phi(T_{2})^{(+)}\right)  \right]=\zeta\lbrack3,2,1,2]$ is
\setlength{\unitlength}{10pt}
\begin{center}
\begin{picture}(20,3)
\put(0,1){\line(0,1){2}}
\put(1,1){\line(0,1){2}}
\put(2,2){\line(0,1){1}}
\put(0,1){\line(1,0){1}}
\put(0,2){\line(1,0){2}}
\put(0,3){\line(1,0){2}}
\put(2,1){\makebox(2,2){$\rightarrow$}}
\put(4,0){\line(0,1){3}}
\put(5,0){\line(0,1){3}}
\put(6,2){\line(0,1){1}}
\put(4,0){\line(1,0){1}}
\put(4,1){\line(1,0){1}}
\put(4,2){\line(1,0){2}}
\put(4,3){\line(1,0){2}}
\put(6,1){\makebox(2,2){$\rightarrow$}}
\put(8,0){\line(0,1){3}}
\put(9,0){\line(0,1){3}}
\put(10,1){\line(0,1){2}}
\put(8,0){\line(1,0){1}}
\put(8,1){\line(1,0){2}}
\put(8,2){\line(1,0){2}}
\put(8,3){\line(1,0){2}}
\put(10,1){\makebox(2,2){$\rightarrow$}}
\put(12,0){\line(0,1){3}}
\put(13,0){\line(0,1){3}}
\put(14,1){\line(0,1){2}}
\put(15,2){\line(0,1){1}}
\put(12,0){\line(1,0){1}}
\put(12,1){\line(1,0){2}}
\put(12,2){\line(1,0){3}}
\put(12,3){\line(1,0){3}}
\put(15,1){\makebox(2,2){$\rightarrow$}}
\put(17,0){\line(0,1){3}}
\put(18,0){\line(0,1){3}}
\put(19,1){\line(0,1){2}}
\put(20,1){\line(0,1){2}}
\put(17,0){\line(1,0){1}}
\put(17,1){\line(1,0){3}}
\put(17,2){\line(1,0){3}}
\put(17,3){\line(1,0){3}}
\end{picture},
\end{center}
and that of 
$\zeta\rightarrow\zeta\left[  \mathrm{FE}\left(  \Phi(T_{3})^{(+)}\right)  \right]=
\zeta\left[  \mathrm{FE}\left(  \Phi(T_{4})^{(+)}\right)  \right]=\zeta\lbrack2,1,2,3]$ is
\setlength{\unitlength}{10pt}
\begin{center}
\begin{picture}(21,3)
\put(0,1){\line(0,1){2}}
\put(1,1){\line(0,1){2}}
\put(2,2){\line(0,1){1}}
\put(0,1){\line(1,0){1}}
\put(0,2){\line(1,0){2}}
\put(0,3){\line(1,0){2}}
\put(2,1){\makebox(2,2){$\rightarrow$}}
\put(4,1){\line(0,1){2}}
\put(5,1){\line(0,1){2}}
\put(6,1){\line(0,1){2}}
\put(4,1){\line(1,0){2}}
\put(4,2){\line(1,0){2}}
\put(4,3){\line(1,0){2}}
\put(6,1){\makebox(2,2){$\rightarrow$}}
\put(8,1){\line(0,1){2}}
\put(9,1){\line(0,1){2}}
\put(10,1){\line(0,1){2}}
\put(11,2){\line(0,1){1}}
\put(8,1){\line(1,0){2}}
\put(8,2){\line(1,0){3}}
\put(8,3){\line(1,0){3}}
\put(11,1){\makebox(2,2){$\rightarrow$}}
\put(13,1){\line(0,1){2}}
\put(14,1){\line(0,1){2}}
\put(15,1){\line(0,1){2}}
\put(16,1){\line(0,1){2}}
\put(13,1){\line(1,0){3}}
\put(13,2){\line(1,0){3}}
\put(13,3){\line(1,0){3}}
\put(16,1){\makebox(2,2){$\rightarrow$}}
\put(18,0){\line(0,1){3}}
\put(19,0){\line(0,1){3}}
\put(20,1){\line(0,1){2}}
\put(21,1){\line(0,1){2}}
\put(18,0){\line(1,0){1}}
\put(18,1){\line(1,0){3}}
\put(18,2){\line(1,0){3}}
\put(18,3){\line(1,0){3}}
\end{picture}.
\end{center}
Since $\Phi(T_{3})^{(+)}=\Phi(T_{4})^{(+)}$ and 
$\mathrm{Rect}\left(  \Phi(T_{3})^{(-)}\right)  =\mathrm{Rect}\left(  \Phi(T_{4})^{(-)}\right)  $, 
we have $T_{3}\sim T_{4}$.
\end{ex}

Theorem~\ref{thm:main1} is the immediate consequence of the following two propositions, 
which will be proven in Section~\ref{sec:main11} and Section~\ref{sec:main12}.

\begin{prp} \label{prp:main11}
For all $T\in \mathbf{B}_{n}^{\mathfrak{sp}_{2n}}(\nu)_{\mu}^{\lambda}$, $\Phi(T)\neq\emptyset$ and
\[
\left(  \Phi(T)^{(+)},\mathrm{Rect}\left(\Phi(T)^{(-)}\right)\right)  \in
{\textstyle\coprod\limits_{\xi,\zeta,\eta\in\mathcal{P}_{n}}}
\mathbf{B}_{n}^{(+)}(\xi)_{\zeta}^{\lambda}\times \mathbf{B}_{n}^{(-)}(\eta)_{\zeta}^{\mu}.
\]
\end{prp}

\begin{prp} \label{prp:main12}
Fix $\nu\in\mathcal{P}_{n}$.
For all $(T_{1},T_{2})\in
{\textstyle\coprod\limits_{\xi,\xi,\eta\in\mathcal{P}_{n}}}
\mathbf{B}_{n}^{(+)}(\xi)_{\zeta}^{\lambda}\times \mathbf{B}_{n}^{(-)}(\eta)_{\zeta}^{\mu}$, 
let $T$ be a tableau in $C_{n}\text{-}\mathrm{SST}(\nu)$ such that $T^{(+)}=T_{1}$ and 
$\mathrm{Rect}(T^{(-)})=T_{2}$, where $T^{(+)}$ (resp. $T^{(-)}$) is the part of $T$ filled with 
$\mathscr{C}_{n}^{(+)}$ (resp. $\mathscr{C}_{n}^{(-)}$)-letters.
If $l(\mu)+l(\nu)\leq n$, then we have $\Phi^{-1}(T)\in \mathbf{B}_{n}^{\mathfrak{sp}_{2n}}(\nu)_{\mu}^{\lambda}$.
\end{prp}

\begin{rem}
Keeping the notation in Proposition~\ref{prp:main12}, let $\xi$($\eta$) be the shape of $T_{1}$($T_{2}$).
Then the number of $T$'s satisfying the condition in Proposition~\ref{prp:main12} is given by 
the LR coefficient $c_{\xi,\eta}^{\nu}$~\cite{F}.
In Example~\ref{ex:main}, $c_{(2,2),(2,1)}^{(3,2,1,1)}=c_{(3,1),(2,1)}^{(3,2,1,1)}=1$ and $c_{(2,1,1),(2,1)}^{(3,2,1,1)}=2$.
Thus, we can recover the branching rule (Eq.~\eqref{eq:KK2}).
\end{rem}

We denote by $\Psi$ the inverse of $\Phi$; $\Psi:=\Phi^{-1}$.
This is given explicitly as follows.

\begin{df} \label{df:Psi}
For $T=C_{1}C_{2}\cdots C_{n_{c}}\in C_{n}\text{-}\mathrm{SST}$, 
let $C_{-}^{(x)}$ (resp. $C_{+}^{(y)}$) be the $\mathscr{C}_{n}^{(-)}$ (resp. $\mathscr{C}_{n}^{(+)}$)-letters part of the $x$-th (resp. $y$-th) column of $T$ and let $C^{(x,y)}$ be the $C_{n}$-column 
whose $\mathscr{C}_{n}^{(-)}$ (resp. $\mathscr{C}_{n}^{(+)}$)-letters part is $C_{-}^{(x)}$ (resp. $C_{+}^{(y)}$).
Let $C_{-}^{(x)\dag}$ (resp. $C_{+}^{(y)\dag}$) be the $\mathscr{C}_{n}^{(-)}$ (resp. $\mathscr{C}_{n}^{(+)}$)-letters part of 
$\psi(C^{(x,y)})$ assuming $C^{(x,y)}\in C_{n}\text{-}\mathrm{Col}_{\overline{\mathrm{KN}}}$.
Replace $C_{-}^{(x)}$ (resp. $C_{+}^{(y)}$) in $T$ by $C_{-}^{(x)\dag}$ (resp. $C_{+}^{(y)\dag}$) and denote by $T^{\dag}$ the resulting tableau.
Then we define 
\[
\psi^{(x,y)}(T):=
\begin{cases}
T^{\dag} & (C^{(x,y)}\in C_{n}\text{-}\mathrm{Col}_{\overline{\mathrm{KN}}}),\\
\emptyset & (otherwise),
\end{cases}
\]
and $\psi^{(x,y)}(\emptyset):=\emptyset$.
We define $\Psi^{(x)}:=\psi^{(x,x)}\circ\cdots\circ\psi^{(x,n_{c})}$, 
$\overline{\Psi^{(x)}}:=\Psi^{(x)}\circ\cdots\circ\Psi^{(1)}$ ($1\leq x\leq n_{c}$) and 
$\Psi:=\overline{\Psi^{(n_{c})}}=\Psi^{(n_{c})}\circ\cdots\circ\Psi^{(1)}$.
\end{df}

Provided that $\Psi$ is well-defined on $T\in C_{n}\text{-}\mathrm{SST}_{\overline{\mathrm{KN}}}$, i.e., 
$\Psi(T) \neq \emptyset$,  
$\Psi$ preserves the shape and weight of $T$ by construction.

\begin{lem} \label{lem:equiv1}
Keeping the notation in Definition~\ref{df:Phi}, we can rewrite the map $\Phi^{(x)}$ 
($1 \leq x \leq n_{c}-1$) in the form,  
\begin{equation} \label{eq:reorder}
\phi^{(x,n_{c})}\circ\left(  \phi^{(x+1,n_{c})}\circ\phi^{(x,n_{c}-1)}\right)
\circ\cdots\circ\left(  \phi^{(x+1,x+1)}\circ\phi^{(x,x)}\right)  \circ
(\Phi^{(x+1)})^{-1}.
\end{equation}
\end{lem}

\begin{proof}
The columns updated by $\phi^{(x,n_{c}-i)}$ and those by $\phi^{(x+1,j)}$ have no common columns 
($i=1,\ldots,n_{c}-x;j=n_{c}-i+1,\ldots,n_{c}$).
So we can move maps $\phi^{(x,n_{c}-1)},\phi^{(x,n_{c}-2)},\ldots$ successively to the right of 
$\phi^{(x+1,n_{c})}$ in \eqref{eq:reorder} to obtain
\[
\phi^{(x,n_{c})}\circ\phi^{(x,n_{c}-1)}\circ\cdots\circ\phi^{(x,x)}\circ
\phi^{(x+1,n_{c})}\circ\cdots\circ\phi^{(x+1,x+1)}\circ (\Phi^{(x+1)})^{-1}=
\Phi^{(x)}.
\]
\end{proof}

The following result may be proven in much the same way as in Lemma~\ref{lem:equiv1}.

\begin{lem} \label{lem:equiv2}
Keeping the notation in Definition~\ref{df:Psi}, we can rewrite the map $\Psi^{(x)}$ 
($2\leq x \leq n_{c}$) in the form, 
\[
\left(  \psi^{(x-1,x-1)}\circ\psi^{(x,x)}\right)  \circ\cdots\circ
\left(\psi^{(x-1,n_{c}-1)}\circ\psi^{(x,n_{c})}\right)  \circ
\psi^{(x-1,n_{c})}\circ (\Psi^{(x-1)})^{-1}.
\]
\end{lem}

\section{Properties of $\Phi$} \label{sec:Phi}

In this section, we investigate the properties of the map $\Phi$ and show that $\Phi$ is well-defined on 
$C_{n}\text{-}\mathrm{SST}_{\mathrm{KN}}$ and 
$\Phi(T) \in C_{n}\text{-}\mathrm{SST}(\lambda)$ 
for all $T \in C_{n}\text{-}\mathrm{SST}_{\mathrm{KN}}(\lambda)$.

\begin{lem} \label{lem:welldf1}
Suppose that $T=C_{1}C_{2}\cdots C_{n_{c}}\in C_{n}\text{-}\mathrm{SST}_{\mathrm{KN}}$.
The map $\phi^{(x,y)}$ is well-defined on
\[
\Tilde{T}:=\phi ^{(x,y-1)}\circ \cdots \circ \phi ^{(x,x)}\circ \overline{\Phi ^{(x+1)}}(T)
\quad (2 \leq x+1\leq y\leq n_{c}).
\]
Here, we assume that $\Tilde{T}\neq \emptyset$ and that in the updating process of the tableau from $T$ to $\Tilde{T}$ the semistandardness of the $\mathscr{C}_{n}^{(+)}$-letters part of the tableau is preserved.
\end{lem}

\begin{proof}
Let $C_{-}^{(x)}$ (resp. $C_{+}^{(y)}$) be the $\mathscr{C}_{n}^{(-)}$ (resp. $\mathscr{C}_{n}^{(+)}$)-letters part of the $x$-th (resp. $y$-th) column of $\Tilde{T}$.
Let $C^{(x,y)}$ be the column whose  $\mathscr{C}_{n}^{(+)}$ (resp. $\mathscr{C}_{n}^{(-)}$)-letters part is $C_{+}^{(y)}$ (resp.$C_{-}^{(x)}$).
If $C^{(x,y)}$ is KN-admissible, then we can apply $\phi^{(x,y)}$ to $\Tilde{T}$.
Suppose that $\Tilde{T}$ has the following configuration.

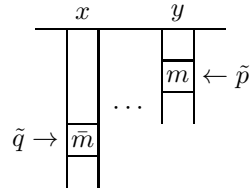
\setlength{\unitlength}{12pt}
\begin{center}
\begin{picture}(8,6)
\put(2,0){\line(0,1){5}}
\put(3,0){\line(0,1){5}}
\put(3,2){\makebox(2,1){$\cdots$}}
\put(5,2){\line(0,1){3}}
\put(6,2){\line(0,1){3}}
\put(2,1){\line(1,0){1}}
\put(2,2){\line(1,0){1}}
\put(5,3){\line(1,0){1}}
\put(5,4){\line(1,0){1}}
\put(1,5){\line(1,0){6}}
\put(2,1){\makebox(1,1){$\Bar{m}$}}
\put(5,3){\makebox(1,1){$m$}}
\put(0,1){\makebox(2,1){$\Tilde{q} \rightarrow$}}
\put(6,3){\makebox(2,1){$\leftarrow \Tilde{p}$}}
\put(2,5){\makebox(1,1){$x$}}
\put(5,5){\makebox(1,1){$y$}}
\end{picture}.
\end{center}
Let $N_{+}$ (resp. $N_{-}$) be the length of the $\mathscr{C}_{n}^{(+)}$ (resp. $\mathscr{C}_{n}^{(-)}$)-letters part of the $y$-th (resp. $x$-th) column and $\Delta N(\geq 0)$ be the difference between the length of the $\mathscr{C}_{n}^{(+)}$-letters part of the $x$-th column and that of the $y$-th column.
Then, $N_{+}+N_{-}+\Delta N=N_{x}$, where $N_{x}$ is the length of the $x$-th column.
In the column $C^{(x,y)}$, $\Bar{m}$ lies at the $(\Tilde{q}-\Delta N)$-th position (from the top).
Hence, if $(\Tilde{q}-\Delta N)-\Tilde{p}+m>N_{+}+N_{-}$, i.e., 
$(\Tilde{q}-\Tilde{p})+m>N_{x}$, then $C^{(x,y)}$ is KN-admissible. 
Let $C_{-}^{(x)\prime}$ be the $\mathscr{C}_{n}^{(-)}$-letters part of the $x$-th column of 
$T^{\prime}:=\psi^{(x,y-1)}(\Tilde{T})$ and 
$C_{+}^{(y-1)\prime}$ be the $\mathscr{C}_{n}^{(+)}$-letters part of the $(y-1)$-st column of $T^{\prime}$.
Let $C^{(x,y-1)}$ be the column whose $\mathscr{C}_{n}^{(+)}$ (resp. $\mathscr{C}_{n}^{(-)}$)-letters part is 
$C_{+}^{(y-1)\prime}$ (resp. $C_{-}^{(x)\prime}$) and $\mathscr{L}^{(x,y-1)}$ be the set of $\mathscr{L}$-letters of $C^{(x,y-1)}$.
We consider the following two cases separately:

\begin{description}
\item[(a)]
$\overline{m}$ appears in the $x$-th column of $T^{\prime}$ and $m\notin \mathscr{L}^{(x,y-1)}$.
\item[(b)]
$\overline{m}$ in the $x$-th column of $\Tilde{T}$ is generated when $\phi ^{(x,y-1)}$ is applied to $T^{\prime}$.
\end{description}

\textbf{Case (a).}
Suppose that the tableau $T^{\prime}$ has the following configuration.

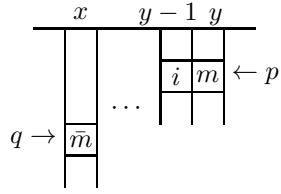
\setlength{\unitlength}{12pt}
\begin{center}
\begin{picture}(9,6)
\put(2,0){\line(0,1){5}}
\put(3,0){\line(0,1){5}}
\put(5,2){\line(0,1){3}}
\put(6,2){\line(0,1){3}}
\put(7,2){\line(0,1){3}}
\put(2,1){\line(1,0){1}}
\put(2,2){\line(1,0){1}}
\put(5,3){\line(1,0){2}}
\put(5,4){\line(1,0){2}}
\put(1,5){\line(1,0){7}}
\put(2,5){\makebox(1,1){$x$}}
\put(4.5,5){\makebox(1.5,1){${\mathstrut y-1}$}}
\put(6.25,5){\makebox(1,1){${\mathstrut y}$}}
\put(3,2){\makebox(2,1){$\cdots$}}
\put(2,1){\makebox(1,1){$\Bar{m}$}}
\put(5,3){\makebox(1,1){$i$}}
\put(6,3){\makebox(1,1){$m$}}
\put(0,1){\makebox(2,1){$q \rightarrow$}}
\put(7,3){\makebox(2,1){$\leftarrow p$}}
\end{picture}.
\end{center}
By the assumption of \textbf{(a)}, $m\notin \mathscr{L}^{(x,y-1)}$ so that $i<m$ 
(if $m \in \mathscr{L}^{(x,y-1)}$, then $\Bar{m}$ in the $x$-th column of $T^{\prime}$ disappears by $\phi^{(x,y-1)}$).
Let us set $\left\{ l \in \mathscr{L}^{(x,y-1)} \relmiddle| \Bar{l}\prec\Bar{m}\prec\overline{l^{\ast}}\right\}
 =:\{l_{r+1},\ldots,l_{r+s}=l_{max}\}$.
If this set is empty ($s=0$), then the position of $\Bar{m}$ does not change 
when $\phi^{(x,y-1)}$ is applied to $T^{\prime}$.
In this case, we have
$(q-p)+\max(i,m)=(q-p)+m>N_{x}$ because $C^{(x,y-1)}$ is KN-admissible ($\Tilde{T}\neq \emptyset$).
This inequality still holds when $\phi ^{(x,y-1)}$ is applied to $T^{\prime}$ so that 
$C^{(x,y)}$ is KN-admissible.
Now suppose that the above set is not empty ($s\geq 1$).
We adopt the second kind algorithm for $\phi^{(x,y-1)}$ here.
Let us assume that $\sharp\left\{ l\in \mathscr{L}^{(x,y-1)} \relmiddle| l_{max}^{\ast}<l<m \right\}=t $.
Since the number of $l$'s such that $l_{max}^{\ast}<l<l_{max}$ ($l\in \mathscr{L}^{(x,y-1)}$) is $s+t-1$, 
we have
\begin{equation} \label{eq:phi_xy_1}
q_{max}^{\ast}-p_{max}^{\ast}+l_{max}^{\ast}\geq N_{x}-(s+t-1)
\end{equation}
by Lemma~\ref{lem:coKN_lower}, 
where $p_{max}^{\ast}$ is the position of $l_{max}^{\ast}$ in the $(y-1)$-st column and $q_{max}^{\ast}$ is the position of $\overline{l_{max}^{\ast}}$ in the $x$-th column of $\phi^{(x,y-1)}(T^{\prime})=\Tilde{T}$.
Initially, the tableau $T^{\prime}$ has the following configuration, 
where the left (resp. right) part is the $\mathscr{C}_{n}^{(-)}$ (resp. $\mathscr{C}_{n}^{(+)}$)-letters one ($i<m<l_{r+1}<\ldots <l_{r+s}=l_{max}$).

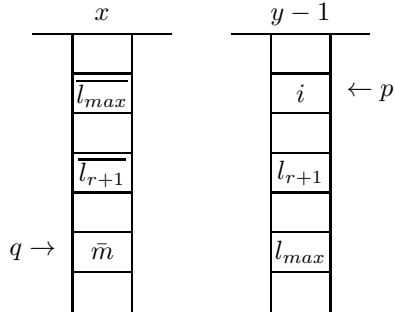
\setlength{\unitlength}{15pt}
\begin{center}
\begin{picture}(10.5,8)
\put(2,0){\line(0,1){7}}
\put(3.5,0){\line(0,1){7}}
\put(2,1){\line(1,0){1.5}}
\put(2,2){\line(1,0){1.5}}
\put(2,3){\line(1,0){1.5}}
\put(2,4){\line(1,0){1.5}}
\put(2,5){\line(1,0){1.5}}
\put(2,6){\line(1,0){1.5}}
\put(1,7){\line(1,0){3.5}}
\put(2,1){\makebox(1.5,1){$\Bar{m}$}}
\put(2,3){\makebox(1.5,1){$\overline{l_{r+1}}$}}
\put(2,5){\makebox(1.5,1){$\overline{l_{max}}$}}
\put(2,7){\makebox(1.5,1){$x$}}
\put(0,1){\makebox(2,1){$q \rightarrow$}}
\put(7,0){\line(0,1){7}}
\put(8.5,0){\line(0,1){7}}
\put(7,1){\line(1,0){1.5}}
\put(7,2){\line(1,0){1.5}}
\put(7,3){\line(1,0){1.5}}
\put(7,4){\line(1,0){1.5}}
\put(7,5){\line(1,0){1.5}}
\put(7,6){\line(1,0){1.5}}
\put(6,7){\line(1,0){3.5}}
\put(7,1){\makebox(1.5,1){$l_{max}$}}
\put(7,3){\makebox(1.5,1){$l_{r+1}$}}
\put(7,5){\makebox(1.5,1){$i$}}
\put(7,7){\makebox(1.5,1){$y-1$}}
\put(8.5,5){\makebox(2,1){$\leftarrow p$}}
\end{picture}.
\end{center}
Let us divide this case further into the following two cases:

\begin{description}
\item[(a-1)]
$i<l_{max}^{\ast}$.
\item[(a-2)]
$l_{max}^{\ast}<i$.
\end{description}
Note that $i\neq l_{max}^{\ast}$ because $i\in C^{(x,y-1)}$ and $l_{max}^{\ast}\notin C^{(x,y-1)}$.

\textbf{Case (a-1).}
The filling diagram of the $C^{(x,y-1)}$ has the following configuration 
before the operation for $l_{max}\rightarrow l_{max}^{\ast}$.

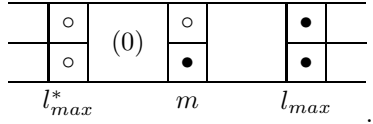
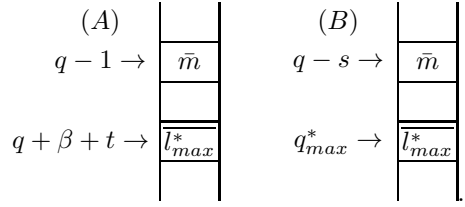
\setlength{\unitlength}{15pt}
\begin{center}
\begin{picture}(9,3)
\put(1,1){\line(0,1){2}}
\put(2,1){\line(0,1){2}}
\put(4,1){\line(0,1){2}}
\put(5,1){\line(0,1){2}}
\put(7,1){\line(0,1){2}}
\put(8,1){\line(0,1){2}}
\put(0,1){\line(1,0){9}}
\put(0,2){\line(1,0){2}}
\put(4,2){\line(1,0){1}}
\put(7,2){\line(1,0){2}}
\put(0,3){\line(1,0){9}}
\put(1,1){\makebox(1,1){$\circ$}}
\put(1,2){\makebox(1,1){$\circ$}}
\put(4,1){\makebox(1,1){$\bullet$}}
\put(4,2){\makebox(1,1){$\circ$}}
\put(7,1){\makebox(1,1){$\bullet$}}
\put(7,2){\makebox(1,1){$\bullet$}}
\put(2,1){\makebox(2,2){$(0)$}}
\put(1,0){\makebox(1,1){$l_{max}^{\ast}$}}
\put(4,0){\makebox(1,1){$m$}}
\put(7,0){\makebox(1,1){$l_{max}$}}
\end{picture}.
\end{center}
Here the number of $(\pm)$-slots in region $(0)$ is $t$.
There are no $\emptyset$-slots in this region.
Also, there are no $(\times)$-slots in this region.
Otherwise, it would contradict the maximality of $l_{max}$ in 
$\left\{  l\in \mathscr{L}^{(x,y-1)} \relmiddle| \Bar{l}\prec\Bar{m}\prec\overline{l^{\ast}}\right\}$.
Let us assume that the number of $(+)$-slots and that of $(-)$-slots in region $(0)$ are 
$\alpha$ and $\beta$, respectively.
Then we have
\begin{equation} \label{eq:phi_xy_2}
l_{max}^{\ast}=m-(\alpha+\beta+t)-1.
\end{equation}
When the operation (A) for $l_{max}\rightarrow l_{max}^{\ast}$ is finished, 
the $(y-1)$-st column of the updated tableau has the left configuration in the figure below.

\setlength{\unitlength}{15pt}
\begin{center}
\begin{picture}(12,7)
\put(3,0){\line(0,1){7}}
\put(5,0){\line(0,1){7}}
\put(3,1){\line(1,0){2}}
\put(3,2){\line(1,0){2}}
\put(3,3){\line(1,0){2}}
\put(3,4){\line(1,0){2}}
\put(3,5){\line(1,0){2}}
\put(3,6){\line(1,0){2}}
\put(3,1){\makebox(2,1){$l_{r+s-1}$}}
\put(3,3){\makebox(2,1){$l_{max}^{\ast}$}}
\put(3,5){\makebox(2,1){$i$}}
\put(1,5){\makebox(3,1)[l]{$p \rightarrow$}}
\put(0.5,6){\makebox(1.5,1){$(A)$}}
\put(10,0){\line(0,1){7}}
\put(12,0){\line(0,1){7}}
\put(10,1){\line(1,0){2}}
\put(10,2){\line(1,0){2}}
\put(10,4){\line(1,0){2}}
\put(10,5){\line(1,0){2}}
\put(10,6){\line(1,0){2}}
\put(10,1){\makebox(2,1){$l_{max}^{\ast}$}}
\put(10,2){\makebox(2,2){$A$}}
\put(10,4){\makebox(2,1){$\vdots$}}
\put(10,5){\makebox(2,1){$i$}}
\put(7,1){\makebox(2,1)[l]{$p_{max}^{\ast} \rightarrow$}}
\put(8,5){\makebox(3,1)[l]{$p \rightarrow$}}
\put(7.5,6){\makebox(1.5,1){$(B)$}}
\end{picture}.
\end{center}
In the operation (B), $s-1$ $\mathscr{L}^{(x,y-1)}$-letters $l_{r+1},\ldots,l_{r+s-1}$ 
together with $t$ $\mathscr{L}^{(x,y-1)}$-letters are 
relocated just above the box containing $l_{max}^{\ast}$ 
so that the $(y-1)$-st column of the updated tableau has the right configuration, 
where $A$ is the block of $s+t-1$ boxes with $\mathscr{L}^{(x,y-1)}$-letters.
Therefore, we have
\begin{equation} \label{eq:phi_xy_3}
p_{max}^{\ast}\geq p+s+t.
\end{equation}
Note that $p_{max}^{\ast}$ does not change under subsequent operations for 
$l_{r+s-1}\rightarrow l_{r+s-1}^{\ast},\ldots,l_{1}\rightarrow l_{1}^{\ast}$.
The $x$-th column of the tableau has the left configuration (A) in the figure below  
when the operation (A) for $l_{max}\rightarrow l_{max}^{\ast}$ is finished.
When the entry $\overline{l_{max}^{\ast}}$ appears below $\Bar{m}$, 
the position of the box containing $\Bar{m}$ is changed from $q$ to $q-1$.
Since there are $\beta + t$ boxes with $\overline{\mathscr{J}^{(x)}}$-letters between the box containing $\Bar{m}$ and 
that containing $\overline{l_{max}^{\ast}}$, 
the position of the box containing $\overline{l_{max}^{\ast}}$ is $q+\beta +t$.

\setlength{\unitlength}{15pt}
\begin{center}
\begin{picture}(11.5,5)
\put(4,0){\line(0,1){5}}
\put(5.5,0){\line(0,1){5}}
\put(4,1){\line(1,0){1.5}}
\put(4,2){\line(1,0){1.5}}
\put(4,3){\line(1,0){1.5}}
\put(4,4){\line(1,0){1.5}}
\put(4,1){\makebox(1.5,1){$\overline{l_{max}^{\ast}}$}}
\put(4,3){\makebox(1.5,1){$\Bar{m}$}}
\put(0,1){\makebox(4,1){$q+\beta+t \rightarrow$}}
\put(1,3){\makebox(3,1){$q-1 \rightarrow$}}
\put(2,4){\makebox(1,1){$(A)$}}
\put(8,4){\makebox(1,1){$(B)$}}
\put(10,0){\line(0,1){5}}
\put(11.5,0){\line(0,1){5}}
\put(10,1){\line(1,0){1.5}}
\put(10,2){\line(1,0){1.5}}
\put(10,3){\line(1,0){1.5}}
\put(10,4){\line(1,0){1.5}}
\put(10,1){\makebox(1.5,1){$\overline{l_{max}^{\ast}}$}}
\put(10,3){\makebox(1.5,1){$\Bar{m}$}}
\put(7,1){\makebox(3,1){$q_{max}^{\ast} \rightarrow$}}
\put(7,3){\makebox(3,1){$q-s \rightarrow$}}
\end{picture}.
\end{center}
When the operation (B) for $l_{max}\rightarrow l_{max}^{\ast}$ is finished, 
the $x$-th column of the updated tableau has the right configuration (B) in the above figure. 
Since $s-1$ $\overline{\mathscr{L}^{(x,y-1)}}$-letters $\overline{l_{r+s-1}},\ldots,\overline{l_{r+1}}$ 
lying above the box containing $\Bar{m}$ before the operation (B) for $l_{max}\rightarrow l_{max}^{\ast}$ are relocated below  
$\overline{l_{max}^{\ast}}$, the position of $\Bar{m}$ is changed from $q-1$ to $q-1-(s-1)=q-s$.
Likewise, the position of the box containing $\overline{l_{max}^{\ast}}$ is changed from $q+\beta +t$ to
\begin{equation} \label{eq:phi_xy_4}
q_{max}^{\ast}=q+\beta+t-(s+t-1)=q+\beta -s+1,
\end{equation}  
which does not change under subsequent operations for  
$l_{r+s-1}\rightarrow l_{r+s-1}^{\ast},\ldots,l_{1}\rightarrow l_{1}^{\ast}$.
From Eqs.~\eqref{eq:phi_xy_1}, \eqref{eq:phi_xy_2}, and \eqref{eq:phi_xy_4}, we have
\begin{equation} \label{eq:phi_xy_5}
(q-s)-p_{max}^{\ast}+m  =q_{max}^{\ast}-p_{max}^{\ast}+l_{max}^{\ast}+\alpha+t
\geq N_{x}-s+\alpha+1.
\end{equation}
Combining Eqs.~\eqref{eq:phi_xy_3} and \eqref{eq:phi_xy_5}, we have
$(q-s)-p+m\geq N_{x}+t+\alpha+1 > N_{x}$.
Here the position of $m$ in the $y$-th column of $\Tilde{T}$ is $p$ and 
that of $\Bar{m}$ in the $x$-th column of $\Tilde{T}$ is $q-s$.
Therefore, $C^{(x,y)}$ is KN-admissible.

\textbf{Case (a-2).}
Let us assume that $i\notin \mathscr{L}^{(x,y-1)}$.
The proof for the case when $i\in \mathscr{L}^{(x,y-1)}$ is similar.
The filling diagram of the column $C^{(x,y-1)}$ has the following configuration 
before the operation for  $l_{max}\rightarrow l_{max}^{\ast}$.

\setlength{\unitlength}{15pt}
\begin{center}
\begin{picture}(12,3)
\put(1,1){\line(0,1){2}}
\put(2,1){\line(0,1){2}}
\put(4,1){\line(0,1){2}}
\put(5,1){\line(0,1){2}}
\put(7,1){\line(0,1){2}}
\put(8,1){\line(0,1){2}}
\put(10,1){\line(0,1){2}}\put(11,1){\line(0,1){2}}
\put(0,1){\line(1,0){12}}
\put(0,2){\line(1,0){2}}
\put(4,2){\line(1,0){1}}
\put(7,2){\line(1,0){1}}
\put(10,2){\line(1,0){2}}
\put(0,3){\line(1,0){12}}
\put(1,1){\makebox(1,1){$\circ$}}
\put(1,2){\makebox(1,1){$\circ$}}
\put(4,1){\makebox(1,1){$\circ$}}
\put(4,2){\makebox(1,1){$\bullet$}}
\put(7,1){\makebox(1,1){$\bullet$}}
\put(7,2){\makebox(1,1){$\circ$}}
\put(10,1){\makebox(1,1){$\bullet$}}
\put(10,2){\makebox(1,1){$\bullet$}}
\put(2,1){\makebox(2,2){$(2)$}}
\put(5,1){\makebox(2,2){$(1)$}}
\put(1,0){\makebox(1,1){$l_{max}^{\ast}$}}
\put(4,0){\makebox(1,1){$i$}}
\put(7,0){\makebox(1,1){$m$}}
\put(10,0){\makebox(1,1){$l_{max}$}}
\end{picture}.
\end{center}
The total number of $(\pm)$-slots in regions $(1)$ and $(2)$ is $t$.
Let us assume that the number of $(\pm)$-slots in region $(1)$ is $t_{1}$.
There are no $\emptyset$-slots in both regions.
Also, there are no $(\times)$-slots in both regions as in \textbf{(a-1)}.
Let us assume that the number of $(+)$-slots and that of $(-)$-slots in region $(j)$ are 
$\alpha_{j}$ and $\beta_{j}$, respectively ($j=1,2$).
Then
\begin{equation} \label{eq:phi_xy_6}
l_{max}^{\ast}=m-\sum_{i=1}^{2}(\alpha_{i}+\beta_{i})-t-2.
\end{equation}
The updated tableau has the following configuration 
when the operation (A) for $l_{max}\rightarrow l_{max}^{\ast}$ is finished.

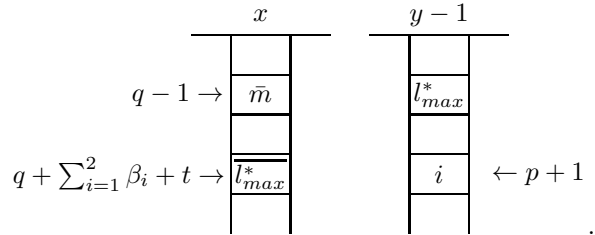
\setlength{\unitlength}{15pt}
\begin{center}
\begin{picture}(14.5,6)
\put(5.5,0){\line(0,1){5}}
\put(7,0){\line(0,1){5}}
\put(5.5,1){\line(1,0){1.5}}
\put(5.5,2){\line(1,0){1.5}}
\put(5.5,3){\line(1,0){1.5}}
\put(5.5,4){\line(1,0){1.5}}
\put(4.5,5){\line(1,0){3.5}}
\put(5.5,1){\makebox(1.5,1){$\overline{l_{max}^{\ast}}$}}
\put(5.5,3){\makebox(1.5,1){$\Bar{m}$}}
\put(5.5,5){\makebox(1.5,1){$x$}}
\put(0,1){\makebox(5.5,1)[l]{$q+\sum_{i=1}^{2}\beta_{i}+t \rightarrow$}}
\put(3,3){\makebox(2.5,1)[l]{$q-1 \rightarrow$}}
\put(10,0){\line(0,1){5}}
\put(11.5,0){\line(0,1){5}}
\put(10,1){\line(1,0){1.5}}
\put(10,2){\line(1,0){1.5}}
\put(10,3){\line(1,0){1.5}}
\put(10,4){\line(1,0){1.5}}
\put(9,5){\line(1,0){3.5}}
\put(10,1){\makebox(1.5,1){$i$}}
\put(10,3){\makebox(1.5,1){$l_{max}^{\ast}$}}
\put(10,5){\makebox(1.5,1){$y-1$}}
\put(12,1){\makebox(2.5,1){$\leftarrow p+1$}}
\end{picture}.
\end{center}
When the operation (B) for $l_{max}\rightarrow l_{max}^{\ast}$ is finished,
the updated tableau has the following configuration.

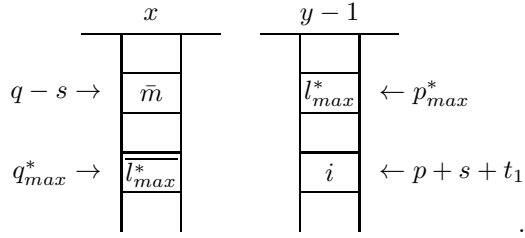
\setlength{\unitlength}{15pt}
\begin{center}
\begin{picture}(13,6)
\put(3,0){\line(0,1){5}}
\put(4.5,0){\line(0,1){5}}
\put(3,1){\line(1,0){1.5}}
\put(3,2){\line(1,0){1.5}}
\put(3,3){\line(1,0){1.5}}
\put(3,4){\line(1,0){1.5}}
\put(2,5){\line(1,0){3.5}}
\put(3,1){\makebox(1.5,1){$\overline{l_{max}^{\ast}}$}}
\put(3,3){\makebox(1.5,1){$\Bar{m}$}}
\put(3,5){\makebox(1.5,1){$x$}}
\put(0,1){\makebox(2.5,1)[r]{$q_{max}^{\ast} \rightarrow$}}
\put(0,3){\makebox(2.5,1)[r]{$q-s \rightarrow$}}
\put(7.5,0){\line(0,1){5}}
\put(9,0){\line(0,1){5}}
\put(7.5,1){\line(1,0){1.5}}
\put(7.5,2){\line(1,0){1.5}}
\put(7.5,3){\line(1,0){1.5}}
\put(7.5,4){\line(1,0){1.5}}
\put(6.5,5){\line(1,0){3.5}}
\put(7.5,1){\makebox(1.5,1){$i$}}
\put(7.5,3){\makebox(1.5,1){$l_{max}^{\ast}$}}
\put(7.5,5){\makebox(1.5,1){$y-1$}}
\put(9.5,1){\makebox(3.5,1)[l]{$\leftarrow p+s+t_{1}$}}
\put(9.5,3){\makebox(3.5,1)[l]{$\leftarrow p_{max}^{\ast}$}}
\end{picture}.
\end{center}
The position of the box containing $i$ in the $(y-1)$-st column is changed from $p+1$ to $p+s+t_{1}$ because $s-1+t_{1}$ $\mathscr{L}^{(x,y-1)}$-letters larger than $i$ are transformed to the corresponding $\mathscr{L}^{(x,y-1)\ast}$-letters and relocated above the box containing $i$.
The position of the box containing $\Bar{m}$ in the $x$-th column is changed from $q-1$ to $q-s$ because $s-1$ $\overline{\mathscr{L}^{(x,y-1)}}$-letters smaller than $\Bar{m}$ are transformed to the corresponding $\overline{\mathscr{L}^{(x,y-1)\ast}}$-letters and relocated below the box containing $\Bar{m}$.
The position of the box containing $\overline{l_{max}^{\ast}}$ in the $x$-th column is changed to
\begin{equation} \label{eq:phi_xy_7}
q_{max}^{\ast}=q+\sum_{i=1}^{2}\beta_{i}+t-(s+t-1)=(q-s)+\sum_{i=1}^{2}\beta_{i}+1,
\end{equation}
because $s-1+t$ $\overline{\mathscr{L}^{(x,y-1)}}$-letters smaller than $\overline{l_{max}^{\ast}}$ are transformed to the corresponding $\overline{\mathscr{L}^{(x,y-1)\ast}}$-letters and relocated below the box containing $\overline{l_{max}^{\ast}}$.
Since $\alpha_{2}$ $\mathscr{I}^{(y-1)}$-letters exist between the box containing $l_{max}^{\ast}$ and that containing $i$ 
in the $(y-1)$-st column,
\begin{equation} \label{eq:phi_xy_8}
p_{max}^{\ast}+\alpha_{2}+1=p+s+t_{1}.
\end{equation}
Note that $p_{max}^{\ast}$ and $q_{max}^{\ast}$ do not change under subsequent operations for 
$l_{r+s-1}\rightarrow l_{r+s-1}^{\ast},\ldots,l_{1}\rightarrow l_{1}^{\ast}$.
From Eqs.~\eqref{eq:phi_xy_1}, \eqref{eq:phi_xy_6}, \eqref{eq:phi_xy_7}, and \eqref{eq:phi_xy_8}, we have
\begin{align*}
(q-s)-p+m  & =q_{max}^{\ast}-p_{max}^{\ast}+l_{max}^{\ast}+\alpha_{1}+s+t+t_{1}\\
& \geq N_{x}+t_{1}+\alpha_{1}+1 > N_{x}.
\end{align*}
Here, the position of the box containing $m$ in the $y$-th column of $\Tilde{T}$ is $p$ and 
that of $\Bar{m}$ in the $x$-th column of $\Tilde{T}$ is $q-s$.
Therefore, $C^{(x,y)}$ is KN-admissible.

\textbf{Case (b).}
In this case, we can write 
$m=l_{i}^{\ast}\in \mathscr{L}^{(x,y-1)\ast}=\{l_{1}^{\ast},l_{2}^{\ast},\ldots,l_{c}^{\ast}\}$.
Let us set
$\{l_{p+1},\ldots,l_{p+r}\}:=\left\{l \in \mathscr{L}^{(x,y-1)} \relmiddle| l_{i}^{\ast}<l<l_{i}\right\}$ 
(if $r=0$, then this set is considered to be empty).
We adopt the first kind algorithm for $\phi^{(x,y-1)}$ here.
When the operation for $l_{i}\rightarrow l_{i}^{\ast}=m$ is finished, 
the updated tableau has the left configuration in the figure below, 
where $A$ is the block consisting of $s$ boxes ($s\geq 1$).

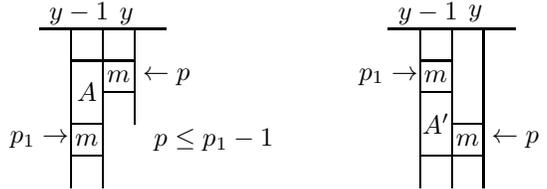
\setlength{\unitlength}{12pt}
\begin{center}
\begin{picture}(17,6)
\put(2,0){\line(0,1){5}}
\put(3,0){\line(0,1){5}}
\put(4,2){\line(0,1){3}}
\put(2,1){\line(1,0){1}}
\put(2,2){\line(1,0){1}}
\put(3,3){\line(1,0){1}}
\put(2,4){\line(1,0){2}}
\put(1,5){\line(1,0){4}}
\put(1.5,5){\makebox(1.5,1){${\mathstrut y-1}$}}
\put(3.25,5){\makebox(1,1){${\mathstrut y}$}}
\put(2,1){\makebox(1,1){$m$}}
\put(2,2){\makebox(1,2){$A$}}
\put(3,3){\makebox(1,1){$m$}}
\put(0,1){\makebox(2,1){$p_{1}\rightarrow$}}
\put(4,3){\makebox(2,1){$\leftarrow p$}}
\put(5,1){\makebox(3,1){$p\leq p_{1}-1$}}

\put(13,0){\line(0,1){5}}
\put(14,0){\line(0,1){5}}
\put(15,0){\line(0,1){5}}
\put(13,1){\line(1,0){2}}
\put(13,3){\line(1,0){1}}
\put(13,4){\line(1,0){1}}
\put(14,2){\line(1,0){1}}
\put(12,5){\line(1,0){4}}
\put(12.5,5){\makebox(1.5,1){$y-1$}}
\put(14.25,5){\makebox(1,1){$y$}}
\put(13,1){\makebox(1,2){$A^{\prime}$}}
\put(13,3){\makebox(1,1){$m$}}
\put(14,1){\makebox(1,1){$m$}}
\put(11,3){\makebox(2,1){$p_{1}\rightarrow$}}
\put(15,1){\makebox(2,1){$\leftarrow p$}}
\end{picture}.
\end{center}
The right configuration is not allowed, 
where $A^{\prime}$ is the block consisting of $s^{\prime}$ boxes ($s^{\prime} \geq 0$).
This can be seen as follows.
Suppose that the entry in the $p_{1}$-th box in the $(y-1)$-st column of the initial tableau $T^{\prime}$ is $j$.
When the operation for $l_{i+1}\rightarrow l_{i+1}^{\ast}$ is finished, 
$l_{i+1}^{\ast},\ldots,l_{c}^{\ast}$ lie below the box containing $j$ in the $(y-1)$-st column
so that the $p_{1}$-th box in the $(y-1)$-st column still has the entry $j$.
The operation for $l_{i}\rightarrow l_{i}^{\ast}$ replaces the entry $j$ with $l_{i}^{\ast}=m$.
This implies that $j>l_{i}^{\ast}=m$ by Lemma~\ref{lem:col_sst1}, 
which contradicts the semistandardness of the $\mathscr{C}_{n}^{(+)}$-letters part of $T^{\prime}$ 
so that the right configuration cannot happen.
When a sequence of operations for $l_{p+r}\rightarrow l_{p+r}^{\ast},\ldots,l_{p+1}\rightarrow l_{p+1}^{\ast}$ 
is finished, 
the position of $m=l_{i}^{\ast}$ in the $(y-1)$-st column becomes to be $p^{\prime}=p_{1}+r$, 
which does not change under subsequent operations.
Since $p\leq p_{1}-1$, we have $p^{\prime}\geq p+r+1$.
On the other hand, by Lemma~\ref{lem:coKN_lower}, we have 
$(q-p^{\prime})+m\geq N_{x}-r$, 
where $q$ is the position of $\Bar{m}=\overline{l_{i}^{\ast}}$ in the $x$-th column of $\Tilde{T}$.
Combining these, we have that $(q-p)+m > N_{x}$, i.e., $C^{(x,y)}$ is KN-admissible.
\end{proof}

\begin{lem} \label{lem:kink1}
Suppose that $T=C_{1}C_{2}\cdots C_{n_{c}} \in C_{n}\text{-}\mathrm{SST}_{\mathrm{KN}}$.
Let us set 
\[
\Tilde{T}=\phi^{(x,y-1)}\circ\cdots\circ\phi^{(x,x)}\circ\overline{\Phi^{(x+1)}}(T) \quad (2 \leq x+1 \leq y \leq n_{c}).
\]
Here, we assume that $\Tilde{T}\neq \emptyset$ and that in the updating process of the tableau from $T$ to $\Tilde{T}$ the semistandardness of the $\mathscr{C}_{n}^{(+)}$-letters part of the tableau is preserved.
\begin{itemize}
\item[(1).]
Suppose that $\Tilde{T}$ has the following configuration, where the left (resp. right) part is the $\mathscr{C}_{n}^{(-)}$ (resp. $\mathscr{C}_{n}^{(+)}$)-letters one $(p\leq q< r\leq s)$.

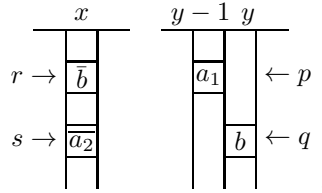
\setlength{\unitlength}{12pt}
\begin{center}
\begin{picture}(10,6)
\put(2,0){\line(0,1){5}}
\put(3,0){\line(0,1){5}}
\put(6,0){\line(0,1){5}}
\put(7,0){\line(0,1){5}}
\put(8,0){\line(0,1){5}}
\put(2,1){\line(1,0){1}}
\put(2,2){\line(1,0){1}}
\put(2,3){\line(1,0){1}}
\put(2,4){\line(1,0){1}}
\put(6,3){\line(1,0){1}}
\put(6,4){\line(1,0){1}}
\put(7,1){\line(1,0){1}}
\put(7,2){\line(1,0){1}}
\put(1,5){\line(1,0){3}}
\put(5,5){\line(1,0){4}}
\put(2,5){\makebox(1,1){$x$}}
\put(5.5,5){\makebox(1.5,1){${\mathstrut y-1}$}}
\put(7.25,5){\makebox(1,1){${\mathstrut y}$}}
\put(2,1){\makebox(1,1){$\overline{a_{2}}$}}
\put(2,3){\makebox(1,1){$\Bar{b}$}}
\put(6,3){\makebox(1,1){$a_{1}$}}
\put(7,1){\makebox(1,1){$b$}}
\put(0,1){\makebox(2,1){$s \rightarrow$}}
\put(0,3){\makebox(2,1){$r \rightarrow$}}
\put(8,1){\makebox(2,1){$\leftarrow q$}}
\put(8,3){\makebox(2,1){$\leftarrow p$}}
\end{picture}.
\end{center}
Then we have 
\[
(q-p)+(s-r)<b-\min(a_{1},a_{2}).
\] 

\item[(2).]
Let $\mathscr{J}^{(x)}$ be the set of $\mathscr{J}$-letters in the $x$-th column and 
$\mathscr{I}^{(y)}$ be the set of $\mathscr{I}$-letters in the $y$-th column and set 
$\mathscr{L}^{(x,y)}:=\mathscr{J}^{(x)}\cap \mathscr{I}^{(y)}$.
If $\sharp\left\{  l\in \mathscr{L}^{(x,y)} \relmiddle| l^{\ast}<b<l\right\}  =\delta$, 
then we have 
\[
(q-p)+(s-r)<b-\min(a_{1},a_{2})-\delta
\]
in the above configuration in $\Tilde{T}$.
\end{itemize}
\end{lem}

\begin{proof}
Note that the tableau $\Tilde{T}$ does not have the following configuration.

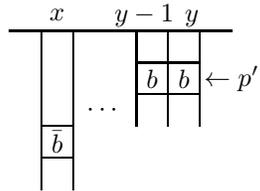
\setlength{\unitlength}{12pt}
\begin{center}
\begin{picture}(8,6)
\put(1,0){\line(0,1){5}}
\put(2,0){\line(0,1){5}}
\put(4,2){\line(0,1){3}}
\put(5,2){\line(0,1){3}}
\put(6,2){\line(0,1){3}}
\put(1,1){\line(1,0){1}}
\put(1,2){\line(1,0){1}}
\put(4,3){\line(1,0){2}}
\put(4,4){\line(1,0){2}}
\put(0,5){\line(1,0){7}}
\put(1,5){\makebox(1,1){$x$}}
\put(3.5,5){\makebox(1.5,1){${\mathstrut y-1}$}}
\put(5.25,5){\makebox(1,1){${\mathstrut y}$}}
\put(1,1){\makebox(1,1){$\Bar{b}$}}
\put(2,2){\makebox(2,1){$\cdots$}}
\put(4,3){\makebox(1,1){$b$}}
\put(5,3){\makebox(1,1){$b$}}
\put(6,3){\makebox(2,1){$\leftarrow p^{\prime}$}}
\end{picture}.
\end{center}
Otherwise, the entry in the $p^{\prime}$-th position in the $(y-1)$-st column of 
$T^{\prime}:=\psi^{(x,y-1)}(\Tilde{T})$ would be strictly larger than $b$ by Lemma~\ref{lem:col_sst2}.
This contradicts the semistandardness of the $\mathscr{C}^{(+)}$-letters part of $T^{\prime}$.
Therefore, the case when $p=q$ and $r=s$ must be excluded.
The case when $r=s$ and $p<q$ must be also excluded because $r=s$ implies $a=b$, 
which contradicts the semistandardness of $\Tilde{T}$.
In particular, $a<b$.

Let us start by proving (1).
Firstly, we set $a_{1}=a_{2}=a$.
Let $C_{-}^{(x)}$ (resp. $C_{+}^{(y-1)}$) be the $\mathscr{C}_{n}^{(-)}$ (resp. $\mathscr{C}_{n}^{(+)}$)-letters part of 
the $x$-th (resp. $(y-1)$-st) column and 
let $C^{(x,y-1)}$ be the column whose $\mathscr{C}_{n}^{(+)}$ (resp. $\mathscr{C}_{n}^{(-)}$)-letters part is 
$C_{+}^{(y-1)}$ (resp. $C_{-}^{(x)}$).
Let $\mathscr{I}^{(y-1)}$ be the set of $\mathscr{I}$-letters in the $(y-1)$-st column in $\Tilde{T}$ and set 
$\mathscr{L}^{(x,y-1)}:=\mathscr{J}^{(x)}\cap \mathscr{I}^{(y-1)}=:\left\{ l_{1},\ldots,l_{c} \right\}$.
Let $l_{k_{0}}$ be the largest $\mathscr{L}^{(x,y-1)}$-letters such that $l_{k_{0}}<b$.
The entry $a$ can be written as $l_{k_{0}-k+1}$ for some $k$ ($k=1,\ldots,k_{0}$).
Let $p_{k}$ (resp. $q_{k}$) be the position of the entry $l_{k_{0}-k+1}$ (resp. $\overline{l_{k_{0}-k+1}}$) in the $(y-1)$-st (resp. $x$-th) column in $\Tilde{T}$.
We proceed by induction on $k$.

\textbf{(I).}
Let $k=1$.
We first consider the case when $p_{1}<q$.
Suppose that the tableau $\Tilde{T}$ has the following configuration $(p_{1}< q< r< s_{1}, k_{0}\in \{1,\ldots,c \})$.

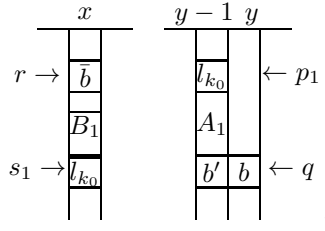
\setlength{\unitlength}{12pt}
\begin{center}
\begin{picture}(10,7)
\put(2,0){\line(0,1){6}}
\put(3,0){\line(0,1){6}}
\put(6,0){\line(0,1){6}}
\put(7,0){\line(0,1){6}}
\put(8,0){\line(0,1){6}}
\put(2,1){\line(1,0){1}}
\put(2,2){\line(1,0){1}}
\put(2,4){\line(1,0){1}}
\put(2,5){\line(1,0){1}}
\put(6,1){\line(1,0){2}}
\put(6,2){\line(1,0){2}}
\put(6,4){\line(1,0){1}}
\put(6,5){\line(1,0){1}}
\put(6,2){\line(1,0){2}}
\put(1,6){\line(1,0){3}}
\put(5,6){\line(1,0){4}}
\put(2,6){\makebox(1,1){$x$}}
\put(5.5,6){\makebox(1.5,1){${\mathstrut y-1}$}}
\put(7.25,6){\makebox(1,1){${\mathstrut y}$}}
\put(2,1){\makebox(1,1){$\overline{l_{k_{0}}}$}}
\put(2,2){\makebox(1,2){$\overline{B_{1}}$}}
\put(2,4){\makebox(1,1){$\Bar{b}$}}
\put(6,1){\makebox(1,1){$b^{\prime}$}}
\put(6,2){\makebox(1,2){$A_{1}$}}
\put(6,4){\makebox(1,1){$l_{k_{0}}$}}
\put(7,1){\makebox(1,1){$b$}}
\put(0,1){\makebox(2,1){$s_{1} \rightarrow$}}
\put(0,4){\makebox(2,1){$r \rightarrow$}}
\put(8,1){\makebox(2,1){$\leftarrow q$}}
\put(8,4){\makebox(2,1){$\leftarrow p_{1}$}}
\end{picture},
\end{center}
where $A_{1}\cap B_{1}=\emptyset$, i.e., $A_{1}$ and $B_{1}$ have no $\mathscr{L}^{(x,y-1)}$-letters.
In this configuration, $l_{k_{0}}<b^{\prime}<b$ and $A_{1}\cap B_{1}=\emptyset$ so that
$\left\vert A_{1}\right\vert +\left\vert B_{1}\right\vert =(q-p_{1}-1)+(s_{1}-r-1)\leq
\left\vert \left\{  l_{k_{0}}+1,\ldots,b-1\right\} \backslash \{ b^{\prime} \} \right\vert $, i.e., $(q-p_{1})+(s_{1}-r)\leq b-l_{k_{0}}$.
Let us assume that
\begin{equation} \label{eq:kink1_1}
(q-p_{1})+(s_{1}-r) = b-l_{k_{0}}.
\end{equation}
We claim that $T^{\prime}=\psi^{(x,y-1)}(\Tilde{T})$ cannot be semistandard 
under the condition of Eq.~\eqref{eq:kink1_1}.
We follow the first kind algorithm for $\psi^{(x,y-1)}$ here.
The filling diagram of the initial column $C^{(x,y-1)}$ has the following configuration.

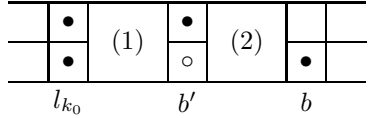
\setlength{\unitlength}{15pt}
\begin{center}
\begin{picture}(9,3)
\put(0,1){\line(1,0){9}}
\put(0,3){\line(1,0){9}}
\put(1,1){\line(0,1){2}}
\put(2,1){\line(0,1){2}}
\put(4,1){\line(0,1){2}}
\put(5,1){\line(0,1){2}}
\put(7,1){\line(0,1){2}}
\put(8,1){\line(0,1){2}}
\put(0,2){\line(1,0){2}}
\put(4,2){\line(1,0){1}}
\put(7,2){\line(1,0){2}}
\put(1,1){\makebox(1,1){$\bullet$}}
\put(1,2){\makebox(1,1){$\bullet$}}
\put(4,1){\makebox(1,1){$\circ$}}
\put(4,2){\makebox(1,1){$\bullet$}}
\put(7,1){\makebox(1,1){$\bullet$}}
\put(2,1){\makebox(2,2){$(1)$}}
\put(5,1){\makebox(2,2){$(2)$}}
\put(1,0){\makebox(1,1){$l_{k_{0}}$}}
\put(4,0){\makebox(1,1){$b^{\prime}$}}
\put(7,0){\makebox(1,1){$b$}}
\end{picture}.
\end{center}
The $b$-th slot is either $(-)$ or $(\pm)$-slot.
Since $A_{1}\cap B_{1}=\emptyset$, regions $(1)$ and $(2)$ have no $(\pm)$-slots 
and the $b^{\prime}$-th slot is $(+)$-slot.
Let us assume that the numbers of $(+)$-slots, $(-)$-slots, and $\emptyset$-slots in region $(i)$ are 
$\alpha_{i}$, $\beta_{i}$, and $\varepsilon_{i}$, respectively $(i=1,2)$.
Then  $q-p_{1}=\alpha_{1}+1$, $s_{1}-r=\sum_{i=1}^{2}\beta_{i}+1$, and 
$b-l_{k_{0}}=\sum_{i=1}^{2}(\alpha_{i}+\beta_{i}+\varepsilon_{i})+2$.
Substituting these into Eq.~\eqref{eq:kink1_1}, we have 
$\alpha_{2}+\varepsilon_{1}+\varepsilon_{2}=0$ so that
$\alpha_{2}=\varepsilon_{1}=\varepsilon_{2}=0$.
Namely, $\emptyset$-slots do not exist in both regions $(1)$ and $(2)$ and 
$(+)$-slots do not exist in region $(2)$.
Set $\left\{  l\in \mathscr{L}^{(x,y-1)} \relmiddle| l<l_{k_{0}}<l^{\dag}\right\} =:\{l_{t+1},\ldots.l_{t+\gamma}\}$ 
(if $\gamma$=0, then this set is considered to be empty).
After these $\mathscr{L}^{(x,y-1)}$-letters are transformed to $\mathscr{L}^{(x,y-1)\dag}$-letters and 
are relocated by $\psi^{(x,y-1)}$, 
the filling diagram of the updated column has the following configuration.
Note that $b<l_{t+1}^{\dag},\ldots,b<l_{t+\gamma}^{\dag}$ because $\varepsilon_{1}=\varepsilon_{2}=0$.

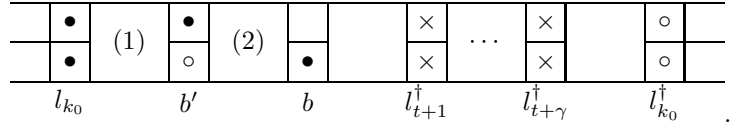
\setlength{\unitlength}{15pt}
\begin{center}
\begin{picture}(18,3)
\put(0,1){\line(1,0){18}}
\put(0,3){\line(1,0){18}}
\put(1,1){\line(0,1){2}}
\put(2,1){\line(0,1){2}}
\put(4,1){\line(0,1){2}}
\put(5,1){\line(0,1){2}}
\put(7,1){\line(0,1){2}}
\put(8,1){\line(0,1){2}}
\put(10,1){\line(0,1){2}}
\put(11,1){\line(0,1){2}}
\put(13,1){\line(0,1){2}}
\put(14,1){\line(0,1){2}}
\put(16,1){\line(0,1){2}}
\put(17,1){\line(0,1){2}}
\put(0,2){\line(1,0){2}}
\put(4,2){\line(1,0){1}}
\put(7,2){\line(1,0){1}}
\put(10,2){\line(1,0){1}}
\put(13,2){\line(1,0){1}}
\put(16,2){\line(1,0){2}}
\put(1,1){\makebox(1,1){$\bullet$}}
\put(1,2){\makebox(1,1){$\bullet$}}
\put(4,1){\makebox(1,1){$\circ$}}
\put(4,2){\makebox(1,1){$\bullet$}}
\put(7,1){\makebox(1,1){$\bullet$}}
\put(10,1){\makebox(1,1){$\times$}}
\put(10,2){\makebox(1,1){$\times$}}
\put(13,1){\makebox(1,1){$\times$}}
\put(13,2){\makebox(1,1){$\times$}}
\put(16,1){\makebox(1,1){$\circ$}}
\put(16,2){\makebox(1,1){$\circ$}}
\put(2,1){\makebox(2,2){$(1)$}}
\put(5,1){\makebox(2,2){$(2)$}}
\put(11,1){\makebox(2,2){$\cdots$}}
\put(1,0){\makebox(1,1){$l_{k_{0}}$}}
\put(4,0){\makebox(1,1){$b^{\prime}$}}
\put(7,0){\makebox(1,1){$b$}}
\put(10,0){\makebox(1,1){$l_{t+1}^{\dag}$}}
\put(13,0){\makebox(1,1){$l_{t+\gamma}^{\dag}$}}
\put(16,0){\makebox(1,1){$l_{k_{0}}^{\dag}$}}
\end{picture}.
\end{center}
Suppose that the operation for $l_{k_{0}} \rightarrow l_{k_{0}}^{\dag}$ is finished.
Then the relocation of the $\gamma +1$ $\mathscr{L}^{(x,y-1)\dag}$-letters, 
$l_{t+1}^{\dag},\ldots,l_{t+\gamma}^{\dag}$, and $l_{k_{0}}^{\dag}$ 
changes the initial position of $b^{\prime}$ from $q$ to $q-(\gamma +1)$. 
If $b\notin \mathscr{L}^{(x,y-1)}$, then $b\notin \mathscr{I}^{(y-1)}$ and the updated tableau has the following configuration, 
which does not change under subsequent operations. 

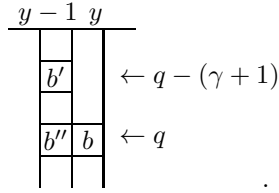
\setlength{\unitlength}{12pt}
\begin{center}
\begin{picture}(8,6)
\put(1,0){\line(0,1){5}}
\put(2,0){\line(0,1){5}}
\put(3,0){\line(0,1){5}}
\put(1,1){\line(1,0){2}}
\put(1,2){\line(1,0){2}}
\put(1,3){\line(1,0){1}}
\put(1,4){\line(1,0){1}}
\put(0,5){\line(1,0){4}}
\put(1,1){\makebox(1,1){$b^{\prime\prime}$}}
\put(2,1){\makebox(1,1){$b$}}
\put(1,3){\makebox(1,1){$b^{\prime}$}}
\put(0.5,5){\makebox(1.5,1){${\mathstrut y-1}$}}
\put(2.25,5){\makebox(1,1){${\mathstrut y}$}}
\put(3.5,1){\makebox(2,1)[l]{$\leftarrow q$}}
\put(3.5,3){\makebox(4.5,1)[l]{$\leftarrow q-(\gamma +1)$}}
\end{picture}.
\end{center}
Any $\mathscr{C}_{n}^{(+)}$-letters larger than $b^{\prime}$ are larger than $b$ 
because there do not exist $(+)$-slots and $(\pm)$-slots in region $(2)$ of the above filling diagram.
Therefore, the entry in the box just below the box containing $b^{\prime}$ in the $(y-1)$-st column is larger than $b$ 
so that 
$b^{\prime\prime}\geq b+1+\gamma$ because there are $\gamma$ boxes between the box containing $b^{\prime}$ and that containing $b^{\prime\prime}$.
This contradicts the semistandardness of the $\mathscr{C}_{n}^{(+)}$-letters part of $T^{\prime}$.
If $b\in \mathscr{L}^{(x,y-1)}$, then the updated tableau has the following configuration.

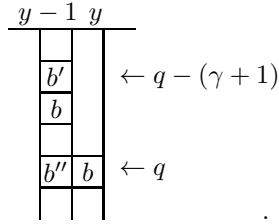
\setlength{\unitlength}{12pt}
\begin{center}
\begin{picture}(8,7)
\put(1,0){\line(0,1){6}}
\put(2,0){\line(0,1){6}}
\put(3,0){\line(0,1){6}}
\put(1,1){\line(1,0){2}}
\put(1,2){\line(1,0){2}}
\put(1,3){\line(1,0){1}}
\put(1,4){\line(1,0){1}}
\put(1,5){\line(1,0){1}}
\put(0,6){\line(1,0){4}}
\put(1,1){\makebox(1,1){$b^{\prime\prime}$}}
\put(2,1){\makebox(1,1){$b$}}
\put(1,3){\makebox(1,1){$b$}}
\put(1,4){\makebox(1,1){$b^{\prime}$}}
\put(0.5,6){\makebox(1.5,1){${\mathstrut y-1}$}}
\put(2.25,6){\makebox(1,1){${\mathstrut y}$}}
\put(3.5,1){\makebox(2,1)[l]{$\leftarrow q$}}
\put(3.5,4){\makebox(4.5,1)[l]{$\leftarrow q-(\gamma +1)$}}
\end{picture}.
\end{center}
After the operation $b\rightarrow b^{\dag}$, the updated configuration turns out to be the same as the previous one and the same argument leads to a contradiction.
Hence, we have $(q-p_{1})+(s_{1}-r)<b-l_{k_{0}}$.

Now let us assume that $p_{1}=q$.
Suppose that the tableau $\Tilde{T}$ has the following configuration.

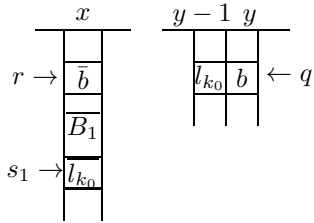
\setlength{\unitlength}{12pt}
\begin{center}
\begin{picture}(10,7)
\put(1.9,0){\line(0,1){6}}
\put(3.1,0){\line(0,1){6}}
\put(6,3){\line(0,1){3}}
\put(7,3){\line(0,1){3}}
\put(8,3){\line(0,1){3}}
\put(1.9,1){\line(1,0){1.2}}
\put(1.9,2){\line(1,0){1.2}}
\put(1.9,4){\line(1,0){1.2}}
\put(1.9,5){\line(1,0){1.2}}
\put(6,4){\line(1,0){2}}
\put(6,5){\line(1,0){2}}
\put(1,6){\line(1,0){3}}
\put(5,6){\line(1,0){4}}
\put(2,6){\makebox(1,1){$x$}}
\put(5.5,6){\makebox(1.5,1){${\mathstrut y-1}$}}
\put(7.25,6){\makebox(1,1){${\mathstrut y}$}}
\put(2,1){\makebox(1,1){$\overline{l_{k_{0}}}$}}
\put(2,2){\makebox(1,2){$\overline{B_{1}}$}}
\put(2,4){\makebox(1,1){$\Bar{b}$}}
\put(6,4){\makebox(1,1){$l_{k_{0}}$}}
\put(7,4){\makebox(1,1){$b$}}
\put(0,1){\makebox(2,1){$s_{1} \rightarrow$}}
\put(0,4){\makebox(2,1){$r \rightarrow$}}
\put(8,4){\makebox(2,1){$\leftarrow q$}}
\end{picture}.
\end{center}
In this configuration, $\left\vert \overline{B_{1}}\right\vert =(s_{1}-r-1)\leq\left\vert \{l_{k_{0}}+1,\ldots,b-1\}\right\vert $, 
i.e., $s_{1}-r\leq b-l_{k_{0}}$.
Let us assume that $s_{1}-r= b-l_{k_{0}}$.
This implies that the block $\overline{B_{1}}$ is filled with consecutive $\overline{\mathscr{J}^{(x)}}$-letters, 
$\overline{b-1},\ldots,\overline{l_{k_{0}}+1}$ (if $l_{k_{0}}+1=b$, then $\overline{B_{1}}$ is empty) so that the filling diagram of the initial column $C^{(x,y-1)}$ has the following configuration.

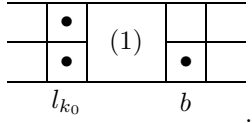
\setlength{\unitlength}{15pt}
\begin{center}
\begin{picture}(6,3)
\put(0,1){\line(1,0){6}}
\put(0,3){\line(1,0){6}}
\put(1,1){\line(0,1){2}}
\put(2,1){\line(0,1){2}}
\put(4,1){\line(0,1){2}}
\put(5,1){\line(0,1){2}}
\put(0,2){\line(1,0){2}}
\put(4,2){\line(1,0){2}}
\put(1,1){\makebox(1,1){$\bullet$}}
\put(1,2){\makebox(1,1){$\bullet$}}
\put(4,1){\makebox(1,1){$\bullet$}}
\put(2,1){\makebox(2,2){$(1)$}}
\put(1,0){\makebox(1,1){$l_{k_{0}}$}}
\put(4,0){\makebox(1,1){$b$}}
\end{picture}.
\end{center}
Region $(1)$ consists of only $(-)$-slots ($l_{k_{0}}+1<b$) or is empty ($l_{k_{0}}+1=b$).
When the operations up to $l_{k_{0}}\rightarrow l_{k_{0}}^{\dag}$ are finished, the entry at the $q$-th position in the $(y-1)$-column is larger than $b$ because region $(1)$ has only $(-)$-slots.
This entry does not change under subsequent operations.
This contradicts the semistandardness of the $\mathscr{C}_{n}^{(+)}$-letters part of $T^{\prime}$.
Hence, we have $s_{1}-r<b-l_{k_{0}}$.

\textbf{(II)}.
We first consider the case when $p_{1}<q$.
We claim that $(q-p_{k+1})+(s_{k+1}-r)<b-l_{k_{0}-k}$ in the following configuration of the tableau $\Tilde{T}$ ($p_{k+1}<p_{k}<q<r<s_{k}<s_{k+1}$).
 
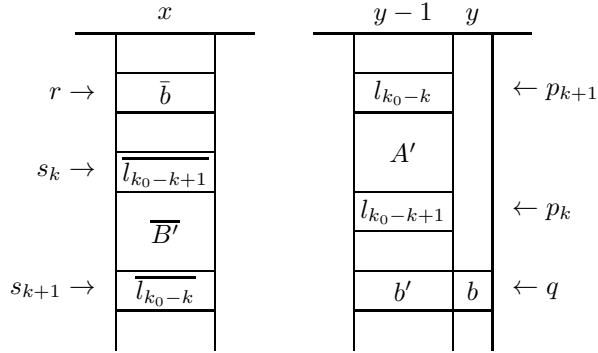
\setlength{\unitlength}{15pt}
\begin{center}
\begin{picture}(16.5,9)
\put(3.5,0){\line(0,1){8}}
\put(6,0){\line(0,1){8}}
\put(9.5,0){\line(0,1){8}}
\put(12,0){\line(0,1){8}}
\put(13,0){\line(0,1){8}}
\put(3.5,1){\line(1,0){2.5}}
\put(3.5,2){\line(1,0){2.5}}
\put(3.5,4){\line(1,0){2.5}}
\put(3.5,5){\line(1,0){2.5}}
\put(3.5,6){\line(1,0){2.5}}
\put(3.5,7){\line(1,0){2.5}}
\put(9.5,1){\line(1,0){3.5}}
\put(9.5,2){\line(1,0){3.5}}
\put(9.5,3){\line(1,0){2.5}}
\put(9.5,4){\line(1,0){2.5}}
\put(9.5,6){\line(1,0){2.5}}
\put(9.5,7){\line(1,0){2.5}}
\put(2.5,8){\line(1,0){4.5}}
\put(8.5,8){\line(1,0){5.5}}
\put(3.5,1){\makebox(2.5,1){$\overline{l_{k_{0}-k}}$}}
\put(3.5,2){\makebox(2.5,2){$\overline{B^{\prime}}$}}
\put(3.5,4){\makebox(2.5,1){$\overline{l_{k_{0}-k+1}}$}}
\put(3.5,6){\makebox(2.5,1){$\Bar{b}$}}
\put(9.5,1){\makebox(2.5,1){$b^{\prime}$}}
\put(9.5,4){\makebox(2.5,2){$A^{\prime}$}}
\put(9.5,3){\makebox(2.5,1){$l_{k_{0}-k+1}$}}
\put(9.5,6){\makebox(2.5,1){$l_{k_{0}-k}$}}
\put(12,1){\makebox(1,1){$b$}}
\put(0,1){\makebox(3,1)[r]{$s_{k+1} \rightarrow$}}
\put(0,4){\makebox(3,1)[r]{$s_{k} \rightarrow$}}
\put(0,6){\makebox(3,1)[r]{$r \rightarrow$}}
\put(13.5,1){\makebox(3,1)[l]{$\leftarrow q$}}
\put(13.5,3){\makebox(3,1)[l]{$\leftarrow p_{k}$}}
\put(13.5,6){\makebox(3,1)[l]{$\leftarrow p_{k+1}$}}
\put(3.5,8){\makebox(2.5,1){$x$}}
\put(9.5,8){\makebox(2.5,1){${\mathstrut y-1}$}}
\put(12,8){\makebox(1,1){${\mathstrut y}$}}
\end{picture}
\end{center}
under the assumption
\begin{equation} \label{eq:kink1_3}
(q-p_{k})+(s_{k}-r)<b-l_{k_{0}-k+1}
\end{equation}
and $b^{\prime}\in \mathscr{I}^{(y-1)}\backslash \mathscr{L}^{(x,y-1)}$.
Since $A^{\prime}\cap B^{\prime}=\emptyset$,
\begin{align*}
\left\vert A^{\prime}\right\vert +\left\vert \overline{B^{\prime}}\right\vert
& =(p_{k}-p_{k+1}-1)+(s_{k+1}-s_{k}-1)\\
& \leq
\begin{cases}
\left\vert \{l_{k_{0}-k}+1,\ldots,l_{k_{0}-k+1}-1\}\right\vert  & (l_{k_{0}-k}+2\leq l_{k_{0}-k+1}),\\
0 & (l_{k_{0}-k}+1=l_{k_{0}-k+1}),
\end{cases}
\end{align*}
i.e., 
\begin{equation} \label{eq:kink1_4}
(p_{k}-p_{k+1})+(s_{k+1}-s_{k})\leq l_{k_{0}-k+1}-l_{k_{0}-k}+1.
\end{equation}
Combining Eqs.~\eqref{eq:kink1_3} and \eqref{eq:kink1_4}, we have $(q-p_{k+1})+(s_{k+1}-r)\leq b-l_{k_{0}-k}$.
Let us assume that
\begin{equation} \label{eq:kink1_5}
(q-p_{k+1})+(s_{k+1}-r)=b-l_{k_{0}-k}.
\end{equation}
The filling diagram of the initial column $C^{(x,y-1)}$ has the following configuration.

\setlength{\unitlength}{15pt}
\begin{center}
\begin{picture}(15,5)
\put(0,1){\line(1,0){15}}
\put(0,3){\line(1,0){15}}
\put(1,1){\line(0,1){2}}
\put(2,1){\line(0,1){2}}
\put(4,1){\line(0,1){2}}
\put(5,1){\line(0,1){2}}
\put(7,1){\line(0,1){2}}
\put(8,1){\line(0,1){2}}
\put(10,1){\line(0,1){2}}
\put(11,1){\line(0,1){2}}
\put(13,1){\line(0,1){2}}
\put(14,1){\line(0,1){2}}
\put(0,2){\line(1,0){2}}
\put(4,2){\line(1,0){1}}
\put(7,2){\line(1,0){1}}
\put(10,2){\line(1,0){1}}
\put(13,2){\line(1,0){2}}
\put(1,1){\makebox(1,1){$\bullet$}}
\put(1,2){\makebox(1,1){$\bullet$}}
\put(4,1){\makebox(1,1){$\bullet$}}
\put(4,2){\makebox(1,1){$\bullet$}}
\put(7,1){\makebox(1,1){$\bullet$}}
\put(7,2){\makebox(1,1){$\bullet$}}
\put(10,1){\makebox(1,1){$\circ$}}
\put(10,2){\makebox(1,1){$\bullet$}}
\put(13,1){\makebox(1,1){$\bullet$}}
\put(11,1){\makebox(2,2){$(2)$}}
\put(0,0){\makebox(3,1){$l_{k_{0}-k}$}}
\put(3,0){\makebox(3,1){$l_{k_{0}-k+1}$}}
\put(7,0){\makebox(1,1){$l_{k_{0}}$}}
\put(10,0){\makebox(1,1){$b^{\prime}$}}
\put(13,0){\makebox(1,1){$b$}}
\put(5,1){\makebox(2,2){$\cdots$}}
\put(2.1,2.5){\makebox(8,1){$
\overbrace{%
\begin{array}
[c]{cccccccccccc}
& & & & & & & & & & &
\end{array}
}$
}}
\put(5,3.5){\makebox(2,1){$(1)$}}
\end{picture}.
\end{center}
Region $(1)$ contains $k$ $(\pm)$-slots and region $(2)$ contains no $(\pm)$-slots. 
Let us assume that the numbers of $(+)$-slots, $(-)$-slots, and $\emptyset$-slots in region $(i)$ are 
$\alpha_{i}$, $\beta_{i}$, and $\varepsilon_{i}$, respectively $(i=1,2)$.
Then $q-p_{k+1}=\alpha_{1}+k+1$, $s_{k+1}-r=\beta_{1}+k+\beta_{2}+1$, and 
$b-l_{k_{0}-k}=\sum_{i=1}^{2}(\alpha_{i}+\beta_{i}+\varepsilon_{i})+k+2$. 
Substituting these into Eq.~\eqref{eq:kink1_5}, we have $k=\alpha_{2}+\varepsilon_{1}+\varepsilon_{2}$.
Therefore, when $\mathscr{L}^{(x,y-1)}$-letters up to $l_{k_{0}}$ are transformed to the corresponding 
$\mathscr{L}^{(x,y-1)\dag}$-letters, at least $k+1-(\varepsilon_{1}+\varepsilon_{2})=\alpha_{2}+1$ of them are larger than $b$.
Suppose that $\sharp\left\{  l\in \mathscr{L}^{(x,y-1)} \relmiddle| l<b<l^{\dag}\right \}  =\gamma$.
Then $\gamma \geq \alpha_{2}+1$.
Suppose that the operation for $l_{k_{0}}\rightarrow l_{k_{0}}^{\dag}$ is finished.
Then the position of the box containing $b^{\prime}$ in the $(y-1)$-st column is changed from $q$ to $q-\gamma^{\prime}$, 
where 
$\gamma^{\prime}=\sharp\left\{  l\in \mathscr{L}^{(x,y-1)} \relmiddle| l<b^{\prime}<l^{\dag}\right \}$.
Since region $(2)$ has $\epsilon _{2}$ $\emptyset$-slots and no $(\pm)$-slots, 
$\gamma^{\prime}=\gamma +\epsilon _{2}$.
The updated tableau has the following configuration.

\setlength{\unitlength}{12pt}
\begin{center}
\begin{picture}(7,6)
\put(3,0){\line(0,1){6}}
\put(4,0){\line(0,1){6}}
\put(5,0){\line(0,1){6}}
\put(3,1){\line(1,0){2}}
\put(3,2){\line(1,0){2}}
\put(3,4){\line(1,0){1}}
\put(3,5){\line(1,0){1}}
\put(2,6){\line(1,0){4}}
\put(3,1){\makebox(1,1){$b^{\prime\prime}$}}
\put(3,2){\makebox(1,2){$C$}}
\put(3,4){\makebox(1,1){$b^{\prime}$}}
\put(4,1){\makebox(1,1){$b$}}
\put(0,1){\makebox(2.5,1)[r]{$q \rightarrow$}}
\put(0,4){\makebox(2.5,1)[r]{$q-\gamma^{\prime} \rightarrow$}}
\put(2.5,6){\makebox(1.5,1){${\mathstrut y-1}$}}
\put(4.25,6){\makebox(1,1){${\mathstrut y}$}}
\end{picture},
\end{center}
where $b^{\prime\prime}\leq b$ and the position of the box containing $b^{\prime}$ does not change under subsequent operations.
Since $\alpha_{2}$ $\mathscr{I}^{(y-1)}$-letters exists between $b^{\prime}$ and $b$, 
$C$ has at most $\alpha_{2}$ letters.
On the other hand, $C$ consists of $\gamma^{\prime}-1$ boxes and $\gamma^{\prime}=\gamma + \varepsilon_{2} \geq \alpha_{2} +1+\varepsilon_{2}$ so that 
$\gamma^{\prime}-1 \geq \alpha_{2}$.
This implies $\gamma^{\prime}-1= \alpha_{2}$ and $b^{\prime\prime}=b$.
Now since $b^{\prime\prime}=b \in \mathscr{L}^{(x,y-1)}$, the entry at the $q$-th position in the $(y-1)$-st column becomes strictly larger than $b$ after the operation $b^{\prime\prime}\rightarrow b^{\dag}$ by Lemma~\ref{lem:col_sst2} and does not change under subsequent operations.
This contradicts the the semistandardness of the $\mathscr{C}_{n}^{(+)}$-letters part of 
$T^{\prime}$.
Hence, we have $(q-p_{k+1})+(s_{k+1}-r)<b-l_{k_{0}-k}$.

Now let us consider the case when $p_{1}=q$.
We claim that $(q-p_{k+1})+(s_{k+1}-r)<b-l_{k_{0}-k}$ in the following configuration of the tableau $\Tilde{T}$ 
($p_{k+1}<p_{k}\leq q<r<s_{k}<s_{k+1}$).

\setlength{\unitlength}{15pt}
\begin{center}
\begin{picture}(16.5,9)
\put(3.5,0){\line(0,1){8}}
\put(6,0){\line(0,1){8}}
\put(9.5,0){\line(0,1){8}}
\put(12,0){\line(0,1){8}}
\put(13,0){\line(0,1){8}}
\put(3.5,1){\line(1,0){2.5}}
\put(3.5,2){\line(1,0){2.5}}
\put(3.5,4){\line(1,0){2.5}}
\put(3.5,5){\line(1,0){2.5}}
\put(3.5,6){\line(1,0){2.5}}
\put(3.5,7){\line(1,0){2.5}}
\put(9.5,1){\line(1,0){3.5}}
\put(9.5,2){\line(1,0){3.5}}
\put(9.5,3){\line(1,0){2.5}}
\put(9.5,4){\line(1,0){2.5}}
\put(9.5,6){\line(1,0){2.5}}
\put(9.5,7){\line(1,0){2.5}}
\put(2.5,8){\line(1,0){4.5}}
\put(8.5,8){\line(1,0){5.5}}
\put(3.5,1){\makebox(2.5,1){$\overline{l_{k_{0}-k}}$}}
\put(3.5,2){\makebox(2.5,2){$\overline{B^{\prime}}$}}
\put(3.5,4){\makebox(2.5,1){$\overline{l_{k_{0}-k+1}}$}}
\put(3.5,6){\makebox(2.5,1){$\Bar{b}$}}
\put(9.5,1){\makebox(2.5,1){$l_{k_{0}}$}}
\put(9.5,4){\makebox(2.5,2){$A^{\prime}$}}
\put(9.5,3){\makebox(2.5,1){$l_{k_{0}-k+1}$}}
\put(9.5,6){\makebox(2.5,1){$l_{k_{0}-k}$}}
\put(12,1){\makebox(1,1){$b$}}
\put(0,1){\makebox(3,1)[r]{$s_{k+1} \rightarrow$}}
\put(0,4){\makebox(3,1)[r]{$s_{k} \rightarrow$}}
\put(0,6){\makebox(3,1)[r]{$r \rightarrow$}}
\put(13.5,1){\makebox(3,1)[l]{$\leftarrow q$}}
\put(13.5,3){\makebox(3,1)[l]{$\leftarrow p_{k}$}}
\put(13.5,6){\makebox(3,1)[l]{$\leftarrow p_{k+1}$}}
\put(3.5,8){\makebox(2.5,1){$x$}}
\put(9.5,8){\makebox(2.5,1){${\mathstrut y-1}$}}
\put(12,8){\makebox(1,1){${\mathstrut y}$}}
\end{picture}
\end{center}
under the assumption
\begin{equation} \label{eq:kink1_6}
(q-p_{k})+(s_{k}-r)<b-l_{k_{0}-k+1}.
\end{equation}
By the same argument as in the case when $p_{1}<q$, we have $(q-p_{k+1})+(s_{k+1}-r)\leq b-l_{k_{0}-k}$.
Let us assume that
\begin{equation} \label{eq:kink1_7}
(q-p_{k+1})+(s_{k+1}-r)=b-l_{k_{0}-k}.
\end{equation}
The filling diagram of the initial column $C^{(x,y-1)}$ has the following configuration.

\setlength{\unitlength}{15pt}
\begin{center}
\begin{picture}(12,5)
\put(0,1){\line(1,0){12}}
\put(0,3){\line(1,0){12}}
\put(1,1){\line(0,1){2}}
\put(2,1){\line(0,1){2}}
\put(4,1){\line(0,1){2}}
\put(5,1){\line(0,1){2}}
\put(7,1){\line(0,1){2}}
\put(8,1){\line(0,1){2}}
\put(10,1){\line(0,1){2}}
\put(11,1){\line(0,1){2}}

\put(0,2){\line(1,0){2}}
\put(4,2){\line(1,0){1}}
\put(7,2){\line(1,0){1}}
\put(10,2){\line(1,0){2}}

\put(1,1){\makebox(1,1){$\bullet$}}
\put(1,2){\makebox(1,1){$\bullet$}}
\put(4,1){\makebox(1,1){$\bullet$}}
\put(4,2){\makebox(1,1){$\bullet$}}
\put(7,1){\makebox(1,1){$\bullet$}}
\put(7,2){\makebox(1,1){$\bullet$}}
\put(10,1){\makebox(1,1){$\bullet$}}

\put(8,1){\makebox(2,2){$(2)$}}
\put(0,0){\makebox(3,1){$l_{k_{0}-k}$}}
\put(3,0){\makebox(3,1){$l_{k_{0}-k+1}$}}
\put(7,0){\makebox(1,1){$l_{k_{0}}$}}
\put(10,0){\makebox(1,1){$b$}}
\put(5,1){\makebox(2,2){$\cdots$}}
\put(2.1,2.5){\makebox(5,1){$
\overbrace{%
\begin{array}
[c]{ccccccc}
& & & & & &  
\end{array}
}$
}}
\put(3.5,3.5){\makebox(2,1){$(1)$}}
\end{picture}.
\end{center}
Region $(1)$ contains $k-1$ $(\pm)$-slots and region $(2)$ contains no $(\pm)$-slots because of the choice of $l_{k_{0}}$.
Let us assume that the numbers of $(+)$-slots, $(-)$-slots, and $\emptyset$-slots in region $(i)$ are 
$\alpha_{i}$, $\beta_{i}$, and $\varepsilon_{i}$, respectively $(i=1,2)$.
Then $q-p_{k+1}=\alpha_{1}+(k-1)+1$, $s_{k+1}-r=\sum_{i=1}^{2}\beta_{i}+k+1$, and 
$b-l_{k_{0}-k}=\sum_{i=1}^{2}(\alpha_{i}+\beta_{i}+\varepsilon_{i})+k+1$. 
Substituting these into Eq.~\eqref{eq:kink1_7}, we have $k=\alpha_{2}+\varepsilon_{1}+\varepsilon_{2}$.
Therefore, when $\mathscr{L}^{(x,y-1)}$-letters up to $l_{k_{0}-1}$ are transformed to the corresponding 
$\mathscr{L}^{(x,y-1)\dag}$-letters, at least $k-\varepsilon_{1}=\alpha_{2}+\varepsilon_{2}$ of them are larger than $l_{k_{0}}$ so that 
$\gamma^{\prime}:=\sharp\left\{  l\in \mathscr{L}^{(x,y-1)} \relmiddle| l<l_{k_{0}}<l^{\dag}\right \}  \geq \alpha_{2} +\varepsilon_{2}$.
The updated tableau just before the operation $l_{k_{0}}\rightarrow l_{k_{0}}^{\dag}$ has the following configuration.

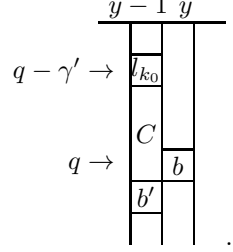
\setlength{\unitlength}{12pt}
\begin{center}
\begin{picture}(6,7)
\put(3,0){\line(0,1){7}}
\put(4,0){\line(0,1){7}}
\put(5,0){\line(0,1){7}}
\put(3,1){\line(1,0){1}}
\put(3,2){\line(1,0){2}}
\put(4,3){\line(1,0){1}}
\put(3,5){\line(1,0){1}}
\put(3,6){\line(1,0){1}}
\put(2,7){\line(1,0){4}}
\put(3,1){\makebox(1,1){$b^{\prime}$}}
\put(3,2){\makebox(1,3){$C$}}
\put(3,5){\makebox(1,1){$l_{k_{0}}$}}
\put(4,2){\makebox(1,1){$b$}}
\put(0,2){\makebox(2.5,1)[r]{$q \rightarrow$}}
\put(0,5){\makebox(2.5,1)[r]{$q-\gamma^{\prime} \rightarrow$}}
\put(2.5,7){\makebox(1.5,1){${\mathstrut y-1}$}}
\put(4.25,7){\makebox(1,1){${\mathstrut y}$}}
\end{picture}.
\end{center}
By the argument of the first paragraph of the proof, $b\notin C$ even if $k_{k_{0}}^{\ast}<b$ (it is clear $b\notin C$ if $k_{k_{0}}^{\ast}>b$) so that $C$ has at most $\alpha_{2}$ letters because $\alpha_{2}$ $\mathscr{I}^{(y-1)}$ letters exist between $l_{k_{0}}$ and $b$.
On the other hand, $C$ consists of $\gamma^{\prime}$ boxes and $\gamma^{\prime} \geq \alpha_{2}$.
This implies $\gamma^{\prime}=\alpha_{2}$ and $C$ consists of consecutive $\alpha_{2}$ letters, $l_{k_{0}}+1,\ldots,b-1$, i.e., $\beta_{2}=\varepsilon_{2}=0$.
If $b \notin \mathscr{L}^{(x,y-1)}$, then $b \notin \mathscr{I}^{(y-1)}$ so that $b^{\prime}>b$.
When the operation $l_{k_{0}}\rightarrow l_{k_{0}}^{\dag}$ is finished, the entry at the $q$-th position in the $(y-1)$-st column is $b^{\prime}$, which does not change under subsequent operations.
This contradicts the semistandardness of the $\mathscr{C}_{n}^{(+)}$-letters part of $T^{\prime}$ because $b^{\prime}>b$.
If $b\in  \mathscr{L}^{(x,y-1)}$, then $b^{\prime}=b$.
When the operation $l_{k_{0}}\rightarrow l_{k_{0}}^{\dag}$ followed by $b\rightarrow b^{\dag}$ is finished, the entry at the $q$-th position in the $(y-1)$-st column is strictly larger than $b$ by Lemma~\ref{lem:col_sst2} and does not change under subsequent operations.
This is also a contradiction.

From \textbf{(I)} and \textbf{(II)}, we have, by induction, 
\[
(q-p)+(s-r)<b-a
\] 
in the configuration depicted in the statement of Lemma~\ref{lem:kink1} with $a_{1}=a_{2}=a$.

Next, we assume that $a_{1}<a_{2}$.
The proof for the case when $a_{1}>a_{2}$ is similar.
We consider the following two cases separately:

\begin{description}
\item[(a)]
$a_{2}$ appears in the $(y-1)$-st column.
\item[(b)]
$a_{2}$ does not appear in the $(y-1)$-st column.
\end{description}

\textbf{Case (a).}
The tableau $\Tilde{T}$ has the following configuration.

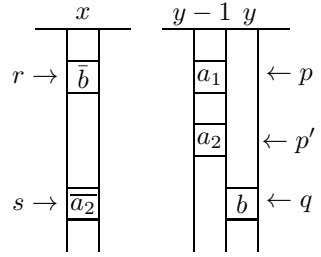
\setlength{\unitlength}{12pt}
\begin{center}
\begin{picture}(10,8)
\put(2,0){\line(0,1){7}}
\put(3,0){\line(0,1){7}}
\put(6,0){\line(0,1){7}}
\put(7,0){\line(0,1){7}}
\put(8,0){\line(0,1){7}}
\put(2,1){\line(1,0){1}}
\put(2,2){\line(1,0){1}}
\put(2,5){\line(1,0){1}}
\put(2,6){\line(1,0){1}}
\put(7,1){\line(1,0){1}}
\put(7,2){\line(1,0){1}}
\put(6,3){\line(1,0){1}}
\put(6,4){\line(1,0){1}}
\put(6,5){\line(1,0){1}}
\put(6,6){\line(1,0){1}}
\put(1,7){\line(1,0){3}}
\put(5,7){\line(1,0){4}}
\put(2,7){\makebox(1,1){$x$}}
\put(5.5,7){\makebox(1.5,1){${\mathstrut y-1}$}}
\put(7.25,7){\makebox(1,1){${\mathstrut y}$}}
\put(2,1){\makebox(1,1){$\overline{a_{2}}$}}
\put(2,5){\makebox(1,1){$\Bar{b}$}}
\put(6,3){\makebox(1,1){$a_{2}$}}
\put(6,5){\makebox(1,1){$a_{1}$}}
\put(7,1){\makebox(1,1){$b$}}
\put(0,1){\makebox(2,1){$s \rightarrow$}}
\put(0,5){\makebox(2,1){$r \rightarrow$}}
\put(8,1){\makebox(2,1){$\leftarrow q$}}
\put(8,3){\makebox(2,1){$\leftarrow p^{\prime}$}}
\put(8,5){\makebox(2,1){$\leftarrow p$}}
\end{picture}.
\end{center}
Since $p^{\prime}-p-1 \leq \left\vert \{a_{1}+1,\ldots,a_{2}-1\}\right\vert $, we have 
$p^{\prime}-p \leq a_{2}-a_{1}$.
On the other hand, $(q-p^{\prime})+(s-r)<b-a_{2}$ so that we have
\[
(q-p)+(s-r)<b-a_{1}=b-\min(a_{1},a_{2}).
\]

\textbf{Case (b).}
Let $j$ be the smallest entry such that $a_{2}<j$ and $j$ (resp. $\Bar{j}$) appears in the $(y-1)$-st (resp. $x$-th) column. 
The tableau $\Tilde{T}$ has the following configuration.

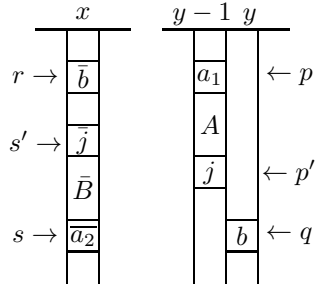
\setlength{\unitlength}{12pt}
\begin{center}
\begin{picture}(10,9)
\put(2,0){\line(0,1){8}}
\put(3,0){\line(0,1){8}}
\put(6,0){\line(0,1){8}}
\put(7,0){\line(0,1){8}}
\put(8,0){\line(0,1){8}}
\put(2,1){\line(1,0){1}}
\put(2,2){\line(1,0){1}}
\put(2,4){\line(1,0){1}}
\put(2,5){\line(1,0){1}}
\put(2,6){\line(1,0){1}}
\put(2,7){\line(1,0){1}}
\put(7,1){\line(1,0){1}}
\put(7,2){\line(1,0){1}}
\put(6,3){\line(1,0){1}}
\put(6,4){\line(1,0){1}}
\put(6,6){\line(1,0){1}}
\put(6,7){\line(1,0){1}}
\put(1,8){\line(1,0){3}}
\put(5,8){\line(1,0){4}}
\put(2,8){\makebox(1,1){$x$}}
\put(5.5,8){\makebox(1.5,1){${\mathstrut y-1}$}}
\put(7.25,8){\makebox(1,1){${\mathstrut y}$}}
\put(2,1){\makebox(1,1){$\overline{a_{2}}$}}
\put(2,2){\makebox(1,2){$\Bar{B}$}}

\put(2,4){\makebox(1,1){$\Bar{j}$}}
\put(2,6){\makebox(1,1){$\Bar{b}$}}
\put(6,3){\makebox(1,1){$j$}}
\put(6,4){\makebox(1,2){$A$}}
\put(6,6){\makebox(1,1){$a_{1}$}}
\put(7,1){\makebox(1,1){$b$}}
\put(0,1){\makebox(2,1){$s \rightarrow$}}
\put(0,4){\makebox(2,1){$s^{\prime} \rightarrow$}}
\put(0,6){\makebox(2,1){$r \rightarrow$}}
\put(8,1){\makebox(2,1){$\leftarrow q$}}
\put(8,3){\makebox(2,1){$\leftarrow p^{\prime}$}}
\put(8,6){\makebox(2,1){$\leftarrow p$}}
\end{picture},
\end{center}
where $A\cap B=\emptyset$.
Since $\left\vert A\right\vert +\left\vert \Bar{B}\right\vert=\left\vert A\cup B \right\vert \leq \left\vert \{a_{1}+1,\ldots,j-1 \} \backslash \{a_{2}\} \right\vert =j-a_{2}-2$, 
we have $p^{\prime}-p+s-s^{\prime} \leq j-a_{1}$.
On the other hand, $(q-p^{\prime})+(s^{\prime}-r)<b-j$ so that we have 
\[
(q-p)+(s-r)<b-a_{1}=b-\min(a_{1},a_{2}).
\]

If such an entry $j$ does not exist, the tableau $\Tilde{T}$ has the following configuration.
 
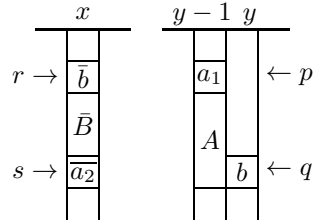
\setlength{\unitlength}{12pt}
\begin{center}
\begin{picture}(10,7)
\put(2,0){\line(0,1){6}}
\put(3,0){\line(0,1){6}}
\put(6,0){\line(0,1){6}}
\put(7,0){\line(0,1){6}}
\put(8,0){\line(0,1){6}}
\put(2,1){\line(1,0){1}}
\put(2,2){\line(1,0){1}}
\put(2,4){\line(1,0){1}}
\put(2,5){\line(1,0){1}}
\put(6,1){\line(1,0){2}}
\put(6,4){\line(1,0){1}}
\put(6,5){\line(1,0){1}}
\put(7,2){\line(1,0){1}}
\put(1,6){\line(1,0){3}}
\put(5,6){\line(1,0){4}}
\put(2,6){\makebox(1,1){$x$}}
\put(5.5,6){\makebox(1.5,1){${\mathstrut y-1}$}}
\put(7.25,6){\makebox(1,1){${\mathstrut y}$}}
\put(2,1){\makebox(1,1){$\overline{a_{2}}$}}
\put(2,2){\makebox(1,2){$\Bar{B}$}}
\put(2,4){\makebox(1,1){$\Bar{b}$}}

\put(6,1){\makebox(1,3){$A$}}
\put(6,4){\makebox(1,1){$a_{1}$}}
\put(7,1){\makebox(1,1){$b$}}
\put(0,1){\makebox(2,1){$s \rightarrow$}}
\put(0,4){\makebox(2,1){$r \rightarrow$}}
\put(8,1){\makebox(2,1){$\leftarrow q$}}
\put(8,4){\makebox(2,1){$\leftarrow p$}}
\end{picture},
\end{center}
where $A\cap B=\emptyset$ and $a_{2}\notin A$.
Furthermore, $b\notin A$ because of the argument of the first paragraph of the proof.
Since $\left\vert A\right\vert +\left\vert \Bar{B}\right\vert=\left\vert A\cup B \right\vert \leq \left\vert \{a_{1}+1,\ldots,b-1 \} \backslash \{a_{2}\} \right\vert = b-a_{1}-2$, we have
\[
(q-p)+(s-r)<b-a_{1}=b-\min(a_{1},a_{2}).
\]

Now let us prove part (2).
We set $a_{1}=a_{2}=a$.
The proof for the case when $a_{1}\neq a_{2}$ is same as that of (1).
Note that $\phi^{(x,y)}$ is well-defined by Lemma~\ref{lem:welldf1}.
Let $C_{+}^{(y)}$ be the $\mathscr{C}^{(+)}$-letters part of the $y$-th column of 
$\Tilde{T}$ and let $C^{(x,y)}$ be the column whose $\mathscr{C}^{(+)}$ (resp. $\mathscr{C}^{(-)}$)-letters part is $C_{+}^{(y)}$ (resp. $C_{-}^{(x)}$).
If $b=l_{c}$, then $\delta=0$ and we have nothing to prove.
Suppose that $b=l_{k_{0}^{\prime}}$ ($k_{0}^{\prime}<c$) and
\begin{equation} \label{eq:kink1_8}
q_{k+1}-p+s-r_{k+1}<l_{k+1}-a-\delta_{k+1}
\end{equation}
holds, where $\delta_{k+1}=\sharp\left\{  l\in \mathscr{L}^{(x,y)} \relmiddle| l^{\ast}<l_{k+1}<l\right\}$, 
$q_{k+1}$ is the position of $l_{k+1}$ in the $y$-th column, and 
$r_{k+1}$ is the position of $\overline{l_{k+1}}$ in the $x$-th column in $\Tilde{T}$ 
($k=c-1,\ldots,k_{0}^{\prime}$).
Suppose that the operation for $l_{k+1}\rightarrow l_{k+1}^{\ast}$ is finished.
The filling diagram of the updated column has the following configuration.

\setlength{\unitlength}{15pt}
\begin{center}
\begin{picture}(15,3)
\put(1,1){\line(0,1){2}}
\put(2,1){\line(0,1){2}}
\put(4,1){\line(0,1){2}}
\put(5,1){\line(0,1){2}}
\put(7,1){\line(0,1){2}}
\put(8,1){\line(0,1){2}}
\put(10,1){\line(0,1){2}}
\put(11,1){\line(0,1){2}}
\put(13,1){\line(0,1){2}}
\put(14,1){\line(0,1){2}}
\put(0,1){\line(1,0){15}}
\put(0,3){\line(1,0){15}}
\put(0,2){\line(1,0){2}}
\put(4,2){\line(1,0){1}}
\put(7,2){\line(1,0){1}}
\put(10,2){\line(1,0){1}}
\put(13,2){\line(1,0){2}}
\put(1,1){\makebox(1,1){$\circ$}}
\put(1,2){\makebox(1,1){$\circ$}}
\put(4,1){\makebox(1,1){$\times$}}
\put(4,2){\makebox(1,1){$\times$}}
\put(7,1){\makebox(1,1){$\times$}}
\put(7,2){\makebox(1,1){$\times$}}
\put(10,1){\makebox(1,1){$\bullet$}}
\put(10,2){\makebox(1,1){$\bullet$}}
\put(13,1){\makebox(1,1){$\circ$}}
\put(13,2){\makebox(1,1){$\circ$}}
\put(1,0){\makebox(1,1){$l_{k}^{\ast}$}}
\put(4,0){\makebox(1,1){$l_{k+1}^{\ast}$}}
\put(7,0){\makebox(1,1){$l_{k+\delta_{k}}^{\ast}$}}
\put(10,0){\makebox(1,1){$l_{k}$}}
\put(13,0){\makebox(1,1){$l_{k+1}$}}
\put(5,1){\makebox(2,2){$\cdots$}}
\put(11,1){\makebox(2,2){$(0)$}}
\end{picture},
\end{center}
where $\delta_{k}=\sharp\left\{ l\in \mathscr{L}^{(x,y)} \relmiddle| l^{\ast}<l_{k}<l\right\}=
\sharp\left\{l_{k+1},\ldots,l_{k+\delta_{k}}\right\}$.
Let us assume that the numbers of $(+)$-slots, $(-)$-slots, and $(\times)$-slots in region $(0)$ are 
$\alpha$, $\beta$, and $\varepsilon$, respectively.
The $(\pm)$-slots and $\emptyset$-slots do not exist in this region.
Then $q_{k+1}=q_{k}+\alpha+1$, $r_{k+1}=r_{k}-\beta-1$, $l_{k+1}=l_{k}+(\alpha+\beta+\varepsilon)+1$, and 
$\delta_{k+1}=(\delta_{k}-1)+\varepsilon$.
Substituting these into Eq.\eqref{eq:kink1_8}, we have $q_{k}-p+s-r_{k}<l_{k}-a-\delta_{k}$.
Therefore, we have, by induction, $(q-p)+(s-r)<b-a-\delta$ in the configuration depicted in the statement of Lemma~\ref{lem:kink1} with $a_{1}=a_{2}=a$.
\end{proof}

\begin{lem} \label{lem:st1}
Suppose that $=C_{1}C_{2}\cdots C_{n_{c}}\in C_{n}\text{-}\mathrm{SST}_{\mathrm{KN}}$.
Let us set 
\[
\Tilde{T}:=\phi^{(x,y-1)}\circ \cdots \circ \phi^{(x,x)}\circ \overline{\Phi^{(x+1)}}(T) \quad (2\leq x+1\leq y \leq n_{c}).
\]
Here, we assume that $\Tilde{T}\neq \emptyset$ and that in the updating process of the tableau from $T$ to $\Tilde{T}$ the semistandardness of the $\mathscr{C}_{n}^{(+)}$-letters part of the tableau is preserved.
Then the $\mathscr{C}_{n}^{(+)}$-letters part of $\phi^{(x,y)}(\Tilde{T})$ is semistandard.
\end{lem}

\begin{proof} 
The map $\phi^{(x,y)}$ is well-defined by Lemma~\ref{lem:welldf1}.
Let $C_{y-1}$ be the $(y-1)$-st column of $\Tilde{T}$ and $C_{y}$ (resp. $C_{y}^{0}$) be the $y$-th column of $\phi^{(x,y)}(\Tilde{T})$ (resp. $\Tilde{T}$).
In what follows, we show that the $\mathscr{C}_{n}^{(+)}$-letters part of the two-column tableau $C_{y-1}C_{y}$ 
in $\phi^{(x,y)}(\Tilde{T})$ is semistandard.
If this is true, the claim of Lemma~\ref{lem:st1} follows because the $\mathscr{C}_{n}^{(+)}$-letters part of $C_{y}C_{y+1}$ 
in $\phi^{(x,y)}(\Tilde{T})$ is guaranteed to be seminstandard by Lemma~\ref{lem:col_sst1}, where $C_{y+1}$ is the $(y+1)$-st column of 
$\phi^{(x,y)}(\Tilde{T})$ ($y\leq n_{c}-1$).
Let us denote by $\mathscr{I}^{(y)}$ the set of $\mathscr{I}$-letters in the $y$-th column of $\Tilde{T}$ and 
by $\mathscr{J}^{(x)}$ the set of $\mathscr{J}$-letters in the $x$-th column of $\Tilde{T}$ and 
set $\mathscr{L}^{(x,y)}:=\mathscr{J}^{(x)}\cap \mathscr{I}^{(y)}=:\left\{l_{1},\ldots,l_{c}\right\}$.
We adopt the second kind algorithm for $\phi^{(x,y)}$ when we treat the $y$-th column, 
while we adopt the first kind one when we treat the $x$-th column.
We claim that $\Delta C_{y-1}[p_{k}^{\prime},q_{k}^{\prime}]\preceq \Delta_{k}(C_{y}^{0})$ for all $k=c,c-1,\ldots,1$ 
so that $\mathscr{C}_{n}^{(+)}$-letters part of $C_{y-1}C_{y}$ is semistandard, where 
$p_{k}^{\prime}$ (resp. $q_{k}^{\prime}$) is the position of the top (resp. bottom) box of the block $\Delta_{k}(C_{y}^{0})$, 
which is defined in the explanation of the second kind algorithm for $\phi$.
The proof is by induction on $k$.
Namely, we prove

\textbf{(I).} 
$\Delta C_{y-1}[p_{c}^{\prime},q_{c}^{\prime}]\preceq \Delta_{c}(C_{y}^{0})$.

\textbf{(II).}
$\Delta C_{y-1}[p_{k}^{\prime},q_{k}^{\prime}]\preceq \Delta_{k}(C_{y}^{0})$ 
under the assumption that 
$\Delta C_{y-1}[p_{k+1}^{\prime},q_{k+1}^{\prime}]\preceq \Delta_{k+1}(C_{y}^{0})$ ($k=c-1,\ldots,1$).

We first prove \textbf{(II)}.  
Suppose that 
$\left\{l\in \mathscr{L}^{(x,y-1)} \relmiddle| l^{\ast}< l_{k} <l \right\}=\left\{l_{k+1},\ldots,l_{k+\delta}\right\}$.
Let $C_{+}^{(y)}$ (resp. $C_{-}^{(x)}$) be the $\mathscr{C}_{n}^{(+)}$ (resp. $\mathscr{C}_{n}^{(-)}$)-letters part of 
the $y$-th (resp. $x$-th) column of $\Tilde{T}$ and let $C^{(x,y)}$ be the column 
whose $\mathscr{C}_{n}^{(+)}$ (resp. $\mathscr{C}_{n}^{(-)}$)-letters part is $C_{+}^{(y)}$ (resp. $C_{-}^{(x)}$).
Suppose that the operation for $l_{k+1}\rightarrow l_{k+1}^{\ast}$ is completed ($k\leq c-1$).
Let $\Tilde{T}^{\prime}$ be the updated tableau and $C^{(x,y)\prime}$ be the resulting column.
Let us assume that  
$\Delta C_{y-1}[p_{k+1}^{\prime},q_{k+1}^{\prime}]\preceq\Delta_{k+1}(C_{y}^{0})$.
The filling diagram of the column $C^{(x,y)\prime}$ has the following configuration.
Here, we assume $r\geq 1$.
The proof for the case when $r=0$ is similar and much simpler.

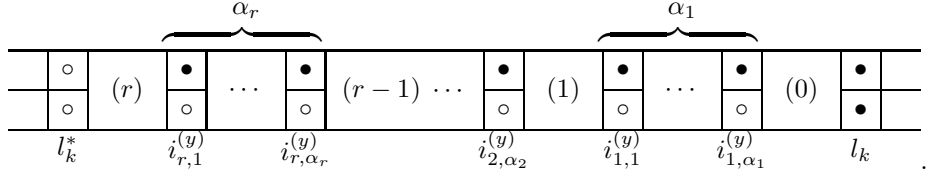
\setlength{\unitlength}{15pt}
\begin{center}
\begin{picture}(23,5)
\put(1,1){\line(0,1){2}}
\put(2,1){\line(0,1){2}}
\put(4,1){\line(0,1){2}}
\put(5,1){\line(0,1){2}}
\put(7,1){\line(0,1){2}}
\put(8,1){\line(0,1){2}}
\put(12,1){\line(0,1){2}}
\put(13,1){\line(0,1){2}}
\put(15,1){\line(0,1){2}}
\put(16,1){\line(0,1){2}}
\put(18,1){\line(0,1){2}}
\put(19,1){\line(0,1){2}}
\put(21,1){\line(0,1){2}}
\put(22,1){\line(0,1){2}}
\put(0,1){\line(1,0){23}}
\put(0,2){\line(1,0){2}}
\put(4,2){\line(1,0){1}}
\put(7,2){\line(1,0){1}}
\put(12,2){\line(1,0){1}}
\put(15,2){\line(1,0){1}}
\put(18,2){\line(1,0){1}}
\put(21,2){\line(1,0){2}}
\put(0,3){\line(1,0){23}}
\put(1,1){\makebox(1,1){$\circ$}}
\put(1,2){\makebox(1,1){$\circ$}}
\put(4,1){\makebox(1,1){$\circ$}}
\put(4,2){\makebox(1,1){$\bullet$}}
\put(7,1){\makebox(1,1){$\circ$}}
\put(7,2){\makebox(1,1){$\bullet$}}
\put(12,1){\makebox(1,1){$\circ$}}
\put(12,2){\makebox(1,1){$\bullet$}}
\put(15,1){\makebox(1,1){$\circ$}}
\put(15,2){\makebox(1,1){$\bullet$}}
\put(18,1){\makebox(1,1){$\circ$}}
\put(18,2){\makebox(1,1){$\bullet$}}
\put(21,1){\makebox(1,1){$\bullet$}}
\put(21,2){\makebox(1,1){$\bullet$}}
\put(2,1){\makebox(2,2){$(r)$}}
\put(5,1){\makebox(2,2){$\cdots$}}
\put(8,1){\makebox(4,2){$(r-1)\;\cdots$}}
\put(13,1){\makebox(2,2){$(1)$}}
\put(16,1){\makebox(2,2){$\cdots$}}
\put(19,1){\makebox(2,2){$(0)$}}
\put(1,0){\makebox(1,1){$l_{k}^{\ast}$}}
\put(4,0){\makebox(1,1){$i_{r,1}^{(y)}$}}
\put(7,0){\makebox(1,1){$i_{r,\alpha_{r}}^{(y)}$}}
\put(12,0){\makebox(1,1){$i_{2,\alpha_{2}}^{(y)}$}}
\put(15,0){\makebox(1,1){$i_{1,1}^{(y)}$}}
\put(18,0){\makebox(1,1){$i_{1,\alpha_{1}}^{(y)}$}}
\put(21,0){\makebox(1,1){$l_{k}$}}
\put(4,2.5){\makebox(4,1){$
\overbrace{
\begin{array}
[c]{cccccc}
& & & & &
\end{array}
}$
}}
\put(5,3.5){\makebox(2,1){$\alpha _{r}$}}
\put(15,2.5){\makebox(4,1){$
\overbrace{
\begin{array}
[c]{cccccc}
& & & & &
\end{array}
}$
}}
\put(16,3.5){\makebox(2,1){$\alpha _{1}$}}
\end{picture}.
\end{center}
Region $(s)$ consists of $(-)$-slots, $(\pm)$-slots, and $(\times)$-slots.
Let us assume that the numbers of $(-)$-slots, $(\pm)$-slots, and $(\times)$-slots in this region are
$\beta_{s}$, $\gamma_{s}$, and $\delta _{s}$ respectively 
and that the position of $(\times)$-slots in region (s) are 
$l_{s,1}^{\ast},\ldots$, and $l_{s,\delta_{s}}^{\ast}$ ($s=0,1,\ldots,r$);
$\left\{  l_{k+1}^{\ast},\ldots,l_{k+\delta}^{\ast}\right\}  =
\left\{l_{0,1}^{\ast},\ldots,l_{0,\delta_{0}}^{\ast},\ldots,l_{r,1}^{\ast},\ldots,l_{r,\delta_{r}}^{\ast}\right\}$.
Between two regions $(s-1)$ and $(s)$, $\alpha_{s}$ $(+)$-slots lie consecutively.

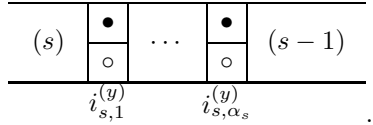
\setlength{\unitlength}{15pt}
\begin{center}
\begin{picture}(9,3)
\put(2,1){\line(0,1){2}}
\put(3,1){\line(0,1){2}}
\put(5,1){\line(0,1){2}}
\put(6,1){\line(0,1){2}}
\put(0,1){\line(1,0){9}}
\put(2,2){\line(1,0){1}}
\put(5,2){\line(1,0){1}}
\put(0,3){\line(1,0){9}}
\put(0,1){\makebox(2,2){$(s)$}}
\put(3,1){\makebox(2,2){$\cdots$}}
\put(6,1){\makebox(3,2){$(s-1)$}}
\put(2,1){\makebox(1,1){$\circ$}}
\put(2,2){\makebox(1,1){$\bullet$}}
\put(5,1){\makebox(1,1){$\circ$}}
\put(5,2){\makebox(1,1){$\bullet$}}
\put(2,0){\makebox(1,1){$i_{s,1}^{(y)}$}}
\put(5,0){\makebox(1,1){$i_{s,\alpha_{s}}^{(y)}$}}
\end{picture}.
\end{center}
The updated tableau $\Tilde{T}^{\prime}$ has the following configuration.
There are no $\mathscr{L}^{(x,y)\ast}$-letters above the box containing $l_{k}$ in the $y$-th column 
because we adopt the second kind algorithm for $\phi^{(x,y)}$ in the $y$-th column, 
while $\overline{\mathscr{L}^{(x,y)\ast}}$-letters may exist below the box containing $\overline{l_{k}}$ in the $x$-th column.

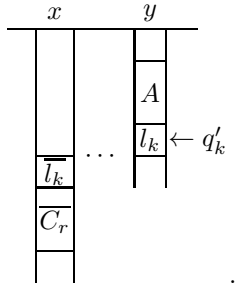
\setlength{\unitlength}{12pt}
\begin{center}
\begin{picture}(7,9)
\put(0.9,0){\line(0,1){8}}
\put(2.1,0){\line(0,1){8}}
\put(4,3){\line(0,1){5}}
\put(5,3){\line(0,1){5}}
\put(0.9,1){\line(1,0){1.2}}
\put(0.9,3){\line(1,0){1.2}}
\put(0.9,4){\line(1,0){1.2}}
\put(4,4){\line(1,0){1}}
\put(4,5){\line(1,0){1}}
\put(4,7){\line(1,0){1}}
\put(0,8){\line(1,0){6}}
\put(1,1){\makebox(1,2){$\overline{C_{r}}$}}
\put(1,3){\makebox(1,1){$\overline{l_{k}}$}}
\put(4,4){\makebox(1,1){$l_{k}$}}
\put(4,5){\makebox(1,2){$A$}}
\put(2,3){\makebox(2,2){$\cdots$}}
\put(1,8){\makebox(1,1){$x$}}
\put(4,8){\makebox(1,1){$y$}}
\put(5,4){\makebox(2,1){$\leftarrow q_{k}^{\prime}$}}
\end{picture}.
\end{center}
where $A$ is the stack of the sequence of blocks 
$L_{r}^{(y)},I_{r}^{(y)},\ldots,I_{1}^{(y)},L_{0}^{(y)}$ in this order (from top to bottom) 
and $\overline{C_{r}}$ is the stack of the sequence of blocks 
$\overline{J_{0}^{(x)}},\overline{J_{1}^{(x)}},\ldots,\overline{J_{r}^{(x)}}$ in this order (from the top).
The block $I_{s}^{(y)}$ consists of consecutive $\alpha_{s}$ $\mathscr{I}^{(y)}\backslash \mathscr{L}^{(x,y)}$-letters
$\{i_{s,1}^{(y)},\ldots,i_{s,\alpha_{s}}^{(y)}\}$, where
\[
i_{s,\alpha_{s}-t+1}^{(y)}=l_{k}-\sum_{i=1}^{s-1}\alpha_{i}-\sum_{i=0}^{s-1}\tau_{i}-t \quad (t=1,\ldots,\alpha_{s})
\]
with $\tau_{i}:=\beta_{i}+\gamma_{i}+\delta_{i}$
and $L_{s}^{(y)}$ is the block of $\gamma_{s}$ $\mathscr{L}^{(x,y)}$-letters ($s=0,1,\ldots,r$).
The block $\overline{J_{s}^{(x)}}$ consists of consecutive $\tau_{s}$ $\mathscr{C}_{n}^{(-)}$-letters 
$\left\{\overline{j_{s,\tau_{s}}^{(x)}},\ldots,\overline{j_{s,1}^{(x)}}\right\}$, where
\[
j_{s,\tau_{s}-t+1}^{(x)}=l_{k}-\sum_{i=1}^{s}\alpha_{i}-\sum_{i=0}^{s-1}\tau_{i}-t
\quad(t=1,\ldots,\tau_{s};s=0,\ldots,r).
\]
Note that $\overline{J_{s}^{(x)}}$ contains $\overline{\mathscr{L}^{(x,y)\ast}}$-letters, 
$\overline{l_{s,1}^{\ast}},\ldots,$ and $\overline{l_{s,\delta_{s}}^{\ast}}$.
Let us assume that the $(y-1)$-st column of $\Tilde{T}^{\prime}$ has the following configuration.

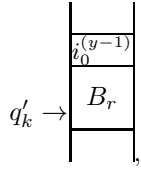
\setlength{\unitlength}{12pt}
\begin{center}
\begin{picture}(4,5)
\put(2,0){\line(0,1){5}}
\put(4,0){\line(0,1){5}}
\put(2,1){\line(1,0){2}}
\put(2,3){\line(1,0){2}}
\put(2,4){\line(1,0){2}}
\put(2,1){\makebox(2,2){$B_{r}$}}
\put(2,3){\makebox(2,1){\small $i_{0}^{(y-1)}$}}
\put(0,1){\makebox(2,1){$q_{k}^{\prime} \rightarrow$}}
\end{picture},
\end{center}
where $B_{r}$ is the stack of the sequence of blocks 
$I_{r}^{(y-1)}, I_{r-1}^{(y-1)},\ldots, I_{1}^{(y-1)}$ in this order (from top to bottom) 
and the position of the bottom box in $B_{r}$ is $q_{k}^{\prime}$ 
(the block $B_{r}$ is not empty because of the assumption of $r\geq 1$).
The block $I_{s}^{(y-1)}$ consists of $\alpha_{k}$ $\mathscr{C}_{n}^{(+)}$-letters 
$\{i_{s,1}^{(y-1)},\ldots,i_{s,\alpha_{s}}^{(y-1)}\}$ 
so that $\left\vert I_{s}^{(y-1)}\right\vert =\left\vert I_{s}^{(y)}\right\vert $
($s=1,\ldots,r$).

\textbf{(i).}
We claim that
$i_{1,\alpha_{1}}^{(y-1)}\leq i_{1,\alpha_{1}}^{(y)}=l_{k}-\tau_{0}-1$.
If this is not true, 
$i_{1,\alpha_{1}}^{(y-1)}\in\{l_{k}-\tau_{0},l_{k}-\tau_{0}+1,\ldots,l_{k}\}$
($i_{1,\alpha_{1}}^{(y-1)}\leq l_{k}$), i.e., $\overline{i_{1,\alpha_{1}}^{(y-1)}}$ is in the block $\overline{J_{0}^{(x)}}$ 
or $\overline{i_{1,\alpha_{1}}^{(y-1)}}=\overline{l_{k}}$.
Suppose that $i_{1,\alpha_{1}}^{(y-1)}=l_{k}-t$ ($t=0,\ldots,\tau_{0}$).
The updated tableau $\Tilde{T}^{\prime}$ has the following configuration.

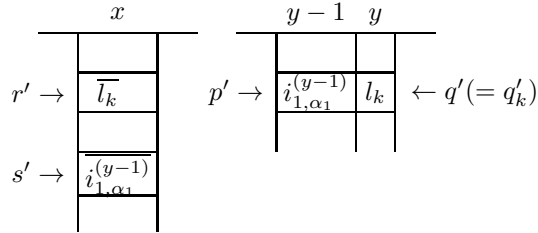
\setlength{\unitlength}{15pt}
\begin{center}
\begin{picture}(14,6)
\put(2,0){\line(0,1){5}}
\put(4,0){\line(0,1){5}}
\put(7,2){\line(0,1){3}}
\put(9,2){\line(0,1){3}}
\put(10,2){\line(0,1){3}}
\put(2,0.9){\line(1,0){2}}
\put(2,2){\line(1,0){2}}
\put(2,3){\line(1,0){2}}
\put(2,4){\line(1,0){2}}
\put(7,3){\line(1,0){3}}
\put(7,4){\line(1,0){3}}
\put(1,5){\line(1,0){4}}
\put(6,5){\line(1,0){5}}
\put(2,0.9){\makebox(2,1){$\overline{i_{1,\alpha _{1}}^{(y-1)}}$}}
\put(2,3){\makebox(1.5,1){$\overline{l_{k}}$}}
\put(7,3){\makebox(2,1){$i_{1,\alpha _{1}}^{(y-1)}$}}
\put(9,3){\makebox(1,1){$l_{k}$}}
\put(0,1){\makebox(2,1){$s^{\prime} \rightarrow$}}
\put(0,3){\makebox(2,1){$r^{\prime} \rightarrow$}}
\put(5,3){\makebox(2,1){$p^{\prime} \rightarrow$}}
\put(10,3){\makebox(4,1){$\leftarrow q^{\prime}(=q_{k}^{\prime})$}}
\put(2,5){\makebox(2,1){$x$}}
\put(7,5){\makebox(2,1){${\mathstrut y-1}$}}
\put(9,5){\makebox(1,1){${\mathstrut y}$}}
\end{picture}.
\end{center}
Let $p_{k}$ and $q_{k}$ be the initial position of $i_{1,\alpha_{1}}^{(y-1)}$ in the $(y-1)$-st column and 
that of $l_{k}$ in the $y$-th column of $\Tilde{T}$, respectively.
We consider the following two cases separately:

\begin{description}
\item[(a)]
$i_{1,\alpha_{1}}^{(y-1)}\notin \mathscr{L}^{(x,y)\ast}$.
\item[(b)]
$i_{1,\alpha_{1}}^{(y-1)}\in \mathscr{L}^{(x,y)\ast}$.
\end{description}

\textbf{Case (a).}
The entry $\overline{i_{1,\alpha_{1}}^{(y-1)}}$ exists initially in the $x$-th column of $\Tilde{T}$.
Let $r_{k}$ and $s_{k}$ be the initial position of $\overline{l_{k}}$ and 
that of $\overline{i_{1,\alpha_{1}}^{(y-1)}}$ in the $x$-th column of $\Tilde{T}$, respectively.
Then $p_{k}=p^{\prime}$ and $q_{k}\geq q^{\prime}$ because $l_{k}$ is relocated upward by the operations for 
$l_{c}\rightarrow l_{c}^{\ast},\ldots,l_{k+1}\rightarrow l_{k+1}^{\ast}$ or still lies at the initial position.
Suppose that $\delta^{\prime}$ $\overline{\mathscr{L}^{(x,y)\ast}}$-letters appear 
between the $r^{\prime}$-th box and the $s^{\prime}$-th box 
in the $x$-th column ($\delta^{\prime}\leq\delta$).
Then $s^{\prime}-r^{\prime}=s_{k}-r_{k}+\delta^{\prime}$ so that
\begin{align} \label{eq:ineq1}
q_{k}-p_{k}+s_{k}-r_{k}  & \geq q^{\prime}-p^{\prime}+s^{\prime}-r^{\prime}-\delta^{\prime}=t-\delta^{\prime} \\
& \geq l_{k}-i_{1,\alpha_{1}}^{(y-1)}-\delta, \nonumber
\end{align}
which contradicts the assertion of Lemma~\ref{lem:kink1}.

\textbf{Case (b).} 
We can write $i_{1,\alpha_{1}}^{(y-1)}=a^{\ast}$ ($a\in \mathscr{L}^{(x,y)}$).
Let $r_{k}$ be the initial position of $\overline{l_{k}}$ in the $x$-th column of $\Tilde{T}$.
Furthermore, let us suppose that the initial entry at the $s_{k}$-th position ($s_{k}\geq r_{k}$) in the $x$-th column of $\Tilde{T}$ is $\Bar{b}$ and that the operation $a \rightarrow a^{\ast}$ replaces the entry $\Bar{b}$ by $\overline{a^{\ast}}$.

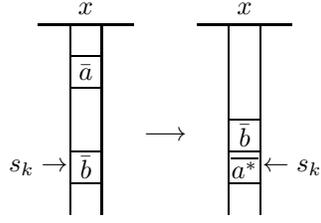
\setlength{\unitlength}{12pt}
\begin{center}
\begin{picture}(11,7)
\put(3,0){\line(0,1){6}}
\put(4,0){\line(0,1){6}}
\put(3,1){\line(1,0){1}}
\put(3,2){\line(1,0){1}}
\put(3,4){\line(1,0){1}}
\put(3,5){\line(1,0){1}}
\put(3,1){\makebox(1,1){$\Bar{b}$}}
\put(3,4){\makebox(1,1){$\Bar{a}$}}
\put(1,1){\makebox(2,1){$s_{k} \rightarrow$}}
\put(5,2){\makebox(2,1){$\longrightarrow$}}
\put(8,0){\line(0,1){6}}
\put(9,0){\line(0,1){6}}
\put(8,1){\line(1,0){1}}
\put(8,2){\line(1,0){1}}
\put(8,3){\line(1,0){1}}
\put(8,1){\makebox(1,1){$\overline{a^{\ast}}$}}
\put(8,2){\makebox(1,1){$\Bar{b}$}}
\put(9,1){\makebox(2,1){$\leftarrow s_{k}$}}
\put(2,6){\line(1,0){3}}
\put(7,6){\line(1,0){3}}
\put(3,6){\makebox(1,1){$x$}}
\put(8,6){\makebox(1,1){$x$}}
\end{picture},
\end{center}
so that $b>i_{1,\alpha_{1}}^{(y-1)}$.
The initial tableau $\Tilde{T}$ has the following configuration.

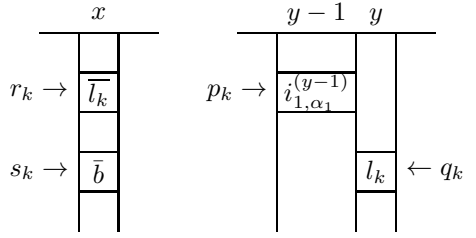
\setlength{\unitlength}{15pt}
\begin{center}
\begin{picture}(12,6)
\put(2,0){\line(0,1){5}}
\put(3,0){\line(0,1){5}}
\put(7,0){\line(0,1){5}}
\put(9,0){\line(0,1){5}}
\put(10,0){\line(0,1){5}}
\put(2,1){\line(1,0){1}}
\put(2,2){\line(1,0){1}}
\put(2,3){\line(1,0){1}}
\put(2,4){\line(1,0){1}}
\put(7,3){\line(1,0){2}}
\put(7,4){\line(1,0){2}}
\put(9,1){\line(1,0){1}}
\put(9,2){\line(1,0){1}}
\put(1,5){\line(1,0){3}}
\put(6,5){\line(1,0){5}}
\put(2,1){\makebox(1,1){$\Bar{b}$}}
\put(2,3){\makebox(1,1){$\overline{l_{k}}$}}
\put(7,3){\makebox(2,1){$i_{1,\alpha _{1}}^{(y-1)}$}}
\put(9,1){\makebox(1,1){$l_{k}$}}
\put(0,1){\makebox(2,1){$s_{k} \rightarrow$}}
\put(0,3){\makebox(2,1){$r_{k} \rightarrow$}}
\put(5,3){\makebox(2,1){$p_{k} \rightarrow$}}
\put(10,1){\makebox(2,1){$\leftarrow q_{k}$}}
\put(2,5){\makebox(1,1){$x$}}
\put(7,5){\makebox(2,1){${\mathstrut y-1}$}}
\put(9,5){\makebox(1,1){${\mathstrut y}$}}
\end{picture}.
\end{center}
Inequality~\eqref{eq:ineq1} still holds in this case and this contradicts the assertion of Lemma~\ref{lem:kink1}.

In both cases, we have $i_{1,\alpha_{1}}^{(y-1)}\leq i_{1,\alpha_{1}}^{(y)}=l_{k}-\tau_{0}-1$ 
and
\[
i_{1,\alpha_{1}-t+1}^{(y-1)}\leq i_{1,\alpha_{1}-t+1}^{(y)}=l_{k}-\tau _{0}-t \quad (t=1,\ldots,\alpha_{1}).
\]

\textbf{(ii).}
Suppose that 
\[
i_{s,1}^{(y-1)}\leq i_{s,1}^{(y)}=l_{k}-\sum_{i=1}^{s}\alpha_{i}-\sum_{i=0}^{s-1}\tau_{i} \quad (s=1,\ldots,r-1).
\]
This is satisfied for $s=1$.
Under this assumption, let us show that
\[
i_{s+1,\alpha_{s+1}}^{(y-1)}\leq i_{s+1,\alpha_{s+1}}^{(y)}=l_{k}-\sum_{i=1}^{s}\alpha_{i}-\sum_{i=0}^{s}\tau_{i}-1.
\]
If this is not true, 
\[
l_{k}-\sum_{i=1}^{s}\alpha_{i}-\sum_{i=0}^{s}\tau_{i} \leq i_{s+1,\alpha_{s+1}}^{(y-1)} \leq i_{s,1}^{(y-1)}-1 
\leq l_{k}-\sum_{i=1}^{s}\alpha_{i}-\sum_{i=0}^{s-1}\tau_{i}-1.
\]
Suppose that $i_{s+1,\alpha_{s+1}}^{(y-1)}=l_{k}-\sum_{i=1}^{s}\alpha_{i}-\sum_{i=0}^{s-1}\tau_{i}-t
=j_{s,\tau_{s}-t+1}^{(x)}$ ($t=1,\ldots,\tau_{s}$).
Then the updated tableau $\Tilde{T}^{\prime}$ has the following configuration.

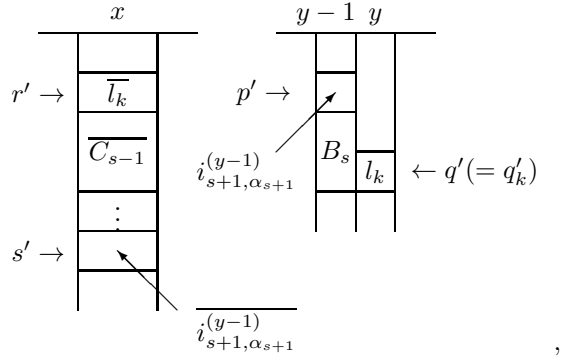
\setlength{\unitlength}{15pt}
\begin{center}
\begin{picture}(14,10)
\put(2,1){\line(0,1){7}}
\put(4,1){\line(0,1){7}}
\put(8,3){\line(0,1){5}}
\put(9,3){\line(0,1){5}}
\put(10,3){\line(0,1){5}}
\put(2,2){\line(1,0){2}}
\put(2,3){\line(1,0){2}}
\put(2,4){\line(1,0){2}}
\put(2,6){\line(1,0){2}}
\put(2,7){\line(1,0){2}}
\put(8,4){\line(1,0){2}}
\put(9,5){\line(1,0){1}}
\put(8,6){\line(1,0){1}}
\put(8,7){\line(1,0){1}}
\put(1,8){\line(1,0){4}}
\put(7,8){\line(1,0){4}}
\put(2,3){\makebox(2,1){$\vdots$}}
\put(2,4){\makebox(2,2){$\overline{C_{s-1}}$}}
\put(2,6){\makebox(2,1){$\overline{l_{k}}$}}
\put(0,2){\makebox(2,1){$s^{\prime} \rightarrow$}}
\put(0,6){\makebox(2,1){$r^{\prime} \rightarrow$}}

\put(4.5,1){\vector(-1,1){1.5}}
\put(5,0){\makebox(2.5,1){$\overline{i_{s+1,\alpha _{s+1}}^{(y-1)}}$}}

\put(8,4){\makebox(1,2){$B_{s}$}}
\put(6,6){\makebox(2,1)[l]{$p^{\prime} \rightarrow$}}
\put(10,4){\makebox(4,1){$\leftarrow q^{\prime}(=q_{k}^{\prime})$}}
\put(9,4){\makebox(1,1){$l_{k}$}}
\put(2,8){\makebox(2,1){$x$}}
\put(7,8){\makebox(2.5,1){${\mathstrut y-1}$}}
\put(9,8){\makebox(1,1){${\mathstrut y}$}}

\put(7,5){\vector(1,1){1.5}}
\put(5,4){\makebox(2.5,1){$i_{s+1,\alpha _{s+1}}^{(y-1)}$}}

\end{picture},
\end{center}
where $\overline{C_{s-1}}$ denotes the stack of blocks, 
$\overline{J_{0}^{(x)}},\ldots,\overline{J_{s-1}^{(x)}}$ in this order (from top to bottom).
Similarly, $B_{s}$ denotes the stack of blocks, 
$I_{s}^{(y-1)},\ldots,,I_{1}^{(y-1)}$ in this order (from top to bottom).

Let $p_{k}$ and $q_{k}$ be the initial position of $i_{s+1,\alpha_{s+1}}^{(y-1)}$ in the $(y-1)$-st column and 
that of $l_{k}$ in the $y$-th column of $\Tilde{T}$, respectively.
We consider the following two cases separately:

\begin{description}
\item[(a)]
$i_{s+1,\alpha_{s+1}}^{(y-1)}\notin \mathscr{L}^{(x,y)\ast}$.
\item[(b)]
$i_{s+1,\alpha_{s+1}}^{(y-1)}\in \mathscr{L}^{(x,y)\ast}$.
\end{description}

\textbf{Case (a).}
The entry $\overline{i_{s+1,\alpha_{s+1}}^{(y-1)}}$ exists initially in the $x$-th column of $\Tilde{T}$.
Let $r_{k}$ and $s_{k}$ be the initial position of $\overline{l_{k}}$ and 
that of $\overline{i_{s+1,\alpha_{s+1}}^{(y-1)}}$ in the $x$-th column of $\Tilde{T}$, respectively.
Then $p_{k}=p^{\prime}$ and $q_{k}\geq q^{\prime}$.
Suppose $\delta^{\prime}$ $\mathscr{L}^{(x,y)\ast}$-letters appear between the $r^{\prime}$-th box and the $s^{\prime}$-th box 
in the $x$-th column ($\delta^{\prime}\leq\delta$).
Then $s^{\prime}-r^{\prime}=s_{k}-r_{k}+\delta^{\prime}$ so that
$q_{k}-p_{k}+s_{k}-r_{k}\geq q^{\prime}-p^{\prime}+s^{\prime}-r^{\prime}-\delta^{\prime}$.
Here
$q^{\prime}-p^{\prime}= \sum_{i=1}^{s}\left\vert I_{i}^{(y-1)}\right\vert =\sum_{i=1}^{s}\alpha_{i}$ and
$s^{\prime}-r^{\prime}=\sum_{i=0}^{s-1}\left\vert \overline{J_{i}^{(x)}}\right\vert +t=\sum_{i=0}^{s-1}\tau_{i}+t$.
Combining these, we have
\begin{equation} \label{eq:ineq2}
q_{k}-p_{k}+s_{k}-r_{k}\geq
\sum_{i=1}^{s}\alpha_{i}+\sum_{i=0}^{s-1}\tau_{i}+t-\delta
=l_{k}-i_{s+1,\alpha_{s+1}}^{(y-1)}-\delta.
\end{equation}
This contradicts the assertion of Lemma~\ref{lem:kink1}.

\textbf{Case (b).} 
We can write $i_{s+1,\alpha_{s+1}}^{(y-1)}=a^{\ast}$ ($a\in \mathscr{L}^{(x,y)}$).
Let $r_{k}$ be the initial position of $\overline{l_{k}}$ in the $x$-th column of $\Tilde{T}$.
Furthermore, let us suppose that the initial entry at the $s_{k}$-th position ($s_{k}\geq r_{k}$) in the $x$-th column of $\Tilde{T}$ is $\Bar{b}$ and that the operation $a \rightarrow a^{\ast}$ replaces the entry $\Bar{b}$ by $\overline{a^{\ast}}$ 
so that $b>i_{s+1,\alpha_{s+1}}^{(y-1)}$.
The initial tableau $\Tilde{T}$ has the following configuration.

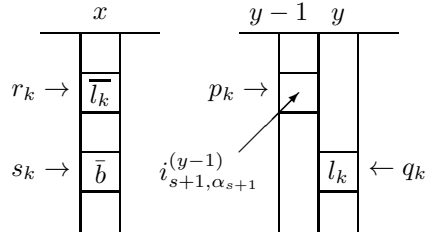
\setlength{\unitlength}{15pt}
\begin{center}
\begin{picture}(11,6)
\put(2,0){\line(0,1){5}}
\put(3,0){\line(0,1){5}}
\put(7,0){\line(0,1){5}}
\put(8,0){\line(0,1){5}}
\put(9,0){\line(0,1){5}}
\put(2,1){\line(1,0){1}}
\put(2,2){\line(1,0){1}}
\put(2,3){\line(1,0){1}}
\put(2,4){\line(1,0){1}}
\put(7,3){\line(1,0){1}}
\put(7,4){\line(1,0){1}}
\put(8,1){\line(1,0){1}}
\put(8,2){\line(1,0){1}}
\put(1,5){\line(1,0){3}}
\put(6,5){\line(1,0){4}}
\put(2,1){\makebox(1,1){$\Bar{b}$}}
\put(2,3){\makebox(1,1){$\overline{l_{k}}$}}
\put(8,1){\makebox(1,1){$l_{k}$}}
\put(0,1){\makebox(2,1){$s_{k} \rightarrow$}}
\put(0,3){\makebox(2,1){$r_{k} \rightarrow$}}
\put(5,3){\makebox(2,1){$p_{k} \rightarrow$}}
\put(9,1){\makebox(2,1){$\leftarrow q_{k}$}}
\put(2,5){\makebox(1,1){$x$}}
\put(6,5){\makebox(2,1){${\mathstrut y-1}$}}
\put(8,5){\makebox(1,1){${\mathstrut y}$}}

\put(6,2){\vector(1,1){1.5}}
\put(4,1){\makebox(2.5,1){$i_{s+1,\alpha _{s+1}}^{(y-1)}$}}

\end{picture}.
\end{center}
Inequality~\eqref{eq:ineq2} still holds in this case and this contradicts the assertion of Lemma~\ref{lem:kink1}.

In both cases, we have $i_{s+1,t}^{(y-1)}\leq i_{s+1,t}^{(y)}$ ($t=1,\ldots,\alpha_{s+1}$).

From \textbf{(i)} and \textbf{(ii)} and by induction, we have
\[
i_{r,1}^{(y-1)}\leq i_{r,1}^{(y)}=l_{k}-\sum_{i=1}^{r}\alpha_{i}-\sum_{i=0}^{r-1}\tau_{i}.
\]

\textbf{(iii).}
We claim that $i_{0}^{(y-1)}\leq l_{k}^{\ast}$.
If this is not true, then
\[
l_{k}-\sum_{i=1}^{r}\alpha_{i}-\sum_{i=0}^{r}\tau_{i}(=l_{k}^{\ast}+1) \leq i_{0}^{(y-1)} \leq i_{r,1}^{(y-1)}-1\leq
l_{k}-\sum_{i=1}^{r}\alpha_{i}-\sum_{i=0}^{r-1}\tau_{i}-1.
\]
Suppose $i_{0}^{(y-1)}=l_{k}-\sum_{i=1}^{r}\alpha_{i}-\sum_{i=0}^{r-1}\tau_{i}-t=j_{r,\tau_{r}-t+1}^{(x)}$  
($t=1,\ldots,\tau_{r}$), then the tableau $\Tilde{T}^{\prime}$ has the following configuration.

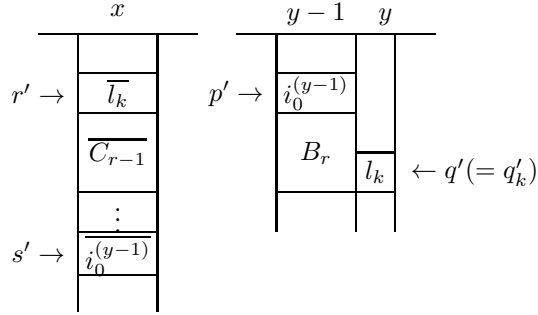
\setlength{\unitlength}{15pt}
\begin{center}
\begin{picture}(14,8)
\put(2,0){\line(0,1){7}}
\put(4,0){\line(0,1){7}}
\put(7,2){\line(0,1){5}}
\put(9,2){\line(0,1){5}}
\put(10,2){\line(0,1){5}}
\put(2,0.9){\line(1,0){2}}
\put(2,2){\line(1,0){2}}
\put(2,3){\line(1,0){2}}
\put(2,5){\line(1,0){2}}
\put(2,6){\line(1,0){2}}
\put(7,5){\line(1,0){2}}
\put(7,6){\line(1,0){2}}
\put(1,7){\line(1,0){2}}
\put(6,7){\line(1,0){2}}
\put(7,3){\line(1,0){3}}
\put(9,4){\line(1,0){1}}
\put(1,7){\line(1,0){4}}
\put(6,7){\line(1,0){5}}
\put(2,0.9){\makebox(2,1){$\overline{i_{0}^{(y-1)}}$}}
\put(2,2){\makebox(2,1){$\vdots$}}
\put(2,3){\makebox(2,2){$\overline{C_{r-1}}$}}
\put(2,5){\makebox(2,1){$\overline{l_{k}}$}}
\put(7,3){\makebox(2,2){$B_{r}$}}
\put(7,5){\makebox(2,1){$i_{0}^{(y-1)}$}}
\put(9,3){\makebox(1,1){$l_{k}$}}
\put(0,1){\makebox(2,1){$s^{\prime} \rightarrow$}}
\put(0,5){\makebox(2,1){$r^{\prime} \rightarrow$}}
\put(5,5){\makebox(2,1){$p^{\prime} \rightarrow$}}
\put(10,3){\makebox(4,1){$\leftarrow q^{\prime}(=q_{k}^{\prime})$}}
\put(2,7,10){\makebox(2,1){$x$}}
\put(7,7){\makebox(2,1){${\mathstrut y-1}$}}
\put(9.25,7){\makebox(1,1){${\mathstrut y}$}}
\end{picture}.
\end{center}
The same argument as in \textbf{(ii)} leads to 
that this configuration contradicts the assertion of Lemma~\ref{lem:kink1}.
Hence we have $i_{0}^{(y-1)}\leq l_{k}^{\ast}$.

\textbf{(iv).}
When the operation (B) for $l_{k}\rightarrow l_{k}^{\ast}$
is finished, the updated tableau has the following configuration.
From the $p_{k}^{\prime}$-th position to the $q_{k}^{\prime}$-th position in the $y$-th column is the block 
$\Delta_{k}(C_{y}^{0})$.

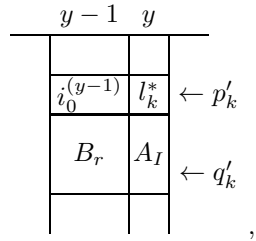
\setlength{\unitlength}{15pt}
\begin{center}
\begin{picture}(7,6)
\put(2,0){\line(0,1){5}}
\put(4,0){\line(0,1){5}}
\put(5,0){\line(0,1){5}}
\put(2,1){\line(1,0){3}}
\put(2,3){\line(1,0){3}}
\put(2,4){\line(1,0){3}}
\put(1,5){\line(1,0){5}}
\put(2,1){\makebox(2,2){$B_{r}$}}
\put(2,3){\makebox(2,1){$i_{0}^{(y-1)}$}}
\put(4,1){\makebox(1,2){$A_{I}$}}
\put(4,3){\makebox(1,1){$l_{k}^{\ast}$}}
\put(5,3){\makebox(2,1){$\leftarrow p_{k}^{\prime}$}}
\put(5,1){\makebox(2,1){$\leftarrow q_{k}^{\prime}$}}
\put(2,5){\makebox(2,1){${\mathstrut y-1}$}}
\put(4,5){\makebox(1,1){${\mathstrut y}$}}
\end{picture},
\end{center}
where $A_{I}$ stands for the stack of the sequence of blocks $I_{r}^{(y)},I_{r-1}^{(y)},\ldots,I_{r}^{(y)}$ in this order (from top to bottom).
Here, $I_{i}^{(y-1)}\preceq I_{i}^{(y)}$ ($i=1,\ldots,r$) so that $B_{r} \preceq A_{I}$ 
and $i_{0}^{(y-1)}\leq l_{k}^{\ast}$.
Therefore, we have 
$\Delta C_{y-1}[p_{k}^{\prime},q_{k}^{\prime}]\preceq\Delta_{k}(C_{y}^{0})$.
The position of $l_{k}^{\ast}$ and those of entries in 
$I_{i}^{(y)}$($i=1,\ldots,r$)
do not change under subsequent operations for 
$l_{k-1}\rightarrow l_{k-1}^{\ast},\ldots,l_{1}\rightarrow l_{1}^{\ast}$.
Thus, the proof of \textbf{(II)} has been completed.

\textbf{(v).}
By the same argument as in \textbf{(i)}, \textbf{(ii)}, and \textbf{(iii)}, it is not hard to show 
$\Delta C_{y-1}[p_{c}^{\prime},q_{c}]\preceq\Delta_{k=c}(C_{y}^{0})$, 
where $p_{c}^{\prime}$ (resp. $q_{c}$) is the position of the top (resp. bottom) box of $\Delta_{k=c}(C_{y}^{0})$.
Note that $q_{c}$ is the initial position of $l_{c}$ in $C_{y}^{0}$.
This completes the proof of \textbf{(I)}.
\end{proof}

The following result may be proven in much the same way as in Lemma~\ref{lem:kink1}.

\begin{lem} \label{lem:kink2}
Suppose that $T=C_{1}C_{2}\cdots C_{n_{c}} \in C_{n}\text{-}\mathrm{SST}_{\mathrm{KN}}$.
Let us set
\begin{align*}
\Tilde{T}  & :=\left(  \phi^{(x+1,y)}\circ\phi^{(x,y-1)}\right)  \circ
\cdots\circ\left(  \phi^{(x+1,x+1)}\circ\phi^{(x,x)}\right)  \\
& \circ(\Phi^{(x+1)})^{-1}\circ\overline{\Phi^{(x+1)}}(T) \quad (2\leq x+1\leq y\leq n_{c}).
\end{align*}
Here, we assume that $\Tilde{T}\neq \emptyset$ and that in the updating process of the tableau from $T$ to $\Tilde{T}$ the semistandardness of the $\mathscr{C}_{n}^{(-)}$-letters part of the tableau is preserved.

\begin{itemize}
\item[(1).]
Suppose that the tableau $\Tilde{T}$ has the following configuration, 
where the left (resp. right) part is the $\mathscr{C}_{n}^{(-)}$ (resp. $\mathscr{C}_{n}^{(+)}$)-letters one $(p\leq q < r\leq s)$.

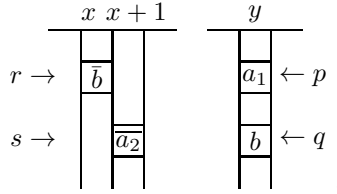
\setlength{\unitlength}{12pt}
\begin{center}
\begin{picture}(10,6)
\put(2,0){\line(0,1){5}}
\put(3,0){\line(0,1){5}}
\put(4,0){\line(0,1){5}}
\put(3,1){\line(1,0){1}}
\put(3,2){\line(1,0){1}}
\put(2,3){\line(1,0){1}}
\put(2,4){\line(1,0){1}}
\put(1,5){\line(1,0){4}}
\put(7,0){\line(0,1){5}}
\put(8,0){\line(0,1){5}}
\put(7,1){\line(1,0){1}}
\put(7,2){\line(1,0){1}}
\put(7,3){\line(1,0){1}}
\put(7,4){\line(1,0){1}}
\put(6,5){\line(1,0){3}}
\put(2,3){\makebox(1,1){$\Bar{b}$}}
\put(3,1){\makebox(1,1){$\overline{a_{2}}$}}
\put(0,1){\makebox(1,1){$s \rightarrow$}}
\put(0,3){\makebox(1,1){$r \rightarrow$}}
\put(1.75,5){\makebox(1,1){${\mathstrut x}$}}
\put(3,5){\makebox(1.5,1){${\mathstrut x+1}$}}
\put(7,1){\makebox(1,1){$b$}}
\put(7,3){\makebox(1,1){$a_{1}$}}
\put(8,1){\makebox(2,1){$\leftarrow q$}}
\put(8,3){\makebox(2,1){$\leftarrow p$}}
\put(7,5){\makebox(1,1){$y$}}
\end{picture}.
\end{center}
Then we have 
\[(q-p)+(s-r)<b-\min(a_{1},a_{2}).
\]

\item[(2).]
Let $\mathscr{J}^{(x)}$ be the set of $\mathscr{J}$-letters in the $x$-th column and 
$\mathscr{I}^{(y)}$ be the set of $\mathscr{I}$-letters in the the $y$-th column and set 
$\mathscr{L}^{(x,y)}:=\mathscr{J}^{(x)}\cap \mathscr{I}^{(y)}$.
If $\sharp\left\{  l\in \mathscr{L}^{(x,y)} \relmiddle| l^{\ast}<b<l\right\}  =\delta$ in $\phi^{(x,y)}(\Tilde{T})$, 
then we have 
\[
(q-p)+(s-r)<b-\min(a_{1},a_{2})-\delta
\]
in the above configuration in $\Tilde{T}$.
\end{itemize}
\end{lem}

\begin{lem} \label{lem:st2}
Suppose that  
$T=C_{1}C_{2}\cdots C_{n_{c}}\in C_{n}\text{-}\mathrm{SST}_{\mathrm{KN}}$.
Let us set 
\[
\Tilde{T}:=\left(  \phi^{(x+1,y)}\circ\phi^{(x,y-1)}\right)  \circ\cdots
\circ\left(  \phi^{(x+1,x+1)}\circ\phi^{(x,x)}\right)  \circ(\Phi^{(x+1)})^{-1}\circ\overline{\Phi^{(x+1)}}(T)
\]
($2\leq x+1\leq y \leq n_{c}$). 
Here, we assume that $\Tilde{T}\neq \emptyset$ and that in the updating process of the tableau from $T$ to $\Tilde{T}$ the semistandardness of the $\mathscr{C}_{n}^{(-)}$-letters part of the tableau is preserved.
Then the $\mathscr{C}_{n}^{(-)}$-letters part of $\phi^{(x,y)}(\Tilde{T})$ is semistandard and 
if $y\leq n_{c}-1$ the $\mathscr{C}_{n}^{(-)}$-letters part of $\left(\phi^{(x+1,y+1)}\circ\phi^{(x,y)}\right)(\Tilde{T})$ is also semistandard.
\end{lem}

\begin{proof}
Let $C_{x}$ (resp. $C_{x}^{0}$) be the $x$-th column of $\phi^{(x,y)}(\Tilde{T})$ (resp. $\Tilde{T}$) and $C_{x+1}$ be the $(x+1)$-st column of $\Tilde{T}$.
In what follows, we show that the $\mathscr{C}_{n}^{(-)}$-letters part of the two-column tableau $C_{x}C_{x+1}$ is semistandard.
If this is true, the claim of Lemma~\ref{lem:st2} is immediate by Lemma~\ref{lem:col_sst1}.

Let us denote by $\mathscr{I}^{(y)}$ the set of $\mathscr{I}$-letters in the $y$-th column of $\Tilde{T}$ 
and by $\mathscr{J}^{(x)}$ the set of $\mathscr{J}$-letters in the $x$-th column of $\Tilde{T}$ and set 
$\mathscr{L}^{(x,y)}:=\mathscr{J}^{(x)}\cap \mathscr{I}^{(y)}=:\left\{  l_{1},\ldots,l_{c}\right\}$.
We adopt the second kind algorithm for $\phi^{(x,y)}$ when we treat the $x$-th column, 
while we adopt the first kind one when we treat the $y$-th column.
We claim that $\overline{\Delta_{k}}(C_{x}^{0})\preceq\Delta C_{x+1}[p_{k}^{\prime},q_{k}^{\prime}]$ 
for all $k=c,c-1,\ldots,1$ so that the $\mathscr{C}_{n}^{(-)}$-letters part of $C_{x}C_{x+1}$ is semistandard, 
where $p_{k}^{\prime}$ (resp. $q_{k}^{\prime}$) is the position of the top (resp. bottom) box of $\overline{\Delta_{k}}(C_{x}^{0})$.
The proof is by induction on $k$.
Namely, we prove

\textbf{(I).}
$\overline{\Delta_{c}}(C_{x}^{0})\preceq\Delta C_{x+1}[p_{c}^{\prime},q_{c}^{\prime}]$.

\textbf{(II).}
$\overline{\Delta_{k}}(C_{x}^{0})\preceq\Delta C_{x+1}[p_{k}^{\prime},q_{k}^{\prime}]$ under the assumption that 
$\overline{\Delta_{k+1}}(C_{x}^{0})\preceq\Delta C_{x+1}[p_{k+1}^{\prime},q_{k+1}^{\prime}]$ 
($k=c-1,\ldots,1$).

We first prove \textbf{(II)}.
Suppose that 
$\left\{  l\in \mathscr{L}^{(x,y)} \relmiddle| l^{\ast}<l_{k}<l\right\}  =:\left\{  l_{k+1},\ldots,l_{k+\delta}\right\} $.
Let $C_{+}^{(y)}$ (resp. $C_{-}^{(x)}$) be the $\mathscr{C}_{n}^{(+)}$ (resp. $\mathscr{C}_{n}^{(-)})$-letters part of 
the $y$-th (resp. $x$-th) column of $\Tilde{T}$ 
and let $C^{(x,y)}$ be the column whose $\mathscr{C}_{n}^{(+)}$ (resp. $\mathscr{C}_{n}^{(-)})$-letters part is $C_{+}^{(y)}$ (resp. $C_{-}^{(x)}$).
Suppose that the operation for $l_{k+1}\rightarrow l_{k+1}^{\ast}$ is finished ($k\leq c-1$).
Let $\Tilde{T}^{\prime}$ be the updated tableau and $C^{(x,y)\prime}$ be the resulting column.
Let us assume that  
$\overline{\Delta_{k+1}}(C_{x})\preceq\Delta C_{x+1}[p_{k+1}^{\prime},q_{k+1}^{\prime}]$.
The filling diagram of the column $C^{(x,y)\prime}$ has the following configuration.
Here, we assume that $r\geq 1$.
The proof for the case when $r=0$ is similar and much simpler.

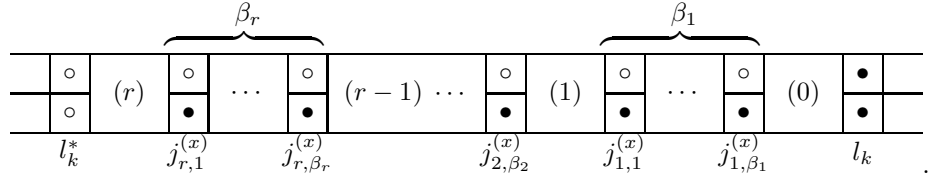
\setlength{\unitlength}{15pt}
\begin{center}
\begin{picture}(23,5)
\put(1,1){\line(0,1){2}}
\put(2,1){\line(0,1){2}}
\put(4,1){\line(0,1){2}}
\put(5,1){\line(0,1){2}}
\put(7,1){\line(0,1){2}}
\put(8,1){\line(0,1){2}}
\put(12,1){\line(0,1){2}}
\put(13,1){\line(0,1){2}}
\put(15,1){\line(0,1){2}}
\put(16,1){\line(0,1){2}}
\put(18,1){\line(0,1){2}}
\put(19,1){\line(0,1){2}}
\put(21,1){\line(0,1){2}}
\put(22,1){\line(0,1){2}}
\put(0,1){\line(1,0){23}}
\put(0,2){\line(1,0){2}}
\put(4,2){\line(1,0){1}}
\put(7,2){\line(1,0){1}}
\put(12,2){\line(1,0){1}}
\put(15,2){\line(1,0){1}}
\put(18,2){\line(1,0){1}}
\put(21,2){\line(1,0){2}}
\put(0,3){\line(1,0){23}}
\put(1,1){\makebox(1,1){$\circ$}}
\put(1,2){\makebox(1,1){$\circ$}}
\put(4,1){\makebox(1,1){$\bullet$}}
\put(4,2){\makebox(1,1){$\circ$}}
\put(7,1){\makebox(1,1){$\bullet$}}
\put(7,2){\makebox(1,1){$\circ$}}
\put(12,1){\makebox(1,1){$\bullet$}}
\put(12,2){\makebox(1,1){$\circ$}}
\put(15,1){\makebox(1,1){$\bullet$}}
\put(15,2){\makebox(1,1){$\circ$}}
\put(18,1){\makebox(1,1){$\bullet$}}
\put(18,2){\makebox(1,1){$\circ$}}
\put(21,1){\makebox(1,1){$\bullet$}}
\put(21,2){\makebox(1,1){$\bullet$}}
\put(2,1){\makebox(2,2){$(r)$}}
\put(5,1){\makebox(2,2){$\cdots$}}
\put(8,1){\makebox(4,2){$(r-1)\;\cdots$}}
\put(13,1){\makebox(2,2){$(1)$}}
\put(16,1){\makebox(2,2){$\cdots$}}
\put(19,1){\makebox(2,2){$(0)$}}
\put(1,0){\makebox(1,1){$l_{k}^{\ast}$}}
\put(4,0){\makebox(1,1){$j_{r,1}^{(x)}$}}
\put(7,0){\makebox(1,1){$j_{r,\beta_{r}}^{(x)}$}}
\put(12,0){\makebox(1,1){$j_{2,\beta_{2}}^{(x)}$}}
\put(15,0){\makebox(1,1){$j_{1,1}^{(x)}$}}
\put(18,0){\makebox(1,1){$j_{1,\beta_{1}}^{(x)}$}}
\put(21,0){\makebox(1,1){$l_{k}$}}
\put(4,2.5){\makebox(4,1){$
\overbrace{
\begin{array}
[c]{cccccc}
& & & & &
\end{array}
}$
}}
\put(5,3.5){\makebox(2,1){$\beta _{r}$}}
\put(15,2.5){\makebox(4,1){$
\overbrace{
\begin{array}
[c]{cccccc}
& & & & &
\end{array}
}$
}}
\put(16,3.5){\makebox(2,1){$\beta _{1}$}}
\end{picture}.
\end{center}
Region $(s)$ contains  $(+)$-slots, $(\pm)$-slots, and $(\times)$-slots.
Let us assume that the numbers of $(+)$-slots, $(\pm)$-slots, and $(\times)$-slots in this region are  
$\alpha_{s}$, $\gamma_{s}$, and $\delta_{s}$, respectively 
and that the position of $(\times)$-slots in region $(s)$ are 
$l_{s,1}^{\ast},\ldots$, and $l_{s,\delta_{s}}^{\ast}$ ($s=0,1,\ldots,r$);
$\left\{  l_{k+1}^{\ast},\ldots,l_{k+\delta}^{\ast}\right\}  =
\left\{l_{0,1}^{\ast},\ldots,l_{0,\delta_{0}}^{\ast},\ldots,l_{r,1}^{\ast},\ldots,l_{r,\delta_{r}}^{\ast}\right\}  $.
Between two regions $(s-1)$ and $(s)$, $\beta_{s}$ $(-)$-slots lie consecutively.

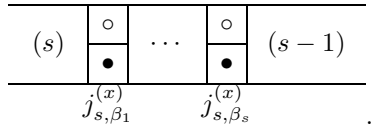
\setlength{\unitlength}{15pt}
\begin{center}
\begin{picture}(9,3)
\put(2,1){\line(0,1){2}}
\put(3,1){\line(0,1){2}}
\put(5,1){\line(0,1){2}}
\put(6,1){\line(0,1){2}}
\put(0,1){\line(1,0){9}}
\put(2,2){\line(1,0){1}}
\put(5,2){\line(1,0){1}}
\put(0,3){\line(1,0){9}}
\put(0,1){\makebox(2,2){$(s)$}}
\put(3,1){\makebox(2,2){$\cdots$}}
\put(6,1){\makebox(3,2){$(s-1)$}}
\put(2,1){\makebox(1,1){$\bullet$}}
\put(2,2){\makebox(1,1){$\circ$}}
\put(5,1){\makebox(1,1){$\bullet$}}
\put(5,2){\makebox(1,1){$\circ$}}
\put(2,0){\makebox(1,1){$j_{s,\beta_{1}}^{(x)}$}}
\put(5,0){\makebox(1,1){$j_{s,\beta_{s}}^{(x)}$}}
\end{picture}.
\end{center}
The updated tableau $\Tilde{T}^{\prime}$ has the following configuration.
There are no $\overline{\mathscr{L}^{(x,y)\ast}}$-letters below the box containing $\overline{l_{k}}$ in the $x$-th column 
because we adopt the second kind algorithm for $\phi^{(x,y)}$ in the $x$-th column, 
while $\mathscr{L}^{(x,y)\ast}$-letters may exist above the box containing $l_{k}$ in the $y$-th column.

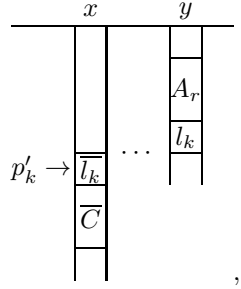
\setlength{\unitlength}{12pt}
\begin{center}
\begin{picture}(7,9)
\put(2,0){\line(0,1){8}}
\put(3,0){\line(0,1){8}}
\put(5,3){\line(0,1){5}}
\put(6,3){\line(0,1){5}}
\put(2,1){\line(1,0){1}}
\put(2,3){\line(1,0){1}}
\put(2,4){\line(1,0){1}}
\put(5,4){\line(1,0){1}}
\put(5,5){\line(1,0){1}}
\put(5,7){\line(1,0){1}}
\put(0,8){\line(1,0){7}}
\put(3,3){\makebox(2,2){$\cdots$}}
\put(2,1){\makebox(1,2){$\overline{C}$}}
\put(2,3){\makebox(1,1){$\overline{l_{k}}$}}
\put(0,3){\makebox(2,1)[l]{$p_{k}^{\prime} \rightarrow$}}
\put(5,4){\makebox(1,1){$l_{k}$}}
\put(5,5){\makebox(1,2){$A_{r}$}}
\put(2,8){\makebox(1,1){$x$}}
\put(5,8){\makebox(1,1){$y$}}
\end{picture},
\end{center}
where $A_{r}$ is the stack of the sequence of blocks $I_{r}^{(y)},I_{r-1}^{(y)},\ldots,I_{0}^{(y)}$ in this order (from top to bottom) and 
$\overline{C}$ is the stack of the sequence of blocks 
$\overline{L_{0}^{(x)}},\overline{J_{1}^{(x)}},\ldots,\overline{J_{r}^{(x)}},\overline{L_{r}^{(x)}}$ 
in this order (from top to bottom).
The block $\overline{J_{s}^{(x)}}$ consists of consecutive 
$\overline{\mathscr{J}^{(x)}\backslash \mathscr{L}^{(x,y)}}$-letters 
$\{\overline{j_{s,\beta_{s}}^{(x)}},\ldots,\overline{j_{s,1}^{(x)}}\}$,
where
\[
j_{s,\beta_{s}-t+1}^{(x)}=l_{k}-\sum_{i=1}^{s-1}\beta_{i}-\sum_{i=0}^{s-1}\tau_{i}-t \quad (t=1,\ldots,\beta_{s}) 
\]
with $\tau_{i}:=\alpha_{i}+\gamma_{i}+\delta_{i}$ 
and $\overline{L_{s}^{(x)}}$ is the block of $\gamma_{s}$ $\overline{\mathscr{L}^{(x,y)}}$-letters
($s=0,1,\ldots,r$).
The block $I_{s}^{(y)}$ consists of consecutive $\tau_{s}$ $\mathscr{C}_{n}^{(+)}$-letters 
$\{i_{s,1}^{(y)},\ldots,i_{s,\tau_{s}}^{(y)}\}$,
where
\[
i_{s,\tau_{s}-t+1}^{(y)}=l_{k}-\sum_{i=1}^{s}\beta_{i}-\sum_{i=0}^{s-1}
\tau_{i}-t \quad (t=1,\ldots,\tau_{s};s=0,\ldots,r).
\]
Let us assume that the $(x+1)$-st column has the following configuration.

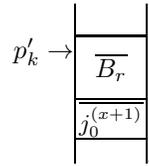
\setlength{\unitlength}{12pt}
\begin{center}
\begin{picture}(4.25,5)
\put(2,0){\line(0,1){5}}
\put(4.25,0){\line(0,1){5}}
\put(2,0,75){\line(1,0){2.25}}
\put(2,2){\line(1,0){2.25}}
\put(2,4){\line(1,0){2.25}}
\put(2,0.75){\makebox(2.25,1.25){\small $\overline{j_{0}^{(x+1)}}$}}
\put(2,2){\makebox(2.25,2){$\overline{B_{r}}$}}
\put(0,3){\makebox(2,1){$p_{k}^{\prime} \rightarrow$}}
\end{picture},
\end{center}
where $\overline{B_{r}}$ is the stack of the sequence of blocks 
$\overline{J_{1}^{(x+1)}},\overline{J_{2}^{(x+1)}},\ldots,\overline{J_{r}^{(x+1)}}$ in this order (from top to bottom) 
and the position of the top box in $\overline{B_{r}}$ is $p_{k}^{\prime}$ 
(the block $\overline{B_{r}}$ is not empty because of the assumption of $r\geq 1$).
The block $\overline{J_{s}^{(x+1)}}$ consists of $\beta_{s}$ $\mathscr{C}_{n}^{(-)}$-letters 
$\left\{  \overline{j_{s,\beta_{s}}^{(x+1)}},\ldots,\overline{j_{s,1}^{(x+1)}}\right\}  $ 
so that $\left\vert \overline{J_{s}^{(x+1)}}\right\vert =\left\vert \overline{J_{s}^{(x)}}\right\vert $ 
($s=1,\ldots,r$).

\textbf{(i).}
We claim that 
$\overline{j_{1,\beta_{1}}^{(x)}}\preceq\overline{j_{1,\beta_{1}}^{(x+1)}}$, i.e., 
$j_{1,\beta_{1}}^{(x+1)}\leq j_{1,\beta_{1}}^{(x)}=l_{k}-\tau_{0}-1$.
If this is not true, 
$j_{1,\beta_{1}}^{(x+1)}\in\{l_{k}-\tau_{0},l_{k}-\tau_{0}+1,\ldots,l_{k}\}$ 
$(\overline{l_{k}}\preceq\overline{j_{1,\beta_{1}}^{(x+1)}})$, i.e., 
$j_{1,\beta_{1}}^{(x+1)}$ is in the block $I_{0}^{(y)}$ or $j_{1,\beta_{1}}^{(x+1)}=l_{k}$.
Suppose $j_{1,\beta_{1}}^{(x+1)}=l_{k}-t$
($t=0,1,\ldots,\tau_{0}$).
The updated tableau $\Tilde{T}^{\prime}$ has the following configuration.

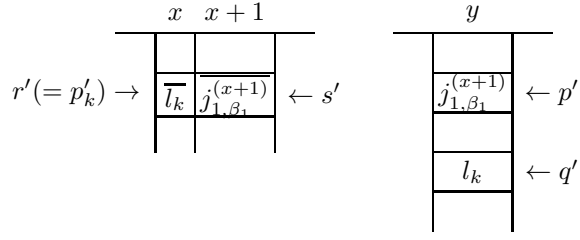
\setlength{\unitlength}{15pt}
\begin{center}
\begin{picture}(15,6)
\put(4,2){\line(0,1){3}}
\put(5,2){\line(0,1){3}}
\put(7,2){\line(0,1){3}}
\put(11,0){\line(0,1){5}}
\put(13,0){\line(0,1){5}}
\put(4,2.9){\line(1,0){3}}
\put(4,4){\line(1,0){3}}
\put(11,1){\line(1,0){2}}
\put(11,2){\line(1,0){2}}
\put(11,3){\line(1,0){2}}
\put(11,4){\line(1,0){2}}
\put(3,5){\line(1,0){5}}
\put(10,5){\line(1,0){4}}
\put(4,2.9){\makebox(1,1){$\overline{l_{k}}$}}
\put(5,2.9){\makebox(2,1){$\overline{j_{1,\beta _{1}}^{(x+1)}}$}}
\put(11,1){\makebox(2,1){$l_{k}$}}
\put(11,3){\makebox(2,1){$j_{1,\beta _{1}}^{(x+1)}$}}
\put(4,5){\makebox(1,1){${\mathstrut x}$}}
\put(5,5){\makebox(2,1){${\mathstrut x+1}$}}
\put(11,5){\makebox(2,1){$y$}}
\put(0,3){\makebox(4,1){$r^{\prime}(=p_{k}^{\prime}) \rightarrow$}}
\put(7,3){\makebox(2,1){$\leftarrow s^{\prime}$}}
\put(13,1){\makebox(2,1){$\leftarrow q^{\prime}$}}
\put(13,3){\makebox(2,1){$\leftarrow p^{\prime}$}}
\end{picture}.
\end{center}
Let $r_{k}$ and $s_{k}$ be the initial position of $\overline{l_{k}}$ in the $x$-th column and 
that of $\overline{j_{1,\beta_{1}}^{(x+1)}}$ in the $(x+1)$-st column of $\Tilde{T}$, respectively.
We consider the following two cases separately:
\begin{description}
\item[(a)]
$j_{1,\beta_{1}}^{(x+1)}\notin \mathscr{L}^{(x,y)\ast}$.
\item[(b)]
$j_{1,\beta_{1}}^{(x+1)}\in \mathscr{L}^{(x,y)\ast}$.
\end{description}

\textbf{Case (a).}
The entry $j_{1,\beta_{1}}^{(x+1)}$ exists initially in the $y$-th column of $\Tilde{T}$.
Let $p_{k}$ and $q_{k}$ be the initial position of $j_{1,\beta_{1}}^{(x+1)}$ and that of $l_{k}$ 
in the $y$-th column of $\Tilde{T}$, respectively.
Then $s_{k}=s^{\prime}$ and $r_{k} \leq r^{\prime}$ because $\overline{l_{k}}$ is relocated downward by previous operations for 
$l_{c}\rightarrow l_{c}^{\ast},\ldots,l_{k+1}\rightarrow l_{k+1}^{\ast}$ or still lies at the initial position.
Suppose that $\delta^{\prime}$ $\mathscr{L}^{(x,y)\ast}$-letters appear between the $p^{\prime}$-th box and the $q^{\prime}$-th box 
in the $y$-th column ($\delta^{\prime}\leq\delta$).
Then $q^{\prime}-p^{\prime}=q_{k}-p_{k}+\delta^{\prime}$ so that
\begin{align} \label{eq:ineq3}
q_{k}-p_{k}+s_{k}-r_{k}  & \geq q^{\prime}-p^{\prime}+s^{\prime}-r^{\prime}-\delta^{\prime}=t-\delta^{\prime}\\
& \geq l_{k}-j_{1,\beta_{1}}^{(x+1)}-\delta \nonumber,
\end{align}
which contradicts the assertion of Lemma~\ref{lem:kink2}.

\textbf{Case (b).} 
We can write $j_{1,\beta_{1}}^{(x+1)}=a^{\ast}$ ($a\in \mathscr{L}^{(x,y)}$).
Let $q_{k}$ be the initial position of $l_{k}$ in the $y$-th column of $\Tilde{T}$.
Furthermore, let us suppose that the initial entry at the $p_{k}$-th position ($p_{k}\leq q_{k}$) in the $y$-th column of $\Tilde{T}$ is $b$ and that the operation $a \rightarrow a^{\ast}$ replaces the entry $b$ by $a^{\ast}$.

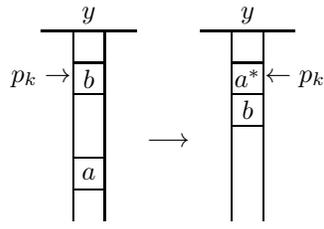
\setlength{\unitlength}{12pt}
\begin{center}
\begin{picture}(11,7)
\put(3,0){\line(0,1){6}}
\put(4,0){\line(0,1){6}}
\put(3,1){\line(1,0){1}}
\put(3,2){\line(1,0){1}}
\put(3,4){\line(1,0){1}}
\put(3,5){\line(1,0){1}}
\put(3,1){\makebox(1,1){$a$}}
\put(3,4){\makebox(1,1){$b$}}
\put(1,4){\makebox(2,1){$p_{k} \rightarrow$}}
\put(5,2){\makebox(2,1){$\longrightarrow$}}
\put(8,0){\line(0,1){6}}
\put(9,0){\line(0,1){6}}
\put(8,3){\line(1,0){1}}
\put(8,4){\line(1,0){1}}
\put(8,5){\line(1,0){1}}
\put(8,3){\makebox(1,1){$b$}}
\put(8,4){\makebox(1,1){$a^{\ast}$}}

\put(9,4){\makebox(2,1){$\leftarrow p_{k}$}}
\put(2,6){\line(1,0){3}}
\put(7,6){\line(1,0){3}}
\put(3,6){\makebox(1,1){$y$}}
\put(8,6){\makebox(1,1){$y$}}
\end{picture},
\end{center}
so that $b>j_{1,\beta_{1}}^{(x+1)}$.
The initial tableau $\Tilde{T}$ has the following configuration.

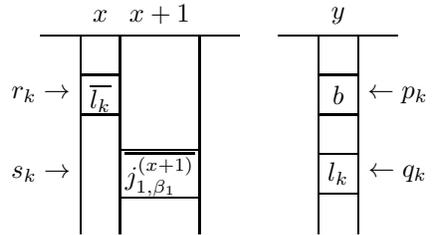
\setlength{\unitlength}{15pt}
\begin{center}
\begin{picture}(11,6)
\put(2,0){\line(0,1){5}}
\put(3,0){\line(0,1){5}}
\put(5,0){\line(0,1){5}}
\put(8,0){\line(0,1){5}}
\put(9,0){\line(0,1){5}}
\put(2,3){\line(1,0){1}}
\put(2,4){\line(1,0){1}}
\put(3,0.9){\line(1,0){2}}
\put(3,2.1){\line(1,0){2}}
\put(8,1){\line(1,0){1}}
\put(8,2){\line(1,0){1}}
\put(8,3){\line(1,0){1}}
\put(8,4){\line(1,0){1}}
\put(1,5){\line(1,0){5}}
\put(7,5){\line(1,0){3}}
\put(2,2.9){\makebox(1,1){$\overline{l_{k}}$}}
\put(3,1){\makebox(2,1){$\overline{j_{1,\beta _{1}}^{(x+1)}}$}}
\put(8,1){\makebox(1,1){$l_{k}$}}
\put(8,3){\makebox(1,1){$b$}}
\put(2,5){\makebox(1,1){${\mathstrut x}$}}
\put(3,5){\makebox(2,1){${\mathstrut x+1}$}}
\put(8,5){\makebox(1,1){$y$}}
\put(0,1){\makebox(2,1){$s_{k} \rightarrow$}}
\put(0,3){\makebox(2,1){$r_{k} \rightarrow$}}
\put(9,1){\makebox(2,1){$\leftarrow q_{k}$}}
\put(9,3){\makebox(2,1){$\leftarrow p_{k}$}}
\end{picture}.
\end{center}
Inequality~\eqref{eq:ineq3} still holds in this case and this contradicts the assertion of Lemma~\ref{lem:kink2}.

In both cases, we have
$j_{1,\beta_{1}}^{(x+1)}\leq j_{1,\beta_{1}}^{(x)}=l_{k}-\tau_{0}-1$
and
\[
j_{1,\beta_{1}-t+1}^{(x+1)}\leq j_{1,\beta_{1}-t+1}^{(x)}=l_{k}-\tau_{0}-t \quad (t=1,\ldots,\beta_{1}).
\]

\textbf{(ii).}
Suppose that
\[
j_{s,1}^{(x+1)}\leq j_{s,1}^{(x)}=l_{k}-\sum_{i=1}^{s}\beta_{i}-\sum_{i=0}^{s-1}\tau_{i} \quad (s=1,\ldots,r-1).
\]
This is satisfied for $s=1$.
Under these assumptions, let us show that
\[
j_{s+1,\beta_{s+1}}^{(x+1)}\leq j_{s+1,\beta_{s+1}}^{(x)}=l_{k}-\sum_{i=1}^{s}\beta_{i}-\sum_{i=0}^{s}\tau_{i}-1.
\]
If this is not true,
\[
l_{k}-\sum_{i=1}^{s}\beta_{i}-\sum_{i=0}^{s}\tau_{i}  \leq
j_{s+1,\beta_{s+1}}^{(x+1)}\leq j_{s,1}^{(x+1)}-1 \leq l_{k}-\sum_{i=1}^{s}\beta_{i}-\sum_{i=0}^{s-1}\tau_{i}-1.
\]
Suppose
$j_{s+1,\beta_{s+1}}^{(x+1)}=l_{k}-\sum_{i=1}^{s}\beta_{i}-\sum_{i=0}^{s-1}\tau_{i}-t=
i_{s,\tau_{s}-t+1}^{(x)}$ ($t=1,\ldots,\tau_{s}$). 
Then the tableau $\Tilde{T}^{\prime}$ has the following configuration. 

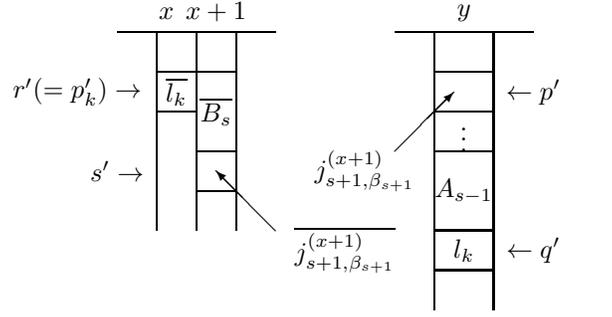
\setlength{\unitlength}{15pt}
\begin{center}
\begin{picture}(15.5,8)
\put(4,2){\line(0,1){5}}
\put(5,2){\line(0,1){5}}
\put(6,2){\line(0,1){5}}
\put(11,0){\line(0,1){7}}
\put(12.5,0){\line(0,1){7}}
\put(5,3){\line(1,0){1}}
\put(5,4){\line(1,0){1}}
\put(4,5){\line(1,0){1}}
\put(4,6){\line(1,0){2}}
\put(3,7){\line(1,0){4}}
\put(11,1){\line(1,0){1.5}}
\put(11,2){\line(1,0){1.5}}
\put(11,4){\line(1,0){1.5}}
\put(11,5){\line(1,0){1.5}}
\put(11,6){\line(1,0){1.5}}
\put(10,7){\line(1,0){3.5}}
\put(3.75,7){\makebox(1,1){${\mathstrut x}$}}
\put(4.75,7){\makebox(1.5,1){${\mathstrut x+1}$}}
\put(11,7){\makebox(1.5,1){$y$}}
\put(4,5){\makebox(1,1){$\overline{l_{k}}$}}
\put(5,4){\makebox(1,2){$\overline{B_{s}}$}}
\put(11,1){\makebox(1.5,1){$l_{k}$}}
\put(11,2){\makebox(1.5,2){$A_{s-1}$}}
\put(11,4){\makebox(1.5,1){$\vdots$}}
\put(0,5){\makebox(4,1){$r^{\prime}(=p_{k}^{\prime}) \rightarrow$}}
\put(2,3){\makebox(2,1){$s^{\prime} \rightarrow$}}
\put(12.5,5){\makebox(2,1){$\leftarrow p^{\prime}$}}
\put(12.5,1){\makebox(2,1){$\leftarrow q^{\prime}$}}

\put(7,2){\vector(-1,1){1.5}}
\put(7.5,1){\makebox(2.5,1){$\overline{j_{s+1,\beta _{s+1}}^{(x+1)}}$}}

\put(10,4){\vector(1,1){1.5}}
\put(8,3){\makebox(2.5,1){$j_{s+1,\beta _{s+1}}^{(x+1)}$}}

\end{picture},
\end{center}
where $A_{s-1}$ denotes the stack of blocks $I_{s-1}^{(y)},\ldots,I_{0}^{(y)}$ in this order (from top to bottom) 
and $\overline{B_{s}}$ denotes the stack of blocks $\overline{J_{1}^{(x+1)}},\ldots,\overline{J_{s}^{(x+1)}}$ 
in this order (from top to bottom).
We consider the following two cases separately:
\begin{description}
\item[(a)]
$j_{s+1,\beta_{s+1}}^{(x+1)}\notin \mathscr{L}^{(x,y)\ast}$.
\item[(b)]
$j_{s+1,\beta_{s+1}}^{(x+1)}\in \mathscr{L}^{(x,y)\ast}$.
\end{description}

\textbf{Case (a).}
The entry $j_{s+1,\beta_{s+1}}^{(x+1)}$ exists initially in the $y$-th column of $\Tilde{T}$.
Let $r_{k}$ and $s_{k}$ be the initial position of $\overline{l_{k}}$ in the $x$-th column and 
that of $\overline{j_{s+1,\beta_{s+1}}^{(x+1)}}$ in the $(x+1)$-st column of $\Tilde{T}$, respectively.
Let $p_{k}$ and $q_{k}$ be the initial position of $j_{s+1,\beta_{s+1}}^{(x+1)}$ 
and that of $l_{k}$ in the $y$-th column of $\Tilde{T}$, respectively.
Then $s_{k}=s^{\prime}$ and $r_{k} \leq r^{\prime}$.
Suppose that $\delta^{\prime}$ $\mathscr{L}^{(x,y)\ast}$-letters appear between the $p^{\prime}$-th box and the $r^{\prime}$-th box 
in the $y$-th column ($\delta^{\prime}\leq\delta$).
Then $q^{\prime}-p^{\prime}=q_{k}-p_{k}+\delta^{\prime}$ so that
$q_{k}-p_{k}+s_{k}-r_{k}\geq q^{\prime}-p^{\prime}+s^{\prime}-r^{\prime}-\delta$.
Here, 
$s^{\prime}-r^{\prime}=\sum_{i=1}^{s}\left\vert J_{i}^{(x+1)}\right\vert =\sum_{i=1}^{s}\beta_{i}$ 
and 
$q^{\prime}-p^{\prime}=\sum_{i=0}^{s-1}\left\vert I_{i}^{(y)}\right\vert +t=\sum_{s=0}^{s-1}\tau_{i}+t$.
Therefore, we have 
\begin{equation} \label{eq:ineq4}
q_{k}-p_{k}+s_{k}-r_{k}\geq\sum_{i=1}^{s}\beta_{i}+\sum_{s=0}^{s-1}\tau_{i}+t-\delta=l_{k}-j_{s+1,\beta_{s+1}}^{(x+1)}-\delta.
\end{equation}
This contradicts the assertion of Lemma~\ref{lem:kink2}.

\textbf{Case (b).}
We can write $j_{s+1,\beta_{s+1}}^{(x+1)}=a^{\ast}$ ($a\in \mathscr{L}^{(x,y)}$).  
Let $q_{k}$ be the initial position of $l_{k}$ in the $y$-th column of $\Tilde{T}$.
Furthermore, let us suppose that the initial entry at the $p_{k}$-th position ($p_{k}\leq q_{k}$) in the $y$-th column of $\Tilde{T}$ is $b$ and that the operation $a \rightarrow a^{\ast}$ replaces the entry $b$ by $a^{\ast}$ so that $b>j_{s+1,\beta_{s+1}}^{(x+1)}$.
The initial tableau $\Tilde{T}$ has the following configuration.

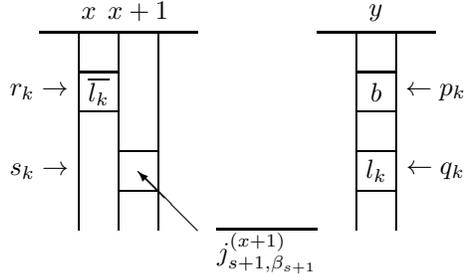
\setlength{\unitlength}{15pt}
\begin{center}
\begin{picture}(12,7)
\put(2,1){\line(0,1){5}}
\put(3,1){\line(0,1){5}}
\put(4,1){\line(0,1){5}}
\put(9,1){\line(0,1){5}}
\put(10,1){\line(0,1){5}}
\put(2,4){\line(1,0){1}}
\put(2,5){\line(1,0){1}}
\put(3,2){\line(1,0){1}}
\put(3,3){\line(1,0){1}}
\put(9,2){\line(1,0){1}}
\put(9,3){\line(1,0){1}}
\put(9,4){\line(1,0){1}}
\put(9,5){\line(1,0){1}}
\put(1,6){\line(1,0){4}}
\put(8,6){\line(1,0){3}}
\put(2,4){\makebox(1,1){$\overline{l_{k}}$}}
\put(9,2){\makebox(1,1){$l_{k}$}}
\put(9,4){\makebox(1,1){$b$}}
\put(1.75,6){\makebox(1,1){${\mathstrut x}$}}
\put(2.75,6){\makebox(1.5,1){${\mathstrut x+1}$}}
\put(9,6){\makebox(1,1){$y$}}
\put(0,2){\makebox(2,1){$s_{k} \rightarrow$}}
\put(0,4){\makebox(2,1){$r_{k} \rightarrow$}}
\put(10,2){\makebox(2,1){$\leftarrow q_{k}$}}
\put(10,4){\makebox(2,1){$\leftarrow p_{k}$}}

\put(5,1){\vector(-1,1){1.5}}
\put(5.5,0){\makebox(2.5,1){$\overline{j_{s+1,\beta _{s+1}}^{(x+1)}}$}}
\end{picture}.
\end{center}
Inequality~\eqref{eq:ineq4} still holds in this case and this contradicts the assertion of Lemma~\ref{lem:kink2}.

In both cases, we have $j_{s+1,t}^{(x+1)}\leq j_{s+1,t}^{(x)}$ ($t=1,\ldots,\beta_{s+1}$).
From \textbf{(i)} and \textbf{(ii)} and by induction, we have
\[
j_{r,1}^{(x+1)}\leq j_{r,1}^{(x)}=l_{k}-\sum_{i=1}^{r}\beta_{i}-\sum_{i=0}^{r-1}\tau_{i}.
\]

\textbf{(iii).}
We claim that
$\overline{l_{k}^{\ast}}\preceq\overline{j_{0}^{(x+1)}}$ ($j_{0}^{(x+1)}\leq l_{k}^{\ast}$).
If this is not true, then
\[
l_{k}-\sum_{i=1}^{r}\beta_{i}-\sum_{i=0}^{r}\tau_{i}(=l_{k}^{\ast}+1)\leq j_{0}^{(x+1)}\leq j_{r,1}^{(x)}-1
\leq l_{k}-\sum_{i=1}^{r}\beta_{i}-\sum_{i=0}^{r-1}\tau_{i}-1.
\]
Suppose 
$j_{0}^{(x+1)}=l_{k}-\sum_{i=1}^{r}\beta_{i}-\sum_{i=0}^{r-1}\tau_{i}-t=i_{r,\tau_{r}-t+1}^{(y)}$ ($t=1,\ldots,\tau_{r}$), 
the updated tableau $\Tilde{T}^{\prime}$  has the following configuration.

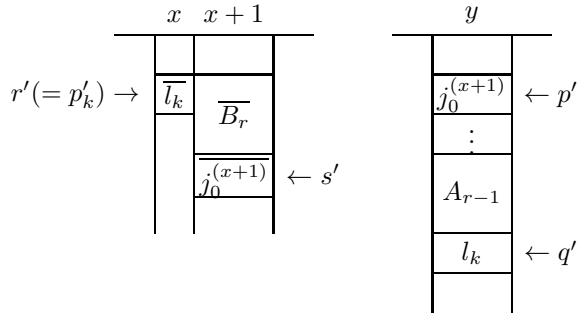
\setlength{\unitlength}{15pt}
\begin{center}
\begin{picture}(15.5,8)
\put(4,2){\line(0,1){5}}
\put(5,2){\line(0,1){5}}
\put(7,2){\line(0,1){5}}
\put(11,0){\line(0,1){7}}
\put(13,0){\line(0,1){7}}
\put(5,2.9){\line(1,0){2}}
\put(5,4){\line(1,0){2}}
\put(4,5){\line(1,0){1}}
\put(4,6){\line(1,0){3}}
\put(3,7){\line(1,0){5}}
\put(11,1){\line(1,0){2}}
\put(11,2){\line(1,0){2}}
\put(11,4){\line(1,0){2}}
\put(11,5){\line(1,0){2}}
\put(11,6){\line(1,0){2}}
\put(10,7){\line(1,0){4}}
\put(4,7){\makebox(1,1){${\mathstrut x}$}}
\put(5,7){\makebox(2,1){${\mathstrut x+1}$}}
\put(11,7){\makebox(2,1){$y$}}
\put(4,5){\makebox(1,1){$\overline{l_{k}}$}}
\put(5,2.9){\makebox(2,1){$\overline{j_{0}^{(x+1)}}$}}
\put(5,4){\makebox(2,2){$\overline{B_{r}}$}}
\put(11,1){\makebox(2,1){$l_{k}$}}
\put(11,2){\makebox(2,2){$A_{r-1}$}}
\put(11,4){\makebox(2,1){$\vdots$}}
\put(11,5){\makebox(2,1){$j_{0}^{(x+1)}$}}
\put(0,5){\makebox(4,1){$r^{\prime}(=p_{k}^{\prime}) \rightarrow$}}
\put(7,3){\makebox(2,1){$\leftarrow s^{\prime}$}}
\put(13,5){\makebox(2,1){$\leftarrow p^{\prime}$}}
\put(13,1){\makebox(2,1){$\leftarrow q^{\prime}$}}
\end{picture}.
\end{center}
The same argument as in \textbf{(ii)} leads to a contradiction.
Hence we have $j_{0}^{(x+1)}\leq l_{k}^{\ast}$.

\textbf{(iv).}
When the operation (B) for $l_{k}\rightarrow l_{k}^{\ast}$ is finished, the updated tableau has the following configuration.

\setlength{\unitlength}{12pt}
\begin{center}
\begin{picture}(6,8)
\put(2,0){\line(0,1){5}}
\put(3.2,0){\line(0,1){5}}
\put(5.4,0){\line(0,1){5}}
\put(2,0.7){\line(1,0){3.4}}
\put(2,2.2){\line(1,0){3.4}}
\put(2,4){\line(1,0){3.4}}
\put(1,5){\line(1,0){5}}
\put(2.2,5){\makebox(1,1){${\mathstrut x}$}}
\put(3.4,5){\makebox(2,1){${\mathstrut x+1}$}}
\put(2.1,1){\makebox(1,1){$\overline{l_{k}^{\ast}}$}}
\put(2.1,2){\makebox(1,2){$\overline{C_{J}}$}}
\put(3.3,1){\makebox(2,1){$\overline{j_{0}^{(x+1)}}$}}
\put(3.3,2){\makebox(2,2){$\overline{B_{r}}$}}
\put(0,0.8){\makebox(2,1){$q_{k}^{\prime} \rightarrow$}}
\put(0,3){\makebox(2,1){$p_{k}^{\prime} \rightarrow$}}
\end{picture},
\end{center}
where $\overline{C_{J}}$ stands for the stack of the sequence of blocks $\overline{J_{1}^{(x)}},\overline{J_{2}^{(x)}},\ldots,\overline{J_{r}^{(x)}}$ in this order (from top to bottom)
Here, $\overline{J_{i}^{(x)}}\preceq\overline{J_{i}^{(x+1)}}$ ($i=1,\ldots,r$) so that 
$\overline{C_{J}}\preceq\overline{B_{r}}$ and 
$\overline{l_{k}^{\ast}}\preceq\overline{j_{0}^{(x+1)}}$.
Therefore,  $\overline{\Delta_{k}}(C_{x}^{0})\preceq\Delta C_{x+1}[p_{k}^{\prime},q_{k}^{\prime}]$.
The position of $\overline{l_{k}^{\ast}}$ and those of entries in $\overline{J_{i}^{(x)}}$ ($i=1,\ldots,r$) 
do not change under subsequent operations for 
$l_{k-1}\rightarrow l_{k-1}^{\ast},\ldots,l_{1}\rightarrow l_{1}^{\ast}$.
Thus, the proof of \textbf{(II)} has been completed.

\textbf{(v).}
By the same argument as in \textbf{(i)}, \textbf{(ii)}, and \textbf{(iii)}, it is not hard to show 
$\overline{\Delta_{k=c}}(C_{x}^{0})\preceq\Delta C_{x+1}[p_{c},q_{c}^{\prime}]$, 
where $p_{c}$ (resp. $q_{c}^{\prime}$) is the position of the top (resp. bottom) box of 
$\overline{\Delta_{k=c}}(C_{x}^{0})$.
Note that $p_{c}$ is the initial position of $\overline{l_{c}}$ in $C_{x}^{0}$.
This completes the proof of \textbf{(I)}.
\end{proof}

We can prove the following Lemma~\ref{lem:st3} and Lemma~\ref{lem:st4} 
in the similar manner of the proof of Lemma~\ref{lem:st1} and Lemma~\ref{lem:st2}.
The proof of Lemma~\ref{lem:st3} uses Lemma~\ref{lem:kink3} instead of Lemma~\ref{lem:kink1} and 
that of Lemma~\ref{lem:st4} uses Lemma~\ref{lem:kink4} instead of Lemma~\ref{lem:kink2}.
Lemma~\ref{lem:kink3} and Lemma~\ref{lem:kink4} can be also proven by the similar manner of the proof of Lemma~\ref{lem:kink1} (2).

\begin{lem} \label{lem:st3}
Suppose that $T=C_{1}C_{2}\cdots C_{n_{c}} \in C_{n}\text{-}\mathrm{SST}_{\mathrm{KN}}$.
Let us set
\[
\Tilde{T}:=\begin{cases}
\overline{\Phi^{(x+1)}}(T) & (1\leq x\leq n_{c}-1), \\
T & (x=n_{c}).
\end{cases}
\]
Here, we assume that $\Tilde{T}\neq \emptyset$ and that in the updating process of the tableau from $T$ to $\Tilde{T}$ the semistandardness of the $\mathscr{C}_{n}^{(+)}$-letters part of the tableau is preserved.
Then the $\mathscr{C}_{n}^{(+)}$-letters part of $\phi^{(x,x)}(\Tilde{T})$ is semistandard.
\end{lem}

\begin{lem} \label{lem:st4}
Suppose that $T=C_{1}C_{2}\cdots C_{n_{c}} \in C_{n}\text{-}\mathrm{SST}_{\mathrm{KN}}$.
Let us set 
\[
\Tilde{T}:=(\Phi^{(x+1)})^{-1}\circ\overline{\Phi^{(x+1)}}(T) \quad (1\leq x\leq n_{c}-1). 
\]
Here, we assume that $\Tilde{T}\neq \emptyset$ and that in the updating process of the tableau from $T$ to $\Tilde{T}$ the semistandardness of the $\mathscr{C}_{n}^{(-)}$-letters part of the tableau is preserved.
Then the $\mathscr{C}_{n}^{(-)}$-letters part of $\phi^{(x,x)}(\Tilde{T})$ and that  of 
$\left(\phi^{(x+1,x+1)}\circ\phi^{(x,x)}\right)(\Tilde{T})$ are semistandard.
\end{lem}

\begin{lem} \label{lem:kink3}
Suppose that $T=C_{1}C_{2}\cdots C_{n_{c}} \in C_{n}\text{-}\mathrm{SST}_{\mathrm{KN}}$.
Let us set
\[
\Tilde{T}:=\begin{cases}
\overline{\Phi^{(x+1)}}(T) & (1\leq x\leq n_{c}-1), \\
T & (x=n_{c}).
\end{cases}
\]
Here, we assume that $\overline{\Phi^{(x+1)}}$ is well-defined on $T$ when $1\leq x\leq n_{c}-1$.
Suppose that $\Tilde{T}$ has the following configuration ($p\leq q<r \leq s$).

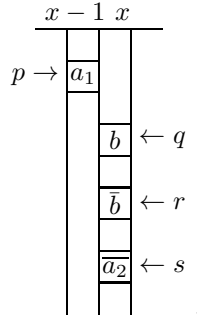
\setlength{\unitlength}{12pt}
\begin{center}
\begin{picture}(6,10)
\put(2,0){\line(0,1){9}}
\put(3,0){\line(0,1){9}}
\put(4,0){\line(0,1){9}}
\put(2,7){\line(1,0){1}}
\put(2,8){\line(1,0){1}}
\put(3,1){\line(1,0){1}}
\put(3,2){\line(1,0){1}}
\put(3,3){\line(1,0){1}}
\put(3,4){\line(1,0){1}}
\put(3,5){\line(1,0){1}}
\put(3,6){\line(1,0){1}}
\put(1,9){\line(1,0){4}}
\put(0,7){\makebox(2,1){$p\rightarrow$}}
\put(4,1){\makebox(2,1){$\leftarrow s$}}
\put(4,3){\makebox(2,1){$\leftarrow r$}}
\put(4,5){\makebox(2,1){$\leftarrow q$}}
\put(1.5,9){\makebox(1.5,1){${\mathstrut x-1}$}}
\put(3.25,9){\makebox(1,1){${\mathstrut x}$}}
\put(2,7){\makebox(1,1){$a_{1}$}}
\put(3,1){\makebox(1,1){$\overline{a_{2}}$}}
\put(3,3){\makebox(1,1){$\Bar{b}$}}
\put(3,5){\makebox(1,1){$b$}}
\end{picture}.
\end{center}
Then we have $(q-p)+(s-r)<b-\min(a_{1},a_{2})$ because the two-column tableau $C_{x-1}C_{x}$ is 
KN-admissible (Definition~\ref{df:KN_C} (C2)).
Let $\mathscr{J}^{(x)}$($\mathscr{I}^{(x)}$) be the set of 
$\mathscr{J}$($\mathscr{I}$)-letters 
in the $x$-th column and set $\mathscr{L}^{(x,x)}:=\mathscr{J}^{(x)}\cap \mathscr{I}^{(x)}$.
If $\sharp\left\{  l\in \mathscr{L}^{(x,x)} \relmiddle| l^{\ast}<b<l\right\} =\delta$, then we have 
$(q-p)+(s-r)<b-\min(a_{1},a_{2})-\delta$ in the above configuration.
\end{lem}

\begin{lem} \label{lem:kink4}
Suppose that $T=C_{1}C_{2}\cdots C_{n_{c}} \in C_{n}\text{-}\mathrm{SST}_{\mathrm{KN}}$.
Let us set
\[
\Tilde{T}:=(\Phi^{(x+1)})^{-1}\circ\overline{\Phi^{(x+1)}}(T) \quad (1\leq x\leq n_{c}-1).
\]
Here, we assume that $\overline{\Phi^{(x+1)}}$ is well-defined on $T$.
Suppose that $\Tilde{T}$ has the following configuration ($p\leq q<r \leq s$).

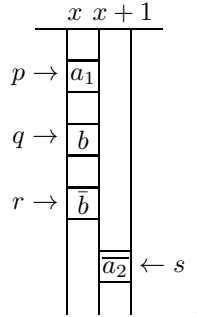
\setlength{\unitlength}{12pt}
\begin{center}
\begin{picture}(6,10)
\put(2,0){\line(0,1){9}}
\put(3,0){\line(0,1){9}}
\put(4,0){\line(0,1){9}}
\put(3,1){\line(1,0){1}}
\put(3,2){\line(1,0){1}}
\put(2,3){\line(1,0){1}}
\put(2,4){\line(1,0){1}}
\put(2,5){\line(1,0){1}}
\put(2,6){\line(1,0){1}}
\put(2,7){\line(1,0){1}}
\put(2,8){\line(1,0){1}}
\put(1,9){\line(1,0){4}}
\put(0,3){\makebox(2,1){$r\rightarrow$}}
\put(0,5){\makebox(2,1){$q\rightarrow$}}
\put(0,7){\makebox(2,1){$p\rightarrow$}}
\put(4,1){\makebox(2,1){$\leftarrow s$}}
\put(1.75,9){\makebox(1,1){${\mathstrut x}$}}
\put(3,9){\makebox(1.5,1){${\mathstrut x+1}$}}
\put(2,3){\makebox(1,1){$\Bar{b}$}}
\put(2,5){\makebox(1,1){$b$}}
\put(2,7){\makebox(1,1){$a_{1}$}}
\put(3,1){\makebox(1,1){$\overline{a_{2}}$}}
\end{picture}.
\end{center}
Then we have $(q-p)+(s-r)<b-\min(a_{1},a_{2})$ because the two-column tableau $C_{x}C_{x+1}$ is KN-admissible.
Let $\mathscr{J}^{(x)}$($\mathscr{I}^{(x)}$) be the set of $\mathscr{J}$($\mathscr{I}$)-letters 
in the $x$-th column and set $\mathscr{L}^{(x,x)}:=\mathscr{J}^{(x)}\cap \mathscr{I}^{(x)}$.
If $\sharp\left\{  l\in \mathscr{L}^{(x,x)} \relmiddle| l^{\ast}<b<l\right\}  =\delta$, then we have 
$(q-p)+(s-r)<b-\min(a_{1},a_{2})-\delta$ in the above configuration.
\end{lem}

\begin{lem} \label{lem:Phi}
Suppose that $T=C_{1}C_{2}\cdots C_{n_{c}} \in C_{n}\text{-}\mathrm{SST}_{\mathrm{KN}}(\lambda)$.
Then $\Phi$ is well-defined on $T$ and $\Phi(T) \in C_{n}\text{-}\mathrm{SST}(\lambda)$. 
\end{lem}

\begin{proof}
\textbf{(I).}
We first prove that $\Phi$ is well-defined on $T$ and that the $\mathscr{C}_{n}^{(+)}$-letters part of $\Phi(T)$ is semistandard.
The map $\overline{\Phi^{(n_{c})}}=\Phi^{(n_{c})}=\phi^{(n_{c},n_{c})}$ is well-defined on $T$ 
because the $n_{c}$-th column of $T$ is KN-admissible and the $\mathscr{C}_{n}^{(+)}$-letters part of $\overline{\Phi^{(n_{c})}}(T)$ 
is semistandard by Lemma~\ref{lem:st3}.

\textbf{(II).}
Suppose that $\overline{\Phi^{(x+1)}}$ is well-defined on $T$, i.e., $\overline{\Phi^{(x+1)}}(T)\neq \emptyset$ 
and the $\mathscr{C}_{n}^{(+)}$-letters part of $\overline{\Phi^{(x+1)}}(T)$ is semistandard 
($x=n_{c}-1,\ldots,1$).
This assumption is satisfied for $x=n_{c}-1$.
\textbf{(i).}
The map $\phi^{(x,x)}$ is well-defined on $\overline{\Phi^{(x+1)}}(T)$ because the $x$-th column of $\overline{\Phi^{(x+1)}}(T)$, i.e., the $x$-th column of $T$ is KN-admissible and 
the $\mathscr{C}_{n}^{(+)}$-letters part of $\phi^{(x,x)}\circ\overline{\Phi^{(x+1)}}(T)$ is semistandard 
by Lemma~\ref{lem:st3}.
\textbf{(ii).}
Let us set $\Tilde{T}=\phi^{(x,y-1)}\circ\cdots\circ\phi^{(x,x)}\circ\overline{\Phi^{(x+1)}}(T)$ ($x+1\leq y\leq n_{c}$).
Suppose that $\Tilde{T}\neq \emptyset$ and that in the updating process of the tableau from $T$ to $\Tilde{T}$ the semistandardness of the $\mathscr{C}_{n}^{(+)}$-letters part of the tableau is preserved.
This assumption is satisfied for $y-1=x$.
Then $\phi^{(x,y)}$ is well-defined on $\Tilde{T}$ by Lemma~\ref{lem:welldf1} and the $\mathscr{C}_{n}^{(+)}$-letters part of 
$\phi^{(x,y)}(\Tilde{T})$ is semistandard by Lemma~\ref{lem:st1}.
From \textbf{(i)} and \textbf{(ii)} and by induction, we have that $\overline{\Phi^{(x)}}=\Phi^{(x)}\circ\overline{\Phi^{(x+1)}}$ 
is well-defined on $T$ and the $\mathscr{C}_{n}^{(+)}$-letters part of $\overline{\Phi^{(x)}}(T)$ is semistandard.
From \textbf{(I)} and \textbf{(II)} and by induction, we conclude that $\Phi$ is well-defined on $T$ and that 
the $\mathscr{C}_{n}^{(+)}$-letters part of $\Phi(T)$ is semistandard.

Now let us show that the $\mathscr{C}_{n}^{(-)}$-letters part of $\Phi(T)$ is semistandard.
Note that all the maps $\phi^{(i,j)}$ ($1\leq i\leq j\leq n_{c}$) are well-defined by the above argument.

\textbf{(I').}
The $\mathscr{C}_{n}^{(-)}$-letters part of $\overline{\Phi^{(n_{c})}}(T)$ is semistandard by Lemma~\ref{lem:col_sst1}.

\textbf{(II').}
Suppose that the $\mathscr{C}_{n}^{(-)}$-letters part of $\overline{\Phi^{(x+1)}}(T)$ and that of $(\Phi^{(x+1)})^{-1}\overline{\Phi^{(x+1)}}(T)$ is semistandard ($x=n_{c}-1,\ldots,1$).
This assumption is satisfied for $x=n_{c}-1$.
\textbf{(i').}
The $\mathscr{C}_{n}^{(-)}$-letters part of $\phi^{(x,x)}\circ(\Phi^{(x+1)})^{-1}\circ\overline{\Phi^{(x+1)}}(T)$ and 
that of $(\phi^{(x+1,x+1)}\circ\phi^{(x,x)})\circ(\Phi^{(x+1)})^{-1}\circ\overline{\Phi^{(x+1)}}(T)$ are semistandard 
by Lemma~\ref{lem:st4}.
\textbf{(ii').}
Let us set 
\[
\Tilde{T}=(\phi^{(x+1,y)}\circ\phi^{(x,y-1)})\circ\cdots\circ
(\phi^{(x+1,x+1)}\circ\phi^{(x,x)})\circ(\Phi^{(x+1)})^{-1}\circ\overline
{\Phi^{(x+1)}}(T)
\]
($x+1\leq y\leq n_{c}$).
Suppose that $\Tilde{T}\neq \emptyset$ and that in the updating process of the tableau from $T$ to $\Tilde{T}$ the semistandardness of the $\mathscr{C}_{n}^{(-)}$-letters part of the tableau is preserved.
Then the $\mathscr{C}_{n}^{(-)}$-letters part of $\phi^{(x,y)}(\Tilde{T})$ and that of 
$(\phi^{(x+1,y+1)}\circ\phi^{(x,y)})(\Tilde{T})$ ($y\leq n_{c}-1$) are seminstandard by Lemma~\ref{lem:st2}.
From \textbf{(i')} and \textbf{(ii')} and by induction, we have that the $\mathscr{C}_{n}^{(-)}$-letters part of 
$\overline{\Phi^{(x)}}(T)=\Phi^{(x)}\circ\overline{\Phi^{(x+1)}}(T)$ is semistandard.
From \textbf{(I')} and \textbf{(II')} and by induction, we conclude that the $\mathscr{C}_{n}^{(-)}$-letters part of $\Phi(T)$ is semistandard.

Since $\Phi$ is well-defined on $T$ so that it preserves the shape of $T$, we have that 
$\Phi(T) \in C_{n}\text{-}\mathrm{SST}(\lambda)$ for 
all $T \in C_{n}\text{-}\mathrm{SST}_{\mathrm{KN}}(\lambda)$.
\end{proof}

\section{Proof of Proposition~\ref{prp:main11}} \label{sec:main11}

In this section, we provide the proof of Proposition~\ref{prp:main11}.
\begin{prp} \label{prp:skew}
Let $W$ be a skew semistandard tableau with entries $\{\Bar{1},\Bar{2},\ldots,\Bar{n}\}$.
We apply the jeu de taquin to $W$ (starting from any inside corner of $W$) to obtain 
a rectification of $W$ denoted by $\mathrm{Rect}(W)$.
The process consists of several steps of Schutzenberger's sliding.
Let us write the whole process as
\[
W=S_{0}\rightarrow S_{1}\rightarrow\cdots\rightarrow S_{m}=\mathrm{Rect}(W).
\]
If $\mathrm{FE}(S_{k})$ is smooth on a Young diagram $\mu$, 
then  $\mathrm{FE}(S_{k+1})$ is also smooth on $\mu$ ($k=0,1,\ldots,m-1$),
where $\mathrm{FE}(S)$ is the far-eastern reading of $S$ neglecting the empty box of $S$.
Therefore, if $\mathrm{FE}(W)$ is smooth on $\mu$, then $\mathrm{FE}(\mathrm{Rect}(W))$ is also smooth on $\mu$.
Conversely, suppose that the far-eastern reading of a semistandard Young tableau $T$ 
filled with $\mathscr{C}_{n}^{(-)}$-letters is smooth on a Young diagram $\mu$, then 
the far-eastern readings of any skew semistandard tableaux whose rectification is $T$ are also smooth on $\mu$.
\end{prp}

\begin{proof}
Let us show that the smoothness is preserved by the jeu de taquin.
Suppose that $S_{k}$ consists of $n_{c}$ columns and 
let the set of letters ($\mathscr{C}_{n}^{(-)}$-letters) in the $x$-th column be $\overline{\mathscr{J}^{(x)}}$ 
($1\leq x \leq n_{c}$).  
It suffices to consider the following case.

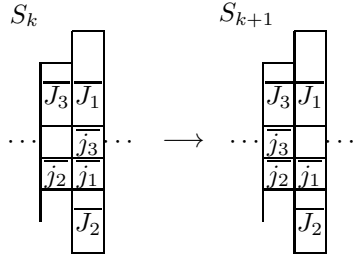
\setlength{\unitlength}{12pt}
\begin{center}
\begin{picture}(11,8)
\put(1,1){\line(0,1){5}}
\put(2,0){\line(0,1){7}}
\put(3,0){\line(0,1){7}}
\put(2,0){\line(1,0){1}}
\put(1,2){\line(1,0){2}}
\put(1,3){\line(1,0){2}}
\put(1,4){\line(1,0){2}}
\put(1,6){\line(1,0){1}}
\put(2,7){\line(1,0){1}}
\put(1,2){\makebox(1,1){$\overline{j_{2}}$}}
\put(2,2){\makebox(1,1){$\overline{j_{1}}$}}
\put(2,3){\makebox(1,1){$\overline{j_{3}}$}}
\put(1,4){\makebox(1,2){$\overline{J_{3}}$}}
\put(2,0){\makebox(1,2){$\overline{J_{2}}$}}
\put(2,4){\makebox(1,2){$\overline{J_{1}}$}}
\put(0,3){\makebox(1,1){$\cdots$}}
\put(3,3){\makebox(1,1){$\cdots$}}
\put(0,7){\makebox(1,1){$S_{k}$}}
\put(5,3){\makebox(1,1){$\longrightarrow$}}
\put(8,1){\line(0,1){5}}
\put(9,0){\line(0,1){7}}
\put(10,0){\line(0,1){7}}
\put(9,0){\line(1,0){1}}
\put(8,2){\line(1,0){2}}
\put(8,3){\line(1,0){2}}
\put(8,4){\line(1,0){2}}
\put(8,6){\line(1,0){1}}
\put(9,7){\line(1,0){1}}
\put(8,2){\makebox(1,1){$\overline{j_{2}}$}}
\put(8,3){\makebox(1,1){$\overline{j_{3}}$}}
\put(9,2){\makebox(1,1){$\overline{j_{1}}$}}
\put(8,4){\makebox(1,2){$\overline{J_{3}}$}}
\put(9,0){\makebox(1,2){$\overline{J_{2}}$}}
\put(9,4){\makebox(1,2){$\overline{J_{1}}$}}
\put(7,3){\makebox(1,1){$\cdots$}}
\put(10,3){\makebox(1,1){$\cdots$}}
\put(7,7){\makebox(1,1){$S_{k+1}$}}
\end{picture},
\end{center}
where $\overline{J_{1}}$ and $\overline{J_{2}}$ are blocks of the $(l+1)$-st column of $S_{k}$ and $S_{k+1}$ and 
$\overline{J_{3}}$ is a block of the $l$-th column of $S_{k}$ and $S_{k+1}$.
In this case, we slide the box containing $\overline{j_{3}}$ in the $(l+1)$-st column into the $l$-th column horizontally. 
Note that 
\[
\max (J_{2})\leq j_{1}-1
\quad \mbox{and} \quad
\min (J_{3})\geq j_{3}+1.
\]
By the rule of Schutzenberger's sliding, we have $\overline{j_{3}}\prec\overline{j_{2}}$ so that
$j_{1}\leq j_{2}<j_{3}$.
Let us set
$\mu^{\prime}:=\mu\lbrack {\mathstrut \overline{\mathscr{J}^{(n_{c})}}},
\ldots,{\mathstrut \overline{\mathscr{J}^{(l+2)}}},{\mathstrut \overline{J_{1}}}]$.
This is a Young diagram by the assumption of Proposition~\ref{prp:skew}.
Let us assume that $j_{1}<j_{2}$.
The proof for the case when $j_{1}=j_{2}$ is similar.
Since $\overline{j_{3}},\overline{j_{1}}$ is smooth on $\mu^{\prime}$,
\[
\mu^{\prime}[\overline{j_{3}}]=(\ldots,\mu_{j_{1}}^{\prime},\mu_{j_{1}%
+1}^{\prime},\ldots,\mu_{j_{2}}^{\prime},\ldots,\mu_{j_{3}}^{\prime}-1,\ldots)
\]
and
\[
\mu^{\prime}[\overline{j_{3}},\overline{j_{1}}]=
(\ldots,\mu_{j_{1}}^{\prime}-1,\mu_{j_{1}+1}^{\prime},\ldots,\mu_{j_{2}}^{\prime},\ldots,\mu_{j_{3}}^{\prime}-1,\ldots)
\]
are Young diagrams so that $\mu_{j_{1}}^{\prime}-1\geq\mu_{j_{1}+1}^{\prime}$.
Since $\overline{j_{3}},\overline{j_{1}},\overline{J_{2}},\overline{J_{3}}$
is smooth on $\mu^{\prime}$,
$\overline{J_{2}}$ is smooth on 
$(\mu_{1}^{\prime},\ldots,\mu_{j_{1}-1}^{\prime},\mu_{j_{1}}^{\prime}-1)$
and $\overline{J_{3}}$ is smooth on
$(\mu_{j_{3}}^{\prime}-1,\ldots)$ and therefore on
$(\mu_{j_{3}}^{\prime},\ldots)$.
Now
\[
\mu^{\prime}[\overline{j_{1}}]=(\ldots,\mu_{j_{1}}^{\prime}-1,\mu_{j_{1}+1}^{\prime},
\ldots,\mu_{j_{2}}^{\prime},\ldots,\mu_{j_{3}}^{\prime},\ldots)
\]
is a Young diagram because
$\mu_{j_{1}}^{\prime}-1\geq\mu_{j_{1}+1}^{\prime}$.
Since $\overline{J_{2}}$ is smooth on 
$(\mu_{1}^{\prime},\ldots,\mu_{j_{1}-1}^{\prime},\mu_{j_{1}}^{\prime}-1)$, 
$\overline{j_{1}},\overline{J_{2}}$ is smooth on $\mu^{\prime}$.
Since $\overline{J_{3}}$ is smooth on $(\mu_{j_{3}}^{\prime},\ldots)$, 
$\overline{j_{1}},\overline{J_{2}},\overline{J_{3}}$
is smooth on $\mu^{\prime}$.
That is, 
$\mu^{\prime}\left[  \overline{j_{1}}\right]  $, 
$\mu^{\prime}\left[  \overline{j_{1}},\overline{J_{2}}\right]  $, and 
$\mu^{\prime}\left[  \overline{j_{1}},\overline{J_{2}},\overline{J_{3}}\right]  $ are all Young diagrams.
\begin{align*}
\mu^{\prime}[\overline{j_{1}},\overline{J_{2}},\overline{J_{3}},\overline{j_{3}}]
& =\mu^{\prime}[\overline{j_{3}},\overline{j_{1}},\overline{J_{2}},\overline{J_{3}}]\\
& =(\ldots,\mu_{j_{1}-1}^{\prime},\mu_{j_{1}}^{\prime}-1,\mu_{j_{1}+1}^{\prime},
\ldots,\mu_{j_{2}}^{\prime},\ldots,\mu_{j_{3}}^{\prime}-1,\ldots)
\end{align*}
and
\begin{align*}
\mu^{\prime}[\overline{j_{1}},\overline{J_{2}},\overline{J_{3}},
\overline{j_{3}},\overline{j_{2}}]  & =\mu^{\prime}[\overline{j_{3}},
\overline{j_{1}},\overline{J_{2}},\overline{J_{3}},\overline{j_{2}}]\\ 
& =(\ldots,\mu_{j_{1}-1}^{\prime},\mu_{j_{1}}^{\prime}-1,\mu_{j_{1}+1}^{\prime},
\ldots,\mu_{j_{2}}^{\prime}-1,\ldots,\mu_{j_{3}}^{\prime}-1,\ldots)
\end{align*}
are Young diagrams because 
$\overline{j_{3}},\overline{j_{1}},\overline{J_{2}},\overline{J_{3}},\overline{j_{2}}$ 
is smooth on $\mu^{\prime}$.
Hence,
$\overline{j_{1}},\overline{J_{2}},\overline{J_{3}},\overline{j_{3}},\overline{j_{2}}$
is smooth on $\mu^{\prime}$.

The ``converse'' part follows from the fact that Schutzenberger's sliding is reversible.
\end{proof}

\begin{ex}
Let $\mu=(3,2,2)$.
The far-eastern reading of the skew semistandard tableau
\setlength{\unitlength}{12pt}
\begin{center}
\begin{picture}(6.5,3)
\put(2.5,0){\line(0,1){1}}
\put(3.5,0){\line(0,1){2}}
\put(4.5,0){\line(0,1){3}}
\put(5.5,1){\line(0,1){2}}
\put(2.5,0){\line(1,0){2}}
\put(2.5,1){\line(1,0){3}}
\put(3.5,2){\line(1,0){2}}
\put(4.5,3){\line(1,0){1}}
\put(2.5,0){\makebox(1,1){$\Bar{3}$}}
\put(3.5,0){\makebox(1,1){$\Bar{1}$}}
\put(3.5,1){\makebox(1,1){$\Bar{2}$}}
\put(4.5,1){\makebox(1,1){$\Bar{1}$}}
\put(4.5,2){\makebox(1,1){$\Bar{3}$}}
\put(0,1){\makebox(2,1){$W=$}}
\end{picture}
\end{center}
is smooth on $\mu$ as we can see the process 
$\mu \rightarrow \mu\lbrack \mathrm{FE}(W)]=\mu\lbrack \Bar{3},\Bar{1},\Bar{2},\Bar{1},\Bar{3}]$ is 
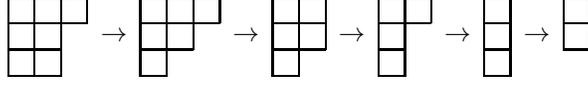
\setlength{\unitlength}{10pt}
\begin{center}
\begin{picture}(22,3)
\put(0,0){\line(0,1){3}}
\put(1,0){\line(0,1){3}}
\put(2,0){\line(0,1){3}}
\put(3,2){\line(0,1){1}}
\put(0,0){\line(1,0){2}}
\put(0,1){\line(1,0){2}}
\put(0,2){\line(1,0){3}}
\put(0,3){\line(1,0){3}}
\put(3,1){\makebox(2,1){$\rightarrow$}}
\put(5,0){\line(0,1){3}}
\put(6,0){\line(0,1){3}}
\put(7,1){\line(0,1){2}}
\put(8,2){\line(0,1){1}}
\put(5,0){\line(1,0){1}}
\put(5,1){\line(1,0){2}}
\put(5,2){\line(1,0){3}}
\put(5,3){\line(1,0){3}}
\put(8,1){\makebox(2,1){$\rightarrow$}}
\put(10,0){\line(0,1){3}}
\put(11,0){\line(0,1){3}}
\put(12,1){\line(0,1){2}}
\put(10,0){\line(1,0){1}}
\put(10,1){\line(1,0){2}}
\put(10,2){\line(1,0){2}}
\put(10,3){\line(1,0){2}}
\put(12,1){\makebox(2,1){$\rightarrow$}}
\put(14,0){\line(0,1){3}}
\put(15,0){\line(0,1){3}}
\put(16,2){\line(0,1){1}}
\put(14,0){\line(1,0){1}}
\put(14,1){\line(1,0){1}}
\put(14,2){\line(1,0){2}}
\put(14,3){\line(1,0){2}}
\put(16,1){\makebox(2,1){$\rightarrow$}}
\put(18,0){\line(0,1){3}}
\put(19,0){\line(0,1){3}}
\put(18,0){\line(1,0){1}}
\put(18,1){\line(1,0){1}}
\put(18,2){\line(1,0){1}}
\put(18,3){\line(1,0){1}}
\put(19,1){\makebox(2,1){$\rightarrow$}}
\put(21,1){\line(0,1){2}}
\put(22,1){\line(0,1){2}}
\put(21,1){\line(1,0){1}}
\put(21,2){\line(1,0){1}}
\put(21,3){\line(1,0){1}}
\end{picture}.
\end{center}
The rectification of $W$ is
\setlength{\unitlength}{12pt}
\begin{center}
\begin{picture}(7.5,3)
\put(4.5,0){\line(0,1){3}}
\put(5.5,0){\line(0,1){3}}
\put(6.5,2){\line(0,1){1}}
\put(7.5,2){\line(0,1){1}}
\put(4.5,0){\line(1,0){1}}
\put(4.5,1){\line(1,0){1}}
\put(4.5,2){\line(1,0){3}}
\put(4.5,3){\line(1,0){3}}
\put(4.5,0){\makebox(1,1){$\Bar{1}$}}
\put(4.5,1){\makebox(1,1){$\Bar{2}$}}
\put(4.5,2){\makebox(1,1){$\Bar{3}$}}
\put(5.5,2){\makebox(1,1){$\Bar{3}$}}
\put(6.5,2){\makebox(1,1){$\Bar{1}$}}
\put(0,1){\makebox(4,1){$\mathrm{Rect}(W)=$}}
\end{picture}
\end{center}
and the far-eastern reading is also smooth on $\mu$ as we can see the process 
$\mu \rightarrow \mu\lbrack \mathrm{FE}(\mathrm{Rect}(W))]=\mu\lbrack \Bar{1},\Bar{3},\Bar{3},\Bar{2},\Bar{1}]$ is 
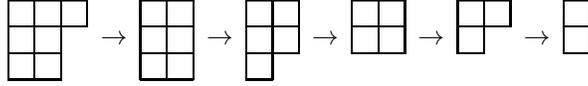
\setlength{\unitlength}{10pt}
\begin{center}
\begin{picture}(22,3)
\put(0,0){\line(0,1){3}}
\put(1,0){\line(0,1){3}}
\put(2,0){\line(0,1){3}}
\put(3,2){\line(0,1){1}}
\put(0,0){\line(1,0){2}}
\put(0,1){\line(1,0){2}}
\put(0,2){\line(1,0){3}}
\put(0,3){\line(1,0){3}}
\put(3,1){\makebox(2,1){$\rightarrow$}}
\put(5,0){\line(0,1){3}}
\put(6,0){\line(0,1){3}}
\put(7,0){\line(0,1){3}}
\put(5,0){\line(1,0){2}}
\put(5,1){\line(1,0){2}}
\put(5,2){\line(1,0){2}}
\put(5,3){\line(1,0){2}}
\put(7,1){\makebox(2,1){$\rightarrow$}}
\put(9,0){\line(0,1){3}}
\put(10,0){\line(0,1){3}}
\put(11,1){\line(0,1){2}}
\put(9,0){\line(1,0){1}}
\put(9,1){\line(1,0){2}}
\put(9,2){\line(1,0){2}}
\put(9,3){\line(1,0){2}}
\put(11,1){\makebox(2,1){$\rightarrow$}}
\put(13,1){\line(0,1){2}}
\put(14,1){\line(0,1){2}}
\put(15,1){\line(0,1){2}}
\put(13,1){\line(1,0){2}}
\put(13,2){\line(1,0){2}}
\put(13,3){\line(1,0){2}}
\put(15,1){\makebox(2,1){$\rightarrow$}}
\put(17,1){\line(0,1){2}}
\put(18,1){\line(0,1){2}}
\put(19,2){\line(0,1){1}}
\put(17,1){\line(1,0){1}}
\put(17,2){\line(1,0){2}}
\put(17,3){\line(1,0){2}}
\put(19,1){\makebox(2,1){$\rightarrow$}}
\put(21,1){\line(0,1){2}}
\put(22,1){\line(0,1){2}}
\put(21,1){\line(1,0){1}}
\put(21,2){\line(1,0){1}}
\put(21,3){\line(1,0){1}}
\end{picture}.
\end{center}
The rectification of 
\setlength{\unitlength}{12pt}
\begin{center}
\begin{picture}(7.5,3)
\put(2.5,0){\line(0,1){1}}
\put(3.5,0){\line(0,1){1}}
\put(4.5,0){\line(0,1){3}}
\put(5.5,1){\line(0,1){2}}
\put(6.5,2){\line(0,1){1}}
\put(2.5,0){\line(1,0){2}}
\put(2.5,1){\line(1,0){3}}
\put(4.5,2){\line(1,0){2}}
\put(4.5,3){\line(1,0){2}}
\put(2.5,0){\makebox(1,1){$\Bar{3}$}}
\put(3.5,0){\makebox(1,1){$\Bar{1}$}}
\put(4.5,1){\makebox(1,1){$\Bar{2}$}}
\put(4.5,2){\makebox(1,1){$\Bar{3}$}}
\put(5.5,2){\makebox(1,1){$\Bar{1}$}}
\put(0,1){\makebox(2,1){$W^{\prime}=$}}
\end{picture}
\end{center}
is the same as $\mathrm{Rect}(W)$ and $\mathrm{FE}(W^{\prime})$ is smooth on $\mu$ as we can see the process 
$\mu \rightarrow \mu\lbrack \mathrm{FE}(W^{\prime})]=\mu\lbrack \Bar{1},\Bar{3},\Bar{2},\Bar{1},\Bar{3}]$ is 
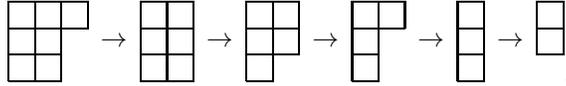
\setlength{\unitlength}{10pt}
\begin{center}
\begin{picture}(21,3)
\put(0,0){\line(0,1){3}}
\put(1,0){\line(0,1){3}}
\put(2,0){\line(0,1){3}}
\put(3,2){\line(0,1){1}}
\put(0,0){\line(1,0){2}}
\put(0,1){\line(1,0){2}}
\put(0,2){\line(1,0){3}}
\put(0,3){\line(1,0){3}}
\put(3,1){\makebox(2,1){$\rightarrow$}}
\put(5,0){\line(0,1){3}}
\put(6,0){\line(0,1){3}}
\put(7,0){\line(0,1){3}}
\put(5,0){\line(1,0){2}}
\put(5,1){\line(1,0){2}}
\put(5,2){\line(1,0){2}}
\put(5,3){\line(1,0){2}}
\put(7,1){\makebox(2,1){$\rightarrow$}}
\put(9,0){\line(0,1){3}}
\put(10,0){\line(0,1){3}}
\put(11,1){\line(0,1){2}}
\put(9,0){\line(1,0){1}}
\put(9,1){\line(1,0){2}}
\put(9,2){\line(1,0){2}}
\put(9,3){\line(1,0){2}}
\put(11,1){\makebox(2,1){$\rightarrow$}}
\put(13,0){\line(0,1){3}}
\put(14,0){\line(0,1){3}}
\put(15,2){\line(0,1){1}}
\put(13,0){\line(1,0){1}}
\put(13,1){\line(1,0){1}}
\put(13,2){\line(1,0){2}}
\put(13,3){\line(1,0){2}}
\put(15,1){\makebox(2,1){$\rightarrow$}}
\put(17,0){\line(0,1){3}}
\put(18,0){\line(0,1){3}}
\put(17,0){\line(1,0){1}}
\put(17,1){\line(1,0){1}}
\put(17,2){\line(1,0){1}}
\put(17,3){\line(1,0){1}}
\put(18,1){\makebox(2,1){$\rightarrow$}}
\put(20,1){\line(0,1){2}}
\put(21,1){\line(0,1){2}}
\put(20,1){\line(1,0){1}}
\put(20,2){\line(1,0){1}}
\put(20,3){\line(1,0){1}}
\end{picture}.
\end{center}
\end{ex}

Suppose that $T \in C_{n}\text{-}\mathrm{SST}_{\mathrm{KN}}$ 
and $T$ consists of $n_{c}$ columns.
To compute $\Phi(T)$, we apply the map of the form $\phi ^{(\centerdot,\centerdot)}$ successively 
to the updated tableau whose entries are updated by preceding application of the map 
of the form $\phi ^{(\centerdot,\centerdot)}$.
To keep track of the updating stage in $\Phi(T)$, let us introduce new notation.
Initially, the set of $\mathscr{I}$ (resp. $\mathscr{J}$)-letters in the $x$-th column of $T$ is written as 
$\mathscr{I}^{(x,i)}$ (resp. $\mathscr{J}^{(x,i)}$) with $i=0$ ($1\leq x \leq n_{c}$).
Whenever the map $\phi^{(x,y)}$ is applied to the updated tableau 
whose entries are updated by preceding application of the map of the form 
$\phi ^{(\centerdot,\centerdot)}$, 
the counter $i$ in $\mathscr{I}^{(y,i)}$ 
is increased by one; $\mathscr{I}^{(y,i)}\rightarrow \mathscr{I}^{(y,i+1)}$ and 
the counter $j$ in $\mathscr{J}^{(x,j)}$ is increased by one; 
$\mathscr{J}^{(x,j)}\rightarrow \mathscr{J}^{(x,j+1)}$.
At the end, i.e., in $\Phi(T)$, 
the set of $\mathscr{I}$ (resp. $\mathscr{J}$)-letters in the $x$-th column is 
$\mathscr{I}^{(x,x)}$ (resp. $\mathscr{J}^{(x,n_{c}-x+1)}$) ($1\leq x\leq n_{c}$).
The letters in $\mathscr{I}^{(x,i)}$ (resp. $\mathscr{J}^{(x,i)}$) are called 
$\mathscr{I}^{(x,i)}$ (resp. $\mathscr{J}^{(x,i)}$)-letters and those in 
$\overline{\mathscr{I}^{(x,i)}}$ (resp. $\overline{\mathscr{J}^{(x,i)}}$) are called 
$\overline{\mathscr{I}^{(x,i)}}$ (resp. $\overline{\mathscr{J}^{(x,i)}}$)-letters.

When a sequence of $\mathscr{C}_{n}^{(+)}$-letters $I$ is smooth on a Young diagram $\lambda$, 
we write $\lambda\left[  \underrightarrow{I}\right]$.
Likewise, when a sequence of $\mathscr{C}_{n}^{(-)}$-letters $\Bar{J}$ is smooth on a Young diagram $\lambda$, 
we write $\lambda\left[  \underrightarrow{\Bar{J}}\right]$.
For example, $\lambda\left[ \underrightarrow{I},\underrightarrow{\Bar{J}}\right]$ implies that 
the sequence of $\mathscr{C}_{n}^{(+)}$-letters $I$ is smooth on $\lambda$ and 
the sequence of  $\mathscr{C}_{n}^{(-)}$-letters $\Bar{J}$ is smooth on the Young diagram 
$\lambda \left[ I \right]$.
We also write $\lambda\left[ \underrightarrow{\mathrm{FE}(T)} \right]$ if $\mathrm{FE}(T)$ is smooth 
on $\lambda$, where $T$ is a semistandard Young or skew tableau.
In this case, we write $\mu\left[ \underrightarrow{\overline{\mathrm{FE}(T)}}\right]=\lambda$, 
where $\mu=\lambda\left[ \underrightarrow{\mathrm{FE}(T)} \right]$ and 
$\overline{\mathrm{FE}(T)}$ is given by changing the unbarred (barred) letters to the corresponding barred (unbarred) letters in $\mathrm{FE}(T)$ and reversing the order of the sequence.

\begin{lem} \label{lem:smooth0}
Let $\lambda$ and $\mu$ be Young diagrams.
If $\lambda\left[  I\right]  =\mu$, where $I$ is the sequence of $\mathscr{C}_{n}^{(+)}$-letters 
$i_{1},i_{2},\ldots,i_{a}$ ($i_{1}<i_{2}<\ldots<i_{a}$), then $I$ is smooth on $\lambda$; 
$\lambda\left[  \underrightarrow{I}\right]  =\mu$.
Similarly, if $\lambda\left[  \Bar{J}\right]  =\mu$, 
where $\Bar{J}$ is the sequence of $\mathscr{C}_{n}^{(-)}$-letters 
$\overline{j_{b}},\ldots,\overline{j_{2}},\overline{j_{1}}$ 
($\overline{j_{b}}\prec\ldots\prec\overline{j_{2}}\prec\overline{j_{1}}$), 
then $\Bar{J}$ is smooth on $\lambda$; $\lambda\left[  \underrightarrow{\overline{J}}\right]  =\mu$.
\end{lem}

\begin{proof}
For $p=2,\ldots,a$, 
$\lambda\left[  i_{1},\ldots,i_{p-1}\right]  =\mu\left[  \overline{i_{a}},\ldots,\overline{i_{p}}\right]  $.
Here, 
$\lambda\left[  i_{1},\ldots,i_{p-1}\right]  _{i_{p}-1}=
\mu\left[ \overline{i_{a}},\ldots,\overline{i_{p}}\right]  _{i_{p}-1}=\mu_{i_{p}-1}$ and 
$\lambda\left[  i_{1},\ldots,i_{p-1}\right]  _{i_{p}}=
\mu\left[ \overline{i_{a}},\ldots,\overline{i_{p}}\right]  _{i_{p}}=\mu_{i_{p}}-1$.
Since $\mu$ is a Young diagram, i.e., $\mu_{i_{p}-1}\geq\mu_{i_{p}}$, we have 
$\lambda\left[  i_{1},\ldots,i_{p-1}\right]  _{i_{p}-1}>
\lambda\left[ i_{1},\ldots,i_{p-1}\right]  _{i_{p}}$.
That is, $\lambda\left[  i_{1},\ldots,i_{p-1}\right]  \left[  i_{p}\right]  $ is a Young diagram.
The proof of the second part is analogous.
\end{proof}

For all $T\in \mathbf{B}_{n}^{\mathfrak{sp}_{2n}}(\nu)_{\mu}^{\lambda}$ with $\nu_{1}=n_{c}$, 
\begin{equation} \label{eq:mu2lambda}
\mu\left[ \underrightarrow{\mathrm{FE}(T)} \right]=
\mu\left[  \underrightarrow{\mathscr{I}^{(n_{c},0)}\vphantom{\overline{\mathscr{J}^{(1,0)}}}},
\underrightarrow{\overline{\mathscr{J}^{(n_{c},0)}}},\ldots,
\underrightarrow{\mathscr{I}^{(1,0)}\vphantom{\overline{\mathscr{J}^{(1,0)}}}},
\underrightarrow{\overline{\mathscr{J}^{(1,0)}}} \right]  =\lambda
\end{equation}
by definition.
Under this condition and the notation introduced above, 
we have the following two lemmas (Lemma~\ref{lem:smooth1} and Lemma~\ref{lem:smooth2}).

 \begin{lem} \label{lem:smooth1}
(1).
Let us define
\[
\lambda^{(x-1)}:=
\begin{cases}
\lambda\left[  \overline{\mathscr{I}^{(1,1)}},\ldots,\overline{\mathscr{I}^{(x-1,x-1)}},
\mathscr{J}^{(1,x-1)},\ldots,\mathscr{J}^{(x-1,1)}\right]  & (2\leq x\leq n_{c}), \\
\lambda & (x=1).
\end{cases}
\]
Then $\lambda^{(x-1)}$ is a Young diagram on which $\overline{\mathscr{I}^{(x,1)}}$ is smooth ($1 \leq x \leq n_{c}$). 

(2).
For $2\leq x\leq n_{c}$, let us assume that
\[
\lambda^{(x-1,i)}:=
\begin{cases}
\lambda\left[  \overline{\mathscr{I}^{(1,1)}},\ldots,
\overline{\mathscr{I}^{(x-1,x-1)}},\mathscr{J}^{(1,x-1)},\ldots,
\mathscr{J}^{(x-i,i)}\right]  & (1\leq i\leq x-1), \\
\lambda\left[  \overline{\mathscr{I}^{(1,1)}},\ldots,
\overline{\mathscr{I}^{(x-1,x-1)}}\right]  & (i=x)
\end{cases}
\]
are all Young diagrams.
Suppose that $\overline{\mathscr{I}^{(x,i)}}$ is smooth on $\lambda^{(x-1,i)}$.
Then we have that $\overline{\mathscr{I}^{(x,i+1)}}$ is smooth on $\lambda^{(x-1,i+1)}$
 ($1\leq i\leq x-1$).

(3).
$\lambda\left[ 
\underrightarrow{\overline{\mathscr{I}^{(1,1)}}},
\underrightarrow{\overline{\mathscr{I}^{(2,2)}}},\ldots,
\underrightarrow{\overline{\mathscr{I}^{(n_{c},n_{c})}}}\right]  $.
\end{lem}

\begin{proof}
Let us begin by giving the proof of (2).
Note that a pair of $\overline{\mathscr{I}^{(x,i+1)}}$ and $\mathscr{J}^{(x-i,i+1)}$ are generated from 
a pair of $\overline{\mathscr{I}^{(x,i)}}$ and $\mathscr{J}^{(x-i,i)}$ by applying $\phi^{(x-i,x)}$ to the updated tableau whose entries are updated by preceding application of the map of the form $\phi^{(\centerdot,\centerdot)}$.
Let us call such sets $\overline{\mathscr{I}^{(x,i)}}$ and $\mathscr{J}^{(x-i,i)}$ to be updated are paired 
and write $\left\langle \overline{\mathscr{I}^{(x,i)}},\mathscr{J}^{(x-i,i)}\right\rangle_{\mathrm{pair}}$ 
($0\leq i\leq x-1; 1\leq x \leq n_{c}$).
Let us set $\mathscr{I}^{(x,i)}=\{i_{1},i_{2},\ldots,i_{a}\}$, 
$\mathscr{J}^{(x-i,i)}=\{j_{1},j_{2},\ldots,j_{b}\}$, 
$\mathscr{I}^{(x,i+1)}=\{i_{1}^{\prime},i_{2}^{\prime},\ldots,i_{a}^{\prime}\}$,  
$\mathscr{J}^{(x-i,i+1)}=\{j_{1}^{\prime},j_{2}^{\prime},\ldots,j_{b}^{\prime}\}$, 
$\mathscr{L}:=\mathscr{I}^{(x,i)}\cap \mathscr{J}^{(x-i,i)}=\{l_{1},l_{2},\ldots,l_{c}\}$, and 
$\mathscr{L}^{\ast}:=\mathscr{I}^{(x,i+1)}\cap \mathscr{J}^{(x-i,i+1)}=\{l_{1}^{\ast},l_{2}^{\ast},\ldots,l_{c}^{\ast}\}$.
Recall that these are ordered sets and are also considered as the sequences of letters.
We write $\Tilde{\lambda}=\lambda^{(x-1,i)}\left[\overline{\mathscr{J}^{(x-1,i)}}\right]=\lambda^{(x-1,i+1)}$ for brevity.

\textbf{(I).}
Let us consider the following three cases separately:
\begin{description}
\item[(a)]
$i_{a}^{\prime}=l_{c}^{\ast}$.
\item[(b)]
$i_{a}^{\prime}\neq l_{c}^{\ast}$ and $i_{a}=l_{c}$.
\item[(c)]
$i_{a}^{\prime}\neq l_{c}^{\ast}$ and $i_{a}\neq l_{c}$.
\end{description}

\textbf{Case (a).}
In this case, $\mathscr{L} \neq \emptyset$ and $l_{c}=i_{a}$.
Indeed, if $l_{c}=i_{p}(\in \mathscr{I}^{(x,i)})$ ($p<a$), then $i_{a} \notin \mathscr{L}$ because 
$i_{a}$ is larger than $l_{c}$ that is the largest letter in $\mathscr{L}$.
This implies that $i_{a}^{\prime}=i_{a}$.
However, this also implies $l_{c}^{\ast}=i_{a}\in \mathscr{I}^{(x,i)}$ due to the assumption of \textbf{(a)}, 
which contradicts the fact that  $l_{c}^{\ast}$ is not an $\mathscr{I}^{(x,i)}$-letter.
To proceed, let us divide this case further into the following two cases:
\begin{description}
\item[(a-1)]
All $\mathscr{I}^{(x,i)}$-letters $i_{1},i_{2},\ldots,i_{a}$ are also $\mathscr{J}^{(x-i,i)}$-letters.
\item[(a-2)]
There exist non-$\mathscr{J}^{(x-i,i)}$-letters in the sequence of $\mathscr{I}^{(x,i)}$-letters $i_{1},i_{2},\ldots,i_{a}$
(That is, there exist letters belonging to $\mathscr{I}^{(x,i)}\backslash \mathscr{L}$ 
in the set $\left\{i_{1},i_{2},\ldots,i_{a}\right\}$).
\end{description}
In case \textbf{(a-1)}, $a=c$.
Then $i_{a=c}^{\prime}=l_{c}^{\ast}$.
According to the algorithm in Definition~\ref{df:split} or Remark~\ref{rem:algorithm1}, 
we can write $l_{c}^{\ast}=j_{r}-1\; (\exists j_{r}\in \mathscr{J}^{(x-i,i)})$.
In case \textbf{(a-2)}, let us choose the largest letter $i_{p}\; (p<a)$ 
from the set of $\mathscr{I}^{(x,i)}$-letters $\{i_{1},i_{2},\ldots,i_{a}\}$ such that 
$i_{p}$ is not a $\mathscr{J}^{(x-i,i)}$-letter (i.e., $i_{p} \in \mathscr{I}^{(x,i)}\backslash \mathscr{L}$).
Now consider the increasing (just by one) sequence of $\mathscr{C}_{n}^{(+)}$-letters
\begin{equation} \label{eq:smooth1_seq}
i_{p}+1,i_{p}+2,\ldots,i_{a}-1.
\end{equation}
By the maximality of $i_{p}$, any letter belonging to $\mathscr{I}^{(x,i)}\backslash \mathscr{L}$ cannot appear in \eqref{eq:smooth1_seq}.
If all of the letters in \eqref{eq:smooth1_seq} are $\mathscr{J}^{(x-i,i)}$-letters, then 
$l_{c}^{\ast}<i_{p}$ so that $i_{a}^{\prime}=i_{p}$, which contradicts the assumption of \textbf{(a)}.
Consequently, there must exist some letters that are not $\mathscr{I}^{(x,i)}$-letters nor $\mathscr{J}^{(x-i,i)}$-letters 
in \eqref{eq:smooth1_seq}.
Denote by $i_{a}-q\; (\exists q\geq1)$ the largest letter among them.
Since $l_{c}=i_{a}$, we have $l_{c}^{\ast}=i_{a}-q$.
By the maximality of $i_{a}-q$, $i_{a}-q+1$ is a $\mathscr{J}^{(x-i,i)}$-letter 
(when $q=1$, $i_{a}=l_{c}$ is a $\mathscr{J}^{(x-i,i)}$-letter).
Therefore, we can write $i_{a}-q+1=j_{r}\;(\exists j_{r}\in \mathscr{J}^{(x-i,i)})$ so that we have $l_{c}^{\ast}=j_{r}-1$.
In both cases \textbf{(a-1)} and \textbf{(a-2)}, we can write 
$i_{a}^{\prime}=l_{c}^{\ast}=j_{r}-1\;(\exists j_{r}\in \mathscr{J}^{(x-i,i)})$.
Since $i_{a}^{\prime}=l_{c}^{\ast}\in \mathscr{I}^{(x,i+1)}$ is the letter generated by $\phi^{(x-i,x)}$, 
$i_{a}^{\prime}\notin \mathscr{J}^{(x-i,i)}$.
By the assumption of (2) of Lemma~\ref{lem:smooth1}, 
$\Tilde{\lambda}\left[ \mathscr{J}^{(x-i,i)}\right]=\lambda^{(x-1,i)}$ is a Young diagram so that 
\[
\Tilde{\lambda}\left[ \mathscr{J}^{(x-i,i)}\right]  _{i_{a}^{\prime}}\geq
\Tilde{\lambda}\left[  \mathscr{J}^{(x-i,i)}\right]  _{i_{a}^{\prime}+1}.
\]
The left-hand side of this inequality is $\Tilde{\lambda}_{i_{a}^{\prime}}$ because $i_{a}^{\prime}\notin \mathscr{J}^{(x-i,i)}$, 
while the right-hand side is $\Tilde{\lambda}_{i_{a}^{\prime}+1}+1$ because $i_{a}^{\prime}+1=j_{r}\in \mathscr{J}^{(x-i,i)}$ 
and thereby $\Tilde{\lambda}_{i_{a}^{\prime}}>\Tilde{\lambda}_{i_{a}^{\prime}+1}$.

\textbf{Case (b).}
Firstly, let us show that 
we can write $i_{a}^{\prime}=i_{p}\; (\exists i_{p}\in \mathscr{I}^{(x,i)}\backslash \mathscr{L})$.
Since $i_{a}^{\prime}\notin \mathscr{L}^{\ast}$ 
so that we can write $i_{a}^{\prime}=i_{p}\; (\exists i_{p}\in \mathscr{I}^{(x,i)}\backslash \mathscr{L})$ 
because $i_{a}^{\prime}\in(\mathscr{I}^{(x,i)}\backslash \mathscr{L})\sqcup \mathscr{L}^{\ast}$.
In this case, $p\leq a-1$. Otherwise, $i_{a}^{\prime}=i_{a}=l_{c}$, which is a contradiction.
To proceed, let us consider the following three cases separately:
\begin{description}
\item[(b-1)]
$p=a-1$ and $i_{a}^{\prime}=i_{p=a-1}=i_{a}-1=l_{c}-1$.
\item[(b-2)]
$p\leq a-1$ and $i_{p}<i_{a}-1$.
\item[(b-3)]
$p<a-1$ and $i_{p}=i_{a}-1$.
\end{description}
In case \textbf{(b-1)}, we can write $l_{c}=j_{r}\; (\exists j_{r}\in \mathscr{J}^{(x-i,i)})$ so that 
$i_{a}^{\prime}=j_{r}-1$.
In case \textbf{(b-2)}, there must exist a sequence of $\mathscr{J}$-letters 
$j_{q},\ldots,j_{q+m}$ such that 
$i_{p}<j_{q+k}<i_{a}\; (k=0,1,\ldots,m)$ and
\begin{align*}
j_{q}-i_{p}  &  =1,\\
j_{q+k}-j_{q+k-1}  &  =1\quad (k=1,\ldots,m),\\
i_{a}-j_{q+m}  &  =1.
\end{align*}
Otherwise, $l_{c}^{\ast}$ cannot be smaller than $i_{a}^{\prime}=i_{p}(\in \mathscr{I}^{(x,i)}\backslash \mathscr{L})$.
The existence of such a sequence implies $i_{a}^{\prime}=i_{p}=j_{q}-1$.
Case \textbf{(b-3)} must be excluded because the inequalities 
$i_{p}<\cdots<i_{a-1}<i_{a}$ are not satisfied.
In both cases \textbf{(b-1)} and \textbf{(b-2)}, we can write $i_{a}^{\prime}=j_{r}-1\; (\exists j_{r}\in \mathscr{J}^{(x-i,i)})$.
Now since $\Tilde{\lambda}\left[ \mathscr{J}^{(x-i,i)}\right]$ is a Young diagram, 
\[
\Tilde{\lambda}\left[ \mathscr{J}^{(x-i,i)}\right]  _{i_{a}^{\prime}}\geq
\Tilde{\lambda}\left[\mathscr{J}^{(x-i,i)}\right]  _{i_{a}^{\prime}+1}.
\]
The left-hand side of this inequality is 
$\Tilde{\lambda}\left[ \mathscr{J}^{(x-i,i)}\right]  _{i_{p}}=\Tilde{\lambda}_{i_{p}}=\Tilde{\lambda}_{i_{a}^{\prime}}$ 
because $i_{p}\in \mathscr{I}^{(x,i)}\backslash \mathscr{L}$, i.e., $i_{p}\notin \mathscr{J}^{(x-i,i)}$, 
while the right-hand side is 
$\Tilde{\lambda}_{i_{a}^{\prime}+1}+1$ because $i_{a}^{\prime}+1=j_{r}\in \mathscr{J}^{(x-i,i)}$ so that 
$\Tilde{\lambda}_{i_{a}^{\prime}}>\Tilde{\lambda}_{i_{a}^{\prime}+1}$.

\textbf{Case (c).}
Let us show that $i_{a}^{\prime}=i_{a}$.
If $\mathscr{L}=\emptyset$, this is obvious.
If $\mathscr{L}\neq \emptyset$, the $\mathscr{I}^{(x,i)}$-letter $i_{a}$ is larger than $l_{c}$ 
that is the largest letter in $\mathscr{L}$ so that 
the $\mathscr{I}^{(x,i)}$-letter $i_{a}$ is not a $\mathscr{J}^{(x-i,i)}$-letter, 
which implies $i_{a}^{\prime}=i_{a}$.
By the assumption of (2) of Lemma~\ref{lem:smooth1}, 
$\overline{\mathscr{I}^{(x,i)}}$ is smooth on 
$\Tilde{\lambda}\left[\mathscr{J}^{(x-i,i)}\right]$ so that 
\[
\Tilde{\lambda}\left[ \mathscr{J}^{(x-i,i)}\right]  _{i_{a}^{\prime}}>
\Tilde{\lambda}\left[ \mathscr{J}^{(x-i,i)}\right]  _{i_{a}^{\prime}+1}.
\]
The left-hand side of this inequality is $\Tilde{\lambda}_{i_{a}^{\prime}}$ because 
$i_{a}^{\prime}=i_{a}\notin \mathscr{J}^{(x-i,i)}$, while the right-hand side is $\Tilde{\lambda}_{i_{a}^{\prime}+1}+\delta$ 
($\delta\in\{0,1\}$).
Therefore, we have $\Tilde{\lambda}_{i_{a}^{\prime}}>\Tilde{\lambda}_{i_{a}^{\prime}+1}$.
In \textbf{(I)}, we have verified that 
$\Tilde{\lambda}_{i_{a}^{\prime}}>\Tilde{\lambda}_{i_{a}^{\prime}+1}$, that is, 
$\Tilde{\lambda}\lbrack\overline{i_{a}^{\prime}}]$ is a Young diagram for all possible cases.

\textbf{(II).}
Let us suppose that 
$\Tilde{\lambda}^{\ast(k+1)}:=
\Tilde{\lambda}\lbrack\overline{i_{a}^{\prime}},\ldots,\overline{i_{k+1}^{\prime}}]$
is a Young diagram ($k+1\leq a$).
In what follows, we prove 
$\Tilde{\lambda}^{\ast(k+1)}[\overline{i_{k}^{\prime}}]=\Tilde{\lambda}^{\ast(k)}$
is also a Young diagram, i.e., 
$\Tilde{\lambda}_{i_{k}^{\prime}}^{\ast(k+1)}>\Tilde{\lambda}_{i_{k}^{\prime}+1}^{\ast(k+1)}$.
Note that $\mathscr{J}^{(x-i,i)}$ is smooth on $\Tilde{\lambda}$ 
by the assumption of (2) of Lemma~\ref{lem:smooth1} and by Lemma~\ref{lem:smooth0}.
Let us consider the following three cases separately:
\begin{description}
\item[(a)]
$i_{k}^{\prime}\in \mathscr{I}^{(x,i+1)}\backslash \mathscr{L}^{\ast}(=\mathscr{I}^{(x,i)}\backslash \mathscr{L})$.
\item[(b)]
$i_{k+1}^{\prime}\in \mathscr{I}^{(x,i+1)}\backslash \mathscr{L}^{\ast}$ 
and $i_{k}^{\prime}\in \mathscr{L}^{\ast}$.
\item[(c)]
$i_{k+1}^{\prime}\in \mathscr{L}^{\ast}$ 
and $i_{k}^{\prime}\in \mathscr{L}^{\ast}$.
\end{description}

\textbf{Case (a).}
We can write 
$i_{k}^{\prime}=i_{p}\; (\exists i_{p}\in \mathscr{I}^{(x,i)}\backslash \mathscr{L})$ and 
\[
\Tilde{\lambda}_{i_{k}^{\prime}}^{\ast(k+1)}=
\Tilde{\lambda}\lbrack\overline{i_{a}^{\prime}},\ldots,\overline{i_{k+1}^{\prime}}]_{i_{p}}=
\Tilde{\lambda}_{i_{p}},
\]
where we have used the fact that 
$i_{p}\notin\{i_{a}^{\prime},\ldots,i_{k+1}^{\prime}\}$ ($\because i_{k}^{\prime}=i_{p}$).
In order to compute $\Tilde{\lambda}_{i_{k}^{\prime}+1}^{\ast(k+1)}=\Tilde{\lambda}_{i_{p}+1}^{\ast(k+1)}$, 
we divide this case further into the following three cases:
\begin{description}
\item[(a-1)]
$i_{p}+1\in \mathscr{I}^{(x,i+1)}$.
\item[(a-2)]
$i_{p}+1\notin \mathscr{I}^{(x,i+1)}$ and $i_{p}+1\in \mathscr{L}$.
\item[(a-3)]
$i_{p}+1\notin \mathscr{I}^{(x,i+1)}$ and $i_{p}+1\notin \mathscr{L}$.
\end{description}
In case \textbf{(a-1)}, by noting $i_{p},i_{p}+1\in \mathscr{I}^{(x,i+1)}$, we have 
$i_{k+1}^{\prime}=i_{k}^{\prime}+1=i_{p}+1$.
Then 
\[
\Tilde{\lambda}_{i_{k}^{\prime}+1}^{\ast(k+1)}=
\Tilde{\lambda}\lbrack\overline{i_{a}^{\prime}},\ldots,\overline{i_{k+1}^{\prime}}=
\overline{i_{p}+1}]_{i_{p}+1}=\Tilde{\lambda}_{i_{p}+1}-1
\]
so that we obtain
\[
\Tilde{\lambda}_{i_{k}^{\prime}}^{\ast(k+1)}=\Tilde{\lambda}_{i_{p}}>
\Tilde{\lambda}_{i_{p}+1}-1=
\Tilde{\lambda}_{i_{k}^{\prime}+1}^{\ast(k+1)}.
\]
In both cases \textbf{(a-2)} and \textbf{(a-3)}, 
$\Tilde{\lambda}_{i_{k}^{\prime}+1}^{\ast(k+1)}=\Tilde{\lambda}_{i_{p}+1}^{\ast (k+1)}=\Tilde{\lambda}_{i_{p}+1}$ 
because $i_{p}+1\notin \mathscr{I}^{(x,i+1)}$.
Since $\overline{\mathscr{I}^{(x,i)}}$ is smooth on 
$\Tilde{\lambda}\left[ \mathscr{J}^{(x-i,i)}\right]$ 
by the assumption of (2), 
\begin{equation} \label{eq:smooth1_1}
\Tilde{\lambda}\left[  \mathscr{J}^{(x-i,i)},\overline{i_{a}},\ldots,\overline{i_{p+1}}\right]  _{i_{p}}>
\Tilde{\lambda}\left[  \mathscr{J}^{(x-i,i)},\overline{i_{a}},\ldots,\overline{i_{p+1}}\right]  _{i_{p}+1}.
\end{equation}
In case \textbf{(a-2)}, the left-hand side of Eq.~\eqref{eq:smooth1_1} is $\Tilde{\lambda}_{i_{p}}$ 
because $i_{p}\in \mathscr{I}^{(x,i)}\backslash \mathscr{L}$, i.e, $i_{p}\notin \mathscr{J}^{(x-i,i)}$.
The right-hand side is $\Tilde{\lambda}_{i_{p}+1}$ because $i_{p}+1\in \mathscr{L}$ 
($\overline{i_{p}+1}$ appears once in $\{\overline{i_{a}},\ldots,\overline{i_{p+1}}\}$ and 
$i_{p}+1$ appears once in $\mathscr{J}^{(x-i,i)}$).
Therefore, $\Tilde{\lambda}_{i_{p}}>\Tilde{\lambda}_{i_{p}+1}$ so that we have 
$\Tilde{\lambda}_{i_{k}^{\prime}}^{\ast(k+1)}=\Tilde{\lambda}_{i_{p}}>
\Tilde{\lambda}_{i_{p}+1}=\Tilde{\lambda}_{i_{k}^{\prime}+1}^{\ast(k+1)}$.
In case \textbf{(a-3)}, $i_{p}+1\notin \mathscr{I}^{(x,i)}$ because 
$i_{p}+1\notin(\mathscr{I}^{(x,i)}\backslash \mathscr{L})\sqcup \mathscr{L}^{\ast}$ and 
$i_{p}+1\notin \mathscr{L}$. The left-hand side of Eq.~\eqref{eq:smooth1_1} is $\Tilde{\lambda}_{i_{p}}$ 
because $i_{p}\in \mathscr{I}^{(x,i)}\backslash \mathscr{L}$, i.e., $i_{p}\notin \mathscr{J}^{(x-i,i)}$, 
while the right-hand side is $\Tilde{\lambda}_{i_{p}+1}+\delta$ ($\delta\in\{0,1\}$) because $i_{p}+1\notin \mathscr{I}^{(x,i)}$.
Therefore, $\Tilde{\lambda}_{i_{p}}>\Tilde{\lambda}_{i_{p}+1}+\delta
\geq\Tilde{\lambda}_{i_{p}+1}$ so that we have 
$\Tilde{\lambda}_{i_{k}^{\prime}}^{\ast(k+1)}=\Tilde{\lambda}_{i_{p}}>
\Tilde{\lambda}_{i_{p}+1}=\Tilde{\lambda}_{i_{k}^{\prime}+1}^{\ast(k+1)}$.

\textbf{Case (b).}
In this case, $\mathscr{L}^{\ast}\neq \emptyset$ and we can write 
$i_{k}^{\prime}=l_{r}^{\ast}\; (\exists l_{r}^{\ast}\in \mathscr{L}^{\ast})$.
We divide this case further into the following two cases 
according to the algorithm in Definition~\ref{df:split} or Remark~\ref{rem:algorithm1}:
\begin{description}
\item[(b-1)]
$l_{r}^{\ast}=i_{p}-1\quad (\exists i_{p}\in \mathscr{I}^{(x,i)}\backslash \mathscr{L})$.
\item[(b-2)]
$l_{r}^{\ast}=j_{q}-1\quad (\exists j_{q}\in \mathscr{J}^{(x-i,i)})$.
\end{description}
Note that the situation that $l_{r}^{\ast}=l_{r+1}^{\ast}-1\; (r\neq c)$ cannot happen.
Indeed, if $l_{r}^{\ast}=l_{r+1}^{\ast}-1\; (r\neq c)$, then 
$i_{k}^{\prime}=l_{r+1}^{\ast}-1$.
Since $l_{r+1}^{\ast} \in \mathscr{I}^{(x,i+1)}$, this implies $i_{k+1}^{\prime}=l_{r+1}^{\ast}$, 
which contradicts the assumption of \textbf{(b)}. 
In case \textbf{(b-1)}, $i_{k+1}^{\prime}=i_{p}$ because $i_{p} \in \mathscr{I}^{(x,i+1)}$ and $i_{k}^{\prime}=i_{p}-1$.
Then
\[
\Tilde{\lambda}_{i_{k}^{\prime}}^{\ast(k+1)}=
\Tilde{\lambda}\lbrack\overline{i_{a}^{\prime}},\ldots,\overline{i_{k+1}^{\prime}}=
\overline{i_{p}}]_{i_{p}-1}=\Tilde{\lambda}_{i_{p}-1},
\]
and
\[
\Tilde{\lambda}_{i_{k}^{\prime}+1}^{\ast(k+1)}=
\Tilde{\lambda}\lbrack\overline{i_{a}^{\prime}},\ldots,\overline{i_{k+1}^{\prime}}=
\overline{i_{p}}]_{i_{p}}=\Tilde{\lambda}_{i_{p}}-1.
\]
From these two equations, we have 
$\Tilde{\lambda}_{i_{k}^{\prime}}^{\ast(k+1)}>\Tilde{\lambda}_{i_{k}^{\prime}+1}^{\ast(k+1)}$.
In case \textbf{(b-2)},
\[
\Tilde{\lambda}_{i_{k}^{\prime}}^{\ast(k+1)}=
\Tilde{\lambda}\lbrack\overline{i_{a}^{\prime}},\ldots,\overline{i_{k+1}^{\prime}}]_{i_{k}^{\prime}}=
\Tilde{\lambda}_{i_{k}^{\prime}}=\Tilde{\lambda}_{j_{q}-1}.
\]
On the other hand,
\[
\Tilde{\lambda}_{i_{k}^{\prime}+1}^{\ast(k+1)}=
\Tilde{\lambda}\lbrack\overline{i_{a}^{\prime}},\ldots,\overline{i_{k+1}^{\prime}}]_{i_{k}^{\prime}+1}=
\Tilde{\lambda}_{i_{k}^{\prime}+1}=\Tilde{\lambda}_{j_{q}},
\]
where we have used the fact that $i_{k}^{\prime}+1<i_{k+1}^{\prime}$.
This is shown as follows.
If $i_{k}^{\prime}+1=i_{k+1}^{\prime}$, 
then 
$j_{q}=i_{k}^{\prime}+1=i_{k+1}^{\prime}$.
This implies that $j_{q}$ is an $\mathscr{I}^{(x,i)}$-letter that is not a $\mathscr{J}^{(x-i,i)}$-letter 
due to the assumption of \textbf{(b)}, which is a contradiction.
Now since $\mathscr{J}^{(x-i,i)}$ is smooth on $\Tilde{\lambda}$, 
we have
\[
\Tilde{\lambda}\lbrack j_{1},\ldots,j_{q-1}]_{j_{q}-1}>
\Tilde{\lambda}\lbrack j_{1},\ldots,j_{q-1}]_{j_{q}}.
\]
By noting $j_{q}-1=l_{r}^{\ast}\notin\{j_{1},\ldots,j_{q-1}\}$, 
the left-hand side of this inequality is found to be $\Tilde{\lambda}_{j_{q}-1}$, 
while the right-hand side is clearly $\Tilde{\lambda}_{j_{q}}$.
Hence, we have 
$\Tilde{\lambda}_{i_{k}^{\prime}}^{\ast(k+1)}>\Tilde{\lambda}_{i_{k}^{\prime}+1}^{\ast(k+1)}$.

\textbf{Case (c).}
In this case, $\mathscr{L}^{\ast}\neq \emptyset$ and we can write $i_{k}^{\prime}=l_{r}^{\ast}$
and $i_{k+1}^{\prime}=l_{r+1}^{\ast}\;(\exists r\in\{1,\ldots,c-1\})$.
According to the algorithm in Definition~\ref{df:split} or Remark~\ref{rem:algorithm1},
let us consider the following three cases separately:
\begin{description}
\item[(c-1)]
$l_{r}^{\ast}=i_{p}-1\quad (\exists i_{p}\in \mathscr{I}^{(x,i)}\backslash \mathscr{L})$.
\item[(c-2)]
$l_{r}^{\ast}=j_{q}-1\quad (\exists j_{q}\in \mathscr{J}^{(x-i,i)})$.
\item[(c-3)]
$l_{r}^{\ast}=l_{r+1}^{\ast}-1\quad (r\neq c)$.
\end{description}
In case \textbf{(c-1)}, we have $i_{p}\in \mathscr{I}^{(x,i+1)}$ and $i_{k}^{\prime}=i_{p}-1$.
This implies $i_{k+1}^{\prime}=i_{p}$.
However, this also implies $l_{r+1}^{\ast}=i_{p}\in \mathscr{I}^{(x,i)}\backslash \mathscr{L}=
\mathscr{I}^{(x,i+1)}\backslash \mathscr{L}^{\ast}$, 
which is clearly a contradiction, and thereby this case must be excluded.
In case \textbf{(c-2)},
\[
\Tilde{\lambda}_{i_{k}^{\prime}}^{\ast(k+1)}=
\Tilde{\lambda}\lbrack\overline{i_{a}^{\prime}},\ldots,\overline{i_{k+1}^{\prime}}]_{i_{k}^{\prime}}=
\Tilde{\lambda}_{i_{k}^{\prime}}=\Tilde{\lambda}_{j_{q}-1}.
\]
On the other hand,
\[
\Tilde{\lambda}_{i_{k}^{\prime}+1}^{\ast(k+1)}=
\Tilde{\lambda}\lbrack\overline{i_{a}^{\prime}},\ldots,\overline{i_{k+1}^{\prime}}]_{i_{k}^{\prime}+1}=
\Tilde{\lambda}_{i_{k}^{\prime}+1}=\Tilde{\lambda}_{j_{q}},
\]
where we have used the fact that $i_{k+1}^{\prime}>i_{k}^{\prime}+1$.
This is shown as follows. 
If $i_{k+1}^{\prime}=i_{k}^{\prime}+1$, 
then we have $l_{r+1}^{\ast}=i_{k+1}^{\prime}=i_{k}^{\prime}+1=l_{r}^{\ast}+1=j_{q}$, 
which contradicts the fact that $l_{r+1}^{\ast}$ is not a $\mathscr{J}^{(x-i,i)}$-letter.
Now since $\mathscr{J}^{(x-i,i)}$ is smooth on $\Tilde{\lambda}$, 
we have
\[
\Tilde{\lambda}\lbrack j_{1},\ldots,j_{q-1}]_{j_{q}-1}>
\Tilde{\lambda}\lbrack j_{1},\ldots,j_{q-1}]_{j_{q}}.
\]
The left-hand side of this inequality is $\Tilde{\lambda}_{j_{q}-1}$ 
because $j_{q}-1=l_{r}^{\ast}\notin\{j_{1},\ldots,j_{q-1}\}$, 
while the right-hand side is clearly $\Tilde{\lambda}_{j_{q}}$.
Hence, we have 
$\Tilde{\lambda}_{i_{k}^{\prime}}^{\ast(k+1)}>\Tilde{\lambda}_{i_{k}^{\prime}+1}^{\ast(k+1)}$.
In case \textbf{(c-3)}, by noting $i^{\prime}_{k+1}=i^{\prime}_{k}+1$, we have
\[
\Tilde{\lambda}_{i_{k}^{\prime}+1}^{\ast(k+1)} =
\Tilde{\lambda}\lbrack\overline{i_{a}^{\prime}},\ldots,\overline{i_{k+1}^{\prime}}]_{i_{k}^{\prime}+1}=
\Tilde{\lambda}_{i_{k}^{\prime}+1}-1,
\]
while
\[
\Tilde{\lambda}_{i_{k}^{\prime}}^{\ast(k+1)}=
\Tilde{\lambda}\lbrack\overline{i_{a}^{\prime}},\ldots,\overline{i_{k+1}^{\prime}}]_{i_{k}^{\prime}}=
\Tilde{\lambda}_{i_{k}^{\prime}}.
\]
Hence, we have
$\Tilde{\lambda}_{i_{k}^{\prime}}^{\ast(k+1)}>\Tilde{\lambda}_{i_{k}^{\prime}+1}^{\ast(k+1)}$.
In \textbf{(II)}, we have verified that 
$\Tilde{\lambda}_{i_{k}^{\prime}}^{\ast(k+1)}>\Tilde{\lambda}_{i_{k}^{\prime}+1}^{\ast(k+1)}$, that is, 
$\Tilde{\lambda}^{\ast(k+1)}[\overline{i_{k}^{\prime}}]$ 
is a Young diagram for all possible cases.
From \textbf{(I)} and \textbf{(II)} and by induction, 
we have completed the proof of (2) of Lemma~\ref{lem:smooth1}.

The proof of (1) is as follows.
We proceed by induction on $x$.
Since the sequence of letters $\mathscr{J}^{(1,0)},\overline{\mathscr{I}^{(1,0)}}$ is smooth on $\lambda$, it is not hard to show that $\overline{\mathscr{I}^{(1,1)}}$ is smooth on 
$\lambda$ by using the same argument as in (2); 
$\lambda \left[ \mathscr{J}^{(1,0)}\right]$ is a Young diagram on which  
$\overline{\mathscr{I}^{(1,0)}}$ is smooth by Eq.\eqref{eq:mu2lambda} so that 
$\overline{\mathscr{I}^{(1,1)}}$ is smooth on $\lambda$ ($x=1$). 
For $2\leq x\leq n_{c}$,
\[
\lambda^{(x-1)}  =\lambda\left[  \overline{\mathscr{I}^{(1,1)}},\ldots,\overline{\mathscr{I}^{(x-1,x-1)}},
\mathscr{J}^{(1,x-1)},\ldots,\mathscr{J}^{(x-1,1)}\right]
\]
is written as 
\begin{equation} \label{eq:paired}
\lambda\left[  \overline{\mathscr{I}^{(1,0)}},\ldots,\overline{\mathscr{I}^{(x-1,0)}},
\mathscr{J}^{(1,0)},\ldots,\mathscr{J}^{(x-1,0)}\right].
\end{equation}
This is shown as follows.
Since
\begin{align*}
&  \left\langle \overline{\mathscr{I}^{(1,0)}},\mathscr{J}^{(1,0)}\right\rangle_{\mathrm{pair}}, \\
&  \left\langle \overline{\mathscr{I}^{(2,0)}},\mathscr{J}^{(2,0)}\right\rangle_{\mathrm{pair}} ,
\left\langle \overline{\mathscr{I}^{(2,1)}},J^{(1,1)}\right\rangle_{\mathrm{pair}} ,\\
&  \qquad \qquad \qquad ,\ldots,\\
&  \left\langle \overline{\mathscr{I}^{(x-1,0)}},\mathscr{J}^{(x-1,0)}\right\rangle_{\mathrm{pair}} ,
\ldots,
\left\langle \overline{\mathscr{I}^{(x-1,x-2)}},\mathscr{J}^{(1,x-2)}\right\rangle_{\mathrm{pair}},
\end{align*}
we can increase the counter in the paired sets appeared in Eq.~\eqref{eq:paired} by one successively 
keeping the shape of Eq.~\eqref{eq:paired} because the corresponding map 
$\phi^{(\centerdot,\centerdot)}$ is weight-preserving.
At the end,  Eq.~\eqref{eq:paired} turns out to be 
\[
\lambda\left[  \overline{\mathscr{I}^{(1,1)}},\ldots,\overline{\mathscr{I}^{(x-1,x-1)}},
\mathscr{J}^{(1,x-1)},\ldots,\mathscr{J}^{(x-1,1)}\right]  =\lambda^{(x-1)}.
\]
Thus, 
\[
\lambda^{(x-1)}=
\lambda\left[ \mathscr{J}^{(1,0)},\overline{\mathscr{I}^{(1,0)}},\ldots,
\mathscr{J}^{(x-1,0)},\overline{\mathscr{I}^{(x-1,0)}}\right]
\]
is a Young diagram on which the sequence of letters 
$\mathscr{J}^{(x,0)},\overline{\mathscr{I}^{(x,0)}}$ is smooth by Eq.~\eqref{eq:mu2lambda}.
Hence, we can show that $\overline{\mathscr{I}^{(x,1)}}$ is smooth on $\lambda^{(x-1)}$ 
by using the same argument as in (2).

The proof of (3) is as follows.
We proceed by induction on $x$ and $i$.

\textbf{(I).}
We have that $\overline{\mathscr{I}^{(1,1)}}$ is smooth on $\lambda$ 
by (1) ($x=1$) and 
$\lambda\left[  \overline{\mathscr{I}^{(1,1)}},\mathscr{J}^{(1,1)}\right]$ is 
a Young diagram by (1) ($x=2$).

\textbf{(II).}
For $2\leq x\leq n_{c}$, let us assume that
\[
\lambda^{(x-1,i)}=
\begin{cases}
\lambda \left[ \overline{\mathscr{I}^{(1,1)}},\ldots,\overline{\mathscr{I}^{(x-1,x-1)}},
\mathscr{J}^{(1,x-1)},\ldots,\mathscr{J}^{(x-i,i)}\right]  & (1\leq i\leq x-1),\\
\lambda\left[ \overline{\mathscr{I}^{(1,1)}},\ldots,\overline{\mathscr{I}^{(x-1,x-1)}}\right] & (i=x)
\end{cases}
\]
are all Young diagrams (for $x=2$ this assumption is satisfied by \textbf{(I)}).
\textbf{(i).}
By (1), $\overline{\mathscr{I}^{(x,1)}}$ is smooth on 
$\lambda^{(x-1)}=\lambda^{(x-1,1)}$.
\textbf{(ii).}
For $1\leq i\leq x-1$, suppose that $\overline{\mathscr{I}^{(x,i)}}$ is smooth on 
$\lambda^{(x-1,i)}$ (for $i=1$ this is satisfied).
Thus, $\overline{\mathscr{I}^{(x,i+1)}}$ is smooth on $\lambda^{(x-1,i+1)}$ 
by the claim of (2).
From \textbf{(i)} and \textbf{(ii)} and by induction, we have that 
$\overline{\mathscr{I}^{(x,i)}}$ is smooth on $\lambda^{(x-1,i)}$ ($1\leq i\leq x$).
For $1 \leq i \leq x-1$, since
\begin{multline*}
\left\langle \overline{\mathscr{I}^{(x,i)}},\mathscr{J}^{(x-i,i)}\right\rangle_{\mathrm{pair}},
\left\langle \overline{\mathscr{I}^{(x,i+1)}},\mathscr{J}^{(x-i-1,i+1)}\right\rangle_{\mathrm{pair}},\\
\ldots, 
\left\langle \overline{\mathscr{I}^{(x,x-1)}},\mathscr{J}^{(1,x-1)}\right\rangle_{\mathrm{pair}},
\end{multline*} 
we have
\begin{align*}
\lambda^{(x-1,i)}\left[ \overline{\mathscr{I}^{(x,i)}}\right]
& =\lambda\left[  \overline{\mathscr{I}^{(1,1)}},\ldots,
\overline{\mathscr{I}^{(x-1,x-1)}},\overline{\mathscr{I}^{(x,i)}},\mathscr{J}^{(1,x-1)},
\ldots,\mathscr{J}^{(x-i,i)}\right]  \\
& =\lambda\left[ \overline{\mathscr{I}^{(1,1)}},\ldots,
\overline{\mathscr{I}^{(x-1,x-1)}},\overline{\mathscr{I}^{(x,x)}},\mathscr{J}^{(1,x)},
\ldots,\mathscr{J}^{(x-i,i+1)}\right]  \\
& =\lambda^{(x,i+1)} \quad (1\leq i \leq x-1),
\end{align*}
and 
\[
\lambda^{(x-1,x)}\left[  \overline{\mathscr{I}^{(x,x)}}\right]
=\lambda\left[  \overline{\mathscr{I}^{(1,1)}},\ldots,
\overline{\mathscr{I}^{(x,x)}}\right]  =\lambda^{(x,x+1)}.
\]
Namely, $\lambda^{(x,i+1)}$ ($1\leq i \leq x$) are all Young diagrams.
By (1), $\lambda^{(x,1)}=\lambda^{(x)}$ is a Young diagram.
That is,
\[
\lambda^{(x,i)}=
\begin{cases}
\lambda\left[  \overline{\mathscr{I}^{(1,1)}},\ldots,
\overline{\mathscr{I}^{(x,x)}},\mathscr{J}^{(1,x)},\ldots,\mathscr{J}^{(x+1-i,i)}\right]  & 
(1\leq i\leq x), \\
\lambda\left[  \overline{\mathscr{I}^{(1,1)}},\ldots,
\overline{\mathscr{I}^{(x,x)}}\right]  & (i=x+1)
\end{cases}
\]
are all Young diagrams.
The claim follows from \textbf{(I)} and \textbf{(II)} and by induction on $x$.
\end{proof}

\begin{lem} \label{lem:smooth2}
(1).
Let us define
\[
\mu^{(x+1)}:=\begin{cases}
\mu\left[  \overline{\mathscr{J}^{(n_{c},1)}},\ldots,\overline{\mathscr{J}^{(x+1,n_{c}-x)}},
\mathscr{I}^{(n_{c},n_{c}-x)},\ldots,\mathscr{I}^{(x+1,1)}\right]  
& (1\leq x\leq n_{c}-1), \\
\mu & (x=n_{c}).
\end{cases}
\]
Then $\mu^{(x+1)}$ is a Young diagram on which $\overline{\mathscr{J}^{(x,1)}}$ is smooth.

(2).
For $1\leq x\leq n_{c}-1$, let us assume that
\[
\mu^{(x+1,i)}:=
\begin{cases}
\Tilde{\mu}^{(x+1)}\left[\mathscr{I}^{(n_{c},n_{c}-x)},\ldots,
\mathscr{I}^{(x+i,i)}\right] & (1\leq i\leq n_{c}-x), \\
\Tilde{\mu}^{(x+1)} & (i=n_{c}-x+1)
\end{cases}
\]
are all Young diagrams, where 
$\Tilde{\mu}^{(x+1)}:=\mu\left[  \underrightarrow{\overline{\mathscr{J}^{(n_{c},1)}}},\ldots,
\underrightarrow{\overline{\mathscr{J}^{(x+1,n_{c}-x)}}}\right]$.
Suppose that $\overline{\mathscr{J}^{(x,i)}}$ is smooth on $\mu^{(x+1,i)}$.
Then we have that $\overline{\mathscr{J}^{(x,i+1)}}$ is smooth on 
$\mu^{(x+1,i+1)}$ ($1\leq i\leq n_{c}-x$).

(3).
$\mu\left[ 
\underrightarrow{\overline{\mathscr{J}^{(n_{c},1)}}},\ldots,
\underrightarrow{\overline{\mathscr{J}^{(2,n_{c}-1)}}},
\underrightarrow{\overline{\mathscr{J}^{(1,n_{c})}}}\right]  $.
\end{lem}

\begin{proof}
The proof of (1) of Lemma~\ref{lem:smooth2} is as follows.
Since the sequence of letters $\mathscr{I}^{(n_{c},0)},\overline{\mathscr{J}^{(n_{c},0)}}$ is 
smooth on $\mu$, 
it is not hard to show that $\overline{\mathscr{J}^{(n_{c},1)}}$ is smooth on $\mu$ 
by using the same argument as in Lemma~\ref{lem:smooth1} (2).
For $1\leq x\leq n_{c}-1$,
\begin{align*}
\mu^{(x+1)}  & =\mu\left[  \overline{\mathscr{J}^{(n_{c},1)}},\ldots,\overline{\mathscr{J}^{(x+1,n_{c}-x)}},
\mathscr{I}^{(n_{c},n_{c}-x)},\ldots,\mathscr{I}^{(x+1,1)}\right]  \\
& =\mu\left[  \mathscr{I}^{(n_{c},0)},\overline{\mathscr{J}^{(n_{c},0)}},\ldots
\mathscr{I}^{(x+1,0)},\overline{\mathscr{J}^{(x+1,0)}}\right]
\end{align*}
is a Young diagram on which the sequence of letters $\mathscr{I}^{(x,0)},\overline{\mathscr{J}^{(x,0)}}$ is smooth by Eq.~\eqref{eq:mu2lambda}.
We can show that $\overline{\mathscr{J}^{(x,1)}}$ is smooth on $\mu^{(x+1)}$ 
by using the same argument as in Lemma~\ref{lem:smooth1} (2).
The proof of the rest part runs as in Lemma~\ref{lem:smooth1} (2) and (3).
\end{proof}

\begin{proof}[Proof of Proposition~\ref{prp:main11}]
Let $T\in \mathbf{B}_{n}^{\mathfrak{sp}_{2n}}(\nu)_{\mu}^{\lambda}$ and suppose that $T$ consists of $n_{c}$ columns.
By Lemma~\ref{lem:Phi}, we have $\Phi(T)\in C_{n}\text{-}\mathrm{SST}(\nu)$.
By Lemma~\ref{lem:smooth1} and Lemma~\ref{lem:smooth2}, we have 
$\lambda\left[ 
\underrightarrow{\overline{\mathscr{I}^{(1,1)}}},\ldots,
\underrightarrow{\overline{\mathscr{I}^{(n_{c},n_{c})}}}  \right]$ 
and 
$\mu\left[ 
\underrightarrow{\overline{\mathscr{J}^{(n_{c},1)}}},\ldots,
\underrightarrow{\overline{\mathscr{J}^{(1,n_{c})}}}  \right]$. 
Let us set 
$\zeta:=\lambda\left[  \underrightarrow{\overline{\mathscr{I}^{(1,1)}}},\ldots,
\underrightarrow{\overline{\mathscr{I}^{(n_{c},n_{c})}}}\right]  $, i.e., 
$\zeta\left[  \underrightarrow{\mathrm{FE}(\Phi(T)^{(+)})}\right]$, 
which is $\mu$ by Eq.~\eqref{eq:mu2lambda}.
Here,
\begin{align*}
\zeta\left[  \mathscr{J}^{(1,n_{c})},\ldots,\mathscr{J}^{(n_{c},1)}\right] &=
\lambda\left[
\mathscr{J}^{(1,n_{c})},\overline{\mathscr{I}^{(1,1)}},\ldots,
\mathscr{J}^{(n_{c},1)},\overline{\mathscr{I}^{(n_{c},n_{c})}}\right] \\
&=
\lambda\left[ \overline{\mathrm{FE}\left(\Phi(T)\right)}\right].
\end{align*}
Since $\Phi$ is weight-preserving, 
$\lambda\left[ \overline{\mathrm{FE}\left(\Phi(T)\right)}\right]=
\lambda\left[ \overline{\mathrm{FE}(T)}\right]=\mu$.
Combining these, we have  
$\mu\left[ 
\underrightarrow{\overline{\mathscr{J}^{(n_{c},1)}}},\ldots,
\underrightarrow{\overline{\mathscr{J}^{(1,n_{c})}}}  \right]=\zeta$, i.e.,  
$\mu\left[  \underrightarrow{\mathrm{FE}(\Phi(T)^{(-)})}\right]  =\zeta$ and therefore 
$\mu\left[  \underrightarrow{\mathrm{FE}(\mathrm{Rect}(\Phi(T)^{(-)}))}\right]
=\zeta$ by Proposition~\ref{prp:skew}.
Hence, we have $\Phi(T)^{(+)}\in \mathbf{B}_{n}^{(+)}(\xi)_{\zeta}^{\lambda}$ and 
$\mathrm{Rect}(\Phi(T)^{(-)})\in \mathbf{B}_{n}^{(-)}(\eta)_{\zeta}^{\mu}$, where 
$\xi$ and $\eta$ are the shapes of $\Phi(T)^{(+)}$ and $\mathrm{Rect}(\Phi(T)^{(-)})$, respectively.
\end{proof}

\section{Properties of $\Psi$} \label{sec:Psi}

Throughout this section, the tableau $T$ is that described in Proposition~\ref{prp:main12}.
The purpose of this section is to show that the map $\Psi$ is well-defined 
and $\Psi(T)\in C_{n}\text{-}\mathrm{SST}_{\mathrm{KN}}(\nu)$.

\begin{lem}
The map $\psi^{(x,n_{c})}$ is well-defined on
\[
\Tilde{T}:=\begin{cases}
\overline{\Psi^{(x-1)}}(T) & (2\leq x\leq n_{c}), \\
T & (x=1).
\end{cases}
\]
Here we assume $\Tilde{T}\neq \emptyset$.
\end{lem}

\begin{proof} 
When $x=1$, let $\Delta q$ be the offset given by the difference between 
the length of the $\mathscr{C}_{n}^{(+)}$-letters part of the first column of $T$ and 
that of the $n_{c}$-th column of $T$.
Suppose that the tableau $\Tilde{T}$ has the following configuration.

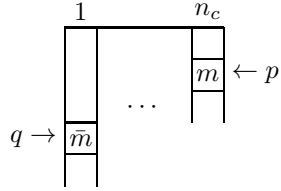
\setlength{\unitlength}{12pt}
\begin{center}
\begin{picture}(9,6)
\put(2,0){\line(0,1){5}}
\put(3,0){\line(0,1){5}}
\put(6,2){\line(0,1){3}}
\put(7,2){\line(0,1){3}}
\put(2,1){\line(1,0){1}}\put(2,2){\line(1,0){1}}
\put(6,3){\line(1,0){1}}\put(6,4){\line(1,0){1}}
\put(2,5){\line(1,0){5}}
\put(4,2){\makebox(1,1){$\cdots$}}
\put(0,1){\makebox(2,1){$q\rightarrow$}}
\put(7,3){\makebox(2,1){$\leftarrow p$}}
\put(2,5){\makebox(1,1){$1$}}
\put(6,5){\makebox(1,1){$n_{c}$}}
\put(2,1){\makebox(1,1){$\Bar{m}$}}
\put(6,3){\makebox(1,1){$m$}}
\end{picture}.
\end{center}
Since $\Bar{m}$ appears in $T_{2}\in \mathbf{B}_{n}^{(-)}(\eta)_{\zeta}^{\mu}$, $m\leq l(\mu)$.
Furthermore, it is obvious that $(q-\Delta q-p)\leq l(\nu)$.
Recall that we assume that $l(\mu)+l(\nu)\leq n$ in Theorem~\ref{thm:main1}.
Hence, $n-m\geq l(\mu)+l(\nu)-m\geq(q-\Delta q-p)$ so that $\psi^{(x,n_{c})}$ is well-defined on $\Tilde{T}$ (Definition~\ref{df:coad}).

When $2\leq x\leq n_{c}$, let $\Delta q$ be the offset given by the difference between 
the length of the $\mathscr{C}_{n}^{(+)}$-letters part of the $x$-th column of $T$ and 
that of the $n_{c}$-th column of $\Tilde{T}$.
Suppose that the tableau $\Tilde{T}$ has the following configuration.

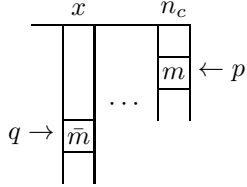
\setlength{\unitlength}{12pt}
\begin{center}
\begin{picture}(8,6)
\put(2,0){\line(0,1){5}}
\put(3,0){\line(0,1){5}}
\put(5,2){\line(0,1){3}}
\put(6,2){\line(0,1){3}}
\put(2,1){\line(1,0){1}}\put(2,2){\line(1,0){1}}
\put(5,3){\line(1,0){1}}\put(5,4){\line(1,0){1}}
\put(1,5){\line(1,0){5}}
\put(3,2){\makebox(2,1){$\cdots$}}
\put(0,1){\makebox(2,1){$q\rightarrow$}}
\put(6,3){\makebox(2,1){$\leftarrow p$}}
\put(2,5){\makebox(1,1){$x$}}
\put(5,5){\makebox(1,1){$n_{c}$}}
\put(2,1){\makebox(1,1){$\Bar{m}$}}
\put(5,3){\makebox(1,1){$m$}}
\end{picture}.
\end{center}
Note that the $\mathscr{C}_{n}^{(-)}$-letters part of the $x$-th column is unchanged under 
application of $\overline{\Psi^{(x-1)}}$ so that $\Bar{m}$ in the $x$-th column in $\Tilde{T}$ 
lies at the original position of $T$, and thereby $m\leq l(\mu)$.
Let $m^{\prime}$ be the entry at the $p$-th position of the $n_{c}$-th column of the original tableau $T$.
Then $m^{\prime}\leq m$ by Lemma~\ref{lem:col_sst2} so that $\min(m,m^{\prime})\leq l(\mu)$.
Hence, we have $n-\min(m,m^{\prime})\geq l(\mu)+l(\nu)-\min(m,m^{\prime})\geq(q-\Delta q-p)$.
That is, $\psi^{(x,n_{c})}$ is well-defined on $\Tilde{T}$.
\end{proof}

\begin{lem} \label{lem:welldf2}
The map $\psi^{(x,y)}$ is well-defined on
\[
\Tilde{T}:=\psi ^{(x,y+1)}\circ \cdots \circ \psi ^{(x,n_{c})}\circ \overline{\Psi^{(x-1)}}(T)
\quad (1\leq x\leq y\leq n_{c}).
\]
Here, we assume that $\Tilde{T}\neq \emptyset$ and that in the updating process of the tableau from $T$ to $\Tilde{T}$ the semistandardness of the $\mathscr{C}_{n}^{(+)}$-letters part of the tableau is preserved.
\end{lem}

\begin{proof}
Let $C_{-}^{(x)}$ (resp. $C_{+}^{(y)}$) be the $\mathscr{C}_{n}^{(-)}$ (resp. $\mathscr{C}_{n}^{(+)}$)-letters part of 
the $x$-th (resp. $y$-th) column of $\Tilde{T}$.
Let $C^{(x,y)}$ be the column whose  $\mathscr{C}_{n}^{(+)}$ (resp. $\mathscr{C}_{n}^{(-)}$)-letters part is  
$C_{+}^{(y)}$ (resp. $C_{-}^{(x)}$).
If $C^{(x,y)}$ is KN-coadmissible, then we can apply $\psi^{(x,y)}$ to $\Tilde{T}$.
Suppose that $\Tilde{T}$ has the following configuration.

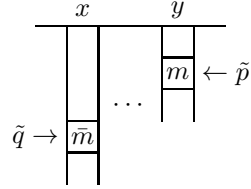
\setlength{\unitlength}{12pt}
\begin{center}
\begin{picture}(8,6)
\put(2,0){\line(0,1){5}}
\put(3,0){\line(0,1){5}}
\put(3,2){\makebox(2,1){$\cdots$}}
\put(5,2){\line(0,1){3}}
\put(6,2){\line(0,1){3}}
\put(2,1){\line(1,0){1}}
\put(2,2){\line(1,0){1}}
\put(5,3){\line(1,0){1}}
\put(5,4){\line(1,0){1}}
\put(1,5){\line(1,0){6}}
\put(2,1){\makebox(1,1){$\Bar{m}$}}
\put(5,3){\makebox(1,1){$m$}}
\put(0,1){\makebox(2,1){$\Tilde{q} \rightarrow$}}
\put(6,3){\makebox(2,1){$\leftarrow \Tilde{p}$}}
\put(2,5){\makebox(1,1){$x$}}
\put(5,5){\makebox(1,1){$y$}}
\end{picture}.
\end{center}
If $(\Tilde{q}-\Delta q-\Tilde{p})+m \leq n$, then $C^{(x,y)}$ is KN-coadmissible, 
where $\Delta q (\geq 0)$ is the offset given by the difference between the length of the  $\mathscr{C}_{n}^{(+)}$-letters part of the $x$-th column and that of the $y$-th column of $\Tilde{T}$.
Let $C_{-}^{(x)\prime}$ be the $\mathscr{C}_{n}^{(-)}$-letters part of 
the $x$-th column of $T^{\prime}:=\phi^{(x,y+1)}(\Tilde{T})$ and $C_{+}^{(y+1)\prime}$ be the $\mathscr{C}_{n}^{(+)}$-letters part of the $(y+1)$-st column of $T^{\prime}$.
Let $C^{(x,y+1)}$ be the column whose $\mathscr{C}_{n}^{(+)}$ (resp. $\mathscr{C}_{n}^{(-)}$)-letters part is 
$C_{+}^{(y+1)\prime}$ (resp. $C_{-}^{(x)\prime}$) and $\mathscr{L}^{(x,y+1)}$ be the set of $\mathscr{L}$-letters of  $C^{(x,y+1)}$.
We consider the following two cases separately:

\begin{description}
\item[(a)]
$\overline{m}$ appears in the $x$-th column of $T^{\prime}$ and $m\notin \mathscr{L}^{(x,y+1)}$.
\item[(b)]
$\overline{m}$ in the $x$-th column of $\Tilde{T}$ is generated when $\psi ^{(x,y+1)}$ is applied to $T^{\prime}$.
\end{description}

\textbf{Case (a).}
Suppose that the tableau $T^{\prime}$ has the following configuration.

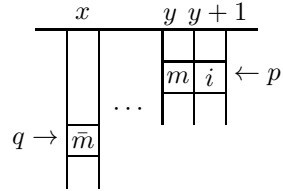
\setlength{\unitlength}{12pt}
\begin{center}
\begin{picture}(9,6)
\put(2,0){\line(0,1){5}}
\put(3,0){\line(0,1){5}}
\put(5,2){\line(0,1){3}}
\put(6,2){\line(0,1){3}}
\put(7,2){\line(0,1){3}}
\put(2,1){\line(1,0){1}}
\put(2,2){\line(1,0){1}}
\put(5,3){\line(1,0){2}}
\put(5,4){\line(1,0){2}}
\put(1,5){\line(1,0){7}}
\put(2,5){\makebox(1,1){$x$}}
\put(4.75,5){\makebox(1,1){${\mathstrut y}$}}
\put(6,5){\makebox(1.5,1){${\mathstrut y+1}$}}
\put(3,2){\makebox(2,1){$\cdots$}}
\put(2,1){\makebox(1,1){$\Bar{m}$}}
\put(5,3){\makebox(1,1){$m$}}
\put(6,3){\makebox(1,1){$i$}}
\put(0,1){\makebox(2,1){$q \rightarrow$}}
\put(7,3){\makebox(2,1){$\leftarrow p$}}
\end{picture}.
\end{center}
By the assumption of \textbf{(a)}, $m \notin \mathscr{L}^{(x,y+1)}$ so that $m<i$ 
(if $m \in \mathscr{L}^{(x,y+1)}$, then $\Bar{m}$ in the $x$-th column of $T^{\prime}$ disappear by $\psi^{(x,y+1)}$).
Let us set $\left\{  l\in \mathscr{L}^{(x,y+1)} \relmiddle| \overline{l^{\dag}}\prec\Bar{m}\prec \Bar{l}\right\}
 =:\{l_{r+1}=l_{min},\ldots,l_{r+s}\}$.
If this set is empty ($s=0$), then the position of $\Bar{m}$ does not change 
when $\psi^{(x,y+1)}$ is applied to $T^{\prime}$.
In this case, we, we have
$(q-\Delta q -p)+\min(m,i)=(q-\Delta q -p)+m \leq n$ by Lemma~\ref{lem:KN_coad} because $C^{(x,y+1)}$ is KN-coadmissible ($\Tilde{T}\neq \emptyset$).
This inequality still holds when $\psi ^{(x,y+1)}$ is applied to $T^{\prime}$ so that 
$C^{(x,y)}$ is KN-coadmissible.
Now suppose that the above set is not empty ($s\geq 1$).
We adopt the second kind algorithm for $\psi^{(x,y+1)}$ here.
Let us assume that $\sharp\left\{ l\in \mathscr{L}^{(x,y+1)} \relmiddle| m< l < l_{min}^{\dag} \right\}=t $.
Since the number of $l$'s such that $l_{min}<l<l_{min}^{\dag}$ is $s+t-1$, we have
\begin{equation} \label{eq:psi_xy_1}
q_{min}^{\dag}-\Delta q -p_{min}^{\dag}+l_{min}^{\dag}\leq n+(s+t-1)+1
\end{equation}
by Lemma~\ref{lem:KN_upper}, where $p_{min}^{\dag}$ is the position of $l_{min}^{\dag}$ in the $(y+1)$-st column and $q_{min}^{\dag}$ is the position of $\overline{l_{min}^{\dag}}$ in the $x$-th column of $\psi^{(x,y+1)}(T^{\prime})=\Tilde{T}$.
Initially, the tableau $T^{\prime}$ has the following configuration, 
where the left (resp. right) part is the $\mathscr{C}_{n}^{(-)}$ (resp. $\mathscr{C}_{n}^{(+)}$)-letters one 
($l_{r+1}=l_{min}<\ldots < l_{r+s}<m<i$).

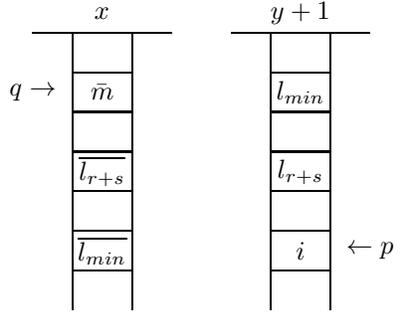
\setlength{\unitlength}{15pt}
\begin{center}
\begin{picture}(10.5,8)
\put(2,0){\line(0,1){7}}
\put(3.5,0){\line(0,1){7}}
\put(2,1){\line(1,0){1.5}}
\put(2,2){\line(1,0){1.5}}
\put(2,3){\line(1,0){1.5}}
\put(2,4){\line(1,0){1.5}}
\put(2,5){\line(1,0){1.5}}
\put(2,6){\line(1,0){1.5}}
\put(1,7){\line(1,0){3.5}}
\put(2,1){\makebox(1.5,1){$\overline{l_{min}}$}}
\put(2,3){\makebox(1.5,1){$\overline{l_{r+s}}$}}
\put(2,5){\makebox(1.5,1){$\Bar{m}$}}
\put(2,7){\makebox(1.5,1){$x$}}
\put(0,5){\makebox(2,1){$q \rightarrow$}}
\put(7,0){\line(0,1){7}}
\put(8.5,0){\line(0,1){7}}
\put(7,1){\line(1,0){1.5}}
\put(7,2){\line(1,0){1.5}}
\put(7,3){\line(1,0){1.5}}
\put(7,4){\line(1,0){1.5}}
\put(7,5){\line(1,0){1.5}}
\put(7,6){\line(1,0){1.5}}
\put(6,7){\line(1,0){3.5}}
\put(7,1){\makebox(1.5,1){$i$}}
\put(7,3){\makebox(1.5,1){$l_{r+s}$}}
\put(7,5){\makebox(1.5,1){$l_{min}$}}
\put(7,7){\makebox(1.5,1){$y+1$}}
\put(8.5,1){\makebox(2,1){$\leftarrow p$}}
\end{picture}.
\end{center}
Let us divide this case further into the following two cases:

\begin{description}
\item[(a-1)]
$i>l_{min}^{\dag}$.
\item[(a-2)]
$l_{min}^{\dag}>i$.
\end{description}
Note that $i\neq l_{min}^{\dag}$ because $i\in C^{(x,y+1)}$ and $l_{min}^{\dag}\notin C^{(x,y+1)}$.

\textbf{Case (a-1).}
The filling diagram of the $C^{(x,y+1)}$ has the following configuration 
before the operation for $l_{min}\rightarrow l_{min}^{\dag}$.

\setlength{\unitlength}{15pt}
\begin{center}
\begin{picture}(9,3)
\put(1,1){\line(0,1){2}}
\put(2,1){\line(0,1){2}}
\put(4,1){\line(0,1){2}}
\put(5,1){\line(0,1){2}}
\put(7,1){\line(0,1){2}}
\put(8,1){\line(0,1){2}}
\put(0,1){\line(1,0){9}}
\put(0,2){\line(1,0){2}}
\put(4,2){\line(1,0){1}}
\put(7,2){\line(1,0){2}}
\put(0,3){\line(1,0){9}}
\put(1,1){\makebox(1,1){$\bullet$}}
\put(1,2){\makebox(1,1){$\bullet$}}
\put(4,1){\makebox(1,1){$\bullet$}}
\put(4,2){\makebox(1,1){$\circ$}}
\put(7,1){\makebox(1,1){$\circ$}}
\put(7,2){\makebox(1,1){$\circ$}}
\put(5,1){\makebox(2,2){$(0)$}}
\put(1,0){\makebox(1,1){$l_{min}$}}
\put(4,0){\makebox(1,1){$m$}}
\put(7,0){\makebox(1,1){$l_{min}^{\dag}$}}
\end{picture}.
\end{center}
Here, the number of $(\pm)$-slots in region $(0)$ is $t$.
There are no $\emptyset$-slots in this region.
Also, there are no $(\times)$-slots in this region.
Otherwise, it would contradict the minimality of $l_{min}$ in 
$\left\{  l\in \mathscr{L}^{(x,y+1)} \relmiddle| \overline{l^{\dag}}\prec\Bar{m}\prec \Bar{l}\right\}$.
Let us assume that the number of $(+)$-slots and that of $(-)$-slots in region $(0)$ are 
$\alpha$ and $\beta$, respectively.
Then we have
\begin{equation} \label{eq:psi_xy_2}
l_{min}^{\dag}=m+(\alpha+\beta+t)+1.
\end{equation}
When the operation (A) for $l_{min}\rightarrow l_{min}^{\dag}$ is finished, 
the $(y+1)$-st column of the updated tableau has the left configuration in the figure below.

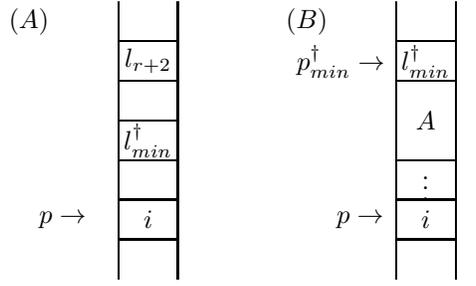
\setlength{\unitlength}{15pt}
\begin{center}
\begin{picture}(12,7)
\put(3,0){\line(0,1){7}}
\put(4.5,0){\line(0,1){7}}
\put(3,1){\line(1,0){1.5}}
\put(3,2){\line(1,0){1.5}}
\put(3,3){\line(1,0){1.5}}
\put(3,4){\line(1,0){1.5}}
\put(3,5){\line(1,0){1.5}}
\put(3,6){\line(1,0){1.5}}
\put(3,1){\makebox(1.5,1){$i$}}
\put(3,3){\makebox(1.5,1){$l_{min}^{\dag}$}}
\put(3,5){\makebox(1.5,1){$l_{r+2}$}}
\put(1,1){\makebox(3,1)[l]{$p \rightarrow$}}
\put(0,6){\makebox(1.5,1){$(A)$}}
\put(10,0){\line(0,1){7}}
\put(11.5,0){\line(0,1){7}}
\put(10,1){\line(1,0){1.5}}
\put(10,2){\line(1,0){1.5}}
\put(10,3){\line(1,0){1.5}}
\put(10,5){\line(1,0){1.5}}
\put(10,6){\line(1,0){1.5}}
\put(10,1){\makebox(1.5,1){$i$}}
\put(10,2){\makebox(1.5,1){$\vdots$}}
\put(10,3){\makebox(1.5,2){$A$}}
\put(10,5){\makebox(1.5,1){$l_{min}^{\dag}$}}
\put(8.5,1){\makebox(3,1)[l]{$p \rightarrow$}}
\put(7.5,5){\makebox(2,1)[l]{$p_{min}^{\dag} \rightarrow$}}
\put(7,6){\makebox(1.5,1){$(B)$}}
\end{picture}.
\end{center}
In the operation (B), $s-1$ $\mathscr{L}^{(x,y+1)}$-letters, $l_{r+2},\ldots,l_{r+s}$ 
together with $t$ $\mathscr{L}^{(x,y+1)}$-letters are 
relocated just below the box containing $l_{min}^{\dag}$ 
so that the $(y+1)$-st column of the updated tableau has the right configuration, 
Hence, we have
\begin{equation} \label{eq:psi_xy_3}
p_{min}^{\dag}\leq p-s-t.
\end{equation}
Note that $p_{min}^{\dag}$ does not change under subsequent operations for 
$l_{r+2}\rightarrow l_{r+2}^{\dag},\ldots,l_{c}\rightarrow l_{c}^{\dag}$.
The $x$-th column of $T^{\prime}$ has the left configuration (A) in the figure below  
when the operation (A) for $l_{min}\rightarrow l_{min}^{\dag}$ is finished.
When the entry $\overline{l_{min}^{\dag}}$ appears above $\Bar{m}$, 
the position of the box containing $\Bar{m}$ is changed from $q$ to $q+1$.
Since there are $\beta + t$ boxes with $\overline{\mathscr{J}^{(x)}}$-letters between the box containing 
$\overline{l_{min}^{\dag}}$ and that containing $\Bar{m}$, 
the position of the box containing $\overline{l_{min}^{\dag}}$ is $q-\beta -t$.

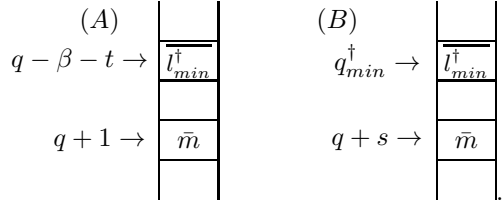
\setlength{\unitlength}{15pt}
\begin{center}
\begin{picture}(12.5,5)
\put(4,0){\line(0,1){5}}
\put(5.5,0){\line(0,1){5}}
\put(4,1){\line(1,0){1.5}}
\put(4,2){\line(1,0){1.5}}
\put(4,3){\line(1,0){1.5}}
\put(4,4){\line(1,0){1.5}}
\put(4,1){\makebox(1.5,1){$\Bar{m}$}}
\put(4,3){\makebox(1.5,1){\small $\overline{l_{min}^{\dag}}$}}
\put(1,1){\makebox(3,1){$q+1 \rightarrow$}}
\put(0,3){\makebox(4,1){$q-\beta-t \rightarrow$}}
\put(2,4){\makebox(1,1){$(A)$}}
\put(8,4){\makebox(1,1){$(B)$}}
\put(11,0){\line(0,1){5}}
\put(12.5,0){\line(0,1){5}}
\put(11,1){\line(1,0){1.5}}
\put(11,2){\line(1,0){1.5}}
\put(11,3){\line(1,0){1.5}}
\put(11,4){\line(1,0){1.5}}
\put(11,1){\makebox(1.5,1){$\Bar{m}$}}
\put(11,3){\makebox(1.5,1){\small $\overline{l_{min}^{\dag}}$}}
\put(8,1){\makebox(3,1){$q+s \rightarrow$}}
\put(8,3){\makebox(3,1){$q_{min}^{\dag} \rightarrow$}}
\end{picture}.
\end{center}
When the operation (B) for $l_{min}\rightarrow l_{min}^{\dag}$ is finished, 
the $x$-th column of the updated tableau has the right configuration (B) in the above figure. 
Since $s-1$ $\overline{\mathscr{L}^{(x,y+1)}}$-letters $\overline{l_{r+s}},\ldots,\overline{l_{r+2}}$ lying  
above the box containing $\Bar{m}$ before the operation (B) for $l_{min}\rightarrow l_{min}^{\dag}$ are relocated above  
$\overline{l_{min}^{\dag}}$, the position of $\Bar{m}$ is changed from $q+1$ to $q+1+(s-1)=q+s$.
Likewise, the position of the box containing $\overline{l_{min}^{\dag}}$ is changed from $q-\beta -t$ to
\begin{equation} \label{eq:psi_xy_4}
q_{min}^{\dag}=q-\beta-t+(s+t-1)=q-\beta+s-1,
\end{equation}  
which does not change under subsequent operations for  
$l_{r+2}\rightarrow l_{r+2}^{\dag},\ldots,l_{c}\rightarrow l_{c}^{\dag}$.
From Eqs.~\eqref{eq:psi_xy_1}, \eqref{eq:psi_xy_2}, and \eqref{eq:psi_xy_4}, we have
\begin{equation} \label{eq:psi_xy_5}
(q+s)-\Delta q -p_{min}^{\dag}+m  =q_{min}^{\dag}-\Delta q -p_{min}^{\dag}+l_{min}^{\dag}-\alpha-t
\leq n+s-\alpha.
\end{equation}
Combining Eqs.~\eqref{eq:psi_xy_3} and \eqref{eq:psi_xy_5}, we have
$(q+s)-\Delta q -p+m \leq n-\alpha -t \leq n$.
Here the position of $m$ in the $y$-th column of $\Tilde{T}$ is $p$ and 
that of $\Bar{m}$ in the $x$-th column is $q+s$.
Therefore, $C^{(x,y)}$ is KN-coadmissible.

\textbf{Case (a-2).}
Let us assume that $i\notin \mathscr{L}^{(x,y+1)}$.
The proof for the case when $i\in \mathscr{L}^{(x,y+1)}$is similar.
The filling diagram of the column $C^{(x,y+1)}$ has the following configuration 
before the operation for  $l_{min}\rightarrow l_{min}^{\dag}$.

\setlength{\unitlength}{15pt}
\begin{center}
\begin{picture}(12,3)
\put(1,1){\line(0,1){2}}
\put(2,1){\line(0,1){2}}
\put(4,1){\line(0,1){2}}
\put(5,1){\line(0,1){2}}
\put(7,1){\line(0,1){2}}
\put(8,1){\line(0,1){2}}
\put(10,1){\line(0,1){2}}
\put(11,1){\line(0,1){2}}
\put(0,1){\line(1,0){12}}
\put(0,2){\line(1,0){2}}
\put(4,2){\line(1,0){1}}
\put(7,2){\line(1,0){1}}
\put(10,2){\line(1,0){2}}
\put(0,3){\line(1,0){12}}
\put(1,1){\makebox(1,1){$\bullet$}}
\put(1,2){\makebox(1,1){$\bullet$}}
\put(4,1){\makebox(1,1){$\bullet$}}
\put(4,2){\makebox(1,1){$\circ$}}
\put(7,1){\makebox(1,1){$\circ$}}
\put(7,2){\makebox(1,1){$\bullet$}}
\put(10,1){\makebox(1,1){$\circ$}}
\put(10,2){\makebox(1,1){$\circ$}}
\put(5,1){\makebox(2,2){$(1)$}}
\put(8,1){\makebox(2,2){$(2)$}}
\put(1,0){\makebox(1,1){$l_{min}$}}
\put(4,0){\makebox(1,1){$m$}}
\put(7,0){\makebox(1,1){$i$}}
\put(10,0){\makebox(1,1){$l_{min}^{\dag}$}}
\end{picture}.
\end{center}
The total number of $(\pm)$-slots in regions $(1)$ and $(2)$ is $t$.
Let us assume that the number of $(\pm)$-slots in region $(1)$ is $t_{1}$.
There are no $\emptyset$-slots in both regions.
Also, there are no $(\times)$-slots in both regions as in \textbf{(a-1)}.
Let us assume that the number of $(+)$-slots and that of $(-)$-slots in region $(j)$ are 
$\alpha_{j}$ and $\beta_{j}$, respectively ($j=1,2$).
Then
\begin{equation} \label{eq:psi_xy_6}
l_{min}^{\dag}=m+\sum_{i=1}^{2}(\alpha_{i}+\beta_{i})+t+2.
\end{equation}
The updated tableau has the following configuration 
when the operation (A) for $l_{min}\rightarrow l_{min}^{\dag}$ is finished.

\setlength{\unitlength}{15pt}
\begin{center}
\begin{picture}(14.5,6)
\put(5.5,0){\line(0,1){5}}
\put(7,0){\line(0,1){5}}
\put(5.5,1){\line(1,0){1.5}}
\put(5.5,2){\line(1,0){1.5}}
\put(5.5,3){\line(1,0){1.5}}
\put(5.5,4){\line(1,0){1.5}}
\put(4.5,5){\line(1,0){3.5}}
\put(5.5,1){\makebox(1.5,1){$\Bar{m}$}}
\put(5.5,3){\makebox(1.5,1){$\overline{l_{min}^{\dag}}$}}
\put(5.5,5){\makebox(1.5,1){$x$}}
\put(3,1){\makebox(2.5,1)[l]{$q+1 \rightarrow$}}
\put(0,3){\makebox(5.5,1)[l]{$q-\sum_{i=1}^{2}\beta_{i}-t \rightarrow$}}
\put(10,0){\line(0,1){5}}
\put(11.5,0){\line(0,1){5}}
\put(10,1){\line(1,0){1.5}}
\put(10,2){\line(1,0){1.5}}
\put(10,3){\line(1,0){1.5}}
\put(10,4){\line(1,0){1.5}}
\put(9,5){\line(1,0){3.5}}
\put(10,1){\makebox(1.5,1){$l_{min}^{\dag}$}}
\put(10,3){\makebox(1.5,1){$i$}}
\put(10,5){\makebox(1.5,1){$y+1$}}
\put(12,3){\makebox(2.5,1)[l]{$\leftarrow p-1$}}
\end{picture}.
\end{center}
When the operation (B) for $l_{min}\rightarrow l_{min}^{\dag}$ is finished,
the updated tableau has the following configuration.

\setlength{\unitlength}{15pt}
\begin{center}
\begin{picture}(13.5,6)
\put(3,0){\line(0,1){5}}
\put(4.5,0){\line(0,1){5}}
\put(3,1){\line(1,0){1.5}}
\put(3,2){\line(1,0){1.5}}
\put(3,3){\line(1,0){1.5}}
\put(3,4){\line(1,0){1.5}}
\put(2,5){\line(1,0){3.5}}
\put(3,1){\makebox(1.5,1){$\Bar{m}$}}
\put(3,3){\makebox(1.5,1){$\overline{l_{min}^{\dag}}$}}
\put(3,5){\makebox(1.5,1){$x$}}
\put(0,1){\makebox(2.5,1)[r]{$q+s \rightarrow$}}
\put(0,3){\makebox(2.5,1)[r]{$q_{min}^{\dag} \rightarrow$}}
\put(7.5,0){\line(0,1){5}}
\put(9,0){\line(0,1){5}}
\put(7.5,1){\line(1,0){1.5}}
\put(7.5,2){\line(1,0){1.5}}
\put(7.5,3){\line(1,0){1.5}}
\put(7.5,4){\line(1,0){1.5}}
\put(6.5,5){\line(1,0){3.5}}
\put(7.5,1){\makebox(1.5,1){$l_{min}^{\dag}$}}
\put(7.5,3){\makebox(1.5,1){$i$}}
\put(7.5,5){\makebox(1.5,1){$y+1$}}
\put(9.5,1){\makebox(4,1)[l]{$\leftarrow p_{min}^{\dag}$}}
\put(9.5,3){\makebox(4,1)[l]{$\leftarrow p-s-t_{1}$}}
\end{picture},
\end{center}
where
\begin{equation} \label{eq:psi_xy_7}
q_{min}^{\dag}=q-\sum_{i=1}^{2}\beta_{i}-t+(s+t-1)=(q+s)-\sum_{i=1}^{2}\beta_{i}-1.
\end{equation}
Since $\alpha_{2}$ $\mathscr{I}^{(y+1)}$-letters exist between the box containing $i$ and 
that containing $l_{min}^{\dag}$,
\begin{equation} \label{eq:psi_xy_8}
p_{min}^{\dag}-\alpha_{2}-1=p-s-t_{1}.
\end{equation}
Note that $p_{min}^{\dag}$ and $q_{min}^{\dag}$ do not change under subsequent operations for 
$l_{r+2}\rightarrow l_{r+2}^{\ast},\ldots,l_{c}\rightarrow l_{c}^{\dag}$.
From Eqs.~\eqref{eq:psi_xy_1}, \eqref{eq:psi_xy_6}, \eqref{eq:psi_xy_7}, and \eqref{eq:psi_xy_8}, we have
\begin{align*}
(q+s)-\Delta q -p+m  & =q_{min}^{\dag}-\Delta q-p_{min}^{\dag}+l_{min}^{\dag}-\alpha_{1}-s-t-t_{1}\\
& \leq n-\alpha_{1}-t_{1} \leq n.
\end{align*}
Here, the position of the box containing $m$ in the $y$-th column of $\Tilde{T}$ is $p$ and 
that of $\Bar{m}$ in the $x$-th column of $\Tilde{T}$ is $q+s$.
Therefore, $C^{(x,y)}$ is KN-coadmissible.

\textbf{Case (b).}
In this case, we can write 
$m=l_{i}^{\dag}\in \mathscr{L}^{(x,y+1)\dag}=\{l_{1}^{\dag},l_{2}^{\dag},\ldots,l_{c}^{\dag}\}$.
Let us set
$\{l_{p+1},\ldots,l_{p+r}\}:=\left\{l \in \mathscr{L}^{(x,y+1)} \relmiddle| l_{i} < l < l_{i}^{\dag}\right\}$ 
(if $r=0$, then this set is considered to be empty).
We adopt the first kind algorithm for $\psi^{(x,y+1)}$ here.
When the operation for $l_{i} \rightarrow l_{i}^{\dag}=m$ is finished, 
the updated tableau has the left configuration in the figure below, 
where $A$ is the block consisting of $s$ boxes ($s\geq 1$).

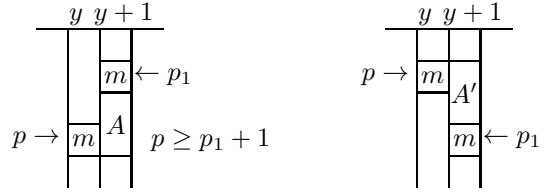
\setlength{\unitlength}{12pt}
\begin{center}
\begin{picture}(17,6)
\put(2,0){\line(0,1){5}}
\put(3,0){\line(0,1){5}}
\put(4,0){\line(0,1){5}}
\put(2,1){\line(1,0){2}}
\put(2,2){\line(1,0){1}}
\put(3,3){\line(1,0){1}}
\put(3,4){\line(1,0){1}}
\put(1,5){\line(1,0){4}}
\put(1.75,5){\makebox(1,1){${\mathstrut y}$}}
\put(3,5){\makebox(1.5,1){${\mathstrut y+1}$}}
\put(2,1){\makebox(1,1){$m$}}
\put(3,1){\makebox(1,2){$A$}}
\put(3,3){\makebox(1,1){$m$}}
\put(0,1){\makebox(2,1){$p \rightarrow$}}
\put(4,3){\makebox(2,1){$\leftarrow p_{1}$}}
\put(5,1){\makebox(3,1){$p\geq p_{1}+1$}}
\put(13,0){\line(0,1){5}}
\put(14,0){\line(0,1){5}}
\put(15,0){\line(0,1){5}}
\put(14,1){\line(1,0){1}}
\put(13,3){\line(1,0){1}}
\put(13,4){\line(1,0){2}}
\put(14,2){\line(1,0){1}}
\put(12,5){\line(1,0){4}}
\put(12.75,5){\makebox(1,1){${\mathstrut y}$}}
\put(14,5){\makebox(1.5,1){${\mathstrut y+1}$}}
\put(14,2){\makebox(1,2){$A^{\prime}$}}
\put(13,3){\makebox(1,1){$m$}}
\put(14,1){\makebox(1,1){$m$}}
\put(11,3){\makebox(2,1){$p \rightarrow$}}
\put(15,1){\makebox(2,1){$\leftarrow p_{1}$}}
\end{picture}.
\end{center}
The right configuration is not allowed, where $A^{\prime}$ is the block consisting of $s^{\prime}$ boxes ($s^{\prime} \geq 0$).
This can be seen as follows. 
Suppose that the entry in the $p_{1}$-th box in the $(y+1)$-st column is $j$ in the initial tableau $T^{\prime}$.
When the operations for $l_{i-1}\rightarrow l_{i-1}^{\dag}$ is finished, 
$l_{1}^{\dag},\ldots,l_{i-1}^{\dag}$ lie above the box containing $j$ in the $(y+1)$-st column 
so that the $p_{1}$-th box in the $(y+1)$-st column still has the entry $j$.
The operation for $l_{i}\rightarrow l_{i}^{\dag}$ replaces the entry $j$ with $l_{i}^{\dag}=m$.
This implies that $j<l_{i}^{\dag}=m$ by Lemma~\ref{lem:col_sst2}, 
which contradicts the semistandardness of the $\mathscr{C}_{n}^{(+)}$-letters part of $T^{\prime}$ 
so that the right configuration cannot happen.
When a sequence of operations for $l_{p+1}\rightarrow l_{p+1}^{\dag},\ldots,l_{p+r}\rightarrow l_{p+r}^{\dag}$ is finished, 
the position of $m=l_{i}^{\dag}$ in the $(y+1)$-st column becomes to be $p^{\prime}=p_{1}-r$, 
which does not change under subsequent operations.
Since $p\geq p_{1}+1$, we have $p^{\prime}\leq p-r-1$.
On the other hand, by Lemma~\ref{lem:KN_upper}, we have 
$(q-\Delta q -p^{\prime})+m\leq n+r+1$, 
where $q$ is the position of $\Bar{m}=\overline{l_{i}^{\dag}}$ in the $x$-th column.
Combining these, we have that $(q-\Delta q-p)+m \leq n$, i.e., $C^{(x,y)}$ is KN-coadmissible.
\end{proof}

The following four lemmas may be proven in the similar manner of the proof of Lemma~\ref{lem:kink1} (Lemma~\ref{lem:kink2}), 
Lemma~\ref{lem:st1}, and Lemma~\ref{lem:st2}.

\begin{lem} \label{lem:kink5}
Let us set 
\[
\Tilde{T}:=\psi^{(x,y+1)}\circ\cdots\circ\psi^{(x,n_{c})}\circ
\overline{\Psi^{(x-1)}}(T) \quad (1\leq x\leq y \leq n_{c}-1).
\]
Here, we assume that $\Tilde{T}\neq \emptyset$ and that in the updating process of the tableau from $T$ to $\Tilde{T}$ the semistandardness of the $\mathscr{C}_{n}^{(+)}$-letters part of the tableau is preserved 
($\overline{\Psi^{(0)}}(T)=T$).

\begin{itemize}
\item[(1).]
Suppose that $\Tilde{T}$ has the following configuration, 
where the left (resp. right) part is the $\mathscr{C}_{n}^{(-)}$ (resp. $\mathscr{C}_{n}^{(+)}$)-letters one $(p\leq q< r\leq s)$.

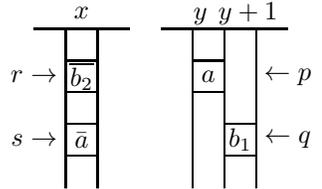
\setlength{\unitlength}{12pt}
\begin{center}
\begin{picture}(10,6)
\put(2,0){\line(0,1){5}}
\put(3,0){\line(0,1){5}}
\put(6,0){\line(0,1){5}}
\put(7,0){\line(0,1){5}}
\put(8,0){\line(0,1){5}}
\put(2,1){\line(1,0){1}}
\put(2,2){\line(1,0){1}}
\put(2,3){\line(1,0){1}}
\put(2,4){\line(1,0){1}}
\put(6,3){\line(1,0){1}}
\put(6,4){\line(1,0){1}}
\put(7,1){\line(1,0){1}}
\put(7,2){\line(1,0){1}}
\put(1,5){\line(1,0){3}}
\put(5,5){\line(1,0){4}}
\put(2,5){\makebox(1,1){$x$}}
\put(5.75,5){\makebox(1,1){${\mathstrut y}$}}
\put(7,5){\makebox(1.5,1){${\mathstrut y+1}$}}
\put(2,1){\makebox(1,1){$\Bar{a}$}}
\put(2,3){\makebox(1,1){$\overline{b_{2}}$}}
\put(6,3){\makebox(1,1){$a$}}
\put(7,1){\makebox(1,1){$b_{1}$}}
\put(0,1){\makebox(2,1){$s \rightarrow$}}
\put(0,3){\makebox(2,1){$r \rightarrow$}}
\put(8,1){\makebox(2,1){$\leftarrow q$}}
\put(8,3){\makebox(2,1){$\leftarrow p$}}
\end{picture}.
\end{center}
Then we have \[
(q-p)+(s-r)<\max(b_{1},b_{2})-a.
\] 

\item[(2).]
Let $\mathscr{J}^{(x)}$ be the set of $\mathscr{J}$-letters in the $x$-th column 
and $\mathscr{I}^{(y)}$ be the set of $\mathscr{I}$-letters in the $y$-th column and set 
$\mathscr{L}^{(x,y)}:=\mathscr{J}^{(x)}\cap \mathscr{I}^{(y)}$.
If $\sharp\left\{  l\in \mathscr{L}^{(x,y)} \relmiddle| l<a<l^{\dag} \right\}  =\delta$ in $\psi^{(x,y)}(\Tilde{T})$, 
then we have \[
(q-p)+(s-r)<\max(b_{1},b_{2})-a-\delta
\]
in the above configuration in $\Tilde{T}$.
\end{itemize}
\end{lem}

\begin{lem} \label{lem:kink6}
Let us set 
\begin{align*}
\Tilde{T}  :=
\left(  \psi^{(x-1,y)}\circ\psi^{(x,y+1)}\right)  &\circ\cdots\circ
\left(\psi^{(x-1,n_{c}-1)}\circ\psi^{(x,n_{c})}\right)  \circ\psi^{(x-1,n_{c})} \\
&\circ(\Psi^{(x-1)})^{-1}\circ\overline{\Psi^{(x-1)}}(T)
\quad (2\leq x\leq y+1 \leq n_{c}).
\end{align*}
Here, we assume that $\Tilde{T}\neq \emptyset$ and that in the updating process of the tableau from $T$ to $\Tilde{T}$ the semistandardness of the $\mathscr{C}_{n}^{(-)}$-letters part of the tableau is preserved.

\begin{itemize}
\item[(1).]
Suppose that the tableau $\Tilde{T}$ has the following configuration, 
where the left (resp. right) part is the $\mathscr{C}_{n}^{(-)}$ (resp. $\mathscr{C}_{n}^{(+)}$)-letters one $(p\leq q < r\leq s)$.

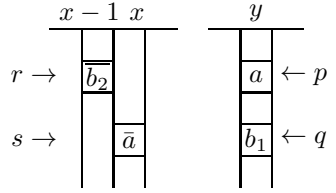
\setlength{\unitlength}{12pt}
\begin{center}
\begin{picture}(10,6)
\put(2,0){\line(0,1){5}}
\put(3,0){\line(0,1){5}}
\put(4,0){\line(0,1){5}}
\put(3,1){\line(1,0){1}}
\put(3,2){\line(1,0){1}}
\put(2,3){\line(1,0){1}}
\put(2,4){\line(1,0){1}}
\put(1,5){\line(1,0){4}}
\put(7,0){\line(0,1){5}}
\put(8,0){\line(0,1){5}}
\put(7,1){\line(1,0){1}}
\put(7,2){\line(1,0){1}}
\put(7,3){\line(1,0){1}}
\put(7,4){\line(1,0){1}}
\put(6,5){\line(1,0){3}}
\put(2,3){\makebox(1,1){$\overline{b_{2}}$}}
\put(3,1){\makebox(1,1){$\Bar{a}$}}
\put(0,1){\makebox(1,1){$s \rightarrow$}}
\put(0,3){\makebox(1,1){$r \rightarrow$}}
\put(1.5,5){\makebox(1.5,1){${\mathstrut x-1}$}}
\put(3.25,5){\makebox(1,1){${\mathstrut x}$}}
\put(7,1){\makebox(1,1){$b_{1}$}}
\put(7,3){\makebox(1,1){$a$}}
\put(8,1){\makebox(2,1){$\leftarrow q$}}
\put(8,3){\makebox(2,1){$\leftarrow p$}}
\put(7,5){\makebox(1,1){$y$}}
\end{picture}.
\end{center}
Then we have \[
(q-p)+(s-r)<\max(b_{1},b_{2})-a.
\]

\item[(2).]
Let $\mathscr{J}^{(x)}$ be the set of $\mathscr{J}$-letters in the $x$-th column 
and $\mathscr{I}^{(y)}$ be the $\mathscr{I}$-letters part of the $y$-th column and set 
$\mathscr{L}^{(x,y)}:=\mathscr{J}^{(x)}\cap \mathscr{I}^{(y)}$.
If $\sharp\left\{  l\in \mathscr{L}^{(x,y)} \relmiddle| l<a<l^{\dag} \right\}  =\delta$ in $\psi^{(x,y)}(\Tilde{T})$, 
then we have 
\[(q-p)+(s-r)<\max(b_{1},b_{2})-a-\delta
\]
in the above configuration in $\Tilde{T}$.
\end{itemize}
\end{lem}

\begin{lem}
Let us set
\[
\Tilde{T}:=\psi^{(x,y+1)}\circ \cdots \circ \psi^{(x,n_{c})}\circ \overline{\Psi^{(x-1)}}(T) \quad (1\leq x \leq y \leq n_{c}-1).
\]
Here, we assume that $\Tilde{T}\neq \emptyset$ and that in the updating process of the tableau from $T$ to $\Tilde{T}$ the semistandardness of the $\mathscr{C}_{n}^{(+)}$-letters part of the tableau is preserved 
($\overline{\Psi^{(0)}}(T)=T$).
Then the $\mathscr{C}_{n}^{(+)}$-letters part of $\psi^{(x,y)}(\Tilde{T})$ is semistandard.
\end{lem}

\begin{lem}
Let us set 
\begin{align*}
\Tilde{T}  :=
\left(  \psi^{(x-1,y)}\circ\psi^{(x,y+1)}\right)  &\circ\cdots\circ
\left(\psi^{(x-1,n_{c}-1)}\circ\psi^{(x,n_{c})}\right)  \circ\psi^{(x-1,n_{c})} \\
&\circ(\Psi^{(x-1)})^{-1}\circ\overline{\Psi^{(x-1)}}(T)
\quad (2\leq x\leq y\leq n_{c}-1).
\end{align*}
Here, we assume that $\Tilde{T}\neq \emptyset$ and that in the updating process of the tableau from $T$ to $\Tilde{T}$ the semistandardness of the $\mathscr{C}_{n}^{(-)}$-letters part of the tableau is preserved.
Then the $\mathscr{C}_{n}^{(-)}$-letters part of $\psi^{(x,y)}(\Tilde{T})$ and that of 
$\left(\psi^{(x-1,y-1)}\circ\psi^{(x,y)}\right)(\Tilde{T})$ are semistandard.
\end{lem}

The following two lemmas (Lemma~\ref{lem:kink7} and Lemma~\ref{lem:kink8}), 
which may be proven in the similar manner of the proof of Lemma~\ref{lem:kink1} and Lemma~\ref{lem:kink2}, 
guarantee that $\Psi(T)$ satisfies the KN-admissible condition on adjacent columns (Definition~\ref{df:KN_C} (C2)). 

\begin{lem} \label{lem:kink7}
Let us set 
\[
\Tilde{T}=\psi^{(x,x)}\circ\cdots\circ\psi^{(x,n_{c})}\circ
\overline{\Psi^{(x-1)}}(T) \quad (2\leq x\leq n_{c}).
\] 
Here, we assume that $\Tilde{T}\neq \emptyset$ and that in the updating process of the tableau from $T$ to $\Tilde{T}$ the semistandardness of the $\mathscr{C}_{n}^{(+)}$-letters part of the tableau is preserved.
Suppose that $\Tilde{T}$ has the following configuration, 
where the left (resp. right) part is the $\mathscr{C}_{n}^{(-)}$ (resp. $\mathscr{C}_{n}^{(+)}$)-letters one ($p\leq q<r \leq s$).

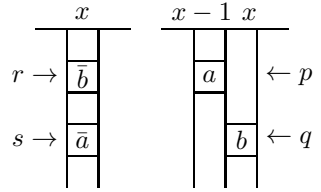
\setlength{\unitlength}{12pt}
\begin{center}
\begin{picture}(10,6)
\put(2,0){\line(0,1){5}}
\put(3,0){\line(0,1){5}}
\put(6,0){\line(0,1){5}}
\put(7,0){\line(0,1){5}}
\put(8,0){\line(0,1){5}}
\put(2,1){\line(1,0){1}}
\put(2,2){\line(1,0){1}}
\put(2,3){\line(1,0){1}}
\put(2,4){\line(1,0){1}}
\put(6,3){\line(1,0){1}}
\put(6,4){\line(1,0){1}}
\put(7,1){\line(1,0){1}}
\put(7,2){\line(1,0){1}}
\put(1,5){\line(1,0){3}}
\put(5,5){\line(1,0){4}}
\put(2,5){\makebox(1,1){$x$}}
\put(5.5,5){\makebox(1.5,1){${\mathstrut x-1}$}}
\put(7.25,5){\makebox(1,1){${\mathstrut x}$}}
\put(2,1){\makebox(1,1){$\Bar{a}$}}
\put(2,3){\makebox(1,1){$\Bar{b}$}}
\put(6,3){\makebox(1,1){$a$}}
\put(7,1){\makebox(1,1){$b$}}
\put(0,1){\makebox(2,1){$s \rightarrow$}}
\put(0,3){\makebox(2,1){$r \rightarrow$}}
\put(8,1){\makebox(2,1){$\leftarrow q$}}
\put(8,3){\makebox(2,1){$\leftarrow p$}}
\end{picture}.
\end{center}
Then we have $(q-p)+(s-r)<b-a$.
\end{lem}


\begin{lem} \label{lem:kink8}
Let us set 
\begin{align*}
\Tilde{T}  :=
\left(  \psi^{(x-1,x-1)}\circ\psi^{(x,x)}\right)  &\circ\cdots\circ
\left(\psi^{(x-1,n_{c}-1)}\circ\psi^{(x,n_{c})}\right)  \circ\psi^{(x-1,n_{c})} \\
&\circ(\Psi^{(x-1)})^{-1}\circ\overline{\Psi^{(x-1)}}(T)
\quad (2\leq x\leq n_{c}),
\end{align*}
Here, we assume that $\Tilde{T}\neq \emptyset$ and that in the updating process of the tableau from $T$ to $\Tilde{T}$ the semistandardness of the $\mathscr{C}_{n}^{(-)}$-letters part of the tableau is preserved.
Suppose that the tableau $\Tilde{T}$ has the following configuration, 
where the left (resp. right) part is the $\mathscr{C}_{n}^{(-)}$ (resp. $\mathscr{C}_{n}^{(+)}$)-letters one $(p\leq q < r\leq s)$.

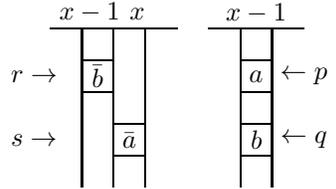
\setlength{\unitlength}{12pt}
\begin{center}
\begin{picture}(10,6)
\put(2,0){\line(0,1){5}}
\put(3,0){\line(0,1){5}}
\put(4,0){\line(0,1){5}}
\put(3,1){\line(1,0){1}}
\put(3,2){\line(1,0){1}}
\put(2,3){\line(1,0){1}}
\put(2,4){\line(1,0){1}}
\put(1,5){\line(1,0){4}}
\put(7,0){\line(0,1){5}}
\put(8,0){\line(0,1){5}}
\put(7,1){\line(1,0){1}}
\put(7,2){\line(1,0){1}}
\put(7,3){\line(1,0){1}}
\put(7,4){\line(1,0){1}}
\put(6,5){\line(1,0){3}}
\put(2,3){\makebox(1,1){$\Bar{b}$}}
\put(3,1){\makebox(1,1){$\Bar{a}$}}
\put(0,1){\makebox(1,1){$s \rightarrow$}}
\put(0,3){\makebox(1,1){$r \rightarrow$}}
\put(1.5,5){\makebox(1.5,1){${\mathstrut x-1}$}}
\put(3.25,5){\makebox(1,1){${\mathstrut x}$}}
\put(7,1){\makebox(1,1){$b$}}
\put(7,3){\makebox(1,1){$a$}}
\put(8,1){\makebox(2,1){$\leftarrow q$}}
\put(8,3){\makebox(2,1){$\leftarrow p$}}
\put(6.5,5){\makebox(2,1){$x-1$}}
\end{picture}.
\end{center}
Then we have $(q-p)+(s-r)<b-a$.
\end{lem}

\begin{lem} \label{lem:Psi}
Suppose that $\overline{\Psi^{(x-1)}}$ is well-defined on $T$ and $\overline{\Psi^{(x-1)}}(T)$ is semistandard 
($2\leq x \leq n_{c}$).
Then $\overline{\Psi^{(x)}}$ is well-defined on $T$ and $\overline{\Psi^{(x)}}(T)$ is semistandard.
Therefore, $\Psi$ is well-defined on $T$ by induction and 
$\Psi(T) \in C_{n}\text{-}\mathrm{SST}_{\mathrm{KN}}(\nu)$. 
\end{lem}

\begin{proof}
The proof is analogous to that of Lemma~\ref{lem:Phi}.
Each column of $\Psi(T)$ satisfies the KN-admissible condition (Definition~\ref{df:KN_C} (C1)) 
because $\psi^{(x,x)}=(\phi^{(x,x)})^{-1}$ is well-defined ($1\leq x \leq n_{c}$) and 
any pair of adjacent columns in $\Psi(T)$ satisfies the KN-admissible condition 
(Definition~\ref{df:KN_C} (C2)) by Lemma~\ref{lem:kink7} and Lemma~\ref{lem:kink8}.
Since $\Psi$ is well-defined on $T$ so that it preserves the shape of $T$, we have that 
 $\Psi(T) \in C_{n}\text{-}\mathrm{SST}_{\mathrm{KN}}(\nu)$.
\end{proof}

\section{Proof of Proposition~\ref{prp:main12}} \label{sec:main12}

In this section, we provide the proof of Proposition~\ref{prp:main12}.
Let $T\in C_{n}\text{-}\mathrm{SST}(\nu)$ be the tableau described in Proposition~\ref{prp:main12} 
with $n_{c}$ columns.
We use the same notation as in Section~\ref{sec:main11} to keep track of the updating stage in $\Psi(T)$.
Initially, the set of $\mathscr{I}$ (resp. $\mathscr{J}$)-letters in the $x$-th column of $T$ is written as 
$\mathscr{I}^{(x,i)}$ (resp. $\mathscr{J}^{(x,i)}$) with $i=0$ ($1\leq x \leq n_{c}$).
Whenever the map $\psi^{(x,y)}$ is applied to the updated tableau whose entries are updated by preceding application of the map of the form $\psi ^{(\centerdot,\centerdot)}$, 
the counter $i$ in $\mathscr{J}^{(x,i)}$ is increased by one; $\mathscr{J}^{(x,i)}\rightarrow \mathscr{J}^{(x,i+1)}$ and the counter $j$ in $\mathscr{I}^{(y,j)}$ 
is increased by one; $\mathscr{I}^{(y,j)}\rightarrow \mathscr{I}^{(y,j+1)}$.
At the end, i.e., in $\Psi(T)$ the set of $\mathscr{I}$ (resp. $\mathscr{J}$)-letters in the $x$-th column is 
$\mathscr{I}^{(x,x)}$ (resp. $\mathscr{J}^{(x,n_{c}-x+1)}$) ($1\leq x\leq n_{c}$).
The letters in $\mathscr{I}^{(x,i)}$ (resp. $\mathscr{J}^{(x,i)}$) are called 
$\mathscr{I}^{(x,i)}$ (resp. $\mathscr{J}^{(x,i)}$)-letters and those in 
$\overline{\mathscr{I}^{(x,i)}}$ (resp. $\overline{\mathscr{J}^{(x,i)}}$) are called 
$\overline{\mathscr{I}^{(x,i)}}$ (resp. $\overline{\mathscr{J}^{(x,i)}}$)-letters as in Section~\ref{sec:main11}.

For all $(T_{1},T_{2})\in \mathbf{B}_{n}^{(+)}(\xi)_{\zeta}^{\lambda}\times \mathbf{B}_{n}^{(-)}(\eta)_{\zeta}^{\mu}$, 
$\zeta\left[  \underrightarrow{\mathrm{FE}(T_{1})}\right]  =\lambda$ by definition, i.e., 
\begin{equation} \label{eq:zeta2lambda}
\zeta\left[  \underrightarrow{\mathscr{I}^{(n_{c},0)}},\ldots,
\underrightarrow{\mathscr{I}^{(1,0)}}\right]  =\lambda.
\end{equation}
Furthermore, $\mu\left[  \underrightarrow{\mathrm{FE}(T_{2})}\right]  =\zeta$ by definition 
and therefore  $\mu\left[  \underrightarrow{\mathrm{FE}(T^{(-)})}\right]  =\zeta$ by Proposition~\ref{prp:skew}, i.e., 
\begin{equation} \label{eq:mu2zeta}
\mu\left[  \underrightarrow{\overline{\mathscr{J}^{(n_{c},0)}}},\ldots,
\underrightarrow{\overline{\mathscr{J}^{(1,0)}}}\right]  =\zeta.
\end{equation}
Under these conditions and the notation introduced above, we have the following lemma.

\begin{lem} \label{lem:smooth3}
(1).
Let us define
\[
\mu^{(i)\prime}:=
\begin{cases}
\mu\left[  \overline{\mathscr{J}^{(n_{c},0)}},\ldots,\overline{\mathscr{J}^{(i+1,0)}}\right]
 & (0\leq i\leq n_{c}-1), \\
\mu & (i=n_{c}).
\end{cases}
\]
Then we have $\mathscr{I}^{(n_{c},i)}$ is smooth on $\mu^{(i)\prime}$ 
($1 \leq i \leq n_{c}$).

(2).
Let us define
\[
\Tilde{\mu}^{(x)}:=
\begin{cases}
\mu\left[ \mathscr{I}^{(n_{c},n_{c})},\overline{\mathscr{J}^{(n_{c},1)}},\ldots,
\mathscr{I}^{(x+1,x+1)},\overline{\mathscr{J}^{(x+1,n_{c}-x)}}\right]
& (1\leq x\leq n_{c}-1),\\
\mu & (x=n_{c})
\end{cases}
\]
and $\mu^{(x)}:=\Tilde{\mu}^{(x)}\left[ \mathscr{I}^{(x,x)}\right]$.
For $2\leq x \leq n_{c}$, let us assume that $\mu^{(x)}$ and   
\[
\mu^{(x,i)}:=\mu^{(x)}\left[  \overline{\mathscr{J}^{(x,n_{c}-x+1)}},\ldots,
\overline{\mathscr{J}^{(i,n_{c}-x+1)}}\right] \quad (1\leq i\leq x) 
\]
are all Young diagrams.
Suppose that $\mathscr{I}^{(x-1,i-1)}$ is smooth on $\mu^{(x,i)}$.
Then we have that $\mathscr{I}^{(x-1,i)}$ is smooth on $\mu^{(x,i+1)}$ 
($1 \leq i \leq x-1$).

(3).
$\mu\left[ \underrightarrow{\mathscr{I}^{(n_{c},n_{c})}\vphantom{\overline{\mathscr{J}^{(n_{c},1)}}}},
\underrightarrow{\overline{\mathscr{J}^{(n_{c},1)}}},\ldots,
\underrightarrow{\mathscr{I}^{(1,1)}\vphantom{\overline{\mathscr{J}^{(n_{c},1)}}}},
\underrightarrow{\overline{\mathscr{J}^{(1,n_{c})}}}\right]=\lambda$.
\end{lem}

\begin{proof}
Let us begin by giving the proof of (2).
Note that the pair of $\mathscr{I}^{(x-1,i)}$ and $\overline{\mathscr{J}^{(i,n_{c}-x+2)}}$ is generated from 
the pair of $\mathscr{I}^{(x-1,i-1)}$ and $\overline{\mathscr{J}^{(i,n_{c}-x+1)}}$ by applying $\psi^{(i,x-1)}$ 
to the updated tableau whose entries are updated by preceding application of the map 
of the form $\psi^{(\centerdot,\centerdot)}$.
Let us call such sets $\mathscr{I}^{(x-1,i-1)}$ and $\overline{\mathscr{J}^{(i,n_{c}-x+1)}}$ to be updated are paired and write 
$\left\langle \mathscr{I}^{(x-1,i-1)},\overline{\mathscr{J}^{(i,n_{c}-x+1)}}\right\rangle_{\mathrm{pair}}$ 
($1\leq i\leq x-1; 2\leq x\leq n_{c}+1$) as in Section~\ref{sec:main11}.
Let us set 
$\mathscr{I}^{(x-1,i-1)}=\{i_{1},i_{2},\ldots,i_{a}\}$, 
$\mathscr{J}^{(i,n_{c}-x+1)}=\{j_{1},j_{2},\ldots,j_{b}\}$, 
$\mathscr{I}^{(x-1,i)}=\{i_{1}^{\prime},i_{2}^{\prime},\ldots,i_{a}^{\prime}\}$, 
$\mathscr{J}^{(i,n_{c}-x+2)}=\{j_{1}^{\prime},j_{2}^{\prime},\ldots,j_{b}^{\prime}\}$, 
$\mathscr{L}:=\mathscr{I}^{(x-1,i-1)}\cap \mathscr{J}^{(i,n_{c}-x+1)}=\{l_{1},l_{2},\ldots,l_{c}\}$, and 
$\mathscr{L}^{\dag}:=\mathscr{I}^{(x-1,i)}\cap \mathscr{J}^{(i,n_{c}-x+2)}=\{l_{1}^{\dag},l_{2}^{\dag},\ldots,l_{c}^{\dag}\}$.
Recall that these are ordered sets and are also considered as the sequences of letters.
We write $\Tilde{\mu}=\mu^{(x,i)}\left[ \mathscr{J}^{(i,n_{c}-x+1)}\right]=\mu^{(x,i+1)}$ for brevity.

\textbf{(I).}
Let us consider the following three cases separately:
\begin{description}
\item[(a)]
$i_{1}^{\prime}=l_{1}^{\dag}$.
\item[(b)]
$i_{1}^{\prime}\neq l_{1}^{\dag}$ and $i_{1}=l_{1}$.
\item[(c)]
$i_{1}^{\prime}\neq l_{1}^{\dag}$ and $i_{1}\neq l_{1}$.
\end{description}

\textbf{Case (a).}
In this case, $\mathscr{L}^{\dag}\neq \emptyset$ and $i_{1}=l_{1}$.
Indeed, if $l_{1}=i_{p}$ ($p>1$), then $i_{1} \notin \mathscr{L}$ 
because $i_{1}$ is smaller than $l_{1}$ that is the smallest letter in $\mathscr{L}$.
This implies $i_{1}^{\prime}=i_{1}$.
However, this also implies $l_{1}^{\dag}=i_{1}\in \mathscr{I}^{(x-1,i-1)}$ due to the assumption of \textbf{(a)}, 
which contradicts the fact that $l_{1}^{\dag}$ is not an $\mathscr{I}^{(x-1,i-1)}$-letter.
To proceed, let us divide this case further into the following two cases:

\begin{description}
\item[(a-1)]
All $\mathscr{I}^{(x-1,i-1)}$-letters $i_{1},i_{2},\ldots,i_{a}$ are also $\mathscr{J}^{(i,n_{c}-x+1)}$-letters.
\item[(a-2)]
There exist non-$\mathscr{J}^{(i,n_{c}-x+1)}$-letters in the sequence of $\mathscr{I}^{(x-1,i-1)}$-letters 
$i_{1},i_{2},\ldots,i_{a}$
(That is, there exist some letters belonging to $\mathscr{I}^{(x-1,i-1)}\backslash \mathscr{L}$ in 
$\left\{i_{1},i_{2},\ldots,i_{a}\right\}$).
\end{description}
In case \textbf{(a-1)}, we have $i_{1}^{\prime}=l_{1}^{\dag}$.
According to the Remark~\ref{rem:algorithm2}, 
we can write $l_{1}^{\dag}=j_{r}+1\; (\exists j_{r}\in \mathscr{J}^{(i,n_{c}-x+1)})$.
In case \textbf{(a-2)}, let us choose the smallest letter $i_{p}$ ($p>1$) from the set of 
$\mathscr{I}^{(x-1,i-1)}$-letters $i_{1},i_{2},\ldots,i_{a}$ 
such that $i_{p}$ is not a $\mathscr{J}^{(i,n_{c}-x+1)}$-letter (i.e., $i_{p} \in \mathscr{I}^{(x-1,i-1)}\backslash \mathscr{L}$).
Now consider the increasing (just by one) sequence of $\mathscr{C}_{n}^{(+)}$-letters
\begin{equation} \label{eq:seq2}
i_{1}+1,i_{1}+2,\ldots,i_{p}-1
\end{equation}
By the minimality of $i_{p}$, any letter belonging to $\mathscr{I}^{(x-1,i-1)}\backslash \mathscr{L}$ 
cannot appear in \eqref{eq:seq2}.
If all of the letters in \eqref{eq:seq2} are $\mathscr{J}^{(i,n_{c}-x+1)}$-letters, then $l_{1}^{\dag}>i_{p}$ so that 
$i_{1}^{\prime}=i_{p}$, which contradicts the assumption of \textbf{(a)}.
Consequently, there must exist some letters that are not $\mathscr{I}^{(x-1,i-1)}$-letters 
nor $\mathscr{J}^{(i,n_{c}-x+1)}$-letters in the sequence~\eqref{eq:seq2}.
Denote by $i_{1}+q\; (\exists q\geq1)$ the smallest letter among them.
Since $l_{1}=i_{1}$, we have $l_{1}^{\dag}=i_{1}+q$.
By the minimality of $i_{1}+q$, $i_{1}+q-1$ is a $\mathscr{J}^{(i,n_{c}-x+1)}$-letter 
(when $q=1$, $i_{1}=l_{1}$ is a $\mathscr{J}^{(i,n_{c}-x+1)}$-letter).
Hence, we can write $i_{1}+q-1=j_{r}\; (\exists j_{r}\in \mathscr{J}^{(i,n_{c}-x+1)})$ so that 
$i_{1}^{\prime}=l_{1}^{\dag}=j_{r}+1$.
Since $i_{1}^{\prime}=l_{1}^{\dag}\in \mathscr{I}^{(x-1,i)}$ is the letter generated by $\psi^{(i,x-1)}$, 
$i_{1}^{\prime}\notin \mathscr{J}^{(i,n_{c}-x+1)}$.
By the assumption of (2) of Lemma~\ref{lem:smooth3}, 
$\Tilde{\mu}\left[ \overline{\mathscr{J}^{(i,n_{c}-x+1)}}\right]  =\mu^{(x,i)}$ is a Young diagram so that 
\[
\Tilde{\mu}\left[  \overline{\mathscr{J}^{(i,n_{c}-x+1)}}\right]  _{i_{1}^{\prime}-1}
\geq\Tilde{\mu}\left[  \overline{\mathscr{J}^{(i,n_{c}-x+1)}}\right]  _{i_{1}^{\prime}}.
\]
The left-hand side of this inequality is $\Tilde{\mu}_{i_{1}^{\prime}-1}-1$ 
because $i_{1}^{\prime}-1=j_{r}\in \mathscr{J}^{(i,n-x+1)}$, 
while the right-hand side is $\Tilde{\mu}_{i_{1}^{\prime}}$ because 
$i_{1}^{\prime}\notin \mathscr{J}^{(i,n_{c}-x+1)}$ and thereby $\Tilde{\mu}_{i_{1}^{\prime}-1}>\Tilde{\mu}_{i_{1}^{\prime}}$.

\textbf{Case (b).}
Firstly, let us show that we can write 
$i_{1}^{\prime}=i_{p}\; (\exists i_{p}\in \mathscr{I}^{(x-1,i-1)}\backslash \mathscr{L})$. 
Since $i_{1}^{\prime} \notin \mathscr{L}^{\dag}$ 
we can write $i_{1}^{\prime}=i_{p}\; (\exists i_{p}\in \mathscr{I}^{(x-1,i-1)}\backslash \mathscr{L})$, 
because $i_{1}^{\prime}\in(\mathscr{I}^{(x-1,i-1)}\backslash \mathscr{L})\sqcup \mathscr{L}^{\dag}$.
In this case, $p\geq 2$.
Otherwise $i_{1}^{\prime}=i_{1}=l_{1}$, which is a contradiction.
To proceed, let us consider the following three cases separately:
\begin{description}
\item[(b-1)]
$p=2$ and $i_{1}^{\prime}=i_{p=2}=i_{1}+1=l_{1}+1$.
\item[(b-2)]
$p\geq2$ and $i_{p}>i_{1}+1$.
\item[(b-3)]
$p>2$ and $i_{p}=i_{1}+1$.
\end{description}
In case \textbf{(b-1)}, we can write $l_{1}=j_{r}\; (\exists j_{r}\in \mathscr{J}^{(i,n_{c}-x+1)})$ and 
$i_{1}^{\prime}=j_{r}+1$.
In case \textbf{(b-2)}, there must exist a sequence of $\mathscr{J}^{(i,n_{c}-x+1)}$-letters $j_{q},\ldots,j_{q+m}$ 
such that $i_{1}<j_{q+k}<i_{p}\; (k=0,1,\ldots,m)$ and
\begin{align*}
j_{q}-i_{1}  & =1,\\
j_{q+k}-j_{q+k-1}  & =1\quad (k=1,\ldots,m),\\
i_{p}-j_{q+m}  & =1.
\end{align*}
Otherwise, $l_{1}^{\dag}$ cannot be larger than 
$i_{1}^{\prime}=i_{p}(\in \mathscr{I}^{(x-1,i-1)}\backslash \mathscr{L})$.
The existence of such a sequence implies $i_{1}^{\prime}=i_{p}=j_{q+m}+1$.
Case \textbf{(b-3)} must be excluded because the inequalities 
$i_{1}<i_{2}<\cdots<i_{p}$ do not hold.
In both cases \textbf{(b-1)} and \textbf{(b-2)}, 
we can write $i_{1}^{\prime}=j_{r}+1\; (\exists j_{r}\in \mathscr{J}^{(i,n_{c}-x+1)})$.
Now since $\Tilde{\mu}\left[ \overline{\mathscr{J}^{(i,n_{c}-x+1)}}\right]$ is a Young diagram,
\[
\Tilde{\mu}\left[  \overline{\mathscr{J}^{(i,n_{c}-x+1)}}\right]  _{i_{1}^{\prime}-1}
\geq\Tilde{\mu}\left[  \overline{\mathscr{J}^{(i,n_{c}-x+1)}}\right]  _{i_{1}^{\prime}}.
\]
The left-hand side of this inequality is $\Tilde{\mu}_{i_{1}^{\prime}-1}-1$ 
because $i_{1}^{\prime}-1=j_{r}\in \mathscr{J}^{(i,n_{c}-x+1)}$, 
while the right-hand side is $\Tilde{\mu}_{i_{p}}=\Tilde{\mu}_{i_{1}^{\prime}}$ 
because $i_{p}\in \mathscr{I}^{(x-1,i-1)}\backslash \mathscr{L}$, 
i.e., $i_{p}\notin \mathscr{J}^{(i,n_{c}-x+1)}$ so that 
$\Tilde{\mu}_{i_{1}^{\prime}-1}>\Tilde{\mu}_{i_{1}^{\prime}}$.

\textbf{Case (c).}
Let us show that $i_{1}^{\prime}=i_{1}$.
If $\mathscr{L}=\emptyset$, this is obvious.
If $\mathscr{L}\neq \emptyset$, the $\mathscr{I}^{(x-1,i-1)}$-letter $i_{1}$ is smaller than $l_{1}$ 
that is the smallest letter in $\mathscr{L}$ 
so that the $\mathscr{I}^{(x-1,i-1)}$-letter $i_{1}$ is not a $\mathscr{J}^{(i,n_{c}-x+1)}$-letter, 
which implies $i_{1}^{\prime}=i_{1}$.
By the assumption of (2) of Lemma~\ref{lem:smooth3}, 
$\mathscr{I}^{(x-1,i-1)}$ is smooth on 
$\mu^{(x,i)}=\Tilde{\mu}\left[ \overline{\mathscr{J}^{(i,n_{c}-x+1)}}\right]$ 
so that
\[
\Tilde{\mu}\left[  \overline{\mathscr{J}^{(i,n_{c}-x+1)}}\right]  _{i_{1}-1}
>\Tilde{\mu}\left[  \overline{\mathscr{J}^{(i,n_{c}-x+1)}}\right]  _{i_{1}}.
\]
The left-hand side of this inequality is $\Tilde{\mu}_{i_{1}-1}-\delta$ ($\delta\in\{0,1\}$), 
while the right-hand side is $\Tilde{\mu}_{i_{1}}$ because $i_{1}\notin \mathscr{J}^{(i,n_{c}-x)}$.
Therefore, we have 
$\Tilde{\mu}_{i_{1}^{\prime}-1}=\Tilde{\mu}_{i_{1}-1}>\Tilde{\mu}_{i_{1}}=\Tilde{\mu}_{i_{1}^{\prime}}$.
In \textbf{(I)}, we have verified that 
$\Tilde{\mu}_{i_{1}^{\prime}-1}>\Tilde{\mu}_{i_{1}^{\prime}}$, that is,  
$\Tilde{\mu}\lbrack i_{1}^{\prime}]$ is a Young diagram for all possible cases.

\textbf{(II).}
Let us suppose that  
$\Tilde{\mu}^{\dag(k-1)}=\Tilde{\mu}\lbrack i_{1}^{\prime},\ldots,i_{k-1}^{\prime}]$
is a Young diagram ($k-1 \geq 1$).
We prove that $\Tilde{\mu}^{\dag(k-1)}[i_{k}^{\prime}]$ 
is also a Young diagram.
Note that $\mathscr{J}^{(i,n_{c}-x+1)}$ is smooth on $\Tilde{\mu}$ 
by Lemma~\ref{lem:smooth0}.
Let us consider the following three cases separately:
\begin{description}
\item[(a)]
$i_{k}^{\prime}\in \mathscr{I}^{(x-1,i)}\backslash \mathscr{L}^{\dag}
(=\mathscr{I}^{(x-1,i-1)}\backslash \mathscr{L})$.
\item[(b)]
$i_{k-1}^{\prime}\in \mathscr{I}^{(x-1,i)}\backslash \mathscr{L}^{\dag}$ and 
$i_{k}^{\prime}\in \mathscr{L}^{\dag}$.
\item[(c)]
$i_{k-1}^{\prime}\in \mathscr{L}^{\dag}$ and $i_{k}^{\prime}\in \mathscr{L}^{\dag}$.
\end{description}

\textbf{Case (a).}
We can write $i_{k}^{\prime}=i_{p}\; (\exists i_{p}\in I\backslash L)$ and
\[
\Tilde{\mu}_{i_{k}^{\prime}}^{\dag(k-1)}=
\Tilde{\mu}\lbrack i_{1}^{\prime},\ldots,i_{k-1}^{\prime}]_{i_{k}^{\prime}}=\Tilde{\mu}_{i_{p}}.
\]
In order to compute 
$\Tilde{\mu}_{i_{p}-1}^{\dag(k-1)}=\Tilde{\mu}\lbrack i_{1}^{\prime},\ldots,i_{k-1}^{\prime}]_{i_{p}-1}$, 
we divide this case further into the following three cases:
\begin{description}
\item[(a-1)]
$i_{p}-1\in \mathscr{I}^{(x-1,i)}$.
\item[(a-2)]
$i_{p}-1\notin \mathscr{I}^{(x-1,i)}$ and $i_{p}-1\in \mathscr{L}$.
\item[(a-3)]
$i_{p}-1\notin \mathscr{I}^{(x-1,i)}$ and $i_{p}-1\notin \mathscr{L}$.
\end{description}
In case \textbf{(a-1)}, we have $i_{k-1}^{\prime}=i_{p}-1$ because $i_{k}^{\prime}=i_{p}$.
Then
\[
\Tilde{\mu}_{i_{k}^{\prime}-1}^{\dag(k-1)}=
\Tilde{\mu}\lbrack i_{1}^{\prime},\ldots,i_{k-1}^{\prime}=i_{p}-1]_{i_{p}-1}=\Tilde{\mu}_{i_{p}-1}+1
\]
so that we obtain
\[
\Tilde{\mu}_{i_{k}^{\prime}-1}^{\dag(k-1)}=\Tilde{\mu}_{i_{p}-1}+1>
\Tilde{\mu}_{i_{p}}=\mu_{i_{k}^{\prime}}^{\dag(k-1)}.
\]
In both cases \textbf{(a-2)} and \textbf{(a-3)}, we have
$\Tilde{\mu}_{i_{k}^{\prime}-1}^{\dag(k-1)}=\Tilde{\mu}_{i_{p}-1}$ because $i_{p}-1\notin \mathscr{I}^{(x-1,i)}$.
By the assumption of (2) of Lemma~\ref{lem:smooth3}, 
$\mathscr{I}^{(x-1,i-1)}$ is smooth on 
$\mu^{(x,i)}=\Tilde{\mu}\left[ \overline{\mathscr{J}^{(i,n_{c}-x+1)}}\right]$, 
\begin{equation} \label{eq:smooth3_1}
\Tilde{\mu}\left[  \overline{\mathscr{J}^{(i,n_{c}-x+1)}},i_{1},\ldots,i_{p-1}\right]_{i_{p}-1}>
\Tilde{\mu}\left[  \overline{\mathscr{J}^{(i,n_{c}-x+1)}},i_{1},\ldots,i_{p-1}\right]  _{i_{p}}.
\end{equation}
In case \textbf{(a-2)}, the left-hand side of Eq.\eqref{eq:smooth3_1} is $\Tilde{\mu}_{i_{p}-1}$ 
because $i_{p}-1\in \mathscr{L}$ ($i_{p}-1$ appears once in $\{i_{1},\ldots,i_{p-1}\}$ and $\overline{i_{p}-1}$ 
appears once in $\overline{\mathscr{J}^{(i,n_{c}-x+1)}}$).
The right-hand side is $\Tilde{\mu}_{i_{p}}$ because $i_{p}\in \mathscr{I}^{(x-1,i-1)}\backslash \mathscr{L}$, i.e., 
$i_{p}\notin \mathscr{J}^{(i,n_{c}-x+1)}$.
Therefore, we have $\Tilde{\mu}_{i_{p}-1}>\Tilde{\mu}_{i_{p}}$ so that 
$\Tilde{\mu}_{i_{k}^{\prime}-1}^{\dag(k-1)}=\Tilde{\mu}_{i_{p}-1}>
\Tilde{\mu}_{i_{p}}=\Tilde{\mu}_{i_{k}^{\prime}}^{\dag(k-1)}$.
In case \textbf{(a-3)}, $i_{p}-1\notin \mathscr{I}^{(x-1,i-1)}$ because 
$i_{p}-1\notin(\mathscr{I}^{(x-1,i-1)}\backslash \mathscr{L})\sqcup \mathscr{L}^{\dag}$ 
and $i_{p}-1\notin \mathscr{L}$. 
The left-hand side of Eq.~\eqref{eq:smooth3_1} is $\Tilde{\mu}_{i_{p}-1}-\delta$ ($\delta\in\{0,1\}$), 
while the right-hand side is $\Tilde{\mu}_{i_{p}-1}$ because 
$i_{p}\in \mathscr{I}^{(x-1,i-1)}\backslash \mathscr{L}$, i.e., $i_{p}\notin \mathscr{J}^{(i,n_{c}-x+1)}$.
Therefore, $\Tilde{\mu}_{i_{p}-1}-\delta>\Tilde{\mu}_{i_{p}}$ so that 
$\Tilde{\mu}_{i_{k}^{\prime}-1}^{\dag(k-1)}=\Tilde{\mu}_{i_{p}-1}>
\Tilde{\mu}_{i_{p}}=\Tilde{\mu}_{i_{k}^{\prime}}^{\dag(k-1)}$.

\textbf{Case (b).}
In this case, $\mathscr{L}^{\dag}\neq \emptyset$ and we can write 
$i_{k}^{\prime}=l_{r}^{\dag}$ ($\exists l_{r}^{\dag}\in \mathscr{L}^{\dag}$).
We divide this case further into the following two cases according to Remark~\ref{rem:algorithm2}:
\begin{description}
\item[(b-1)]
$l_{r}^{\dag}=i_{p}+1\quad (\exists i_{p}\in \mathscr{I}^{(x-1,i-1)}\backslash \mathscr{L})$.
\item[(b-2)]
$l_{r}^{\dag}=j_{q}+1\quad (\exists j_{q}\in \mathscr{J}^{(i,n_{c}-x+1)})$.
\end{description}
The situation that $l_{r}^{\dag}=l_{r-1}^{\dag}+1\;(r\neq1)$ cannot happen.
Indeed, if $l_{r}^{\dag}=l_{r-1}^{\dag}+1\; (r\neq1)$, then $i_{k}^{\prime}=l_{r-1}^{\dag}+1$.
Since $l_{r-1}^{\dag} \in \mathscr{I}^{(x-1,i)}$, this implies $i_{k-1}^{\prime}=l_{r-1}^{\dag}$, 
which contradicts the assumption of \textbf{(b)}.
In case \textbf{(b-1)}, $i_{k-1}^{\prime}=i_{p}$ because 
$i_{p}\in \mathscr{I}^{(x-1,i-1)}$ and $i_{k}^{\prime}=i_{p}+1$.
Then
\[
\Tilde{\mu}_{i_{k}^{\prime}}^{\dag(k-1)}=\Tilde{\mu}\lbrack i_{1}^{\prime},\ldots,i_{k-1}^{\prime}=
i_{p}]_{i_{p}+1}=\Tilde{\mu}_{i_{p}+1},
\]
and
\[
\Tilde{\mu}_{i_{k}^{\prime}-1}^{\dag(k-1)}=\Tilde{\mu}\lbrack i_{1}^{\prime},\ldots,i_{k-1}^{\prime}=
i_{p}]_{i_{p}}=\Tilde{\mu}_{i_{p}}+1.
\]
From these two equations, we have
$\Tilde{\mu}_{i_{k}^{\prime}-1}^{\dag(k-1)}>\Tilde{\mu}_{i_{k}^{\prime}}^{\dag(k-1)}$.
In case \textbf{(b-2)},
\[
\Tilde{\mu}_{i_{k}^{\prime}}^{\dag(k-1)}=
\Tilde{\mu}\lbrack i_{1}^{\prime},\ldots,i_{k-1}^{\prime}]_{i_{k}^{\prime}}=
\Tilde{\mu}_{i_{k}^{\prime}}=\Tilde{\mu}_{j_{q}+1}.
\]
On the other hand,
\[
\Tilde{\mu}_{i_{k}^{\prime}-1}^{\dag(k-1)}=
\Tilde{\mu}\lbrack i_{1}^{\prime},\ldots,i_{k-1}^{\prime}]_{i_{k}^{\prime}-1}=
\Tilde{\mu}_{i_{k}^{\prime}-1}=\Tilde{\mu}_{j_{q}},
\]
where we have used the fact that $i_{k}^{\prime}-1>i_{k-1}^{\prime}$.
This is shown as follows.
If $i_{k}^{\prime}-1=i_{k-1}^{\prime}$, then we have 
$j_{q}=i_{k}^{\prime}-1=i_{k-1}^{\prime}$.
This implies that 
$j_{q}$ is an $\mathscr{I}^{(x-1,i-1)}$-letter but is not a $\mathscr{J}^{(i,n_{c}-x+1)}$-letter 
due to the assumption of \textbf{(b)}, 
which is a contradiction. 
Now since $\overline{\mathscr{J}^{(i,n_{c}-x+1)}}$ is smooth on $\Tilde{\mu}$, 
we have
\[
\Tilde{\mu}\lbrack\overline{j_{b}},\ldots,\overline{j_{q+1}}]_{j_{q}}>
\Tilde{\mu}\lbrack\overline{j_{b}},\ldots,\overline{j_{q+1}}]_{j_{q}+1}.
\]
By noting that $j_{q}+1=l_{r}^{\dag}\notin\{j_{b},\ldots,j_{q+1}\}$, 
the right-hand side of the above inequality is found to be $\Tilde{\mu}_{j_{q}+1}$, 
while the left-hand side is clearly $\Tilde{\mu}_{j_{q}}$.
Hence, $\Tilde{\mu}_{i_{k}^{\prime}-1}^{\dag(k-1)}>\Tilde{\mu}_{i_{k}^{\prime}}^{\dag(k-1)}$.

\textbf{Case (c).}
In this case, $\mathscr{L}^{\dag}\neq\emptyset$ and we can write 
$i_{k-1}^{\prime}=l_{r-1}^{\dag}$ and 
$i_{k}^{\prime}=l_{r}^{\dag}\; (\exists r\in\{2,\ldots,c\})$.
According to the algorithm described above Eq.~\eqref{eq:map2} or Remark~\ref{rem:algorithm2}, 
let us consider the following three cases separately:
\begin{description}
\item[(c-1)]
$l_{r}^{\dag}=i_{p}+1\quad (\exists i_{p}\in \mathscr{I}^{(x-1,i-1)}\backslash \mathscr{L})$.
\item[(c-2)]
$l_{r}^{\dag}=j_{q}+1\quad (\exists j_{q}\in \mathscr{J}^{(i,n_{c}-x+1)})$.
\item[(c-3)]
$l_{r}^{\dag}=l_{r-1}^{\dag}+1\quad (r\neq1)$.
\end{description}
In case \textbf{(c-1)}, $i_{k-1}^{\prime}=i_{p}$ 
because $i_{p}\in \mathscr{I}^{(x-1,i)}$ and $i_{k}^{\prime}=i_{p}+1$.
Then 
$l_{r-1}^{\dag}=i_{p}\in \mathscr{I}^{(x-1,i-1)}\backslash \mathscr{L} =
\mathscr{I}^{(x-1,i)}\backslash \mathscr{L}^{\dag}$, 
which derives a contradiction, and thereby case \textbf{(c-1)} must be excluded.
In case \textbf{(c-2)},
\[
\Tilde{\mu}_{i_{k}^{\prime}}^{\dag(k-1)}=
\Tilde{\mu}\lbrack i_{1}^{\prime},\ldots,i_{k-1}^{\prime}]_{i_{k}^{\prime}}=
\Tilde{\mu}_{i_{k}^{\prime}}=\Tilde{\mu}_{j_{q}+1}.
\]
On the other hand,
\[
\Tilde{\mu}_{i_{k}^{\prime}-1}^{\dag(k-1)}=
\Tilde{\mu}\lbrack i_{1}^{\prime},\ldots,i_{k-1}^{\prime}]_{i_{k}^{\prime}-1}=
\Tilde{\mu}_{i_{k}^{\prime}-1}=\Tilde{\mu}_{j_{q}},
\]
where we have used the fact that $i_{k-1}^{\prime}<i_{k}^{\prime}-1$.
This is shown as follows.
If $i_{k-1}^{\prime}=i_{k}^{\prime}-1$, then, 
$l_{r-1}^{\dag}=i_{k-1}^{\prime}=i_{k}^{\prime}-1=j_{q}$,  
which contradicts the fact that 
$l_{r-1}^{\dag}$ is not a $\mathscr{J}^{(i,n_{c}-x+1)}$-letter.
Now since $\overline{\mathscr{J}^{(i,n_{c}-x+1)}}$ is smooth on  $\Tilde{\mu}$, 
we have
\[
\Tilde{\mu}\lbrack\overline{j_{b}},\ldots,\overline{j_{q+1}}]_{j_{q}}>
\Tilde{\mu}\lbrack\overline{j_{b}},\ldots,\overline{j_{q+1}}]_{j_{q}+1}.
\]
By noting $j_{q}+1=l_{r}^{\dag}\notin\{j_{b},\ldots,j_{q+1}\}$, 
the right-hand side of the above inequality is seen to be $\Tilde{\mu}_{j_{q}+1}$, 
while the left-hand side is clearly $\Tilde{\mu}_{j_{q}}$.
Hence, we have $\Tilde{\mu}_{i_{k}^{\prime}-1}^{\dag(k-1)}>
\Tilde{\mu}_{i_{k}^{\prime}}^{\dag(k-1)}$.
In case \textbf{(c-3)}, since $i_{k}^{\prime}-1=i_{k-1}^{\prime}(=l_{r-1}^{\dag})$,
\[
\Tilde{\mu}_{i_{k}^{\prime}-1}^{\dag(k-1)}=
\Tilde{\mu}\lbrack i_{1}^{\prime},\ldots,i_{k-1}^{\prime}=l_{r-1}^{\dag}]_{l_{r-1}^{\dag}}=
\Tilde{\mu}_{i_{k}^{\prime}-1}+1,
\]
while
\[
\Tilde{\mu}_{i_{k}^{\prime}}^{\dag(k-1)}=
\Tilde{\mu}\lbrack i_{1}^{\prime},\ldots,i_{k-1}^{\prime}]_{i_{k}^{\prime}}=\Tilde{\mu}_{i_{k}^{\prime}}.
\]
Hence, we have 
$\Tilde{\mu}_{i_{k}^{\prime}-1}^{\dag(k-1)}>\Tilde{\mu}_{i_{k}^{\prime}}^{\dag(k-1)}$.
In \textbf{(II)}, we have verified that 
$\Tilde{\mu}_{i_{k}^{\prime}-1}^{\dag(k-1)}>\Tilde{\mu}_{i_{k}^{\prime}}^{\dag(k-1)}$, that is, 
$\Tilde{\mu}^{\dag(k-1)}[i_{k}^{\prime}]$ is a Young diagram for all possible cases.
From \textbf{(I)} and \textbf{(II)} and by induction, we have completed the proof of (2) of Lemma~\ref{lem:smooth3}.

The proof of (1) is as follows.
We proceed by induction on $i$.
Since $\mu^{(0)\prime}=\mu\left[\overline{\mathscr{J}^{(n_{c},0)}},\ldots,
\overline{\mathscr{J}^{(1,0)}}\right]  =\zeta$ and 
$\mathscr{I}^{(n_{c},0)}$ is smooth on $\zeta$ 
by Eqs.~\eqref{eq:zeta2lambda} and \eqref{eq:mu2zeta}, 
we have that $\mathscr{I}^{(n_{c},0)}$ is smooth on $\mu^{(0)\prime}$.
For $i=0,\ldots,n_{c}-1$, suppose that 
$\mathscr{I}^{(n_{c},i)}$ is smooth on $\mu^{(i)\prime}$.
This is satisfied for $i=0$.
Then we have that $\mathscr{I}^{(n_{c},i+1)}$ is smooth on $\mu^{(i+1)\prime}$ 
by the same argument as in (2).

The proof of (3) is as follows.
We proceed by induction on $x$ and $i$.

\textbf{(I).}
We have
$\mu\left[ \overline{\mathscr{J}^{(n_{c},0)}},\ldots,\overline{\mathscr{J}^{(1,0)}}\right]
\left[\underrightarrow{\mathscr{I}^{(n_{c},0)}}\right],\ldots,
\mu\left[  \overline{\mathscr{J}^{(n_{c},0)}}\right]  \left[\underrightarrow{\mathscr{I}^{(n_{c},n_{c}-1)}}\right]$, 
and $\mu\left[ \underrightarrow{\mathscr{I}^{(n_{c},n_{c})}}\right]=\mu^{(n_{c})}$ 
are all Young diagrams by the claim of (1).
For $1 \leq i \leq n_{c}$,  
\[
\left\langle \mathscr{I}^{(n_{c},i-1)},\overline{\mathscr{J}^{(i,0)}}\right\rangle_{\mathrm{pair}},
\left\langle\mathscr{I}^{(n_{c},i)},\overline{\mathscr{J}^{(i+1,0)}}\right\rangle_{\mathrm{pair}},\ldots,  
\left\langle \mathscr{I}^{(n_{c},n_{c}-1)},\overline{\mathscr{J}^{(n_{c},0)}}\right\rangle_{\mathrm{pair}},
\]
so that we have 
\begin{align*}
\mu\left[  \overline{\mathscr{J}^{(n_{c},0)}},\ldots,\overline{\mathscr{J}^{(i,0)}}\right]
\left[  \mathscr{I}^{(n_{c},i-1)}\right]
&=\mu\left[  \overline{\mathscr{J}^{(n_{c},1)}},\ldots,\overline{\mathscr{J}^{(i,1)}}\right]
\left[ \mathscr{I}^{(n_{c},n_{c})}\right] \\
&=\mu^{(n_{c})}\left[ \overline{\mathscr{J}^{(n_{c},1)}},\ldots,\overline{\mathscr{J}^{(i,1)}}\right].
\end{align*}
Hence,
\[
\mu^{(n_{c},i)}=\mu^{(n_{c})}\left[ \overline{\mathscr{J}^{(n_{c},1)}},\ldots,
\overline{\mathscr{J}^{(i,1)}}\right]
\quad (1\leq i\leq n_{c}) 
\]
are all Young diagrams and the smoothness of 
$\overline{\mathscr{J}^{(n_{c},1)}},\ldots,\overline{\mathscr{J}^{(i,1)}}$ on $\mu^{(n_{c})}$ follows from 
Lemma~\ref{lem:smooth0}.

\textbf{(II).}
For $x=n_{c},\ldots,2$, let us assume that $\mu^{(x)}$ and 
\[
\mu^{(x,i)}=\mu^{(x)}\left[ \underrightarrow{\overline{\mathscr{J}^{(x,n_{c}-x+1)}}},\ldots,
\underrightarrow{\overline{\mathscr{J}^{(i,n_{c}-x+1)}}}\right] \quad (1\leq i\leq x)
\]
are all Young diagrams, where $\mu^{(x)}$ is defined in (2).
For $x=n_{c}$ this is satisfied by \textbf{(I)}.
\textbf{(i).}
For $i=1$,
\begin{align*}
\mu^{(x,1)}  =&\mu^{(x)}\left[  \overline{\mathscr{J}^{(x,n_{c}-x+1)}},\ldots,
\overline{\mathscr{J}^{(1,n_{c}-x+1)}}\right]  \\
=&\mu\left[  \mathscr{I}^{(n_{c},n_{c})},\ldots,\mathscr{I}^{(x,x)}\right]  
\left[\overline{\mathscr{J}^{(n_{c},1)}},\ldots,\overline{\mathscr{J}^{(x,n_{c}-x+1)}}\right] \\
&\quad \left[\overline{\mathscr{J}^{(x-1,n_{c}-x+1)}},\ldots,\overline{\mathscr{J}^{(1,n_{c}-x+1)}}\right].  \\
\end{align*}
The right-hand side of this equation is written as
\[
\mu\left[ \overline{\mathscr{J}^{(n_{c},0)}},\ldots,\overline{\mathscr{J}^{(1,0)}}\right]
\left[  \mathscr{I}^{(n_{c},0)},\ldots,\mathscr{I}^{(x,0)}\right]  
=\zeta\left[  \underrightarrow{\mathscr{I}^{(n_{c},0)}},\ldots,
\underrightarrow{\mathscr{I}^{(x,0)}}\right] \quad (\because \eqref{eq:mu2zeta})
\]
because
\begin{align*}
\left\langle \mathscr{I}^{(n_{c},0)},\overline{\mathscr{J}^{(1,0)}}\right\rangle_{\mathrm{pair}} ,&\ldots
,\left\langle \mathscr{I}^{(n_{c},n_{c}-1)},\overline{\mathscr{J}^{(n_{c},0)}}\right\rangle_{\mathrm{pair}}, \\
& \ldots \\
\left\langle \mathscr{I}^{(x,0)},\overline{\mathscr{J}^{(1,n_{c}-x)}}\right\rangle_{\mathrm{pair}} ,&\ldots,
\left\langle \mathscr{I}^{(x,x-1)},\overline{\mathscr{J}^{(x,n_{c}-x)}}\right\rangle_{\mathrm{pair}}.
\end{align*}
Thus, we have that $\mathscr{I}^{(x-1,0)}$ is smooth on $\mu^{(x,1)}$ 
by Eq.~\eqref{eq:zeta2lambda}.
\textbf{(ii).}
For $i=1,\ldots,x-1$, suppose that 
$\mathscr{I}^{(x-1,i-1)}$ is smooth on $\mu^{(x,i)}$ 
(for $i=1$ this is satisfied).
Then we have that 
$\mathscr{I}^{(x-1,i)}$ is smooth on $\mu^{(x,i+1)}$ 
by the same argument as in (2).
From \textbf{(i)} and \textbf{(ii)} and by induction on $i$, we have that 
\[
\mu^{(x,1)}\left[  \underrightarrow{\mathscr{I}^{(x-1,0)}}\right]  ,\ldots,
\mu^{(x,x)}\left[  \underrightarrow{\mathscr{I}^{(x-1,x-1)}}\right]
\] 
are all Young diagrams.
Here, 
\begin{align} \label{eq:muxi}
&\mu^{(x,i)}\left[ \mathscr{I}^{(x-1,i-1)}\right] \\  
 =&\mu\left[ \mathscr{I}^{(n_{c},n_{c})},\overline{\mathscr{J}^{(n_{c},1)}},\ldots,
\mathscr{I}^{(x,x)},\overline{\mathscr{J}^{(x,n_{c}-x+1)}}\right]  \notag \\
& \quad \left[  \overline{\mathscr{J}^{(x-1,n_{c}-x+1)}},\ldots,\overline{\mathscr{J}^{(i,n_{c}-x+1)}}\right]
\left[ \mathscr{I}^{(x-1,i-1)}\right] \quad (1\leq i \leq x-1). \notag
\end{align}
and
$\mu^{(x,x)}\left[ \mathscr{I}^{(x-1,x-1)}\right]  =\mu^{(x-1)}$ so that 
\[
\mu^{(x-1)}=\mu^{(x,x)}\left[  \underrightarrow{\mathscr{I}^{(x-1,x-1)}}\right]
=\mu^{(x)}\left[  \underrightarrow{\overline{\mathscr{J}^{(x,n_{c}-x+1)}}},
\underrightarrow{\mathscr{I}^{(x-1,x-1)}\vphantom{ \overline{\mathscr{J}^{(x,n_{c}-x+1)}} }}\right]
\]
by Lemma~\ref{lem:smooth0} and by the assumption of \textbf{(II)}.
Since
\begin{multline*}
\left\langle \mathscr{I}^{(x-1,i-1)},\overline{\mathscr{J}^{(i,n_{c}-x+1)}}\right\rangle_{\mathrm{pair}},
\left\langle \mathscr{I}^{(x-1,i)},\overline{\mathscr{J}^{(i+1,n_{c}-x+1)}}\right\rangle_{\mathrm{pair}}, 
\\
\ldots,\left\langle \mathscr{I}^{(x-1,x-2)},\overline{\mathscr{J}^{(x-1,n_{c}-x+1)}}\right\rangle_{\mathrm{pair}} 
\quad (1 \leq i \leq x-1),
\end{multline*} 
the right-hand side of Eq.~\eqref{eq:muxi} is written as 
\begin{align*}
& \mu\left[ \mathscr{I}^{(n_{c},n_{c})},\overline{\mathscr{J}^{(n_{c},1)}},\ldots,
\mathscr{I}^{(x,x)},\overline{\mathscr{J}^{(x,n_{c}-x+1)}}\right]  \\
&\quad \left[  \overline{\mathscr{J}^{(x-1,n_{c}-x+2)}},\ldots,\overline{\mathscr{J}^{(i,n_{c}-x+2)}}\right]
\left[ \mathscr{I}^{(x-1,x-1)}\right]  \\
=&\mu^{(x-1)}\left[ \overline{\mathscr{J}^{(x-1,n_{c}-x+2)}},\ldots,\overline{\mathscr{J}^{(i,n_{c}-x+2)}}\right].
\end{align*}
Hence, 
\[
\mu^{(x-1)}
=\mu^{(x)}\left[  \underrightarrow{\overline{\mathscr{J}^{(x,n_{c}-x+1)}}},
\underrightarrow{\mathscr{I}^{(x-1,x-1)}\vphantom{ \overline{\mathscr{J}^{(x,n_{c}-x+1)}} }}\right]
\]
and 
\[
\mu^{(x-1,i)}=\mu^{(x-1)}\left[ \underrightarrow{\overline{\mathscr{J}^{(x-1,n_{c}-x+2)}}},\ldots,
\underrightarrow{\overline{\mathscr{J}^{(i,n_{c}-x+2)}}}\right] \quad (1\leq i\leq x-1) 
\]
are all Young diagrams.
The smoothness of $\overline{\mathscr{J}^{(x,n_{c}-x+2)}},\ldots,\overline{\mathscr{J}^{(i,n_{c}-x+2)}}$ 
on $\mu^{(x-1)}$ follows from Lemma~\ref{lem:smooth0}.
From \textbf{(I)} and \textbf{(II)} and by induction on $x$, we have 
$\mu^{(1)}\left[ \underrightarrow{\overline{\mathscr{J}^{(1,n_{c})}}}\right]=
\mu\left[ \underrightarrow{\mathscr{I}^{(n_{c},n_{c})}\vphantom{\overline{\mathscr{J}^{(n_{c},1)}}}},
\underrightarrow{\overline{\mathscr{J}^{(n_{c},1)}}},\ldots,
\underrightarrow{\mathscr{I}^{(1,1)}\vphantom{\overline{\mathscr{J}^{(n_{c},1)}}}},
\underrightarrow{\overline{\mathscr{J}^{(1,n_{c})}}}\right]=\mu\left[ \mathrm{FE}\left(\Psi(T)\right)\right]$.
Since $\Psi$ is weight-preserving, 
\begin{align*}
\mu\left[ \mathrm{FE}\left(\Psi(T)\right)\right]=\mu\left[ \mathrm{FE}(T)\right]=&\mu\left[  \mathscr{I}^{(n_{c},0)},\overline{\mathscr{J}^{(n_{c},0)}},
\ldots,\mathscr{I}^{(1,0)},\overline{\mathscr{J}^{(1,0)}}\right]  \\
 =&\zeta\left[  \mathscr{I}^{(n_{c},0)},\ldots,\mathscr{I}^{(1,0)}\right]  =\lambda.
\end{align*}
The last line is due to Eqs.~\eqref{eq:zeta2lambda} and \eqref{eq:mu2zeta}.
This completes the proof.
\end{proof}

\begin{proof}[Proof of Proposition~\ref{prp:main12}]
Let $T$ be the tableau described in Proposition~\ref{prp:main12}.
Suppose that $T$ consists of $n_{c}$ columns.
By Lemma~\ref{lem:Psi}, we have $\Psi(T)\in C_{n}\text{-}\mathrm{SST}_{\mathrm{KN}}(\nu)$.
By Lemma~\ref{lem:smooth3}, we have
$\mu\left[ 
\underrightarrow{\mathscr{I}^{(n_{c},n_{c})}\vphantom{\overline{\mathscr{J}^{(n_{c},1)}}}},
\underrightarrow{\overline{\mathscr{J}^{(n_{c},1)}}},\ldots,
\underrightarrow{\mathscr{I}^{(1,1)}\vphantom{\overline{\mathscr{J}^{(n_{c},1)}}}}
\underrightarrow{\overline{\mathscr{J}^{(1,n_{c})}}}\right]  =\lambda$.
This completes the proof. 
\end{proof}

\section{Main Theorem II} \label{sec:main2}

In this section, we will show that LR crystals of $C_{n}$-type 
$\mathbf{B}_{n}^{\mathfrak{sp}_{2n}}(\nu)_{\mu}^{\lambda}$ defined by Eq.~\eqref{eq:LRcrystal2} are identical to LR crystals of type $B_{n}$ or $D_{n}$ $\mathbf{B}_{n}^{\mathfrak{g}}(\nu)_{\mu}^{\lambda}$ ($\mathfrak{g}=\mathfrak{so}_{2n+1}$ or $\mathfrak{so}_{2n}$) in the stable region, $l(\mu)+l(\nu)\leq n$ (Theorems~\ref{thm:main_B} and \ref{thm:main_D}).
Here $\mathfrak{so}_{N}=\mathfrak{so}(N,\mathbb{C})$ ($N=2n+1$ or $2n$) is the special orthogonal Lie algebra.
Consequently, Theorem~\ref{thm:main1} with $\mathfrak{sp}_{2n}$ being replaced by $\mathfrak{so}_{2n+1}$ or $\mathfrak{so}_{2n}$ holds and it provides the crystal interpretation of the branching rule (Eq.~\eqref{eq:KK2}).

\subsection{LR crystals of $B_{n}$-type} \label{sec:crystal_B}

The odd special orthogonal Lie algebra $\mathfrak{so}(2n+1,\mathbb{C})=\mathfrak{so}_{2n+1}$ is the classical Lie algebra of $B_{n}$-type.
Using the standard unit vectors $\epsilon_{i}\in\mathbb{Z}^{n}$ ($i=1,2,\ldots,n$), the simple roots are expresses as
\begin{align*}
\alpha_{i}  & =\epsilon_{i}-\epsilon_{i+1}\quad (i=1,2,\ldots,n-1),\\
\alpha_{n}  & =\epsilon_{n},
\end{align*}
and the fundamental weights as
\begin{align*}
\omega_{i} & =\epsilon_{1}+\epsilon_{2}+\cdots+\epsilon_{i}\quad (i=1,2,\ldots,n-1),\\
\omega_{n} & =\frac{1}{2}(\epsilon_{1}+\epsilon_{2}+\cdots+\epsilon_{n}).
\end{align*}

Let 
$\Tilde{\lambda}=a_{1}\omega_{1}+\cdots+a_{n}\omega_{n}$ 
($a_{i}\in\mathbb{Z}_{\geq0}$) be a dominant integral weight.
Then $\Tilde{\lambda}$ can be written as 
$\Tilde{\lambda}=\lambda_{1}\epsilon_{1}+\cdots+\lambda_{n}\epsilon_{n}$, where
\begin{align*}
\lambda_{1}  & =a_{1}+a_{2}+\cdots+a_{n-1}+\frac{1}{2}a_{n}, \\
\lambda_{2}  & =a_{2}+\cdots+a_{n-1}+\frac{1}{2}a_{n}, \\
& \vdots\\
\lambda_{n}  & =\frac{1}{2}a_{n}.
\end{align*}

Here, we do not need to consider the spin representation for the finite-dimensional irreducible $U_{q}(\mathfrak{so}_{2n+1})$-module $V_{q}^{\mathfrak{so}_{2n+1}}(\omega_{n})$ as explained later so that $\frac{1}{2}a_{n}\in\mathbb{Z}_{\geq0}$.
Hence we can associate a Young diagram $\lambda=(\lambda_{1},\ldots,\lambda_{n})$ to $\Tilde{\lambda}$ and simplify the original definition of $B_{n}$-tableaux~\cite{N}.
Throughout this section, $B_{n}$-tableaux are referred to as $B_{n}$-tableaux without spin columns associated with the spin representations. 

\begin{df}[\cite{HK,N}]
\begin{itemize}
\item[(1)]
Let $\lambda$ be a Young diagram with at most $n$ rows.
A $B_{n}$-tableau of shape $\lambda$ is a tableau obtained by filling the boxes in $\lambda$ with entries from the set
\[
\{1,2,\ldots,n,0,\overline{n},\ldots,\overline{1}\}
\]
equipped with the total order
\[
1\prec2\prec\cdots\prec n\prec0\prec\overline{n}\prec\cdots\prec\overline{1}.
\]

\item[(2)]
A $B_{n}$-tableau is said to be semistandard if
\begin{itemize}
\item[(a)]
the entries in each rows are weakly increasing, but zeros cannot be repeated;
\item[(b)]
the entries in each column are strictly increasing, but zeros can be repeated.
\end{itemize}
\end{itemize}
\end{df}

We denote by $B_{n}\text{-}\mathrm{SST}(\lambda)$ the set of all semistandard $B_{n}$-tableaux of shape $\lambda$.
For a tableau $T\in B_{n}\text{-}\mathrm{SST}(\lambda)$, we define its weight to be 
\[
\mathrm{wt}(T):=\sum_{i=1}^{n}(k_{i}-\overline{k_{i}})\epsilon_{i},
\]
where $k_{i}$ (resp. $\overline{k_{i}}$) is the number of $i$'s (resp. $\Bar{i}$'s) 
appearing in $T$.

\begin{df}[\cite{HK,N}] \label{df:KN_B}
A tableau $T\in B_{n}\text{-}\mathrm{SST}(\lambda)$ is said to be \emph{KN-admissible} when the following conditions are satisfied.

\begin{itemize}

\item[(B1)]
If $T$ has a column of the form
\setlength{\unitlength}{12pt}
\begin{center}
\begin{picture}(3,6)
\put(2,0){\line(0,1){6}}
\put(3,0){\line(0,1){6}}
\put(2,0){\line(1,0){1}}
\put(2,1){\line(1,0){1}}
\put(2,2){\line(1,0){1}}
\put(2,4){\line(1,0){1}}
\put(2,5){\line(1,0){1}}
\put(2,6){\line(1,0){1}}
\put(2,1){\makebox(1,1){$\Bar{i}$}}
\put(2,4){\makebox(1,1){$i$}}
\put(0,1){\makebox(2,1){$q\rightarrow$}}
\put(0,4){\makebox(2,1){$p\rightarrow$}}
\end{picture},
\end{center}
then we have
$(q-p)+i > N$,
where $N$ is the length of the column.

\item[(B2)]
If $T$ has a pair of adjacent columns having one of the following configurations with $p\leq q < r\leq s$ and $a\leq b<n$:
\setlength{\unitlength}{12pt}
\begin{center}
\begin{picture}(7,5)
\put(3,0){\line(0,1){5}}
\put(6,0){\line(0,1){5}}
\put(3,0){\makebox(1,1){$\Bar{a}$}}
\put(3,1){\makebox(1,1){$\Bar{b}$}}
\put(3,3){\makebox(1,1){$b$}}
\put(2,4){\makebox(1,1){$a$}}
\put(6,0){\makebox(1,1){$\Bar{a}$}}
\put(5,1){\makebox(1,1){$\Bar{b}$}}
\put(5,3){\makebox(1,1){$b$}}
\put(5,4){\makebox(1,1){$a$}}
\put(4,0){\makebox(1,1)[l]{$,$}}
\put(0,0){\makebox(2,1){$s\rightarrow$}}
\put(0,1){\makebox(2,1){$r\rightarrow$}}
\put(0,3){\makebox(2,1){$q\rightarrow$}}
\put(0,4){\makebox(2,1){$p\rightarrow$}}
\end{picture},
\end{center}
then we have $(q-p)+(s-r)<b-a$.

\item[(B3)]
If $T$ has a pair of adjacent columns having one of the following configurations with $p\leq q<r=q+1\leq s$ and $a<n$:
\setlength{\unitlength}{12pt}
\begin{center}
\begin{picture}(13,5)
\put(3,0){\line(0,1){5}}
\put(6,0){\line(0,1){5}}
\put(9,0){\line(0,1){5}}
\put(12,0){\line(0,1){5}}
\put(2,4){\makebox(1,1){$a$}}
\put(3,0){\makebox(1,1){$\Bar{a}$}}
\put(3,1.5){\makebox(1,1){$\Bar{n}$}}
\put(3,2.5){\makebox(1,1){$n$}}

\put(5,4){\makebox(1,1){$a$}}
\put(6,0){\makebox(1,1){$\Bar{a}$}}
\put(6,1.5){\makebox(1,1){$0$}}
\put(6,2.5){\makebox(1,1){$n$}}

\put(8,4){\makebox(1,1){$a$}}
\put(9,0){\makebox(1,1){$\Bar{a}$}}
\put(9,1.5){\makebox(1,1){$0$}}
\put(9,2.5){\makebox(1,1){$0$}}

\put(11,4){\makebox(1,1){$a$}}
\put(12,0){\makebox(1,1){$\Bar{a}$}}
\put(12,1.5){\makebox(1,1){$\Bar{n}$}}
\put(12,2.5){\makebox(1,1){$0$}}

\put(4,0){\makebox(1,1)[l]{$,$}}
\put(7,0){\makebox(1,1)[l]{$,$}}
\put(10,0){\makebox(1,1)[l]{$,$}}
\put(0,0){\makebox(2,1){$s\rightarrow$}}
\put(0,1.5){\makebox(2,1){$r\rightarrow$}}
\put(0,2.5){\makebox(2,1){$q\rightarrow$}}
\put(0,4){\makebox(2,1){$p\rightarrow$}}
\end{picture},
\end{center}

\begin{center}
\begin{picture}(13,5)
\put(3,0){\line(0,1){5}}
\put(6,0){\line(0,1){5}}
\put(9,0){\line(0,1){5}}
\put(12,0){\line(0,1){5}}

\put(2,1.5){\makebox(1,1){$\Bar{n}$}}
\put(2,2.5){\makebox(1,1){$n$}}
\put(2,4){\makebox(1,1){$a$}}
\put(3,0){\makebox(1,1){$\Bar{a}$}}

\put(5,1.5){\makebox(1,1){$0$}}
\put(5,2.5){\makebox(1,1){$n$}}
\put(5,4){\makebox(1,1){$a$}}
\put(6,0){\makebox(1,1){$\Bar{a}$}}

\put(8,1.5){\makebox(1,1){$0$}}
\put(8,2.5){\makebox(1,1){$0$}}
\put(8,4){\makebox(1,1){$a$}}
\put(9,0){\makebox(1,1){$\Bar{a}$}}

\put(11,1.5){\makebox(1,1){$\Bar{n}$}}
\put(11,2.5){\makebox(1,1){$0$}}
\put(11,4){\makebox(1,1){$a$}}
\put(12,0){\makebox(1,1){$\Bar{a}$}}

\put(4,0){\makebox(1,1)[l]{$,$}}
\put(7,0){\makebox(1,1)[l]{$,$}}
\put(10,0){\makebox(1,1)[l]{$,$}}
\put(0,0){\makebox(2,1){$s\rightarrow$}}
\put(0,1.5){\makebox(2,1){$r\rightarrow$}}
\put(0,2.5){\makebox(2,1){$q\rightarrow$}}
\put(0,4){\makebox(2,1){$p\rightarrow$}}
\end{picture},
\end{center}
then we have $(q-p)+(s-r)=s-p-1<n-a$.

\item[(B4)]
The tableau $T$ cannot have a pair of adjacent columns having one of the following configurations with $p<s$:
\setlength{\unitlength}{12pt}
\begin{center}
\begin{picture}(13,3)
\put(3,0){\line(0,1){3}}
\put(6,0){\line(0,1){3}}
\put(9,0){\line(0,1){3}}
\put(12,0){\line(0,1){3}}

\put(2,2){\makebox(1,1){$n$}}
\put(3,0){\makebox(1,1){$\Bar{n}$}}

\put(5,2){\makebox(1,1){$n$}}
\put(6,0){\makebox(1,1){$0$}}

\put(8,2){\makebox(1,1){$0$}}
\put(9,0){\makebox(1,1){$0$}}

\put(11,2){\makebox(1,1){$0$}}
\put(12,0){\makebox(1,1){$\Bar{n}$}}

\put(5,0){\makebox(1,1)[l]{$,$}}
\put(8,0){\makebox(1,1)[l]{$,$}}
\put(11,0){\makebox(1,1)[l]{$,$}}
\put(0,0){\makebox(2,1){$s\rightarrow$}}

\put(0,2){\makebox(2,1){$p\rightarrow$}}
\end{picture}.
\end{center}

\end{itemize}
\end{df}

We denote by $B_{n}\text{-}\mathrm{SST}_{\mathrm{KN}}(\lambda)$ the set of all KN-admissible semistandard $B_{n}$-tableau (without spin columns) of shape $\lambda$.

A crystal $\mathcal{B}^{\mathfrak{so}_{2n+1}}(\lambda)$ associated with the finite-dimensional irreducible $U_{q}(\mathfrak{so}_{2n+1})$-module $V_{q}^{\mathfrak{so}_{2n+1}}(\lambda)$ of a dominant integral weight $\Tilde{\lambda}$ is defined in the same way as in Section~\ref{sec:crystal_C}.
As a set, 
$\mathcal{B}^{\mathfrak{so}_{2n+1}}(\lambda)$ is $B_{n}\text{-}\mathrm{SST}_{\mathrm{KN}}(\lambda)$.
The crystal structure of 
$\mathcal{B}^{\mathfrak{so}_{2n+1}}(\lambda)$ is given by the crystal graph of 
$\mathcal{B}^{\mathfrak{so}_{2n+1}}(\square)$, the tensor product rule, and the far-eastern reading of $T\in\mathcal{B}^{\mathfrak{so}_{2n+1}}(\lambda)$.
The crystal graph of 
$\mathcal{B}^{\mathfrak{so}_{2n+1}}(\square)$ is given by as follows:
\setlength{\unitlength}{12pt}
\begin{center}
\begin{picture}(21,1.5)
\put(0,0){\framebox(1,1){$1$}}
\put(2.5,0){\framebox(1,1){$2$}}
\put(7.5,0){\framebox(1,1){$n$}}
\put(10,0){\framebox(1,1){$0$}}
\put(12.5,0){\framebox(1,1){$\Bar{n}$}}
\put(17.5,0){\framebox(1,1){$\Bar{2}$}}
\put(20,0){\framebox(1,1){$\Bar{1}$}}
\put(1,0.5){\vector(1,0){1.5}}
\put(3.5,0.5){\vector(1,0){1.5}}
\put(6,0.5){\vector(1,0){1.5}}
\put(8.5,0.5){\vector(1,0){1.5}}
\put(11,0.5){\vector(1,0){1.5}}
\put(13.5,0.5){\vector(1,0){1.5}}
\put(16,0.5){\vector(1,0){1.5}}
\put(18.5,0.5){\vector(1,0){1.5}}
\put(5,0){\makebox(1,1){$\cdots$}}
\put(15,0){\makebox(1,1){$\cdots$}}
\put(1,0.5){\makebox(1.5,1){$1$}}
\put(3.5,0.5){\makebox(1.5,1){$2$}}
\put(5.9,0.5){\makebox(1.5,1){$n-1$}}
\put(8.5,0.5){\makebox(1.5,1){$n$}}
\put(11,0.5){\makebox(1.5,1){$n$}}
\put(13.6,0.5){\makebox(1.5,1){$n-1$}}
\put(16,0.5){\makebox(1.5,1){$2$}}
\put(18.5,0.5){\makebox(1.5,1){$1$}}
\end{picture},
\end{center}
where  
$\mathrm{wt}\left(\framebox[12pt]{$i$\rule{0pt}{8pt}}\right)=\epsilon_{i}$, 
$\mathrm{wt}\left(\framebox[12pt]{$0$\rule{0pt}{8pt}}\right)=\epsilon_{i}$, 
and 
$\mathrm{wt}\left(\framebox[12pt]{$\Bar{i}$\rule{0pt}{8pt}}\right)=-\epsilon_{i}$
$(i=1,2,\ldots,n)$.
In $B_{n}$ case, Definition~\ref{df:addition} is still valid, but the following rule has to be added~\cite{N}.
For a Young diagram 
$\lambda=(\lambda_{1},\ldots,\lambda_{n})\in\mathcal{P}_{n}$,
\begin{equation} \label{eq:add0}
\lambda\lbrack0]:=
\begin{cases}
\lambda & (\lambda_{n}>0),\\
(\lambda_{1},\ldots,\lambda_{n-1},-\infty) & (\lambda_{n}=0).
\end{cases}
\end{equation}

The generalized LR rule of $B_{n}$-type is given by:
\begin{thm}[\cite{HK,KN,N}]
Let $\Tilde{\mu}=\sum_{i=1}^{n}\mu_{i}\epsilon_{i}$ and $\Tilde{\nu}=\sum_{i=1}^{n}\nu_{i}\epsilon_{i}$ be dominant integral weights, and $\mu=(\mu_{1},\mu_{2},\ldots,\mu_{n})$ and 
$\nu=(\nu_{1},\nu_{2},\ldots,\nu_{n})$ be the corresponding Young diagrams, respectively.
Then we have the following isomorphism:
\begin{equation} \label{eq:so1}
\mathcal{B}^{\mathfrak{so}_{2n+1}}(\mu)\otimes\mathcal{B}^{\mathfrak{so}_{2n+1}}(\nu)
\simeq
{\textstyle\bigoplus\limits_{\substack{T\in \mathcal{B}^{\mathfrak{so}_{2n+1}}(\nu)\\
\mathrm{FE}(T)=
\framebox[16pt]{$m_{1}$}\otimes
\cdots\otimes 
\framebox[16pt]{$m_{N}$}
}}}
\mathcal{B}^{\mathfrak{so}_{2n+1}}\left(  \mu\lbrack m_{1},m_{2},\ldots,m_{N}]\right),
\end{equation}
where $N=\left\vert \nu\right\vert$.
In the right-hand side of Eq.~\eqref{eq:so1}, we set 
$\mathcal{B}^{\mathfrak{so}_{2n+1}}\left(  \mu\lbrack m_{1},\ldots,m_{N}]\right)  =\emptyset$ 
if the sequence of letters $m_{1},\ldots,m_{N}$ is not smooth on $\mu$.
\end{thm}

Let us denote by $d_{\mu\nu}^{\lambda}$ the number of 
$\mathcal{B}^{\mathfrak{so}_{2n+1}}(\lambda)$ appearing in the right-hand side of Eq.~\eqref{eq:so1}.
Then the multiplicity $d_{\mu\nu}^{\lambda}$ is given by the cardinality of the following set:
\[
\mathbf{B}_{n}^{\mathfrak{so}_{2n+1}}(\nu)_{\mu}^{\lambda}:=
\left\{ T\in\mathcal{B}^{\mathfrak{so}_{2n+1}}(\nu) \relmiddle| 
\mu\left[ \underrightarrow{\mathrm{FE}(T)}\right]
=\lambda\right\}.
\]
In the stable region, i.e., $l(\mu)+l(\nu)\leq n$, a tableau 
$T\in \mathbf{B}_{n}^{\mathfrak{so}_{2n+1}}(\nu)_{\mu}^{\lambda}$ dose not contain zeros.
This is shown as follows.
We can assume that $l(\mu)=n-k$ and $l(\nu)\leq k$ ($k=1,2,\ldots,n-1$) so that $\mu_{n}=\nu_{n}=0$ and $\mu$ and $\nu$ (and therefore $\lambda$) do not contain spin columns.
Suppose that in the far-eastern reading of 
$T\in \mathbf{B}_{n}^{\mathfrak{so}_{2n+1}}(\nu)_{\mu}^{\lambda}$, $0$ appears firstly in the $i$-th box;
\[
\mathrm{FE}(T)=\framebox[15pt]{$m_{1}$\rule{0pt}{8pt}}\otimes \cdots 
\otimes \framebox[30pt]{$m_{i}=0$\rule{0pt}{8pt}}\otimes \cdots.
\] 
Since the sequence of letters $m_{1},\ldots,m_{i}=0$ is smooth on $\mu$,  
$l\left(  \mu\lbrack m_{1},\ldots,m_{i-1}]\right)  =n$.
Otherwise, $\mu\lbrack m_{1},\ldots,m_{i-1}][0]$ would not be a Young diagram by the rule of Eq.~\eqref{eq:add0}.
Hence, $k$ letters $n-k+1,\ldots,n$ must appear in the sequence of letters $m_{1},\ldots,m_{i-1}$ in this order.
This implies $l(\nu)\geq k+1$ because $k+1$ letters $n-k+1,\ldots,n,0$ in $T$ appear at different rows due to the semistandardness of $T$.
This contradicts the assumption that $l(\nu)\leq k$.
Thus, $T$ has no zeros.
Therefore, conditions (B1), (B2), and (B3) in Definition~\ref{df:KN_B} can be replaced by conditions (C1) and (C2) in Definition~\ref{df:KN_C} (with $\lambda$ being replaced by $\nu$) as long as tableaux in $\mathbf{B}_{n}^{\mathfrak{so}_{2n+1}}(\nu)_{\mu}^{\lambda}$ are considered in the stable region.
Condition (B4) in Definition~\ref{df:KN_B} is replaced by:

\begin{itemize}
\item[(B4')]
A tableau $T\in B_{n}\text{-}\mathrm{SST}_{\mathrm{KN}}(\nu)$ cannot have a pair of adjacent columns having the following configuration with $p<s$:
\setlength{\unitlength}{12pt}
\begin{center}
\begin{picture}(4,3)
\put(3,0){\line(0,1){3}}
\put(2,2){\makebox(1,1){$n$}}
\put(3,0){\makebox(1,1){$\Bar{n}$}}
\put(3,2){\makebox(1,1){$n$}}
\put(0,0){\makebox(2,1){$s\rightarrow$}}
\put(0,2){\makebox(2,1){$p\rightarrow$}}
\end{picture}.
\end{center}
\end{itemize}
This is contained in condition (C2) in Definition~\ref{df:KN_C} (with $\lambda$ being replaced by $\nu$).

Combining these, we obtain: 
\begin{thm} \label{thm:main_B}
Fix $\lambda, \mu, \nu \in \mathcal{P}_{n}$.
If $l(\mu)+l(\nu)\leq n$, then we have
$\mathbf{B}_{n}^{\mathfrak{so}_{2n+1}}(\nu)_{\mu}^{\lambda}=
\mathbf{B}_{n}^{\mathfrak{sp}_{2n}}(\nu)_{\mu}^{\lambda}$.
\end{thm}

\subsection{LR crystals of $D_{n}$-type}

The even special orthogonal Lie algebra $\mathfrak{so}(2n,\mathbb{C})=\mathfrak{so}_{2n}$ is the classical Lie algebra of $D_{n}$-type.
Using the standard unit vectors $\epsilon_{i}\in\mathbb{Z}^{n}$ ($i=1,2,\ldots,n$), the simple roots are expressed as
\begin{align*}
\alpha_{i} &  =\epsilon_{i}-\epsilon_{i+1}\quad (i=1,2,\ldots,n-1),\\
\alpha_{n} &  =\epsilon_{n-1}+\epsilon_{n},
\end{align*}
and the fundamental weights as 
\begin{align*}
\omega_{i} &  =\epsilon_{1}+\epsilon_{2}+\cdots+\epsilon_{i}\quad (i=1,2,\ldots,n-2), \\
\omega_{n-1} &  =\frac{1}{2}(\epsilon_{1}+\cdots+\epsilon_{n-1}-\epsilon_{n}), \\
\omega_{n} &  =\frac{1}{2}(\epsilon_{1}+\cdots+\epsilon_{n-1}+\epsilon_{n}). 
\end{align*}

Let 
$\Tilde{\lambda}=a_{1}\omega_{1}+\cdots+a_{n}\omega_{n}$ ($a_{i}\in\mathbb{Z}_{\geq0}$) be a dominant integral weight.
Then $\Tilde{\lambda}$ can be written as 
$\Tilde{\lambda}=\lambda_{1}\epsilon_{1}+\cdots+\lambda_{n}\epsilon_{n}$, where
\begin{align*}
\lambda_{1}  &  =a_{1}+a_{2}+\cdots+a_{n-2}+\frac{1}{2}(a_{n-1}+a_{n}),\\
\lambda_{2}  &  =a_{2}+\cdots+a_{n-2}+\frac{1}{2}(a_{n-1}+a_{n}),\\
&  \vdots\\
\lambda_{n-1}  &  =\frac{1}{2}(a_{n-1}+a_{n}),\\
\lambda_{n}  &  =\frac{1}{2}(a_{n}-a_{n-1}).
\end{align*}

Here we do not consider the spin representations for the finite-dimensional irreducible $U_{q}(\mathfrak{so}_{2n})$-modules $V_{q}^{\mathfrak{so}_{2n}}(\omega_{n-1})$ and $V_{q}^{\mathfrak{so}_{2n}}(\omega_{n})$ as in Section~\ref{sec:crystal_B} so that 
$\lambda_{n-1}, \left\vert \lambda_{n} \right\vert \in \mathbb{Z}_{\geq 0}$.
Hence we can associate a Young diagram 
$\lambda=(\lambda_{1},\ldots,\lambda_{n-1},\left\vert \lambda_{n}\right\vert)$ to $\Tilde{\lambda}$ and simplify the original definition of $D_{n}$-tableaux~\cite{N}.
Throughout this section, $D_{n}$-tableaux are referred to as $D_{n}$-tableaux without spin columns associated with the spin representations. 

\begin{df}[\cite{HK,N}]
\begin{itemize}
\item[(1)]
Let $\lambda$ be a Young diagram with at most $n$ rows.
A $D_{n}$-tableau of shape $\lambda$ is a tableau obtained by filling the boxes in $\lambda$ with entries from the set
\[
\{1,2,\ldots,n,\overline{n},\ldots,\overline{1}\}
\]
equipped with the linear order
\[
1\prec2\prec\cdots\prec n-1\prec
\begin{tabular}{c}
n \\
$\Bar{n}$
\end{tabular}
\prec\overline{n-1}\prec\cdots\prec
\overline{1},
\]
where the order between $n$ and $\Bar{n}$ is not defined.

\item[(2)]
A $D_{n}$-tableau is said to be semistandard if
\begin{itemize}
\item[(a)]
the entries in each rows are weakly increasing, and $n$ and $\Bar{n}$ do not appear simultaneously;
\item[(b)]
the entries in each column are strictly increasing, and $n$ and $\Bar{n}$ can appear successively.
\end{itemize}
\end{itemize}
\end{df}

For a $D_{n}$-tableau $T$, we write
\setlength{\unitlength}{12pt}
\begin{center}
\begin{picture}(8,4)
\put(2,0){\line(0,1){4}}
\put(4,0){\line(0,1){4}}
\put(5,1){\line(0,1){1}}
\put(7,2){\line(0,1){1}}
\put(8,3){\line(0,1){1}}
\put(2,0){\line(1,0){2}}
\put(4,1){\line(1,0){1}}
\put(5,2){\line(1,0){2}}
\put(7,3){\line(1,0){1}}
\put(2,4){\line(1,0){6}}
\put(0,1,5){\makebox(2,1){$T=$}}
\put(2,1){\makebox(2,2){$T^{\pm}$}}
\put(5,2.5){\makebox(1,1){$T^{0}$}}
\end{picture},
\end{center}
where $T^{\pm}=T^{+}$ if $a_{n}\leq a_{n-1}$, $T^{\pm}=T^{-}$ if $a_{n}\geq a_{n-1}$, 
$l(T^{\pm})=n$ and $l(T^{0})\leq n-1$.
We denote by $D_{n}\text{-}\mathrm{SST}(\lambda)$ the set of all semistandard $D_{n}$-tableaux of shape $\lambda$.
For a tableau $T\in D_{n}\text{-}\mathrm{SST}(\lambda)$, we define its weight to be 
\[
\mathrm{wt}(T):=\sum_{i=1}^{n}(k_{i}-\overline{k_{i}})\epsilon_{i},
\]
where $k_{i}$ (resp. $\overline{k_{i}}$) is the number of $i$'s (resp. $\Bar{i}$'s) appearing in $T$.

\begin{df}[\cite{HK,N}] \label{df:KN_D}
A tableau $T\in D_{n}\text{-}\mathrm{SST}(\lambda)$ is said to be \emph{KN-admissible} when the following conditions are satisfied.
\begin{itemize}

\item[(D1)]
If $T$ has a column of the form
\setlength{\unitlength}{12pt}
\begin{center}
\begin{picture}(3,6)
\put(2,0){\line(0,1){6}}
\put(3,0){\line(0,1){6}}
\put(2,0){\line(1,0){1}}
\put(2,1){\line(1,0){1}}
\put(2,2){\line(1,0){1}}
\put(2,4){\line(1,0){1}}
\put(2,5){\line(1,0){1}}
\put(2,6){\line(1,0){1}}
\put(2,1){\makebox(1,1){$\Bar{i}$}}
\put(2,4){\makebox(1,1){$i$}}
\put(0,1){\makebox(2,1){$q\rightarrow$}}
\put(0,4){\makebox(2,1){$p\rightarrow$}}
\end{picture},
\end{center}
then we have $(q-p)+i > N$, where $N$ is the length of the column.

\item[(D2)]
If $T^{+}$ has a column whose $k$-th entry is $n$ (resp. $\Bar{n}$), then $n-k$ is even (resp. odd）.

\item[(D3)]
If $T^{-}$ has a column whose $k$-th entry is $n$ (resp. $\Bar{n}$), then $n-k$ is odd (resp. even）.

\item[(D4)]
If $T$ has a pair of adjacent columns having one of the following configurations with $p\leq q < r\leq s$ and $a\leq b<n$:
\setlength{\unitlength}{12pt}
\begin{center}
\begin{picture}(7,5)
\put(3,0){\line(0,1){5}}
\put(6,0){\line(0,1){5}}
\put(3,0){\makebox(1,1){$\Bar{a}$}}
\put(3,1){\makebox(1,1){$\Bar{b}$}}
\put(3,3){\makebox(1,1){$b$}}
\put(2,4){\makebox(1,1){$a$}}
\put(6,0){\makebox(1,1){$\Bar{a}$}}
\put(5,1){\makebox(1,1){$\Bar{b}$}}
\put(5,3){\makebox(1,1){$b$}}
\put(5,4){\makebox(1,1){$a$}}
\put(4,0){\makebox(1,1)[l]{$,$}}
\put(0,0){\makebox(2,1){$s\rightarrow$}}
\put(0,1){\makebox(2,1){$r\rightarrow$}}
\put(0,3){\makebox(2,1){$q\rightarrow$}}
\put(0,4){\makebox(2,1){$p\rightarrow$}}
\end{picture},
\end{center}
then we have $(q-p)+(s-r)<b-a$.

\item[(D5)]
If $T$ has a pair of adjacent columns having one of the following configurations with $p\leq q<r=q+1\leq s$ and $a<n$:
\setlength{\unitlength}{12pt}
\begin{center}
\begin{picture}(13,5)
\put(3,0){\line(0,1){5}}
\put(6,0){\line(0,1){5}}
\put(9,0){\line(0,1){5}}
\put(12,0){\line(0,1){5}}
\put(2,4){\makebox(1,1){$a$}}
\put(3,0){\makebox(1,1){$\Bar{a}$}}
\put(3,1.5){\makebox(1,1){$\Bar{n}$}}
\put(3,2.5){\makebox(1,1){$n$}}

\put(5,4){\makebox(1,1){$a$}}
\put(6,0){\makebox(1,1){$\Bar{a}$}}
\put(6,1.5){\makebox(1,1){$n$}}
\put(6,2.5){\makebox(1,1){$\Bar{n}$}}

\put(8,1.5){\makebox(1,1){$\Bar{n}$}}
\put(8,2.5){\makebox(1,1){$n$}}
\put(8,4){\makebox(1,1){$a$}}
\put(9,0){\makebox(1,1){$\Bar{a}$}}

\put(11,1.5){\makebox(1,1){$n$}}
\put(11,2.5){\makebox(1,1){$\Bar{n}$}}
\put(11,4){\makebox(1,1){$a$}}
\put(12,0){\makebox(1,1){$\Bar{a}$}}

\put(4,0){\makebox(1,1)[l]{$,$}}
\put(7,0){\makebox(1,1)[l]{$,$}}
\put(10,0){\makebox(1,1)[l]{$,$}}
\put(0,0){\makebox(2,1){$s\rightarrow$}}
\put(0,1.5){\makebox(2,1){$r\rightarrow$}}
\put(0,2.5){\makebox(2,1){$q\rightarrow$}}
\put(0,4){\makebox(2,1){$p\rightarrow$}}
\end{picture},
\end{center}
then we have $(q-p)+(s-r)=s-p-1<n-a$.

\item[(D6)]
The tableau $T$ cannot have a pair of adjacent columns having one of the following configurations with $p<s$:
\setlength{\unitlength}{12pt}
\begin{center}
\begin{picture}(13,3)
\put(3,0){\line(0,1){3}}
\put(6,0){\line(0,1){3}}
\put(9,0){\line(0,1){3}}
\put(12,0){\line(0,1){3}}

\put(2,2){\makebox(1,1){$n$}}
\put(3,0){\makebox(1,1){$n$}}

\put(5,2){\makebox(1,1){$n$}}
\put(6,0){\makebox(1,1){$\Bar{n}$}}

\put(8,2){\makebox(1,1){$\Bar{n}$}}
\put(9,0){\makebox(1,1){$n$}}

\put(11,2){\makebox(1,1){$\Bar{n}$}}
\put(12,0){\makebox(1,1){$\Bar{n}$}}

\put(4,0){\makebox(1,1)[l]{$,$}}
\put(7,0){\makebox(1,1)[l]{$,$}}
\put(10,0){\makebox(1,1)[l]{$,$}}
\put(0,0){\makebox(2,1){$s\rightarrow$}}
\put(0,2){\makebox(2,1){$p\rightarrow$}}
\end{picture}.
\end{center}

\item[(D7)]
If $T$ has a pair of adjacent columns having one of the following configurations with $p\leq q<r\leq s$ and $a<n$;
\setlength{\unitlength}{12pt}
\begin{center}
\begin{picture}(15,6)
\put(4,1){\line(0,1){5}}
\put(7,1){\line(0,1){5}}
\put(10,1){\line(0,1){5}}
\put(13,1){\line(0,1){5}}

\put(3,2){\makebox(1,1){$\Bar{n}$}}
\put(3,5){\makebox(1,1){$a$}}
\put(4,1){\makebox(1,1){$\Bar{a}$}}
\put(4,4){\makebox(1,1){$n$}}

\put(6,2){\makebox(1,1){$n$}}
\put(6,5){\makebox(1,1){$a$}}
\put(7,1){\makebox(1,1){$\Bar{a}$}}
\put(7,4){\makebox(1,1){$\Bar{n}$}}

\put(9,2){\makebox(1,1){$n$}}
\put(9,5){\makebox(1,1){$a$}}
\put(10,1){\makebox(1,1){$\Bar{a}$}}
\put(10,4){\makebox(1,1){$n$}}

\put(12,2){\makebox(1,1){$\Bar{n}$}}
\put(12,5){\makebox(1,1){$a$}}
\put(13,1){\makebox(1,1){$\Bar{a}$}}
\put(13,4){\makebox(1,1){$\Bar{n}$}}

\put(5,1){\makebox(1,1)[l]{$,$}}
\put(11,1){\makebox(1,1)[l]{$,$}}
\put(3,0){\makebox(5,1){$r-q+1=\mathrm{odd},$}}
\put(10,0){\makebox(5,1){$r-q+1=\mathrm{even},$}}

\put(0,1){\makebox(2,1){$s\rightarrow$}}
\put(0,2){\makebox(2,1){$r\rightarrow$}}
\put(0,4){\makebox(2,1){$q\rightarrow$}}
\put(0,5){\makebox(2,1){$p\rightarrow$}}
\end{picture}
\end{center}
then we have $s-p<n-a$.

\end{itemize}
\end{df}

We denote by $D_{n}\text{-}\mathrm{SST}_{\mathrm{KN}}(\lambda)$ the set of all KN-admissible semistandard $D_{n}$-tableau (without spin columns) of shape $\lambda$.

A crystal $\mathcal{B}^{\mathfrak{so}_{2n}}(\lambda)$ associated with the finite-dimensional irreducible $U_{q}(\mathfrak{so}_{2n})$-module $V_{q}^{\mathfrak{so}_{2n}}(\Tilde{\lambda})$ of a dominant integral weight $\Tilde{\lambda}$ is defined in the same way as in Section~\ref{sec:crystal_C}.
As a set, 
$\mathcal{B}^{\mathfrak{so}_{2n}}(\lambda)$ is $B_{n}\text{-}\mathrm{SST}_{\mathrm{KN}}(\lambda)$.
The crystal structure of 
$\mathcal{B}^{\mathfrak{so}_{2n}}(\lambda)$ is given by the crystal graph of 
$\mathcal{B}^{\mathfrak{so}_{2n}}(\square)$, the tensor product rule, and the far-eastern reading of $T\in\mathcal{B}^{\mathfrak{so}_{2n}}(\lambda)$.
The crystal graph of 
$\mathcal{B}^{\mathfrak{so}_{2n}}(\square)$ is given by as follows:

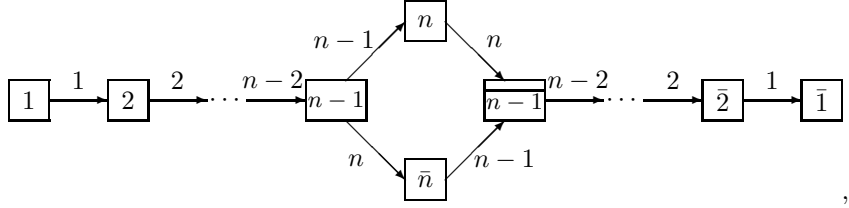
\setlength{\unitlength}{15pt}
\begin{center}
\begin{picture}(21,5)
\put(0,2){\framebox(1,1){$1$}}
\put(2.5,2){\framebox(1,1){$2$}}
\put(7.5,2){\framebox(1.5,1){\small $n-1$}}
\put(10,0){\framebox(1,1){$\Bar{n}$}}
\put(10,4){\framebox(1,1){$n$}}
\put(12,2){\framebox(1.5,1){\small $\overline{n-1}$}}
\put(17.5,2){\framebox(1,1){$\Bar{2}$}}
\put(20,2){\framebox(1,1){$\Bar{1}$}}

\put(1,2.5){\vector(1,0){1.5}}
\put(3.5,2.5){\vector(1,0){1.5}}
\put(6,2.5){\vector(1,0){1.5}}

\put(8.5,2){\vector(1,-1){1.5}}
\put(8.5,3){\vector(1,1){1.5}}
\put(11,0.5){\vector(1,1){1.5}}
\put(11,4.5){\vector(1,-1){1.5}}

\put(13.5,2.5){\vector(1,0){1.5}}
\put(16,2.5){\vector(1,0){1.5}}
\put(18.5,2.5){\vector(1,0){1.5}}

\put(5,2){\makebox(1,1){$\cdots$}}
\put(15,2){\makebox(1,1){$\cdots$}}
\put(1,2.5){\makebox(1.5,1){$1$}}
\put(3.5,2.5){\makebox(1.5,1){$2$}}
\put(5.9,2.5){\makebox(1.5,1){$n-2$}}
\put(8,0.5){\makebox(1.5,1){$n$}}
\put(7.7,3.5){\makebox(1.5,1){$n-1$}}
\put(11.75,0.5){\makebox(1.5,1){$n-1$}}
\put(11.5,3.5){\makebox(1.5,1){$n$}}
\put(13.6,2.5){\makebox(1.5,1){$n-2$}}
\put(16,2.5){\makebox(1.5,1){$2$}}
\put(18.5,2.5){\makebox(1.5,1){$1$}}
\end{picture},
\end{center}
where 
$\mathrm{wt}\left(\framebox[12pt]{$i$\rule{0pt}{8pt}}\right)=\epsilon_{i}$ 
and 
$\mathrm{wt}\left(\framebox[12pt]{$\Bar{i}$\rule{0pt}{8pt}}\right)=-\epsilon_{i}$
$(i=1,2,\ldots,n)$.

Even in $D_{n}$ case, Definition~\ref{df:addition} is valid and the generalized LR rule of $D_{n}$-type is given by:
\begin{thm}[\cite{HK,KN,N}]
Let $\Tilde{\mu}=\sum_{i=1}^{n}\mu_{i}\epsilon_{i}$ and $\Tilde{\nu}=\sum_{i=1}^{n}\nu_{i}\epsilon_{i}$ be dominant integral weights, and 
$\mu=(\mu_{1},\ldots,\mu_{n-1},\left\vert\mu_{n}\right\vert)$ and 
$\nu=(\nu_{1},\ldots,\nu_{n-1},\left\vert\nu_{n}\right\vert)$ be the corresponding Young diagrams, respectively.
Then we have the following isomorphism:
\begin{equation} \label{eq:so2}
\mathcal{B}^{\mathfrak{so}_{2n}}(\mu)\otimes\mathcal{B}^{\mathfrak{so}_{2n}}(\nu)
\simeq
{\textstyle\bigoplus\limits_{\substack{T\in \mathcal{B}^{\mathfrak{so}_{2n}}(\nu)\\
\mathrm{FE}(T)=
\framebox[16pt]{$m_{1}$}\otimes
\cdots\otimes 
\framebox[16pt]{$m_{N}$}
}}}
\mathcal{B}^{\mathfrak{so}_{2n}}\left(  \mu\lbrack m_{1},m_{2},\ldots,m_{N}]\right),
\end{equation}
where $N=\left\vert \nu\right\vert$.
In the right-hand side of Eq.~\eqref{eq:so2}, we set 
$\mathcal{B}^{\mathfrak{so}_{2n}}\left(  \mu\lbrack m_{1},\ldots,m_{N}]\right)  =\emptyset$ 
if the sequence of letters $m_{1},\ldots,m_{N}$ is not smooth on $\mu$.
\end{thm}

Let us denote by $d_{\mu\nu}^{\lambda}$ the number of 
$\mathcal{B}^{\mathfrak{so}_{2n}}(\lambda)$ appearing in the right-hand side of Eq~\eqref{eq:so2}.
Then the multiplicity $d_{\mu\nu}^{\lambda}$ is given by the cardinality of the following set:
\[
\mathbf{B}_{n}^{\mathfrak{so}_{2n}}(\nu)_{\mu}^{\lambda}:=
\left\{ T\in\mathcal{B}^{\mathfrak{so}_{2n}}(\nu) \relmiddle| 
\mu\left[ \underrightarrow{\mathrm{FE}(T)}\right]
=\lambda\right\}.
\]
Suppose that the far-eastern reading of $T\in \mathbf{B}_{n}^{\mathfrak{so}_{2n}}(\nu)_{\mu}^{\lambda}$ is
\[
\mathrm{FE}(T)=\framebox[15pt]{$m_{1}$\rule{0pt}{8pt}}\otimes \framebox[15pt]{$m_{2}$\rule{0pt}{8pt}}\cdots 
\otimes \framebox[15pt]{$m_{N}$\rule{0pt}{8pt}}.
\] 
If $l(\mu)+l(\nu)\leq n$, then the following additional rule is imposed on the sequence of entries, 
$m_{1},m_{2},\ldots,m_{n}$, in order to guarantee the smoothness of $\mathrm{FE}(T)$ on $\mu$:
To each $n$ (resp. $\Bar{n}$) in this sequence, we assign $+$ (resp. $-$) and cancel out all $(+,-)$-pairs.
Then, the resulting sequence must not have $-$'s.

To verify this rule, it is sufficient to show that \framebox[10pt]{$n$} must appear before \framebox[10pt]{$\Bar{n}$} in $\mathrm{FE}(T)$ (if $\Bar{n}$'s exist in $T$). 
This is shown as follows.
If $l(\mu)=n$, then $\nu=\emptyset$.
Excluding this trivial case, we can assume that $l(\mu)\leq n-1$.
Suppose that in the far-eastern reading of 
$T\in \mathbf{B}_{n}^{\mathfrak{so}_{2n}}(\nu)_{\mu}^{\lambda}$, $\Bar{n}$ appears firstly in the $i$-th box;
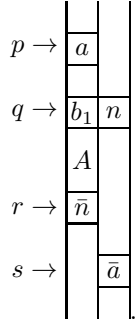
\begin{equation} \label{eq:FEso2}
\mathrm{FE}(T)=\framebox[15pt]{$m_{1}$\rule{0pt}{8pt}}\otimes \cdots 
\otimes \framebox[31pt]{$m_{i}=\Bar{n}$\rule{0pt}{8pt}}\otimes \cdots.
\end{equation} 
Since the sequence of letters $m_{1},\ldots,m_{i}=\Bar{n}$ is smooth on $\mu$,  
$l\left(  \mu\lbrack m_{1},\ldots,m_{i-1}]\right)  =n$.
However, this cannot occur because the sequence of letters $m_{1},\ldots,m_{i-1}$ does not contain $n$ and $l(\mu)\leq n-1$.
The same is true for $\mathbf{B}_{n}^{\mathfrak{sp}_{2n}}(\nu)_{\mu}^{\lambda}$ in the stable region, $l(\mu)+l(\nu)\leq n$.
In particular, a tableau $T\in \mathbf{B}_{n}^{\mathfrak{so}_{2n}}(\nu)_{\mu}^{\lambda}$ does not have vertical dominoes 
{\setlength{\tabcolsep}{3pt}\footnotesize
\begin{tabular}{|c|} \hline
$\Bar{n}$ \\ \hline
$n$ \\ \hline
\end{tabular}}.
Thus, conditions (D1), (D2), (D3), (D4), and (D5) in Definition~\ref{df:KN_D} can be replaced by conditions (C1) and (C2) in Definition~\ref{df:KN_C} (with $\lambda$ being replaced by $\nu$) as long as tableaux in $\mathbf{B}_{n}^{sp_{2n}}(\nu)_{\mu}^{\lambda}$ are considered in the stable region.
Condition (D6) in Definition~\ref{df:KN_D} is replaced by:

\begin{itemize}
\item[(D6')]
A tableau $T\in D_{n}\text{-}\mathrm{SST}_{\mathrm{KN}}(\nu)$ cannot have a pair of adjacent columns having the following configurations with $p<s$:
\setlength{\unitlength}{12pt}
\begin{center}
\begin{picture}(4,3)
\put(3,0){\line(0,1){3}}

\put(2,2){\makebox(1,1){$n$}}
\put(3,0){\makebox(1,1){$\Bar{n}$}}
\put(3,2){\makebox(1,1){$n$}}

\put(0,0){\makebox(2,1){$s\rightarrow$}}
\put(0,2){\makebox(2,1){$p\rightarrow$}}
\end{picture}.
\end{center}
\end{itemize}
This is contained in (C2) in Definition~\ref{df:KN_C} (with $\lambda$ being replaced by $\nu$).
Condition (D7) in Definition~\ref{df:KN_D} is replaced by:

\begin{itemize}
\item[(D7')]
If $T\in D_{n}\text{-}\mathrm{SST}_{\mathrm{KN}}(\nu)$ has a pair of adjacent columns having one of the following configurations with $p\leq q<r\leq s$ and $a<n$;
\setlength{\unitlength}{12pt}
\begin{center}
\begin{picture}(15,6)
\put(4,1){\line(0,1){5}}
\put(11,1){\line(0,1){5}}

\put(3,2){\makebox(1,1){$\Bar{n}$}}
\put(3,5){\makebox(1,1){$a$}}
\put(4,1){\makebox(1,1){$\Bar{a}$}}
\put(4,4){\makebox(1,1){$n$}}

\put(10,2){\makebox(1,1){$n$}}
\put(10,5){\makebox(1,1){$a$}}
\put(11,1){\makebox(1,1){$\Bar{a}$}}
\put(11,4){\makebox(1,1){$n$}}

\put(2,0){\makebox(5,1){$r-q+1=\mathrm{odd},$}}
\put(10,0){\makebox(5,1){$r-q+1=\mathrm{even},$}}

\put(0,1){\makebox(2,1){$s\rightarrow$}}
\put(0,2){\makebox(2,1){$r\rightarrow$}}
\put(0,4){\makebox(2,1){$q\rightarrow$}}
\put(0,5){\makebox(2,1){$p\rightarrow$}}

\put(6,5){\makebox(1,1){$(A)$}}
\put(13,5){\makebox(1,1){$(B)$}}
\end{picture}
\end{center}
then we have $s-p<n-a$.

\end{itemize}
This is due to the fact that $\mathrm{FE}(T)$ of Eq.\eqref{eq:FEso2} is not allowed.

Suppose that $T\in D_{n}\text{-}\mathrm{SST}_{\mathrm{KN}}(\nu)$ has configuration $(A)$ above.

\setlength{\unitlength}{12pt}
\begin{center}
\begin{picture}(4,10)
\put(2,0){\line(0,1){10}}
\put(3,0){\line(0,1){10}}
\put(4,0){\line(0,1){10}}

\put(2,3){\line(1,0){1}}
\put(2,4){\line(1,0){1}}
\put(2,6){\line(1,0){2}}
\put(2,7){\line(1,0){2}}
\put(2,8){\line(1,0){1}}
\put(2,9){\line(1,0){1}}
\put(3,1){\line(1,0){1}}
\put(3,2){\line(1,0){1}}

\put(2,3){\makebox(1,1){$\Bar{n}$}}
\put(2,4){\makebox(1,2){$A$}}
\put(2,6){\makebox(1,1){$b_{1}$}}
\put(2,8){\makebox(1,1){$a$}}
\put(3,1){\makebox(1,1){$\Bar{a}$}}
\put(3,6){\makebox(1,1){$n$}}

\put(0,1){\makebox(2,1){$s\rightarrow$}}
\put(0,3){\makebox(2,1){$r\rightarrow$}}
\put(0,6){\makebox(2,1){$q\rightarrow$}}
\put(0,8){\makebox(2,1){$p\rightarrow$}}

\end{picture}.
\end{center}
Since $r-q+1$ is odd, $A$ has at least one box.
Let $b_{2}$ be the entry at the $(q+1)$-st position in the left column ($a\leq b_{1}<b_{2}\leq n$).
Then,
\[
(q-p)+(s-r)<s-p<n-a=\max(b_{1},n)-a,
\] 
and
\[
(q+1-p)+(s-r) \leq s-p<n-a=\max(b_{2},n)-a.
\]
Thus, the condition for the right configuration of (C2) in Definition~\ref{df:KN_C} is satisfied irrespective of whether $q-p$ is odd or even.
Similarly, the configuration (B) leads to the condition for the left configuration of (C2) in Definition~\ref{df:KN_C}

Combining these, we obtain:
\begin{thm} \label{thm:main_D}
Fix $\lambda, \mu, \nu \in \mathcal{P}_{n}$.
If $l(\mu)+l(\nu)\leq n$, then we have
$\mathbf{B}_{n}^{\mathfrak{so}_{2n}}(\nu)_{\mu}^{\lambda}=
\mathbf{B}_{n}^{\mathfrak{sp}_{2n}}(\nu)_{\mu}^{\lambda}$.
\end{thm}

\subsection*{Acknowledgements}
The author would like to express his gratitude to Professor Susumu Ariki and Professor Ryoichi Kase for helping him in writing this paper and warm encouragement.
He would like to thank Professor Satoshi Naito for helpful discussions and comments.
He also would like to thank Seoul National University for kind hospitality and support during his visit in April 2017.
The results of Section~\ref{sec:main2} were suggested by Professor Jae-Hoon Kwon and were proven in the course of our discussions.


\begin{thebibliography}{99}

\bibitem{DeC}
C. De Concini,
Symplectic standard tableaux,
\textit{Adv. in Math.} \textbf{34} (1979), 1--27.

\bibitem{F}
W. Fulton,
``Young tableaux: with applications to representation theory,''
Cambridge Univ. Press, (1997).

\bibitem{HK}
J. Hong and S.-J. Kang,
``Introduction to quantum groups and crystal bases,''
Graduate Studies in Mathematics \textbf{42}, Amer. Math. Soc., (2002).

\bibitem{Kas1}
M. Kashiwara,
Crystalizing the $q$--analogue of universal enveloping algebras,
\textit{Comm. Math. Phys.} \textbf{133} (1990), 249--260.

\bibitem{Kas2}
M. Kashiwara,
On crystal bases of the $q$--analogue of universal enveloping algebras,
\textit{Duke Math. J.} \textbf{63} (1991), 465--516.

\bibitem{KN}
M. Kashiwara and T. Nakashima,
Crystal graphs for representations of the $q$--analogue of classical Lie algebras,
\textit{J. Algebra} \textbf{165} (1994), 295-345.

\bibitem{Kin}
R. C. King,
Modification rules and products of irreducible representations of the unitary, orthogonal, and symplectic groups,
\textit{J. Math. Phys.} \textbf{12} (1971), 1588--1598.

\bibitem{Koi}
K. Koike,
On the decomposition of tensor products of the representations of the classical groups: by means of the universal characters,
\textit{Adv. in Math.} \textbf{74} (1989), 57--86.

\bibitem{Kwo1}
J.-H. Kwon,
Super duality and crystal bases for quantum orthosymplectic superalgebras,
\textit{Int. Math. Res. Not.} \textbf{23} (2015), 12620--12677.

\bibitem{Kwo2}
J.-H. Kwon,
Super duality and crystal bases for quantum orthosymplectic superalgebras II,
\textit{J. Algebraic Combin.} \textbf{43} (2016), 553--588.

\bibitem{Kwo3}
J.-H. Kwon,
Combinatorial extension of stable branching rules for classical groups,
arXiv:math/1512.01877v4.

\bibitem{Lec}
C. Lecouvey,
Schensted-type correspondence, plactic monoid, and jeu de taquin for type $C_{n}$,
\textit{J. Algebra} \textbf{247} (2002), 295--331.

\bibitem{N}
T. Nakashima,
Crystal base and a generalization of the Littlewod-Richardson rule for the classical Lie algebras,
\textit{Comm. Math. Phys.} \textbf{154} (1993) 215--243.

\bibitem{She}
J. T. Sheats,
A symplectic jeu de taquin bijection between the tableaux of King and of De Concini,
\textit{Trans. Amer. Math. Soc.} \textbf{351} (1999), 3569--3607.

\end{thebibliography}
\end{document}